\documentclass[11pt, oneside, reqno]{amsart}
\usepackage{geometry}
\usepackage{amssymb, amsthm,amsmath,amscd,bbm}
\usepackage[usenames, dvipsnames]{color}
\usepackage[pdfborder={0 0 0}, colorlinks = true, linkcolor=blue, citecolor=OliveGreen]{hyperref}
\usepackage{cleveref} 
\usepackage{enumerate}
\usepackage{array,multirow}
\usepackage{bm}
\usepackage{mathrsfs}
\usepackage{stmaryrd}
\usepackage{graphicx}
\usepackage{todonotes} 
\usepackage{mathtools}

\mathtoolsset{centercolon} %fixes vertical spacing in := 

\makeatletter
  \def\l@subsection{\@tocline{2}{0pt}{3pc}{6pc}{}} 
  \makeatother

\usepackage{mdwlist}

\theoremstyle{plain}
\newtheorem{theorem}[subsubsection]{Theorem}
\newtheorem{proposition}[subsubsection]{Proposition}
\newtheorem{lemma}[subsubsection]{Lemma}
\newtheorem{definition}[subsubsection]{Definition}
\newtheorem{corollary}[subsubsection]{Corollary}

\theoremstyle{remark}
\newtheorem{remark}[subsubsection]{Remark}
\newtheorem{example}[subsubsection]{Example}
\numberwithin{equation}{subsection}

\newcommand{\bbA}{\mathbb A}
\newcommand{\bbC}{\mathbb{C}}
\newcommand{\bbD}{\mathbb{D}}
\newcommand{\bbH}{\mathbb{H}}
\newcommand{\bbK}{\mathbb{K}}
\newcommand{\bbR}{\mathbb{R}}
\newcommand{\bbP}{\mathbb P}
\newcommand{\bbQ}{\mathbb Q}
\newcommand{\bbV}{\mathbb{V}}
\newcommand{\bbZ}{\mathbb{Z}}

\newcommand{\bfH}{\mathbf{H}}

\newcommand{\bfG}{\mathbf G}
\newcommand{\bfk}{\mathbf{k}}
\newcommand{\bfK}{\mathbf{K}}

\newcommand{\bfv}{\mathbf v}

\newcommand{\bfx}{\mathbf{x}}
\newcommand{\bfy}{\mathbf {y}}

\newcommand{\bfalpha}{\pmb \alpha}
\newcommand{\bftau}{\pmb \tau}

\newcommand{\sfa}{\mathsf a}

\newcommand{\sfN}{\mathsf N}
\newcommand{\sfM}{\mathsf M}
\newcommand{\sfp}{\mathsf p}

\newcommand{\ttd}{\mathtt d}

\newcommand{\calE}{\mathcal E}

\newcommand{\calG}{\mathcal G}

\newcommand{\calM}{\mathcal M}

\newcommand{\calO}{\mathcal{O}}
\newcommand{\calP}{\mathcal{P}}

\newcommand{\calX}{\mathcal X}
\newcommand{\calZ}{\mathcal Z}

\newcommand{\scrC}{\mathscr C}
\newcommand{\rk}{\mathrm{rk}}

\newcommand{\ddc}{\mathrm{d d^c}}
\newcommand{\KM}{\mathrm{KM}}
\newcommand{\Id}{\mathrm{Id}}
\newcommand{\tr}{\mathrm{tr}}

\newcommand{\Sym}{\mathrm{Sym}}
\newcommand{\GL}{\mathrm{GL}}
\newcommand{\SL}{\mathrm{SL}}
\newcommand{\GSpin}{\mathrm{GSpin}}

\newcommand{\Mp}{\mathrm{Mp}}
\newcommand{\Res}{\mathrm{Res}}

\newcommand{\Herm}{\mathrm{Her}}

\newcommand{\Mat}{\mathrm{Mat}}

\newcommand{\archW}{\mathcal W}

\newcommand{\lbT}{S}

\newcommand{\Eis}{E}

\newcommand{\normEis}{\mathcal E}

\newcommand{\vbE}{\mathcal E}
\newcommand{\transpose}[1]{{}^{t} #1}

\newcommand{\bigcdot}{\raisebox{-0.20ex}{\scalebox{1.4}{$\cdot$}}}

\newcommand{\ChowHat}[1]{ \widehat{\mathrm{CH}}{}^{#1}} 
\newcommand{\ChowFin}[1]{\mathrm{CH}_{\mathrm{fin}} }
\DeclareMathOperator{\str}{\mathrm{tr_s}}

\DeclareMathOperator{\rank}{\mathrm{rk}}

\newcommand{\nuo}{\nu^{\mathtt o}}
\newcommand{\tildenuo}{\widetilde{\nu}{}^{\mathtt o}}
\newcommand{\xio}{ \xi^{\mathtt o}}
\newcommand{\phio}{\varphi^{\mathtt o}}
\newcommand{\archGreen}{ \lie g^\mathtt{o}}

\newcommand{\lie}[1]{\mathfrak{#1}}
\newcommand{\sm}[1]{\begin{smallmatrix}  #1 \end{smallmatrix}}
\newcommand{\ul}[1]{\underline{#1} }

\title{Green forms and the arithmetic Siegel-Weil formula}
\author{Luis E. Garcia and Siddarth Sankaran}
\begin{document}
\begin{abstract}
We construct natural Green forms for special cycles in orthogonal and unitary Shimura varieties, in all codimensions, and, for compact Shimura varieties of type $\mathrm{O}(p,2)$ and $\mathrm{U}(p,1)$, we show that the resulting local archimedean height pairings are related to special values of derivatives of Siegel Eisentein series. A conjecture put forward by Kudla relates these derivatives to arithmetic intersections of special cycles, and our results  settle the part of his conjecture involving local archimedean heights. 
\end{abstract}

\maketitle
\setcounter{tocdepth}{2}
\tableofcontents

\section{Introduction}

The Arakelov theory of Shimura varieties has been intensively studied since, about twenty years ago, Kudla \cite{KudlaCD} launched a program relating families of special cycles in their arithmetic Chow groups with derivatives of Eisenstein series and Rankin-Selberg L-functions. 

This paper is a contribution to the archimedean aspects of this theory. Building upon previous work of the first author \cite{GarciaSuperconnections}, we use Quillen's formalism of superconnections \cite{QuillenChern} as developed by Bismut-Gillet-Soul\'e \cite{BismutGilletSoule1,BismutGilletSouleDuke90} to construct natural Green forms for special cycles, in all codimensions, on orthogonal and unitary Shimura varieties. We show that these forms have good functorial properties and are compatible with star products. Furthermore, for compact Shimura varieties of type $\GSpin(p,2)$ or $\mathrm{U}(p,1)$, we compute the corresponding local archimedean heights of special cyles explicitly.

These results provide compelling evidence for Kudla's conjectural identity, termed the \emph{arithmetic Siegel-Weil formula}, between special derivatives of Eisenstein series and generating series of arithmetic heights of special cycles. More precisely, we show that the non-holomorphic terms on both sides are equal for these Shimura varieties.

Our methods combine Quillen's extension of Chern-Weil theory with the theory of the theta correspondence. This allows us to use representation theoretic arguments and the Siegel-Weil formula when computing archimedean local heights. In this way, we avoid the highly involved computations that feature in prior work, and give a conceptual explanation for the equality of non-holomorphic terms in Kudla's conjectural identities.

\subsection{Main results}
The remainder of the introduction outlines our results in more detail. Throughout the paper, we treat $\GSpin(p,2)$ and $\mathrm{U}(p,q)$ Shimura varieties in parallel; these are referred to as the \emph{orthogonal} and \emph{unitary} cases, respectively.

 Let $F$ denote a totally real field of degree $[F:\bbQ] = d$ and let $E$ be a CM extension of $F$ equipped with a fixed CM type. Suppose that $\bbV$ is a quadratic space over $F$ in the orthogonal case (resp.\ a Hermitian space over $E$ in the unitary case). We assume that there is one archimedean place $\sigma_1$ such that $\bbV_{\sigma_1}$ satisfies the signature condition 
 \begin{equation} \label{eqn:intro-signature-assumption}
\mathrm{signature}(\bbV_{\sigma_1})  = \begin{cases} (p,2) \text{ with } p>0 & \text{orthogonal case,} \\ (p,q) \text{ with } p,q>0 & \text{unitary case}, \end{cases}
\end{equation}
and $\bbV$ is positive definite at all other archimedean places. Let
\begin{equation}
\bfH = \begin{cases} \Res_{F/\bbQ} \, \GSpin(\bbV), & \text{orthogonal case,} \\
\Res_{F/\bbQ} \, \mathrm{U}(\bbV), & \text{unitary case}, \end{cases}
\end{equation}
and $\mathbb{D}$ be the hermitian symmetric domain attached to $\bfH(\mathbb{R})$; concretely, $\bbD$ parametrizes oriented negative-definite real planes (resp.\ negative definite complex  $q$-dimensional subspaces) in $\bbV_{\sigma_1}$ in the orthogonal (resp.\ unitary) case. For a fixed compact open subgroup $K \subset \bfH(\bbA_f)$, let 
$X_{\bbV,K}$ 
be the corresponding Shimura variety, which has a canonical model
over $\sigma_1(F)$ (resp.\ $\sigma_1(E)$). Then $X_{\bbV,K}$ is a finite disjoint union 
\begin{equation} \label{eq:intro_Shimura_var_components}
X_{\bbV,K}\ = \ \coprod \Gamma \backslash \bbD
\end{equation}
of quotients of $\bbD$ by certain arithmetic subgroups $\Gamma \subset \bfH(\bbQ)$. Let $\mathcal{X}_K$ denote the variety obtained by viewing the canonical model of $X_{\bbV,K}$ as a variety over $\mathrm{Spec}(\mathbb{Q})$. The complex points \begin{equation}
\mathcal{X}_K(\mathbb{C}) = \coprod X_{\bbV[k],K}
\end{equation} 
are a finite disjoint union of Shimura varieties attached to $\bbV[1]:=\bbV$ and its nearby spaces $\bbV[k]$ (see \Cref{section:Green_currents_global}).

The variety $\mathcal{X}_K$ is equipped with a family of rational \emph{special cycles}
\begin{equation}
 \{ Z(T, \varphi_f)  \},
  \end{equation}
as defined by Kudla \cite{KudlaOrthogonal}, that are parametrized by pairs $(T,\varphi_f)$ consisting of a matrix $T \in \Sym_r(F)$ (resp.\ $T \in \Herm_r(E)$) and a $K$-invariant Schwartz function $\varphi_f \in \mathcal{S}(\bbV(\bbA_f)^r)^K$. These cycles generalize the construction of Heegner points on modular curves and Hirzebruch-Zagier cycles on Hilbert modular surfaces.

The irreducible components of the cycle $Z(T,\varphi_f)$ on $X_{\bbV,K}$ (say) admit a complex uniformization by certain complex submanifolds $\mathbb{D}_{\bfv}$ of $\mathbb{D}$ defined as follows: for a collection of vectors $\bfv = (v_1, \dots, v_r) \in \bbV^r$ satisfying
\begin{equation}
T(\bfv) := \left( \tfrac{1}{2}Q( v_i, v_j ) \right)_{i,j=1,\dots, r} = T,
\end{equation}
let
\begin{equation}
\bbD_{\bfv} \ := \ \{   z \in \bbD \ | \  z \perp v_i \, \text{ for all } i=1,\dots, r   \},
\end{equation}
so that $\mathbb{D}_\bfv$, if non-empty, is a hermitian symmetric subdomain of $\mathbb{D}$ of codimension $\tilde r :=\mathrm{rk}(T)$ in the orthogonal case (resp.\ $\tilde r := q \cdot \mathrm{rk}(T)$ in the unitary case). 

As a first step towards defining a Green current for $Z(T,\varphi_f)$, we construct a current $\archGreen(\bfv)$ on $\mathbb{D}$ satisfying the equation
\begin{equation} \label{eq:intro_modified_Green_current}
\ddc \archGreen(\bfv) \ + \ \delta_{\bbD_{\bfv}} \wedge \Omega_{\calE^{\vee}}^{r - \rk(T)} \ = \ [\phio_{\KM}(\bfv)].
\end{equation}
Here $\delta_{\bbD_{\bfv}}$ is the current defined by integration along  $\bbD_{\bfv}$ and $\Omega_{\calE^{\vee}}$
is the top Chern form of  the dual of the tautological bundle $\calE$, see \Cref{subsection:Hermitian symmetric domains attached to orthogonal and unitary groups} below; the form $\phio_{\KM}(\bfv) := e^{2 \pi \tr(T(\bfv))} \varphi_{\KM}(\bfv)$ is, up to a normalizing factor, the Schwartz\footnote{This means that $\varphi_{\KM}(\bfv)$ and all its  derivatives are of exponential decay in $\bfv$.} form $\varphi_{\KM}(\bfv) \in A^{\tilde r,\tilde r}(\mathbb{D})$ introduced by Kudla and Millson \cite{KudlaMillson3}. 

In recent work \cite{GarciaSuperconnections}, the first author introduced a superconnection $\nabla_{\bfv}$ on $\mathbb{D}$ and showed that the component of degree $(\tilde r, \tilde r)$ of the corresponding Chern form agrees with $\phio_{\KM}(\bfv)$. This allows us to apply the general results of \cite{BismutGilletSoule1} to obtain an explicit natural form $\nuo(\bfv)$ satisfying the transgression formula
\begin{equation} 
\ddc \, \nuo(\sqrt t \, \bfv) \ =\  - t \, \frac{d}{d t} \phio_{\KM}(\sqrt t \, \bfv), \qquad \qquad t \in \bbR_{>0}.
\end{equation}
Moreover, the forms $\nuo(t\bfv)$ and $\phio(t \bfv)$ are of exponential decay in $t$ on $\mathbb{D} \backslash \mathbb{D}_\bfv$. 

To define $\archGreen(\bfv)$, assume first that $T$ is a non-degenerate matrix (so that $r=\rk(T)$) and consider the integral 
\begin{equation} \label{eq:intro_def_g0_integral}
\archGreen(\bfv) := \int_1^{\infty} \nuo(\sqrt t \bfv) \frac{dt}{t},
\end{equation}
initially defined on $\mathbb{D} \backslash \mathbb{D}_\bfv$. The estimates in \cite{BismutGilletSouleDuke90} show that $\archGreen(\bfv)$ is smooth on $\mathbb{D} \backslash \mathbb{D}_{\bfv}$, locally integrable on $\mathbb{D}$, and satisfies \eqref{eq:intro_modified_Green_current}, which in this case reduces to Green's equation; in other words, $\archGreen(\bfv)$ is a Green form for $\mathbb{D}_\bfv$, in the terminology of \cite[\textsection 1.1]{BostGilletSoule}.

When $T$ is degenerate, the expression \eqref{eq:intro_def_g0_integral} is not locally integrable on $\mathbb{D}$ in general. We will circumvent this problem by regularizing the integral and obtain a current $\archGreen(\bfv)$ satisfying \eqref{eq:intro_modified_Green_current} for every $\bfv$, see \Cref{subsection:green_form_on_D}.

Returning to the special cycles $Z(T,\varphi_f)$, we define currents $\lie g(T, \bfy, \varphi_f)$ on $\calX_K(\bbC)$ as weighted sums of $\archGreen(\bfv)$ over vectors $\bfv$ satisfying $T(\bfv)=T$; these currents also depend on a parameter  $\bfy \in \Sym_r(F\otimes_{\bbQ} \bbR)_{\gg0}$ or $\bfy \in \Herm_r(E \otimes_{\bbQ}\bbR)_{\gg 0}$ in the orthogonal or unitary cases, respectively. 

\begin{theorem} \label{thm:intro_thm_1}
The current $ \lie g(T, \bfy, \varphi_f)$ satisfies 
	\begin{equation} \label{eqn:Intro Generalized Green Equation}
	\ddc \lie g(T, \bfy, \varphi_f) \ +\ \delta_{Z(T, \varphi_f)(\bbC)} \wedge \Omega_{\calE^{\vee}}^{r - \rank(T)} \ = \ \omega(T,\bfy, \varphi_f),
	\end{equation}
	where $\omega(T,\bfy, \varphi_f)$ is the $T$'th coefficient in the $q$-expansion of the theta function attached to  $\varphi_{\KM}\otimes \varphi_f$.
	
	Moreover, if $T_1$ and $T_2$ are non-degenerate and $Z(T_1, \varphi_1)$ and $Z(T_2, \varphi_2)$ intersect properly, then  
	\[
	\lie g(T_1, \bfy_1, \varphi_1) * \lie g(T_2, \bfy_2, \varphi_2) \ \equiv \ \sum_{T = ( \sm{T_1 &*\\ * &T_2})} \lie g \left( T, \left( \sm{\bfy_1 & \\ & \bfy_2} \right), \varphi_1 \otimes \varphi_2 \right) \ \pmod{ \mathrm{im} \, d + \mathrm{im} \, d^{\mathrm c}}.
	\]
\end{theorem}

Note that our construction is valid for all $T$; for example, when $\bbV$ is anisotropic we obtain
\begin{equation}
\lie g(\mathbf 0, \bfy, \varphi_f)|_{\calX_{K,\sigma_k}(\bbC)} \ = \ - \varphi_f(0) \, \log(\det \sigma_k(\bfy)) \cdot c_{\mathrm{rk}(\vbE)-1}(\vbE^\vee,\nabla)^*  \wedge \Omega_{\vbE^\vee}^{r-1}
\end{equation} 
for each real embedding $\sigma_k$ of $F$.

When $T$ is non-degenerate, \eqref{eqn:Intro Generalized Green Equation} is Green's equation for the cycle $Z(T,\varphi_f)$ and hence $\lie g(T, \bfy, \varphi_f)$ is a Green current (in fact, a Green form) for $Z(T,\varphi_f)$. When $T$ is degenerate, the cycle $Z(T,\varphi_f)$ appears in the ``wrong" codimension; following \cite{KudlaOrthogonal}, this deficiency can be rectified by intersecting with a power of the tautological bundle and, as we discuss in \Cref{sec:arith Chow groups}, solutions to \eqref{eqn:Intro Generalized Green Equation} correspond naturally to Green currents for this modified cycle.

Our main result computes local archimedean heights of special cycles in terms of Siegel Eisenstein series. We restrict our attention to the case that $\bbV$ is anisotropic, so that the corresponding Shimura variety is compact; our assumption on the signature of $\bbV$ ensures that this is the case whenever $F \neq \bbQ$. We further assume that $q=1$ in the unitary case. 
Fix an integer $r \leq p+1$ and let 
\begin{equation}
s_0 \ := \  \frac{p +1 -r }{2},
\end{equation} 
and, for a Schwartz function $\varphi_f \in \mathcal{S}(\bbV(\bbA_f)^r)^K $,  consider the corresponding genus $r$ Siegel Eisenstein series
\begin{equation}
E(\bftau,  \Phi_f,s) \ =\ \sum_T E_T(\bftau,\Phi_f,s) 
\end{equation}
of parallel scalar weight $l = \dim_F(\bbV)/2$ (resp.\ $l =\dim_E(\bbV)$), see \Cref{subsection:Fourier coefficients of scalar weight Eisenstein series}; here $\bftau =\bfx + i \bfy\in \bbH_r^d$, where $\mathbb{H}_r$ is the Siegel (resp.\ Hermitian) upper half-space of genus $r$. Let
\begin{equation}
E_T'(\bftau, \Phi_f,s_0) \ := \ \frac{d}{d s} E_T(\bftau, \Phi_f,s) \Big|_{s= s_0}
\end{equation}
denote the derivative of its Fourier coefficient $E_T(\bftau,\Phi_f,s)$ at $s=s_0$.

\begin{theorem} \label{thm:Intro Green Integral} Suppose that $\bbV$ is anisotropic and, in the unitary case, that $q=1$. Then for any $T$, there is an explicit constant $\kappa(T, \Phi_f)$, given by \Cref{def:KappaTilde}, such that
\begin{equation} \label{eq:intro_main_thm}
\frac{(-1)^r  \kappa_0}{2\mathrm{Vol}(X_{\bbV,K},\Omega_\vbE)}	  \int_{[\calX_K(\bbC)]} \lie g(T, \bfy, \varphi_f) \, \wedge \Omega_{\vbE}^{p+1-r} \, q^T = \  E'_T(\bftau, \Phi_f, s_0)  \ - \  \kappa(T, \Phi_f) \, q^T.
\end{equation}
Here $q^T=e^{2\pi i \mathrm{tr}(T\bftau)}$, and $\kappa_0=1$ if $s_0>0$ and $\kappa_0=2$ if $s_0=0$.
\end{theorem}

As a special case, suppose that $T$ is non-degenerate, so that there is a factorization
\begin{equation}
E_T(\bftau, \Phi_f,s) \ =\ W_{T, \infty}(\bftau,  \Phi^l_{\infty},s) \cdot W_{T,f}(e,  \Phi_f,s),
\end{equation}
where the  factors on the right are the products of the archimedean and non-archimedean local Whittaker functionals, respectively. 
Let 
\begin{equation}
E_T'(\bftau, \Phi_f,s_0)_\infty \ =\ W'_{T, \infty}(\bftau,  \Phi^l_{\infty},s_0) \cdot W_{T,f}(e,  \Phi_f,s_0)
\end{equation}
denote the archimedean contribution to the special derivative. 
%Then \Cref{thm:Intro Green Integral} specializes to the following identities: if $T$ is not totally positive definite, then
%\begin{equation} \label{eq:intro_thm_non_deg_T_not_pos_def}
%\frac{(-1)^r  \kappa_0}{2\mathrm{Vol}(X_{\bbV,K},\Omega_\vbE)} \int_{\mathcal{X}_K(\mathbb{C})} \lie g(T,\bfy,\varphi_f)\wedge \Omega_\vbE^{p+1-r} q^T = \Eis_T'(i \bfy,\Phi_f,s_0)_\infty
%\end{equation}
%and, if $T$ is totally positive definite, then
%\begin{equation} \label{eq:intro_thm_non_deg_T_pos_def}
%\begin{split}
%& \frac{(-1)^r \kappa_0}{2\mathrm{Vol}(X_{\bbV,K},\Omega_\vbE)} \int_{\mathcal{X}_K(\mathbb{C})} \lie g(T,\bfy,\varphi_f)\wedge \Omega_\vbE^{p+1-r} q^T  \\
%& \quad =\Eis_T'(\bfy,\Phi_f,s_0)_\infty -\Eis_T(\bfy,\Phi_f,s_0) \left( \frac{\iota d}{2} \left( r \log \pi - \frac{\Gamma_r'(\iota m/2)}{\Gamma_r(\iota m/2)} \right) + \frac{\iota}{2} \log N_{F/\mathbb{Q}} \det T \right).
%\end{split}
%\end{equation}
%Here $\iota=1$ (resp.\ $\iota=2$) in the orthogonal (resp.\ unitary) case.
Then \Cref{thm:Intro Green Integral} specializes to the identity
\begin{equation} \label{eq:intro_thm_non_deg_T_not_pos_def}
\frac{(-1)^r  \kappa_0}{2\mathrm{Vol}(X_{\bbV,K},\Omega_\vbE)} \int_{[\mathcal{X}_K(\mathbb{C})]} \lie g(T,\bfy,\varphi_f)\wedge \Omega_\vbE^{p+1-r} q^T = \Eis_T'(\bftau,\Phi_f,s_0)_\infty
\end{equation}
if $T$ is not totally positive definite, and to
\begin{equation} \label{eq:intro_thm_non_deg_T_pos_def}
\begin{split}
& \frac{(-1)^r \kappa_0}{2\mathrm{Vol}(X_{\bbV,K},\Omega_\vbE)} \int_{[\mathcal{X}_K(\mathbb{C})]} \lie g(T,\bfy,\varphi_f)\wedge \Omega_\vbE^{p+1-r} q^T  \\
& \quad =\Eis_T'(\bftau,\Phi_f,s_0)_\infty -\Eis_T(\bftau,\Phi_f,s_0) \left( \frac{\iota d}{2} \left( r \log \pi - \frac{\Gamma_r'(\iota m/2)}{\Gamma_r(\iota m/2)} \right) + \frac{\iota}{2} \log N_{F/\mathbb{Q}} \det T \right)
\end{split}
\end{equation}
if $T$ is totally positive definite; here $\iota=1$ (resp.\ $\iota=2$) in the orthogonal (resp.\ unitary) case.

When $T$ is non-degenerate, the proof of the theorem can be summarized as follows: the current $\lie g(T,\bfy, \varphi_f)$ is given by a sum of integrals of the form \eqref{eq:intro_def_g0_integral}, for vectors $\bfv$ with $T(\bfv) = T$. Interchanging the order of integration,  the Siegel-Weil formula relates the left hand side of \eqref{eq:intro_main_thm} to the Fourier coefficient $E_T(\bftau, \Phi(\nu),s_0)$ of an Eisenstein series attached to the Schwartz form 
\begin{equation} \label{eqn:Intro Nu}  
\nu(\bfv) = e^{- 2 \pi \tr(T(\bfv))} \nuo(\bfv). 
\end{equation} 
We then analyze the behaviour of $\nu(\bfv)$ under the action of the metaplectic group $\mathrm{Mp}_{2r}(\mathbb{R})$ (resp.\ the unitary group $\mathrm{U}(r,r)$) via the Weil representation. A multiplicity one argument allows us to identify $ \Phi(\nu)$ explicitly, and in turn relate $E_T(\bftau, \Phi(\nu),s_0)$ to $E'_T(\bftau, \Phi_f,s_0)$ via a lowering operator. To conclude the proof, we apply work of Shimura   \cite{ShimuraConfluent} to derive asymptotic estimates for the Fourier coefficients $E_T(\bftau,\Phi_f,s)$ as $\bfy \to \infty$. 

When $T$ is degenerate, the idea is roughly the same, though additional care is required in handling the regularization, as well as establishing the required asymptotics of $E_T(\bftau,\Phi_f,s)$. 

Prior results of this form have appeared in only a few special cases in the literature. For divisors, the Green function we define specializes to the one defined by Kudla \cite{KudlaCD}, and \Cref{thm:Intro Green Integral}  was proved in \cite{KRYFaltingsHeights} for Shimura curves; a related result for $\mathrm{U}(p,1)$ Shimura varieties over imaginary quadratic fields was proved by Ehlen and the second author \cite{EhlenSankaran}.

In higher codimension much less was known. For (arithmetic) codimension two cycles on Shimura curves, Kudla \cite{KudlaCD} defined Green currents using star products; this construction does not coincide with ours, but does agree modulo exact currents by \Cref{thm:intro_thm_1}, and \Cref{thm:Intro Green Integral} reduces to results proved by elaborate explicit computations in \cite{KudlaCD} and \cite{KudlaRapoportYang}. Similar methods were used by Liu \cite{LiuYifengArithThetaI} for arithmetic codimension $p+1$ cycles on $\mathrm{U}(p,1)$ to prove a star product version of the particular case of the theorem given by \eqref{eq:intro_thm_non_deg_T_not_pos_def}.
Again in the non-degenerate case, a recent preprint of Bruinier and Yang \cite{BruinierYangArithmeticDegrees} proves a star product version of \eqref{eq:intro_thm_non_deg_T_not_pos_def} for arithmetic codimension $p+1$ cycles on $\mathrm{O}(p,2)$ by a different argument involving induction on $p$.

Finally, we place our results in the context of Kudla's conjectures on special cycles in arithmetic Chow groups. Putting aside the difficult issues involved in constructing integral models, the integral appearing in \Cref{thm:Intro Green Integral} is the archimedean contribution to the height of an arithmetic cycle lifting $Z(T,\varphi_f)$; according to Kudla's conjectural arithmetic Siegel-Weil formula, this height should equal the Fourier coefficient of an appropriately normalized version of the Eisenstein series appearing above. The remaining contribution to the arithmetic height is purely algebro-geometric in nature,  and in particular should be independent of $\bfy$; thus \Cref{thm:Intro Green Integral} asserts that the non-holomorphic terms in Kudla's conjectural identity coincide. Put another way, our theorem reduces Kudla's conjecture to a relatively explicit conjectural formula for the analogue of the Faltings height (as in \cite{BostGilletSoule}) of a special cycle $Z(T,\varphi_f)$ in terms of $T$ and $\varphi_f$; we discuss this point in  more detail in \Cref{sec:Arithmetic Height conjecture}.

\subsection{Notation and conventions} 
Let $\bbK$ be a field endowed with a (possibly trivial) involution $a \mapsto \overline{a}$. We write $\mathrm{Sym}_r(\bbK)$ (resp. $\mathrm{Her}_r(\bbK)$) for the group of symmetric (resp. hermitian) $r$-by-$r$ matrices with coefficients in $\bbK$ under matrix addition. For $a \in \mathrm{GL}_r(\bbK)$ and $b \in \mathrm{Her}_r(\bbK)$, let
%\begin{equation}
%\begin{split}
%m(a) &=  \left( \begin{array}{cc} a & 0 \\ 0 & {^t \bar{a}^{-1}} \end{array} \right), \\
%n(b)&=\begin{pmatrix} 1_r & b \\ 0 & 1_r \end{pmatrix}, \\
%w_r &= \begin{pmatrix} & 1_r \\ -1_r & \end{pmatrix}.
%\end{split}
%\end{equation}
\begin{equation}
m(a) =  \left( \begin{array}{cc} a & 0 \\ 0 & {^t \bar{a}^{-1}} \end{array} \right), \quad n(b)=\begin{pmatrix} 1_r & b \\ 0 & 1_r \end{pmatrix}, \quad w_r = \begin{pmatrix} & 1_r \\ -1_r & \end{pmatrix}.
\end{equation}
For $x=(x_1,\ldots,x_r) \in \bbK^r$, we write
\begin{equation}
d(x) =  \mathrm{diag}(x_1, \dots, x_r ) = \begin{pmatrix}x_{1} & & \\ & \ddots & \\ & & x_{r}\end{pmatrix}.
\end{equation}
%For $0 \leq r' \leq r$, we let
%\begin{equation}
%w_{r,r'} = \begin{pmatrix} 1_{r-r'} & & & \\ & 0_{r'} & & 1_{r'} \\ & & 1_{r-r'} & \\ & -1_{r'} & & 0_{r'} \end{pmatrix}
%\end{equation}
%and abbreviate $w_r=w_{r,r}$.

We fix the following standard choice of additive character $\psi=\psi_{F}:F \to \mathbb{C}^\times$ when $F$ is a local field of characteristic zero. If $F=\mathbb{R}$ we set $\psi(x)=e^{2\pi i x}$; if $F=\mathbb{Q}_p$ we choose $\psi=\psi_{\mathbb{Q}_p}$ so that $\psi(p^{-1})=e^{-2\pi i/p}$; if $F$ is a finite extension of $\mathbb{Q}_v$ we set $\psi_F(x)=\psi_{\mathbb{Q}_v}(\mathrm{tr}_{F/\mathbb{Q}_v}(x))$. If $F$ is a global field, we write $\mathbb{A}_F^\times$ for the ideles of $F$ and set $\psi_F:F^\times \backslash \mathbb{A}_F^\times \to \mathbb{C}^\times$, where $\psi_{F}=\otimes_v \psi_{F_v}$ and the product runs over all places $v$ of $F$.

%Let $E_{i,j}$ be the matrix with 1 at the $(i,j)$-th entry at 0 everywhere else.

We denote the connected component of the identity of a Lie group $G$ by $G^0$. 

We denote by $A^*(X)$ (resp. $D^*(X)$) the space of differential forms (resp. currents) on a smooth manifold $X$. Given $\alpha \in A^*(X)=\oplus_{k \geq 0} A^k(X)$, we write $\alpha_{[k]}$ for its component of degree $k$. If $\alpha$ is closed, we write $[\alpha] \in \mathrm{H}^*(X)$ for the cohomology class defined by $\alpha$.

If $X$ is a complex manifold, we let $\mathrm{d}^c = (4\pi i)^{-1}(\partial - \overline{\partial})$, so that $\mathrm{dd}^c = (-2\pi i)^{-1} \partial \overline{\partial}$. We denote by $^*$ the operator on $\oplus_{k \geq 0} A^{k,k}(X)$ acting by multiplication by $(-2\pi i)^{-k}$ on $A^{k,k}(X)$. The canonical orientation on $X$ induces an inclusion $A^{p,q}(X) \subset D^{p,q}(X)$ sending a differential form $\omega$ to the current given by integration against $\omega$ on $X$, which we will denote by $[\omega]$ or simply by $\omega$.

If $f(s)$ is a meromorphic function of a complex variable $s$, we write $\mathop{\mathrm{CT}}\limits_{s=0} f(s)$ for the constant term of its Laurent expansion at $s=0$.

\subsection{Acknowledgments}

This work was done while L. G. was at the University of Toronto and IHES and S.S. was at the University of Manitoba; the authors thank these institutions for providing excellent working conditions. An early draft, with a full proof of the main identity for non-degenerate Fourier coefficients, was circulated and posted online in January 2018; we are grateful to Daniel Disegni, Stephen Kudla and Shouwu Zhang for comments on it, and other helpful conversations. S.S.\ acknowledges financial support from NSERC.

\section{Green forms on hermitian symmetric domains} 

Let $X$ be a complex manifold and $Z \subset X$ be a closed irreducible analytic subset of codimension $c$. A \emph{Green current} for $Z$ is a current $\lie g_Z \in D^{c-1,c-1}(X)$ such that
\begin{equation} \label{eq:def_Green_current}
\ddc \lie g_Z+\delta_Z = [\omega_Z],
\end{equation}
where $\delta_Z$ denotes the current of integration on $Z$ and $\omega_Z$ is a smooth differential form on $X$. A \emph{Green form} is a Green current given by a form that is locally integrable on $X$ and smooth on $X-Z$ (see \cite[\textsection 1.1]{BostGilletSoule}).
%\todo{SS: added italics for def's. LG: OK.}

Here we will construct Green forms for certain complex submanifolds of the hermitian symmetric space $\mathbb{D}$ attached to $\mathrm{O}(p,2)$ or $\mathrm{U}(p,q)$, where $p,q>0$. Throughout the paper we will refer to the case involving $\mathrm{O}(p,2)$ as case 1 (or as the orthogonal case) and to the case involving $\mathrm{U}(p,q)$ as case 2 (or as the unitary case). Our methods apply uniformly in both cases.

%We would like to emphasize that our methods apply uniformly and all results in both cases are closely parallel.

There is a natural hermitian holomorphic vector bundle $\vbE$ over $\mathbb{D}$, and the submanifolds of $\mathbb{D}$ that we consider are the zero loci $Z(s)$ of certain natural holomorphic sections $s \in \mathrm{H}^0((\vbE^\vee)^r)$ ($r \geq 1$). In this setting, the results in \cite{BismutInv90,BismutGilletSoule1, BismutGilletSouleDuke90} (reviewed in \Cref{subsection:Superconnections and characteristic forms of Koszul complexes}) can be applied to construct some natural differential forms on $\mathbb{D}$, as we show in \Cref{subsection:Hermitian symmetric domains attached to orthogonal and unitary groups}. We give explicit examples in \Cref{subsection:explicit formulas in rank one}, showing that this construction recovers some differential forms considered in previous work on special cycles (cf.\ \cite{KudlaCD, KudlaMillson1}), and then (Sections \ref{subsection:Basic properties of the forms varphi and nu} and \ref{subsection:Behaviour under the Weil representation}) we establish the main properties of these forms. Using these results,  in Section \ref{subsection:Green forms} we define some currents related to $Z(s)$, including a Green form for $Z(s)$. The final \Cref{subsection:star_products} 
considers star products and will be used in the proof of \Cref{theorem:star_product_formula_global}.

%Consider the hermitian symmetric domain $\mathbb{D}$ attached to the Lie group $\mathrm{U}(p,q)$, where $pq \neq 0$.

%{\bf TO DO: explain here informally what special cycles on $\mathbb{D}$ are and what Green forms are (mention we don't care about harmonic?). Mention that they are zero loci of sections of the tautological bundle and what differential forms we will define. As we explain all this, reference the relevant subsections. Explain here the whole case 1- case 2 approach.}

\subsection{Superconnections and characteristic forms of Koszul complexes} \label{subsection:Superconnections and characteristic forms of Koszul complexes} %{\bf TO DO: explain here $\mathrm{tr_s}$ and related super-machinery.} 

In this section we review the construction of some characteristic differential forms attached to a pair $(\vbE,u)$, where $\vbE$ is a holomorphic hermitian vector bundle and $u$ is a holomorphic section of its dual. The results in this section are due to Quillen \cite{QuillenChern}, Bismut \cite{BismutInv90} and Bismut-Gillet-Soul\'e \cite{BismutGilletSoule1, BismutGilletSouleDuke90}.

%The exceptions are \Cref{thm:characterization_xi_0}, which characterizes the forms $\xio$ that we will use later to construct Green forms for special cycles on Shimura varieties, and the results in \Cref{subsection:star_products} leading to the rotation invariance of star products proved in \Cref{corollary:rotation_invariance}.

We will use Quillen's formalism of superconnections and related notions of superalgebra. For more details, the reader is referred to \cite{QuillenChern, BGV}.
We briefly recall that a super vector space $V$ is just a complex $\mathbb{Z}/2\mathbb{Z}$-graded vector space; we write $V=V_0 \oplus V_1$ and refer to $V_0$ and $V_1$ as the even and odd part of $V$ respectively. We write $\tau$ for the endomorphism of $V$ determined by $\tau(v)=(-1)^{\mathrm{deg}(v)}v$. The supertrace $\mathrm{tr_s} \colon \mathrm{End}(V) \to \mathbb{C}$ is the linear form defined by
\begin{equation}
\mathrm{tr_s}(u)=\mathrm{tr}(\tau u),
\end{equation}
where $\mathrm{tr}$ denotes the usual trace. Thus if $u=\left(\begin{smallmatrix} a & b \\ c & d \end{smallmatrix}\right)$ with $a \in \mathrm{End}(V_0)$, $d \in \mathrm{End}(V_1)$, $b \in \mathrm{Hom}(V_1,V_0)$ and $c \in \mathrm{Hom}(V_0,V_1)$, then $\mathrm{tr_s}(u)=\mathrm{tr}(a)-\mathrm{tr}(d)$.

%{\bf TO DO: add definition of $d^c$. Design choice: I will do it when I define $\varphi(v)$, $\nu(v)$ and $\xi(v)$, where I will also use the operator $^*$. We put $d^c = (4\pi i)^{-1}(\partial - \overline{\partial})$. Then $dd^c = (-2\pi i)^{-1} \partial \overline{\partial}$ and... }.

% Koszul complexes
\subsubsection{} \label{subsubsection:Koszul_complexes} Let $\vbE$ be a holomorphic vector bundle on a complex manifold $X$ and $u \in \mathrm{H}^0(\vbE^\vee)$ be a holomorphic section of its dual $\vbE^\vee$. Let $K(u)$ be the Koszul complex of $u$: its underlying vector bundle is the exterior algebra $\wedge \vbE$ and its differential $u: \wedge^k \vbE \to \wedge^{k-1} \vbE$ is defined by
\begin{equation}
u(e_1 \wedge \cdots \wedge e_k) = \sum_{1 \leq i \leq k} (-1)^{i+1} u(e_i) e_1 \wedge \cdots \wedge \hat{e_i} \wedge \cdots \wedge e_k.
\end{equation}
The grading on $K(u)$ is given by $K(u)^{-k}=\wedge^k \vbE$, so that $K(u)$ is supported in non-positive degrees. We identify $K(u)$ with the corresponding complex of sheaves of sections of $\wedge^k \vbE$,
%The reason we need to say this is that to define regularity we need to talk about the cohomology of $K(u)$, and this cohomology is taken in the category of coherent sheaves.
and say that $u$ is a \emph{regular section} if the cohomology of $K(u)$ vanishes in negative degrees.
% References: Fulton, Intersection Theory, Appendix A.5; SGA6, Exp. VII.
If $u$ is regular with zero locus $Z(u)$, then $K(u)$ is quasi-isomorphic to $\mathcal{O}_{Z(u)}$ (regarded as a complex supported in degree zero).

% The superconnection on the Koszul complex of a hermitian vector bundle
\subsubsection{} \label{subsection:superconnection_Koszul_complex} Assume now that $\vbE$ is endowed with a hermitian metric $\| \cdot \|_\vbE$. This induces a hermitian metric on $\wedge \vbE$: for any $x \in X$, the subspace $\wedge^k \vbE_x$ is orthogonal to $\wedge^j \vbE_x$ if $j \neq k$, and an orthonormal basis of $\wedge^k \vbE_x$ is given by all elements $e_{i_1} \wedge \cdots \wedge e_{i_k}$, where $1 \leq i_1 < \cdots < i_k \leq \mathrm{rk} \vbE$ and $\{e_1,\ldots,e_{\mathrm{rk} \vbE}\}$ is an orthonormal basis of $\vbE_x$. Let $\nabla$ be the corresponding Chern connection on $\wedge \vbE$. We regard $\wedge \vbE$ as a super vector bundle, with even part $\wedge^{\text{even}} \vbE$ and odd part $\wedge^{\text{odd}} \vbE$, and $u$ as an odd endomorphism of $\wedge \vbE$. Let $u^*$ be the adjoint of $u$, define the superconnection
\begin{equation}
\nabla_u = \nabla + i \sqrt{2\pi}(u+u^*)
\end{equation}
on $\wedge \vbE$, and consider Quillen's Chern form
\begin{equation}
\phio(u)=\phio(\vbE,\| \cdot \|_\vbE,u) = \mathrm{tr_s}(e^{\nabla_u^2}) \in \bigoplus_{k \geq 0} A^{k,k}(X).
\end{equation}

We recall some properties of $\phio(u)$ established (in greater generality) by Quillen \cite{QuillenChern}. The form $\phio(u)$ is closed and functorial: given a holomorphic map of complex manifolds $f:X' \to X$, consider the pullback bundle $(f^*\vbE,f^*\|\cdot\|)$ and the pullback section $f^*u \in \mathrm{H}^0(f^*\vbE^\vee)$. Then
\begin{equation}
\phio(f^*u) = f^*\phio(u).
\end{equation}
Let $^*$ be the operator on $\oplus_{k \geq 0} A^{k,k}(X)$ acting by multiplication by $(-2\pi i)^{-k}$ on $A^{k,k}(X)$. Writing $[\phio(u)^*]$ for the cohomology class of $\phio(u)^*$ and $\mathrm{ch}(\cdot)$ for the Chern character, we have
\begin{equation} \label{eq:cohomology_class_of_varphi0}
[\phio(u)^*] = \mathrm{ch}(\wedge \vbE) = \mathrm{ch}(\wedge^{\text{even}}\vbE)-\mathrm{ch}(\wedge^{\text{odd}}\vbE).
\end{equation}

\subsubsection{} In particular, $[\phio(u)^*]$ depends on $\vbE$, but not on $u$. Thus the forms $\phio(tu)$ for $t \in \mathbb{R}_{>0}$ all belong to the same cohomology class, but as $t \to +\infty$ they concentrate on the zero locus of $u$. More precisely, recall that $\nabla_u^2$ is an even element of the (super)algebra
\begin{equation}
A(X,\mathrm{End}(\wedge \vbE)) := A^*(X) \hat{\otimes}_{\mathcal{C}^\infty(X)} \Gamma(\mathrm{End}(\wedge \vbE)).
\end{equation} 
Given a relatively compact open subset $U \subset X$ whose closure $\overline{U}$ is disjoint from $Z(u)$ and a non-negative integer $k$, consider an algebra seminorm $\| \cdot \|_{\overline{U},\vbE,k}$ on $A(X,\mathrm{End}(\wedge \vbE))$ measuring uniform convergence on $\overline{U}$ of partial derivatives of order at most $k$. We will need an estimate of $\| e^{\nabla_{tu}^2} \|_{\overline{U},\vbE,k}$ for large $t$. To obtain it, write $\nabla^2_{tu}=(\nabla^2_{tu})_{[0]}+R_{tu}$, where $(\nabla^2_{tu})_{[0]}$ has form-degree zero and $R_{tu}$ has form-degree $\geq 1$. Note that $R_{tu}$ is nilpotent and that
\begin{equation}
(\nabla^2_{tu})_{[0]} = -2\pi (tu+tu^*)^2 = -2\pi t^2\|u\|_{\vbE^\vee}^2 \otimes \mathrm{id} \in A^0(X) \otimes \mathrm{End}(\wedge \vbE)
\end{equation}
(here $\|\cdot\|_{\vbE^\vee}$ denotes the unique metric on $\vbE^\vee$ such that the isomorphism $\vbE^\vee \simeq \overline{\vbE}$ induced by $\|\cdot \|_\vbE$ is an isometry). In particular, $(\nabla^2_{tu})_{[0]}$ and $R_{tu}$ commute. Hence we have
\begin{equation} \label{eq:expansion_e_nabla2}
\begin{split}
e^{\nabla^2_{tu}} &= e^{(\nabla^2_{tu})_{[0]}} e^{R_{tu}} \\
&= e^{-2\pi t^2\|u\|^2} \sum_{k=0}^{N} \tfrac{1}{k!}R_{tu}^k
\end{split}
\end{equation}
with $N \leq \dim_\mathbb{R} X$. Let $a$ be any positive real number strictly less than $\mathrm{min}_{x \in \overline{U}}\{ \| u(x) \|^2_{\vbE^\vee} \}$. Since $R_{tu}$ is polynomial in $t$, it follows from \eqref{eq:expansion_e_nabla2} that 
\begin{equation} \label{eq:algebra_norm_estimate}
\| e^{\nabla_{tu}^2}\|_{\overline{U},\vbE,k} \leq C e^{-2\pi at^2}, \quad t \in \mathbb{R}_{>0}
\end{equation}
for some positive real number $C$. Thus a similar bound holds for $\phio(tu)$.

\subsubsection{} It follows from \eqref{eq:cohomology_class_of_varphi0} that the form $\tfrac{d}{dt}\phio(t^{1/2}u)$ (for $t \in \mathbb{R}_{>0}$) is exact, and one can ask for a construction of a functorial transgression of this form. Bismut, Gillet and Soul\'e \cite{BismutGilletSoule1} construct such a transgression, and the resulting form is key to our results. To define it, let $N \in \mathrm{End}(\wedge \vbE)$ be the number operator acting on $\wedge^k \vbE$ by multiplication by $-k$ and set
\begin{equation}
\nuo(u) = \mathrm{tr_s}(Ne^{\nabla_u^2}) \in \bigoplus_{k \geq 0} A^{k,k}(X).
\end{equation}
Then $\nuo(u)$ is functorial with respect to holomorphic maps $f:X' \to X$ and satisfies (\cite[Thm. 1.15]{BismutGilletSoule1})
\begin{equation} \label{eq:nu0_transgression}
-\frac{1}{t}\partial\overline{\partial} \nuo(t^{1/2}u) = \frac{d}{dt}\phio(t^{1/2}u), \quad t > 0.
\end{equation}

\subsubsection{} \label{subsection:def_characterization_xi_0} Assume that the section $u$ has no zeroes on $X$. Then the Koszul complex $K(u)$ is acyclic and \eqref{eq:cohomology_class_of_varphi0} shows that $\phio(u)$ is exact. In this case one can define a characteristic form $\xio(u)$ giving a $\overline{\partial} \partial$-transgression of $\phio(u)$ 
%. Namely, let
by setting
\begin{equation} \label{eq:def_xi0}
\xio(u) = \int_1^{+\infty} \nuo(t^{1/2}u) \frac{dt}{t}.
\end{equation}
The bound \eqref{eq:algebra_norm_estimate} implies that all partial derivatives of $\nuo(t^{1/2}u)$ decrease rapidly as $t \to +\infty$. In particular, the integral converges and defines a form in $\oplus_{k \geq 0} A^{k,k}(X)$; moreover, one can differentiate under the integral sign, so that by \eqref{eq:nu0_transgression} we have
\begin{equation}
\partial \overline{\partial} \xio(u) = \phio(u).
\end{equation}

The form $\xio$ is functorial with respect to holomorphic maps $X' \to X$. In addition, $\xio(tu)$ (for $t \in \mathbb{R}_{>0}$) is rapidly decreasing as $t \to \infty$. More precisely, given a relatively compact open subset $U$ of $X$ with closure $\overline{U}$ and a positive integer $k$, we denote by $\| \cdot \|_{\overline{U},k}$ any  seminorm on $A^*(X)$ measuring uniform convergence on $\overline{U}$ of partial derivatives of order at most $k$. The bound \eqref{eq:algebra_norm_estimate} shows that
\begin{equation}
\|\xio(tu)\|_{\overline{U},k} \leq C e^{-2\pi a t^2} \quad \text{for all } t \geq 1,
\end{equation}
for some positive constants $a$ and $C$.

\subsubsection{} We now go back to the general case where $\vbE$ is a hermitian holomorphic vector bundle on $X$ and $u \in \mathrm{H}^0(\vbE^\vee)$,
and now make the assumption that $u$ is regular (see \Cref{subsubsection:Koszul_complexes}) and the zero locus $Z(u)$ is smooth; thus $K(u)$ is a resolution of $\mathcal{O}_{Z(u)}$ and the codimension of $Z(u)$ (if non-empty) in $X$ equals $\mathrm{rk}(\vbE)$. % ({\bf TO DO: check that this is equivalent to $Z(u)$ empty or $Z(u)$ has codimension $\mathrm{rk}(\vbE)$ in $X$.})
Define
\begin{equation} \label{eq:def_general_g}
\begin{split}
\archGreen(u) &= \xio(u)_{[2\mathrm{rk}(\vbE) -2]}^* \\
&= \left( \tfrac{i}{2\pi} \right)^{\mathrm{rk}(\vbE) -1} \int_1^{\infty} \nuo(t^{1/2} u)_{[2\mathrm{rk}(\vbE)-2]} \frac{dt}{t} . 
\end{split}
\end{equation}
The following proposition is contained in \cite{BismutGilletSouleDuke90}. It shows that $\archGreen(u)$ is a Green form for $Z(u)$. 
\begin{proposition} \phantomsection \label{prop:Green_current_general}
	%\label{prop:Green_current_general}
	\begin{enumerate}
		\item The integral \eqref{eq:def_general_g} converges to a smooth differential form $\archGreen(u) \in A^{\mathrm{rk}(\vbE)-1,\mathrm{rk}(\vbE)-1}(X-Z(u))$.
		\item The form $\archGreen(u)$ is locally integrable on $X$.
		\item As currents on $X$ we have
		\begin{equation*}
		\mathrm{dd}^c\archGreen(u) + \delta_{Z(u)} = \phio(u)_{[2\mathrm{rk}(\vbE)]}^*,
		\end{equation*}
		where $\ddc = \frac{i}{2 \pi} \partial \overline \partial$.
	\end{enumerate}
\end{proposition}
\begin{proof}
	Part $(1)$ has already been discussed in \Cref{subsection:def_characterization_xi_0}. 
	
	Part $(2)$ is shown to hold in the course of the proof of \cite[Thm. 3.3]{BismutGilletSouleDuke90}. In that paper, one considers an immersion $i \colon M' \to M$ of complex manifolds, a vector bundle $\eta$ on $M'$ and a complex $(\xi,v)$ of holomorphic hermitian vector bundles on $M$ that gives a resolution of $i_*\mathcal{O}_{M'}(\eta)$. Assume that $Z(u)$ is non-empty. We set $M=X$, $M'=Z(u)$, $\eta=\mathcal{O}_{Z(u)}$, and for the complex $(\xi,v)$ we take the Koszul complex $(\wedge \vbE,u)$; with these choices, the form $\nuo(t^{1/2}u)$ agrees with the form $\alpha_t$ defined in \cite[(3.12)]{BismutGilletSouleDuke90}.
	
	Let $z_0 \in Z(u)$, choose local coordinates $z_1,\ldots,z_N$ ($N=\dim X$) around $z_0$ such that $Z(u)$ is defined by the equations $z_1=\cdots=z_{\mathrm{rk}(\vbE)}=0$ and let $|y|=(|z_1|^2+\cdots+|z_{\mathrm{rk}(\vbE)}|^2)^{1/2}$. Since the complex normal bundle to $Z(u)$ in $X$ is canonically identified with $\vbE$, the equations \cite[(3.24), (3.26)]{BismutGilletSouleDuke90} show that
	\begin{equation} \label{eqn:Bismut nuo integrability}
	|y|^{2\mathrm{rk}(\vbE)-1} \int_1^{+\infty} \nuo(t^{1/2}u)\frac{du}{u}
	\end{equation}
	is bounded in a neighborhood of $z_0$; part $(2)$ follows since $|y|^{-(2\mathrm{rk}(\vbE)-1)}$ is locally integrable around $z_0$.

	To prove $(3)$, let $\beta$ be a compactly supported form on $X$. Then
	\begin{equation}
	\begin{split}
	\int_X \archGreen(u) \wedge \mathrm{dd}^c \beta &= \lim_{a \to +\infty} \int_1^a \int_X \nuo(t^{1/2}u)_{[2\mathrm{rk}(\vbE)-2]}^* \wedge \mathrm{dd}^c \beta \frac{dt}{t} \\
	&= \lim_{a \to +\infty} \int_1^a \int_X \mathrm{dd}^c \nuo(t^{1/2}u)_{[2\mathrm{rk}(\vbE)-2]}^* \wedge \beta \frac{dt}{t} \\
	&= \lim_{a \to +\infty} \int_1^a \int_X (-t \frac{d}{dt}\phio(t^{1/2}u)_{[2\mathrm{rk}(\vbE)]}^*) \wedge \beta \frac{dt}{t} \\
	&= \int_X \phio(u)_{[2\mathrm{rk}(\vbE)]}^* \wedge \beta - \lim_{a \to +\infty}\int_X \phio(a^{1/2}u)_{[2\mathrm{rk}(\vbE)]}^* \wedge \beta.
	%&= [\phio(u)_{[2\mathrm{rk}(\vbE)]}^*](\beta) - \lim_{a \to +\infty}[\phio(a^{1/2}u)_{[2\mathrm{rk}(\vbE)]}^*](\beta).
	\end{split}
	\end{equation}
	Here the first equality follows from dominated convergence since the integrals 
	\begin{equation} \label{eq:int_nuo_uniformly_bounded}
	|y|^{2\mathrm{rk}(\vbE)-1} \smallint_1^{a} \nuo(t^{1/2}u)^p\frac{du}{u}
	\end{equation}
	for $a \geq 1$ are uniformly bounded in a neighborhood of $z_0$, as shown in the proof of \cite[Thm. 3.3]{BismutGilletSouleDuke90}. 
%	since the integrals 
%	\begin{equation} \label{eq:int_nuo_uniformly_bounded}
%	|y|^{2\mathrm{rk}(\vbE)-1} \int_1^{a} \nuo(t^{1/2}u)^p\frac{du}{u}
%	\end{equation}
%	for $a \geq 1$ are uniformly bounded in a neighborhood of $z_0$, as shown in the proof of \cite[Thm. 3.3]{BismutGilletSouleDuke90}. 
The third equality follows from \eqref{eq:nu0_transgression}. This establishes $(3)$ since $\phio(a^{1/2}u)_{[2\mathrm{rk}(\vbE)]}^*$ approaches $\delta_{Z(u)}$ as $a \to +\infty$, as shown by \cite[(1.14), (1.19)]{BismutGilletSouleDuke90}. 
\end{proof}

\subsection{Hermitian symmetric domains of orthogonal and unitary groups} \label{subsection:Hermitian symmetric domains attached to orthogonal and unitary groups}
% Definitions of the groups and hermitian symmetric domains
\subsubsection{} \label{subsubsection:hermitian_domain_definitions} 
Let 
\begin{equation}
\bbK = \left\{ \begin{array}{cr} \mathbb{R} &  \text{ case 1} \\ \mathbb{C} &  \text{ case 2} \end{array} \right. 
\end{equation} 
and 
\begin{equation}
\sigma = \left\{ \begin{array}{cc} \text{id} &  \text{ case 1 } \\ \text{complex conjugation} &  \text{ case 2}. \end{array} \right. 
\end{equation}
Let $m \geq 2$ be a positive integer and let $V$ be a $\bbK$-vector space of dimension $m$ endowed with a non-degenerate $\sigma$-Hermitian bilinear form $Q$ of signature $(p,q)$. We assume that $pq \neq 0$ and in case 1 we further assume that $q=2$. Let $\mathrm{U}(V):=\mathrm{Aut}(V,Q)$ denote its isometry group, and set
\begin{equation}
G = \mathrm{U}(V)^0 \cong \left\{ \begin{array}{rc} \mathrm{SO}(p,2)^0 &  \text{ case 1 } \\ \mathrm{U}(p,q) &  \text{ case 2}.\end{array} \right. 
\end{equation}

We fix an orthogonal decomposition $V = V^+ \oplus V^-$ 
%with $\dim_{\bbK}V^+=p$ and $\dim_{\bbK}V^-=q$, and 
with $V^+$ and $V^-$ positive and negative definite respectively; in case 1 we also fix an orientation of $V^-$. Let $K$ be the centralizer in $G$ of the isometry of $V$ acting as the identity on $V^+$ and as $-1$ on $V^-$. Then $K$ is a maximal compact subgroup of $G$, given by
\begin{equation}
K = 
\begin{cases}\mathrm{SO}(V^+) \times \mathrm{SO}(V^-), &  \text{case 1,} \\ \mathrm{U}(V^+) \times \mathrm{U}(V^-), &  \text{case 2.}\end{cases}
\end{equation}

Let $\mathrm{Gr}(q,V)$ be the Grassmannian of oriented two-dimensional subspaces of $V$ in case 1, which consists of two copies of the usual Grassmannian. In case 2, we take $\mathrm{Gr}(q,V)$ to be the space of  $q$-dimensional (complex) subspaces of $V$. Let
%n the unitary case; in the  furthermore that these subspaces are oriented, so that in this case $\mathrm{Gr}(q,V)$ consists of two copies of the usual Grassmannian. Let
\begin{equation} \label{eq:def_unitary_D}
\mathbb{D} = \{z \in \mathrm{Gr}(q,V) | \  Q|_z <0 \}
\end{equation}
be the open subset of $\mathrm{Gr}(q,V)$ consisting of negative definite subspaces. Then $\mathbb{D}$ has two connected components in case $1$ and is connected in case 2. Let $z_0 \in \mathbb{D}$ be the point corresponding to $V^-$ and $\mathbb{D}^+$ be the connected component of $\mathbb{D}$ containing $z_0$. Then $G$ acts transitively on $\mathbb{D}^+$ and the stabilizer of $z_0$ is $K$; thus $\mathbb{D}^+ \simeq G/K$ is the symmetric domain associated with $G$.
%In case 2 $G$ acts transitively on $\mathbb{D}$, and the stabilizer of $z_0:=V^- \in \mathbb{D}$ is $K$. Thus in this case $\mathbb{D} \simeq G/K$ is the hermitian symmetric domain associated with $G$. In case 1 $\mathbb{D}$ has two connected components. From now on in this case we fix an orientation of $V^-$ and we let $z_0$ be the corresponding point on $\mathbb{D}$ and $\mathbb{D}^+$ be the connected component of $\mathbb{D}$ containing $z_0$. Then $G$ acts transitively on $\mathbb{D}^+$
%Thus a point $z \in \mathbb{D}$ can be identified with a $q$-dimensional (oriented, in case 1) negative definite subspace $V_z^-$ of $V$; we denote by $z_0 \in \mathbb{D}$ the point with $V_{z_0}^-=V^-$ (corresponding to the coset $eK$ under $\mathbb{D} \simeq G/K$).
In case 2 it is clear that $\mathbb{D}$ carries a $\mathrm{U}(V)$-invariant complex structure; to see this in case 1, one can use the model
%We denote by $\mathbb{D}$ be the hermitian symmetric domain attached to $G$. Thus in case 1 a model for $\mathbb{D}$ is given by
\begin{equation} \label{eq:isomorphism_modelsofD_orthogonal}
\mathbb{D} \simeq \{[v] \in \mathbb{P}(V(\mathbb{C})) | Q(v,v)=0, \ Q(v,\overline{v})<0\}.
\end{equation}
The correspondence between both models sends $z \in \mathbb{D}$ to the line $[e_{z}+ie'_{z}] \in \mathbb{P}(V(\mathbb{C}))$, where $e_z$ and $e_z'$ form an oriented orthogonal basis of $z$ satisfying $Q(e_z,e_z)=Q(e_z',e_z')$.

%A model for the hermitian symmetric domain $\mathbb{D} \simeq G/K$ in case 1 is 
%%We denote by $\mathbb{D}$ be the hermitian symmetric domain attached to $G$. Thus in case 1 a model for $\mathbb{D}$ is given by
%\begin{equation}
%\mathbb{D} = \{[v] \in \mathbb{P}(V_\mathbb{C}) | Q(v,v)=0, \ Q(v,\overline{v})<0\}
%\end{equation}
%whereas in case 2, writing $\mathrm{Gr}(q,V)$ for the Grassmannian of subspaces of $V$ of dimension $q$, we have the model
%\begin{equation}
%\mathbb{D} = \{z \in \mathrm{Gr}(q,V) | \  Q|_z <0 \}.
%\end{equation}
%In both cases, the decomposition $V=V^+ \oplus V^-$ defines a point $z_0 \in \mathbb{D}$ (corresponding to the coset $eK$ under $\mathbb{D} \simeq G/K$).

% Tautological bundle, its metric, sections, and special cycles on $\mathbb{D}$
\subsubsection{} \label{subsubsection:taut_bundle_definitions} Let $\vbE$ be the tautological bundle on $\mathbb{D}$, whose fiber over $z \in \mathbb{D}$ is the subspace $z \subset V$. Thus $\vbE$ is a holomorphic line bundle in case 1 (for it corresponds to the pullback of $\mathcal{O}_{\mathbb{P}(V(\mathbb{C}))}(-1)$ under the isomorphism \eqref{eq:isomorphism_modelsofD_orthogonal}), and a holomorphic vector bundle of rank $q$ in case 2. It carries a natural hermitian metric $h_\vbE$ defined by
\begin{equation}
h_\vbE(v_z) = \left\{ \begin{array}{cc} -Q(v_z,\overline{v_z}) & \quad \text{ case 1} \\ -Q(v_z,v_z) & \quad \text{ case 2} \end{array} \right.
\end{equation}
for $v_z \in \vbE_z=z$. %Here in case 1 $\overline{v_z}$ means v_z regarded as an element of $\overline{z}$, which is $z$ with the conjugate complex structure.
This metric is equivariant for the natural $\mathrm{U}(V)$-equivariant structure on $\vbE$. We denote by $\nabla_\vbE$ the corresponding Chern connection on $\vbE$ and by
\begin{equation}
\Omega = \Omega_\vbE := c^{\mathrm{top}}(\vbE,\nabla_\vbE)^* = \left( \tfrac{i}{2\pi} \right)^{\mathrm{rk} \vbE} \det(\nabla_\vbE^2) \in A^{\mathrm{rk} \vbE,\mathrm{rk} \vbE}(\mathbb{D})
\end{equation}
its Chern-Weil form of top degree.

We denote by $\Omega_\mathbb{D}$ the K\"ahler form
\begin{equation}
\Omega_{\mathbb{D}} = \partial \overline{\partial} \log k_{\mathbb{D}}(z,z),
\end{equation}
where $k_{\mathbb{D}}$ is the Bergmann kernel function of $\mathbb{D}$. As shown in \cite[p. 219]{Wells}, the invariant form $-\tfrac{i}{2\pi}\Omega_\mathbb{D}$ agrees with the first Chern form $c_1(\Omega^{\mathrm{top}}_X)$ of the canonical bundle $\Omega^{\mathrm{top}}_X$ on any quotient $X=\Gamma \backslash \mathbb{D}^+$ by a discrete torsion free subgroup $\Gamma \subset G$. When $\vbE$ has rank one, the canonical bundle on $\mathbb{D}$ is naturally isomorphic to $\vbE^{\otimes p}$ in case 1 (as an application of the adjunction formula shows) and to $\vbE^{\otimes (p+1)}$ in case 2, and so in both cases $-\tfrac{i}{2\pi}\Omega_\mathbb{D}$ is a positive integral multiple of $\Omega_\vbE$.

%{\bf TO DO: for the general U(p,q) case we will also need the Kahler form $\Omega_{\mathbb{D}}$. Need to define it here.}

An element $v \in V$ defines a global holomorphic section $s_v$ of $\vbE^\vee$: for $v'_z \in \vbE_z$, we define
\begin{equation}
s_v(v'_z) = Q(v'_z,v).
\end{equation}
Let $\mathbb{D}_v$ be the zero locus of $s_v$ on $\mathbb{D}$ and set $\mathbb{D}_v^+=\mathbb{D}_v \cap \mathbb{D}^+$. We have
\begin{equation}
\mathbb{D}_v = \{z \in \mathbb{D} | v \perp z \}
\end{equation}
and so $\mathbb{D}_v$ is non-empty only if $Q(v,v)>0$ or $v=0$. Assume that $Q(v,v)>0$, so that the orthogonal complement $v^\perp$ of $v$ has signature $(p-1,q)$. Writing $G_v$ for the stabilizer of $v$ in $G$, we find that $G_v^0 \simeq \mathrm{Aut}(v^\perp,Q)^0$ acts transitively on $\mathbb{D}_v^+$ with stabilizers isomorphic to $\mathrm{SO}(p-1) \times \mathrm{SO}(q)$ in case 1 and to $\mathrm{U}(p-1) \times \mathrm{U}(q)$ in case 2. Thus $\mathbb{D}_v^+$  is the symmetric domain attached to $G_v^0$. We conclude that $\mathrm{codim}_\mathbb{D} \mathbb{D}_v = \mathrm{rk} \vbE^\vee$
%\begin{equation}
%\mathrm{codim}_\mathbb{D} \mathbb{D}_v = \mathrm{rk} \vbE^\vee
%\end{equation}
and hence that $s_v$ is a regular section of $\vbE^\vee$. In fact, this shows that $s_v$ is regular whenever $v \neq 0$, since in the remaining case we have $Q(v,v) \leq 0$ and so $s_v$ does not vanish on $\mathbb{D}$.

More generally, given a positive integer $r$ and a vector ${\bf v} =(v_1,\ldots,v_r) \in V^r$, there is a holomorphic section $s_{\bf v} = (s_{v_1},\ldots,s_{v_r})$ of $(\vbE^{\vee})^r$, with zero locus
\begin{equation} \label{eq:def_D_bfv}
\mathbb{D}_{\bf v} := Z(s_{\bf v}) = \cap_{1 \leq i \leq r} \mathbb{D}_{v_i}.
\end{equation}
Note that $\mathbb{D}_{\bf v}$ depends only on the span $\langle v_1,\ldots,v_r \rangle$. It is non-empty if and only if $\langle v_1,\ldots,v_r \rangle$ is a a positive definite subpace of $(V,Q)$ of positive dimension; in that case, its (complex) codimension in $\mathbb{D}$ is $\mathrm{rk}(\vbE) \cdot \dim_\bbK\langle v_1,\ldots,v_r \rangle$. We set $\mathbb{D}_{\bf v}^+=\mathbb{D}_{\bf v} \cap \mathbb{D}^+$.

% {\bf TO DO: I think $s_v$ is regular whenever $v \neq 0$, since in the remaining case we have $Q(v,v) \leq 0$ and so $\mathbb{D}_v$ is empty and this should imply regularity. Think about it.}

%\subsubsection{} We write $\mathcal{V}$ for the trivial bundle on $\mathbb{D}$ with fiber $V_\mathbb{C}$ in case 1 and $V$ and case 2, and identify it with its sheaf of sections.
%%\begin{equation}
%%\mathcal{V} = \left\{ \begin{array}{cc} V_\mathbb{C} \otimes \mathcal{O}_\mathbb{D} & \quad \text{ case 1} \\ V \otimes \mathcal{O}_\mathbb{D} & \quad \text{ case 2} \end{array} \right.
%%\end{equation}
%There is a holomorphic subbundle $\vbE \subset \mathcal{V}$ over $\mathbb{D}$, whose fiber over $z \in \mathbb{D}$ is
%\begin{equation}
%\vbE_z = \left\{ \begin{array}{cc} \mathbb{C}v \text{ if } z=[v] & \quad \text{ case 1} \\ z & \quad \text{ case 2}.  \end{array} \right.
%\end{equation}
%Thus $\vbE$ is the tautological bundle; it is a line bundle in case 1 and a vector bundle of rank $q$ in case 2. 
%

% Definition of the forms $\varphi$, $\nu$ and $\xi$ on hermitian symmetric domains.
\subsubsection{} \label{subsection:main_definitions_forms} We will now specialize the constructions in \Cref{subsection:Superconnections and characteristic forms of Koszul complexes} to the setting of hermitian symmetric domains.

Let $r$ be a positive integer and $\mathbf{v}=(v_1,\ldots,v_r)$ be an $r$-tuple of vectors in $V$. We write $K(\mathbf{v}) := K(s_{\mathbf{v}})$ for the Koszul complex associated with the section $s_{\mathbf{v}} = (s_{v_1},\ldots,s_{v_r})$ of $(\vbE^r)^\vee$. On its underlying vector bundle $\wedge (\vbE^r)$, we consider the superconnection 
\begin{equation}\label{eq:superconnection_nabla_bfv_definition}
\nabla_{\bf v} := \nabla_{s_{\mathbf{v}}} = \nabla + i \sqrt{2\pi}(s_{\mathbf{v}} + s_{\mathbf{v}}^*),
\end{equation} 
and we define forms
\begin{equation}
\begin{split}
\phio({\mathbf{v}}) &:= \phio(s_{\mathbf{v}})^* = \sum_{k \geq 0} \left( \tfrac{i}{2\pi} \right)^k \mathrm{tr_s}(e^{\nabla_{\mathbf{v}}^2})_{[2k]},  \\
\nuo({\mathbf{v}}) &:= \nuo(s_{\mathbf{v}})^* = \sum_{k \geq 0} \left( \tfrac{i}{2\pi} \right)^k \mathrm{tr_s}(Ne^{\nabla_{\mathbf{v}}^2})_{[2k]},
\end{split}
\end{equation}
where $N$ is the number operator on $\wedge (\vbE^r)$ acting on $\wedge^k (\vbE^r)$ by multiplication by $-k$.

%Let $r$ be a positive integer and $\mathbf{v}=(v_1,\ldots,v_r)$ be an $r$-tuple of vectors in $V$. Then $\mathbf{v}$ defines a global holomorphic section $s_{\mathbf{v}} = (s_{v_1},\ldots,s_{v_r})$ of $(\vbE^r)^\vee$. We have a Koszul complex $K(\mathbf{v}) := K(s_{\mathbf{v}})$ with underlying vector bundle $\wedge (\vbE^r)$ and a superconnection 
%\begin{equation}
%\nabla_{\bf v} := \nabla_{s_{\mathbf{v}}} = \nabla + i \sqrt{2\pi}(s_{\mathbf{v}} + s_{\mathbf{v}}^*)
%\end{equation} 
%on $\wedge(\vbE^r)$, and we define forms
%\begin{equation}
%\begin{split}
%\phio({\mathbf{v}}) &:= \phio(s_{\mathbf{v}})^* = \sum_{k \geq 0} \left( \tfrac{i}{2\pi} \right)^k \mathrm{tr_s}(e^{\nabla_{\mathbf{v}}^2})[2k],  \\
%\nuo({\mathbf{v}}) &:= \nuo(s_{\mathbf{v}})^* = \sum_{k \geq 0} \left( \tfrac{i}{2\pi} \right)^k \mathrm{tr_s}(Ne^{\nabla_{\mathbf{v}}^2})[2k],
%\end{split}
%\end{equation}
%where $N$ is the number operator on $\wedge (\vbE^r)$ acting on $\wedge^k (\vbE^r)$ by multiplication by $-k$.

\begin{definition} \label{def:forms_varphi_nu_xi}
	For $\mathbf{v}=(v_1,\ldots,v_r) \in V^r$, write $Q(\mathbf{v},\mathbf{v}) = Q(v_1,v_1) + \cdots + Q(v_r,v_r)$ and define
	\begin{equation*}
	\begin{split}
	\varphi(\mathbf{v}) &= e^{-\pi Q(\mathbf{v},\mathbf{v})} \,  \phio(\mathbf{v}), \\
	\nu(\mathbf{v}) &= e^{-\pi Q(\mathbf{v},\mathbf{v})} \,  \nuo(\mathbf{v}).
	\end{split}
	\end{equation*}
	Thus $\varphi(\mathbf{v})$ and $\nu(\mathbf{v})$ belong to $\oplus_{k \geq 0} A^{k,k}(\mathbb{D})$. 
\end{definition}

The forms $\varphi({\bf v})$ and $\nu({\bf v})$ were already defined and studied in the setting of general period domains in \cite{GarciaSuperconnections}. 

% Assuming that $s_{\bf v}$ is a regular section, Proposition \ref{prop:Green_current_general} shows that $\archGreen({\bf v})$ is smooth on $\mathbb{D}-\mathbb{D}_{\bf v}$, locally integrable on $\mathbb{D}$ and satisfies the equation
%\begin{equation} \label{eq:Green_form_gv}
%\mathrm{dd}^c[\mathfrak{g}({\bf v})]+\delta_{\mathbb{D}_{\bf v}} = [\phio({\bf v})_{[2\mathrm{rk}(\vbE)]}],
%\end{equation}
%so that $\mathfrak{g}({\bf v})$ is a Green form for $\mathbb{D}_{\bf v}$.

%We recall... {\bf TO DO: explain here that $\varphi[q]=\varphi_{\mathrm{KM}}$}.

% Explicit formulas for $\mathrm{O}(p,2)$ and $\mathrm{U}(p,1)$. Relation with work of Kudla.
\subsection{Explicit formulas for $\mathrm{O}(p,2)$ and $\mathrm{U}(p,1)$} \label{subsection:explicit formulas in rank one} Let us give some explicit formulas when the tautological bundle $\vbE$ is a line bundle. Thus $V$ is a either a real vector space of signature $(p,2)$ (case 1) or a complex vector space of signature $(p,1)$ (case 2). There is a unique hermitian metric on $\vbE^\vee$ making the isomorphism $\overline{\vbE} \cong \vbE^\vee$ induced by $h_\vbE$ an isometry; we denote this metric by $h$ and its Chern connection by $\nabla_{\vbE^\vee}$. For $v \in V$ we have
\begin{equation} \label{eq:varphi_explicit_formula}
\begin{split}
\varphi(v)_{[2]} &= e^{-\pi (Q(v,v)+2h(s_v))} \left(i\frac{\partial h(s_v)\wedge\overline{\partial} h(s_v)}{h(s_v)}-\Omega_\vbE \right), \\
\varphi(v) &= \varphi(v)_{[2]} \wedge \mathrm{Td}^{-1}(\vbE^\vee,\nabla)^*,
\end{split}
\end{equation}
where
\begin{equation}
\mathrm{Td}^{-1}(\vbE^\vee,\nabla) = \det \left( \frac{1-e^{-\nabla_{\vbE^\vee}^2}}{\nabla_{\vbE^\vee}^2} \right)
\end{equation}
denotes the inverse Todd form of $(\vbE^\vee,\nabla_{\vbE^\vee})$. This is a special case of the Mathai-Quillen formula \cite[Thm. 8.5]{MathaiQuillen}; see also \cite[\textsection 3]{GarciaSuperconnections} for a proof in our setting, where it is also shown that the form $\varphi(v)_{[2]}$ coincides with the form $\varphi_{\mathrm{KM}}(v)$ defined by Kudla and Millson in \cite{KudlaMillson1}.

Let us now consider the form $\nu(v)$ for $v \in V$. Here the Koszul complex $K(v)$ has just two terms: $K(v) = (\vbE \xrightarrow{s_v} \mathcal{O}_\mathbb{D})$,
%\begin{equation}
%K(v) = \vbE \xrightarrow{s_v} \mathcal{O}_\mathbb{D},
%\end{equation}
and the operator $N$ acts by zero on $\mathcal{O}_\mathbb{D}$ and by $-1$ on $\vbE$. For the component of degree zero of $\nuo(v)$ we obtain
\begin{equation}\label{eq:computation_nu0_degree_0_case_1}
\begin{split}
\nuo(v)_{[0]} &= \mathrm{tr_s}(Ne^{\nabla_v^2})_{[0]} \\
&= \mathrm{tr_s}(Ne^{(\nabla_v^2)_{[0]}}) \\
&= \mathrm{tr_s}(Ne^{-2\pi h(s_v)}) \\
&= e^{-2\pi h(s_v)}.
\end{split}
\end{equation}
Given $z \in \mathbb{D}$, let $z^\perp$ be the orthogonal complement of $z$ in $V$, so that $V=z \oplus z^\perp$; we write $v_z$ and $v_{z^\perp}$ for the orthogonal projection of $v \in V$ to $z$ and $z^\perp$ respectively. Let $Q_z$ be the (positive definite) Siegel majorant of $Q$ defined by
\begin{equation} \label{eq:def_Siegel_majorant}
Q_z(v,v)=Q(v_{z^\perp},v_{z^\perp})-Q(v_z,v_z).
\end{equation}
Then we have $Q_z(v,v) = Q(v,v)+2h_z(s_v)$ and we conclude that $\nu(v)_{[0]}$ is just the Siegel gaussian:
\begin{equation} \label{eq:nu0_and_Siegel_gaussian}
\nu(v)_{[0]}=\varphi^{\mathrm{SG}}(v):=e^{-\pi Q_z(v,v)}.
\end{equation}
If $v \neq 0$, then using \eqref{eq:computation_nu0_degree_0_case_1} we also obtain an explicit formula for the Green function $\archGreen(s_{v})$ defined in \eqref{eq:def_general_g} (see Example \ref{example:g(v)_rkE_1} below).

% Basic properties of the forms $\varphi$ and $\nu$
% I.e. the straightforward ones + the fact that they are Schwartz forms
\subsection{Basic properties of the forms $\varphi$ and $\nu$} \label{subsection:Basic properties of the forms varphi and nu} %We will next establish some useful properties of the forms $\varphi$ and $\nu$. 

% Restriction to special cycles
\subsubsection{} We first give formulas for the restriction of $\varphi$ and $\nu$ to special cycles. Let $w \in V$ with $Q(w,w)>0$ and recall the group $G_w$ and complex submanifold $\mathbb{D}_w^+ \subset \mathbb{D}^+$ defined in \ref{subsubsection:taut_bundle_definitions}. Then $G_w$ is identified with the isometry group of $(w^\perp,Q)$, and we may identify $\mathbb{D}_w^+$ with the hermitian symmetric domain attached to $G_w^0$. We write $\varphi_{\mathbb{D}_w}(v')$ and $\nu_{\mathbb{D}_w}(v')$ for the forms on $\mathbb{D}_w$ given in \Cref{def:forms_varphi_nu_xi}.

\begin{lemma} \label{lemma:restriction_varphi_nu_special_cycle_type_1}
	Let $v,w \in V$ with $Q(w,w)>0$ and write $v=v'+v''$ with $v' \in w^\perp$ and $v'' \in \langle w \rangle$. Then
	\begin{equation*}
	\begin{split}
	\nuo(v)|_{\mathbb{D}_w} = \nuo_{\mathbb{D}_w}(v'), \quad \nu(v)|_{\mathbb{D}_w} &= e^{-\pi Q(v'',v'')} \nu_{\mathbb{D}_w}(v'), \\
	\phio(v)|_{\mathbb{D}_w} = \phio_{\mathbb{D}_w}(v'), \quad \varphi(v)|_{\mathbb{D}_w} &= e^{-\pi Q(v'',v'')} \varphi_{\mathbb{D}_w}(v').
	\end{split}
	\end{equation*}
\end{lemma}
\begin{proof} 
	Let $\vbE_w$ be the tautological bundle on $\mathbb{D}_w$ whose fiber over $z \in \mathbb{D}_w$ is $z \subset w^\perp$. The restriction of $\vbE$ to $\mathbb{D}_w$ is isometric to $\vbE_w$. For any $v \in V$, this isometry induces an isomorphism $K(v)|_{\mathbb{D}_w} \cong K(v'),$
	%\begin{equation}
	%K(v)|_{\mathbb{D}_w} \cong K(v'),
	%\end{equation}
	where $K(v')$ denotes the Koszul complex with underlying vector bundle $\wedge \vbE_w$ and differential $s_{v'}$. Thus $\nabla_v|_{\mathbb{D}_w}=\nabla_{v'}$ and the lemma follows.
\end{proof}

\subsubsection{} The proof of the next proposition is a straightforward consequence of general properties of Koszul complexes and Chern forms.

\begin{proposition} \label{prop:phi_basic_properties}
	Let $r \geq 1$ and ${\bf v} = (v_1,\ldots,v_r) \in V^r$. Then:
	\begin{enumerate}[(a)]
		\item $\varphi({\bf v}) = \varphi(v_1) \wedge \ldots \wedge \varphi(v_r)$.
		\item $\varphi({\bf v})$ is closed.
		\item $\varphi({\bf v})_{[k]}=0$ if $k < 2r \cdot \mathrm{rk}(\vbE)$.
		\item For every $g \in \mathrm{Aut}(V,Q)$, we have $g^*\varphi(gv_1,\ldots, gv_r) = \varphi(v_1,\ldots,v_r)$. %\todo{ The notation $U(V)$ seems to suggest the unitary case to me. LG: I agree.}
		\item $\varphi(0) = c_{\mathrm{rk}(\vbE)}(\vbE^\vee,\nabla)^* \wedge \mathrm{Td}^{-1}(\vbE^\vee,\nabla)^*$ (here we assume $r=1$).
		\item Let 
		\begin{equation*}
		h \in \left\{ \begin{array}{cc} \mathrm{O}(r), & \quad \text{ case 1 } \\ \mathrm{U}(r), & \quad \text{ case 2.} \end{array} \right.
		\end{equation*} 
		Then $\varphi((v_1,\ldots,v_r) \cdot h) = \varphi(v_1,\ldots,v_r)$.
	\end{enumerate}
\end{proposition}
\begin{proof}
	Except for $(c)$, which follows from $(a)$ and the Mathai-Quillen formula \cite[Thm. 8.5]{MathaiQuillen}, all statements are proved in \cite[Prop. 2.3 and (2.17)]{GarciaSuperconnections} in case 1, and the proof there extends without modification to case 2.
\end{proof}

Consider now the form $\nu$. We write 
\begin{equation}
c(F,\nabla_F)=\det(t\nabla_F^2+1_{\mathrm{rk}(F)})=1+c_1(F,\nabla_F)t+\ldots+c_{\mathrm{rk}(F)}(F,\nabla_F)t^{\mathrm{rk}(F)}
\end{equation}
for the total Chern-Weil form of a vector bundle $F$ with connection $\nabla_F$.

\begin{proposition} \label{prop:nu_properties}
	Let $r \geq 1$ and ${\bf v} = (v_1,\ldots,v_r) \in V^r$.
	\begin{enumerate}[(a)]
		\item We have $\nu({\bf v}) = \sum_{1\leq i \leq r} \nu_i({\bf v})$, where
		\begin{equation*}
		\nu_i({\bf v}) := \nu(v_i) \wedge \varphi(v_1,\ldots,\hat{v_i},\ldots,v_r).
		\end{equation*}
		\item For $t >0$:
		\begin{equation*}
		\mathrm{dd}^c \nuo(t^{1/2}{\bf v})=-t\frac{d}{dt}\phio(t^{1/2}{\bf v}).
		\end{equation*}
		\item $\nu({\bf v})_{[k]}=0$ if $k < 2r \cdot \mathrm{rk}(\vbE)-2$.
		\item For any $g \in \mathrm{Aut}(V,Q)$, we have $g^*\nu(gv_1,\ldots,gv_r) = \nu(v_1,\ldots,v_r)$.
	 \item For the zero vector $\mathbf 0 \in V^r$, we have
		\[
				\nu(\mathbf 0)_{[2  r \cdot \rk(\vbE) - 2]} \ =\    r \cdot c_{\mathrm{rk}(\vbE)-1}(\vbE^\vee,\nabla)^* \, \wedge \, \left( c_{\rk(\vbE)} (\vbE^{\vee}, \nabla )^*\right) ^{r-1}.
		\]
		In particular, $\nu(0)_{[0]}=1$ when $r= \rk(\vbE) =1 $. %\todo{might as well take general $r$ here, we use it later. LG: Agreed.}
%		\item Assume that $r=1$. Then 
%		\begin{equation*}
%		\nu(0)_{[2\mathrm{rk}(\vbE)-2]}= c_{\mathrm{rk}(\vbE)-1}(\vbE^\vee,\nabla)^*.
%		\end{equation*}
%		(In particular, $\nu(0)_{[0]}=1$ when $\vbE$ has rank one.)
		%\item $\nu(0)[2\mathrm{rk}(\vbE)-2]= \left\{ \begin{array}{cc} 1, & \quad \text{ case 1} \\ ?, & \quad \text{ case 2.} \end{array} \right.$
		\item Let 
		\begin{equation*}
		h \in \left\{ \begin{array}{cc} \mathrm{O}(r), & \quad \text{ case 1 } \\ \mathrm{U}(r), & \quad \text{ case 2.} \end{array} \right.
		\end{equation*} 
		Then $\nu((v_1,\ldots,v_r) \cdot h) = \nu(v_1,\ldots,v_r)$.
	\end{enumerate}
\end{proposition}
\begin{proof}
	Recall that $\nu({\bf v}) = e^{-\pi Q({\bf v},{\bf v})} \mathrm{tr_s}(Ne^{\nabla_{{\bf v}}^2})$, where $N$ is the number operator on the Koszul complex $K({\bf v}) \simeq \otimes_{1 \leq i \leq r}K(v_i)$. Letting $N_i$ be the number operator on $K(v_i)$, we can write $N=N_1 + \cdots +N_r$ and $\nabla_{\bf v}^2=\nabla_{v_1}^2+\cdots + \nabla_{v_r}^2$, where $[N_i,\nabla_{v_j}^2]=[\nabla_{v_i}^2,\nabla_{v_j}^2]=0$ for $i \neq j$; this proves $(a)$.
	
	Part $(b)$ follows from \eqref{eq:nu0_transgression}. By part $(a)$ and \Cref{prop:phi_basic_properties}.$(c)$, it suffices to prove $(c)$ when $r=1$. Then $(c)$ is vacuously true if $\mathrm{rk}(\vbE)=1$, and in general it follows from \cite[(3.72),(3.35), Thm. 3.10]{BGSciag}.
	
	For any ${\bf v}$ and any $g \in \mathrm{U}(V)$, the $\mathrm{U}(V)$-equivariant structure on $\vbE$ induces an isomorphism $g^*K(g{\bf v}) \simeq K({\bf v})$ preserving the metric; this proves $(d)$.

	For part $(e)$, first consider the case $r=1$. When $\vbE $ also has rank one, the desired relation follows immediately from \eqref{eq:nu0_and_Siegel_gaussian}. For general $\vbE$, by taking  $u=0$ in \cite[(3.35)]{BGSciag} we find that
		\begin{equation}
		\begin{split}
		\mathrm{tr_s}(Ne^{\nabla^2})_{[2\mathrm{rk}(\vbE)-2]}^* &=- \frac{d}{db}\left. \det \left(\tfrac{i}{2\pi}\nabla_{\vbE^\vee}^2-b1_{\mathrm{rk}(\vbE)}\right)\right|_{b=0} \\
		&= (-1)^{r-1} \frac{d}{db}\left. \det \left(\tfrac{i}{2\pi}\nabla_{\vbE}^2+b1_{\mathrm{rk}(\vbE)}\right)\right|_{b=0} \\
		%&= (-1)^{r-1} \frac{d}{db}b^{\mathrm{rk}(\vbE)} \left. \det \left(\tfrac{i}{2\pi}b^{-1}\nabla_{\vbE}^2+1_{\mathrm{rk}(\vbE)}\right)\right|_{b=0} \\
		%&= (-1)^{r-1} \frac{d}{db}  b^{\mathrm{rk}(\vbE)} \sum_{k=0}^{\mathrm{rk}(\vbE)} c_i(\vbE,\nabla)(b^{-1})^i  |_{b=0} \\
		%&= (-1)^{r-1} \frac{d}{db} \sum_{k=0}^{\mathrm{rk}(\vbE)} c_i(\vbE,\nabla)b^{\mathrm{rk}(\vbE)-i}  |_{b=0} \\
		&= (-1)^{r-1} c_{\mathrm{rk}(\vbE)-1}(\vbE,\nabla)^* \\
		&= c_{\mathrm{rk}(\vbE)-1}(\vbE^\vee,\nabla)^*.
		\end{split}
		\end{equation}
		Note that the number operator $N$ in op.\ cit.\ has sign opposite from ours.  The formula for general $r$  follows from part $(a)$ above together with \Cref{prop:phi_basic_properties}(a),(e).
%	
%	Part $(e)$ follows from the explicit computation \eqref{eq:nu0_and_Siegel_gaussian} when $\vbE$ has rank one. In the general case, by taking $u=0$ in \cite[(3.35)]{BGSciag} we find that
%	\begin{equation}
%	\begin{split}
%	\mathrm{tr_s}(Ne^{\nabla^2})_{[2\mathrm{rk}(\vbE)-2]}^* &=- \frac{d}{db}\left. \det \left(\tfrac{i}{2\pi}\nabla_{\vbE^\vee}^2-b1_{\mathrm{rk}(\vbE)}\right)\right|_{b=0} \\
%	&= (-1)^{r-1} \frac{d}{db}\left. \det \left(\tfrac{i}{2\pi}\nabla_{\vbE}^2+b1_{\mathrm{rk}(\vbE)}\right)\right|_{b=0} \\
%	%&= (-1)^{r-1} \frac{d}{db}b^{\mathrm{rk}(\vbE)} \left. \det \left(\tfrac{i}{2\pi}b^{-1}\nabla_{\vbE}^2+1_{\mathrm{rk}(\vbE)}\right)\right|_{b=0} \\
%	%&= (-1)^{r-1} \frac{d}{db}  b^{\mathrm{rk}(\vbE)} \sum_{k=0}^{\mathrm{rk}(\vbE)} c_i(\vbE,\nabla)(b^{-1})^i  |_{b=0} \\
%	%&= (-1)^{r-1} \frac{d}{db} \sum_{k=0}^{\mathrm{rk}(\vbE)} c_i(\vbE,\nabla)b^{\mathrm{rk}(\vbE)-i}  |_{b=0} \\
%	&= (-1)^{r-1} c_{\mathrm{rk}(\vbE)-1}(\vbE,\nabla)^* \\
%	&= c_{\mathrm{rk}(\vbE)-1}(\vbE^\vee,\nabla)^*.
%	\end{split}
%	\end{equation}
%	Note that the number operator $N$ in op.\ cit.\ has sign opposite from ours. % It does because they place the Koszul complex in positive degrees and we place it in negative ones.

	To prove $(f)$, note that $h$ induces an isometry 
	\begin{equation}
	K(v_1,\ldots,v_r) \xrightarrow{i(h)} K((v_1,\ldots,v_r)\cdot h)
	\end{equation}
	that commutes with $N$ and such that $\nabla_{(v_1,\ldots,v_r)\cdot h}=i(h)^{-1}\nabla_{(v_1,\ldots,v_r)}i(h)$. Thus $(f)$ follows from the conjugation invariance of $\mathrm{tr_s}$.
\end{proof}

The properties of $\varphi({\bf v})$ and $\nu({\bf v})$ established in Propositions \ref{prop:phi_basic_properties} and \ref{prop:nu_properties} also hold for $\phio({\bf v})$ and $\nuo({\bf v})$. In contrast, the next result, showing that the forms $\varphi$ and $\nu$ are rapidly decreasing as functions of ${\bf v}$, really requires the additional factor $e^{-\pi Q({\bf v},{\bf v})}$ to hold (note for example that the restriction of $\phio(t{\bf v})$ to $\mathbb{D}_{\bf v}$ is independent of $t \in \mathbb{R}$).

Let $\mathcal{S}(V^r)$ be the Schwartz space of complex-valued smooth functions on $V^r$ all whose derivatives are rapidly decreasing.

\begin{lemma} \label{lem:Phi and Nu are Schwartz}
	For fixed $z \in \mathbb{D}$ and $r \geq 1$, we have
	\begin{equation*}
	\varphi(\cdot,z), \nu(\cdot,z) \in \mathcal{S}(V^r) \otimes \wedge T_{z}^*\mathbb{D}.
	\end{equation*}
\end{lemma}
\begin{proof}
	By \Cref{prop:phi_basic_properties}.(a) and \Cref{prop:nu_properties}.(a), we may assume that $r=1$. Recall that the quadratic form $Q_z(v) = \tfrac{1}{2}Q(v,v)+h_z(s_v)$ on $V$ is positive definite. Write $\nabla_{v}^2(z)=(\nabla_{v}^2)_{[0]}(z)+S(v,z)$. By Duhamel's formula (\cite[p. 144]{SouleBook}) we have
	\begin{equation} \label{eq:Duhamel}
	\begin{split}
	& e^{-\pi Q(v,v)} e^{\nabla_v^2(z)} = e^{-2\pi Q_z(v)} \\
	& + \sum_{k \geq 1} (-1)^k \int_{\Delta^k} e^{-2\pi (1-t_k)Q_z(v)}S(v,z)e^{-2\pi (t_k-t_{k-1})Q_z(v)} \cdots S(v,z) e^{-2\pi t_1 Q_z(v)} dt_1 \cdots dt_k.
	\end{split}
	\end{equation}
	
	Here $\Delta^k=\{(t_1,\ldots,t_k) \in \mathbb{R}^k| 0 \leq t_1 \leq \ldots \leq t_k \leq 1\}$ is the $k$-simplex and the sum is finite since $S(v,z)$ has positive degree. Let $\| \cdot \|_{\overline{U},k}$ be an algebra seminorm as in \Cref{subsection:superconnection_Koszul_complex} and let $Q_{\overline{U}}$ a positive definite quadratic form on $V$ such that $Q_{\overline{U}} < Q_z$ for all $z \in \overline{U}$. Then
	\begin{equation}
	\| e^{-2\pi Q_z(v)} \|_{\overline{U},k} < C e^{-2\pi Q_{\overline{U}}(v)}, \quad v \in V,
	\end{equation}
	for some positive constant $C$. Since $S(v,z)$ grows linearly with $v$, \eqref{eq:Duhamel} implies that
	\begin{equation} \label{eq:bound_Chern_2}
	\| e^{-\pi Q(v,v)} e^{\nabla_v^2(z)} \|_{\overline{U},k} < C e^{-2\pi Q_{\overline{U}}(v)}, \quad v \in V,
	\end{equation}
	(with different $C$) and hence that the same bound holds for $\| \phi(v) \|_{\overline{U},k}$ where $\phi(v)$ is any derivative of $\varphi$ or $\nu$ in the $v$ variable. %{\bf TO DO: finish this by transcribing the argument leading to \cite[(4.7)]{Garcia}}
\end{proof}

% Behaviour under the Weil representation
\subsection{Behaviour under the Weil representation} \label{subsection:Behaviour under the Weil representation}

% Definition of the group on the symplectic/skew-Hermitian side.
\subsubsection{} \label{subsubsection:symplectic_group_definitions} Let $r$ be a positive integer. We write $0$ and $1_r$ for the identically zero and identity $r$-by-$r$ matrices respectively. Consider the vector space $W_r:=\bbK^r$ endowed with the $\sigma$-skew-Hermitian form determined by
\begin{equation}
\begin{pmatrix} 0 & 1_r \\ -1_r & 0 \end{pmatrix}.
\end{equation}
Its isometry group is the symplectic group
\begin{equation}
\mathrm{Sp}_{2r}(\mathbb{R}) = \left\{ g \in \mathrm{GL}_{2r}(\mathbb{R})  \, \left| \, g \left(\begin{smallmatrix} 0 & 1_r \\ -1_r & 0 \end{smallmatrix} \right) {^t g} = \left(\begin{smallmatrix} 0 & 1_r \\ -1_r & 0 \end{smallmatrix} \right) \right. \right\}
\end{equation}
in case 1 and the quasi-split unitary group
\begin{equation}
\mathrm{U}(r,r) = \left\{g \in \mathrm{GL}_{2r}(\mathbb{C})  \ \left| \  g \left(\begin{smallmatrix} 0 & 1_r \\ -1_r & 0 \end{smallmatrix} \right) {^t \overline{g}} = \left(\begin{smallmatrix} 0 & 1_r \\ -1_r & 0 \end{smallmatrix} \right) \right. \right\}
\end{equation}
in case 2. Denote by $\mathrm{Mp}_{2r}(\mathbb{R})$ the metaplectic double cover of $\mathrm{Sp}_{2r}(\mathbb{R})$, and identify it (as a set) with $\mathrm{Sp}_{2r}(\mathbb{R}) \times \{ \pm 1 \}$ as in \cite{RallisHoweDuality}. Define
\begin{equation} \label{eq:def_Gr'_archimedean}
G_r' =  \begin{cases} \mathrm{Mp}_{2r}(\mathbb{R}), & \text{orthogonal case} \\ \mathrm{U}(r,r), & \text{unitary case.} \end{cases}
\end{equation}

Let $N_r$ and $M_r$ be the subgroups of $G'_r$ given by
\begin{align}
N_r&=\left\{ (n(b),1) \ | \  b \in \mathrm{Sym}_r(\mathbb{R}) \right\} \\
M_r&=\left\{ (m(a),\epsilon) \ \left|  \ a \in \mathrm{GL}_r(\mathbb{R}), \epsilon = \pm 1 \right. \right\}
\end{align}
in case 1 and by
\begin{equation}
\begin{split}
N_r &= \left\{ n(b) | b \in \mathrm{Her}_{r} \right\}, \\
M_r &= \left\{ m(a) | a \in \mathrm{GL}_{r}(\mathbb{C}) \right\}
\end{split}
\end{equation}
in case 2.

We also fix a maximal compact subgroup $K'_r$ of $G_r'$ as follows. In the orthogonal case,  let $K'_r$ be the inverse image under the metaplectic cover of the standard maximal compact subgroup
\begin{equation} \label{eq:def_max_cmpct_Sp2r}
\left\{\left. \begin{pmatrix} a & -b \\ b & a \end{pmatrix} \right| a+ib \in \mathrm{U}(r) \right\} \cong \mathrm{U}(r)
\end{equation}
of $\mathrm{Sp}_{2r}(\mathbb{R})$. In the unitary case, we define $K'_r = G_r' \cap \mathrm{U}(2r)$. Thus in this case $K'_r \simeq \mathrm{U}(r) \times \mathrm{U}(r)$; an explicit isomorphism $\mathrm{U}(r) \times \mathrm{U}(r) \xrightarrow{\simeq} K'_r $ is given by
\begin{equation} \label{eq:isom_Kr_case2}
%\mathrm{U}(r) \times \mathrm{U}(r) &\xrightarrow{\simeq} K'_r \\
(k_1,k_2) \mapsto [k_1,k_2] := \frac{1}{2}\begin{pmatrix} k_1+k_2 & -ik_1+ik_2 \\ ik_1-ik_2 & k_1+k_2\end{pmatrix}.
\end{equation}
%{\bf TO DO: check this again}. DONE!

% Highest weights for $K'_r$.
\subsubsection{} \label{section: Arch Weil representation} Recall that we have fixed the additive character $\psi(x) = e^{2 \pi i x}$; in the unitary case, we  also fix a character $\chi=\chi_V$ of $\mathbb{C}^\times$ such that $\chi|_{\mathbb{R}^\times} = \mathrm{sgn}(\cdot)^m$. 
% Reference: Ichino's paper "On the Siegel-Weil formula for unitary groups"
Then $G_r' \times \mathrm{U}(V)$ acts on $\mathcal{S}(V^r)$ via the Weil representation
\begin{equation}
\omega  = \begin{cases} \omega_{\psi}, & \text{orthogonal case,} \\ \omega_{\psi,\chi}, & \text{unitary case.} \end{cases} 
\end{equation}
(see \cite[\textsection 1]{HarrisKudlaSweet}). Here the action of $\mathrm{U}(V)$ is in both cases given by
\begin{equation}
\omega(g)\phi({\bf v}) = \phi(g^{-1}{\bf v}), \quad g \in \mathrm{U}(V), \ \phi \in \mathcal{S}(V^r).
\end{equation}
To describe the action of $G_r'$, we write
\begin{equation}
\begin{split}
\underline{m}(a) &= \begin{cases} (m(a),1), & \text{ for } a \in \mathrm{GL}_r(\mathbb{R}) \text{ in case 1,} \\ m(a), & \text{ for } a \in \mathrm{GL}_r(\mathbb{C}) \text{ in case 2,} \end{cases} \\
\underline{n}(b) &= \begin{cases} (n(b),1), & \text{ for } b \in \mathrm{Sym}_r(\mathbb{R}) \text{ in case 1,} \\ n(b), & \text{ for } b \in \mathrm{Her}_r \text{ in case 2,} \end{cases} \\
\underline{w}_r &= \begin{cases} (w_r,1), & \text{ case 1,} \\ w_r, & \text{ case 2.} \end{cases}\\
\end{split}
\end{equation}
Let $|\cdot|_\bbK$ denote the normalized absolute value on $\bbK$; thus $|z|_\mathbb{R}=|z|$ and $|z|_\mathbb{C} = z\overline{z}$. Then 
\begin{equation} \label{eq:Weil_rep_formulas}
\begin{split}
\omega(\underline{m}(a))\phi({\bf v}) \ &= \  |\det a|_{\bbK}^{m/2}\phi({\bf v}\cdot a) \cdot \begin{cases} \chi_\psi(\det a) & \text{ case 1 } \\ \chi(\det a) & \text{ case 2,} \end{cases} \\
\omega(\underline{n}(b))\phi({\bf v})  \ &= \  \psi \left(\mathrm{tr}(bT({\bf v}))\right)  \, \phi({\bf v}), \\
\omega(\underline{w}_r)\phi({\bf v}) \  &= \  \gamma_{V^r}  \, \hat{\phi}({\bf v});
\end{split}
\end{equation}
see \cite{KudlaSplitting} and also \cite[\textsection 4.2]{IchinoPullbacks} for explicit formulas for $\chi_\psi$ and $\gamma_{V^r}$.
Here for ${\bf v}=(v_1,\ldots,v_r)$ we define $T({\bf v})=\tfrac{1}{2}(Q(v_i,v_j))$ and $\hat{\phi}({\bf v})$ denotes the Fourier transform
\begin{equation}
\hat{\phi}({\bf v}) \ = \  \int_{V^r} \phi(\mathbf{w}) \, \psi \left( \tfrac{1}{2}\mathrm{tr}_{\mathbb{C}/\mathbb{R}}\mathrm{tr}(Q(\mathbf{v},\mathbf{w}))\right) d\mathbf{w},
\end{equation}
where $d\mathbf{w}$ is the self-dual Haar measure on $V^r$ with respect to $\psi$.

% The irreducible representation of $K'$ spanned by $\nu$

\subsubsection{} \label{subsubsection:varphi_and_nu_under_Weil_rep} For our purposes it is crucial to understand the action of $K'_r$ on the Schwartz forms $\varphi$ and $\nu$. This was done for the form $\varphi$ in \cite{KudlaMillson1}, where it is shown that $\varphi$ spans a one-dimensional representation of $K'_r$. We will show that $\nu$ also generates an irreducible representation of $K'_r$. To describe it we next recall the parametrization of irreducible representations of $K'_r$ by highest weights; we denote by $\pi_\lambda$ the (unique up to isomorphism) representation with highest weight $\lambda$. 
%Given a dominant weight $\lambda=(l_1,\ldots,l_r) \in \mathbb{Z}^r$ (here dominant means that $l_1 \geq \ldots \geq l_r$), we denote by $\pi_\lambda$ the irreducible representation (unique up to isomorphism) of $\mathrm{U}(r)$ with highest weight $\lambda$.

%To describe it, recall the parametrization of irreducible representations of $K'_r$ by highest weights; we denote by $\pi_\lambda$ the (unique up to isomorphism) representation with highest weight $\lambda$. A highest weight for $\mathrm{U}(r)$ is a tuple $\lambda = (l_1,\ldots,l_r) \in \mathbb{Z}^r$ such that $l_1 \geq \cdots \geq l_r$. 
Consider first the orthogonal case. Then $K'_r$ is a double cover of $\mathrm{U}(r)$ and its irreducible representations are $\pi_\lambda$ with 
\begin{equation} \label{weight_notation_case_1}
\lambda=(l_1,\ldots,l_r) \in \mathbb{Z}^r \cup (\tfrac{1}{2}+\mathbb{Z})^r, \quad l_1 \geq \cdots \geq l_r.
\end{equation}
For an integer $k$, we write ${\det}^{k/2}$ for the character of $K_r'$ whose square factors through $\mathrm{U}(r)$ and defines the $k$-th power of the usual determinant character $\det\colon\mathrm{U}(r) \to \mathbb{C}^\times$. Kudla and Millson \cite[Thm. 3.1.(ii)]{KudlaMillson1} show that
\begin{equation} \label{eq:varphi_KM_weight_1}
\omega(k')\varphi = \det(k')^{m/2} \varphi, \quad k' \in K'_r,
\end{equation}
and so $\varphi$ affords the one-dimensional representation $\pi_{l}$ of $K'_r$ with highest weight 
\begin{equation} \label{eq:def_l0_case_1}
l:=\tfrac{m}{2}(1,\ldots,1).
\end{equation}

Now consider the unitary case. Using the isomorphism \eqref{eq:isom_Kr_case2}, %\todo{changed confusing wording. LG: Yes, this is better.}
 any irreducible representation of $K'_r$  is isomorphic to
 %we can write uniquely (up to isomorphism) any irreducible representation of $K'_r$ as 
 $\pi_{\lambda_1} \boxtimes \pi_{\lambda_2}$ for a unique pair $\lambda=(\lambda_1,\lambda_2)$ of dominant weights of $\mathrm{U}(r)$. Let $k(\chi)$ be the unique integer such that $\chi(z)=\left(z/|z|\right)^{k(\chi)}$, and note that $k(\chi)$ and $m$ have the same parity. Kudla and Millson show that 
\begin{equation} \label{eq:varphi_KM_weight_2}
\omega([k_1,k_2])\varphi = (\det k_1)^{(m+k(\chi))/2}(\det k_2)^{(-m+k(\chi))/2} \varphi, \quad k_1,k_2 \in \mathrm{U}(r),
\end{equation}
and so $\varphi$ generates the one-dimensional representation $\pi_{l}$ of $K'_r$ with highest weight 
\begin{equation}\label{eq:def_l0_case_2}
l:=(\tfrac{m+k(\chi)}{2}(1,\ldots,1),\tfrac{-m+k(\chi)}{2}(1,\ldots,1)).
\end{equation}

We will now determine the $K'_r$ representation generated by the Schwartz form $\nu({\bf v})_{[2r-2]}$. Recall that by \Cref{prop:nu_properties} we can write
\begin{equation}
\nu({\bf v}) = \sum_{1 \leq i \leq r} \nu_i({\bf v}),
\end{equation}
where
\begin{equation} 
\nu_i({\bf v}) = e^{-\pi Q(v_i,v_i)}\mathrm{tr_s}(N_i e^{\nabla_{v_i}^2}) \wedge \varphi(v_1,\ldots,\hat{v_i},\ldots,v_r).
\end{equation}

Let $\epsilon_r=1_r$ and, for $1 \leq i <r$, set
\begin{equation} \label{eq:def_epsilon_i}
\epsilon_i= \begin{pmatrix} 1_{i-1} & & & \\ & 0 & & 1 \\ & & 1_{r-i-1} & \\ & -1 & & 0 \end{pmatrix} \in \mathrm{SO}(r).
\end{equation}
Then
\begin{equation}
\nu_i({\bf v}) = \omega(\underline{m}(\epsilon_i))\nu_r({\bf v}), \quad 1 \leq i \leq r;
\end{equation}
thus $\nu(\mathbf{v})_{[2r-2]}$ belongs to the $K'_r$-representation generated by $\nu_r(\mathbf{v})_{[2r-2]}$.

Define a weight $\lambda_0$ of $K_r'$ by setting
\begin{subequations} \label{eq:def_lambda_0}
	\begin{equation} 
		\lambda_0 := (\tfrac{m}{2},\ldots,\tfrac{m}{2}, \tfrac{m}{2}-2)
	\end{equation} 
	in the orthogonal case, and 
	\begin{equation} 
	\lambda_0 := ((\tfrac{m+k(\chi)}{2},\ldots,\tfrac{m+k(\chi)}{2},\tfrac{m+k(\chi)}{2}-1),(\tfrac{-m+k(\chi)}{2},\ldots,\tfrac{-m+k(\chi)}{2},\tfrac{-m+k(\chi)}{2}+1))
	\end{equation}
	in the unitary case.
\end{subequations}

%\begin{equation} \label{eq:def_lambda_0_case_1}
%\lambda_0 := (\tfrac{m}{2},\ldots,\tfrac{m}{2}, \tfrac{m}{2}-2)
%\end{equation}
%in case 1 and
%\begin{equation} \label{eq:def_lambda_0_case_2}
%\lambda_0 := ((\tfrac{m+k(\chi)}{2},\ldots,\tfrac{m+k(\chi)}{2},\tfrac{m+k(\chi)}{2}-1),(\tfrac{-m+k(\chi)}{2},\ldots,\tfrac{-m+k(\chi)}{2},\tfrac{-m+k(\chi)}{2}+1))
%\end{equation}
%in case 2.
%
%\begin{equation} \label{eq:def_lambda_0}
%\lambda_0 := l \ + \ \begin{cases} (0, \dots, 0, -2), & \text{orthogonal case} \\ \big( (0, \dots , - 1), (0, \dots, 1) \big),  & \text{unitary case.} \end{cases}
%\end{equation}
%\todo{previous version very hard to read, is this OK? LG: Yes, I prefer this too.}
%
%\begin{equation} \label{eq:def_lambda_0}
%\lambda_0 := \left\{ \begin{array}{cc} (\tfrac{m}{2},\ldots,\tfrac{m}{2}, \tfrac{m}{2}-2), & \text{case 1} \\ ((\tfrac{m+k(\chi)}{2},\ldots,\tfrac{m+k(\chi)}{2},\tfrac{m+k(\chi)}{2}-1),(\tfrac{-m+k(\chi)}{2},\ldots,\tfrac{-m+k(\chi)}{2},\tfrac{-m+k(\chi)}{2}+1)), & \text{case 2.} \end{array} \right.
%\end{equation}
\begin{lemma} \label{lemma:weight_lambda_0}
	Let $r \geq 1$ and assume that $\mathrm{rk}(\vbE)=1$. Under the action of $K'_r$ via the Weil representation, the form $\nu(\mathbf{v})_{[2r-2]}$ %\in \mathcal{S}(V^r) \otimes \wedge T_z^*\mathbb{D}$ 
	generates an irreducible representation $\pi_{\lambda_0}$ with highest weight $\lambda_0$. The form $\nu_r(\mathbf{v})_{[2r-2]}$ is a highest weight vector in $\pi_{\lambda_0}$.
\end{lemma}

\begin{proof}
	%Consider case 1 first. Let $\mathrm{U}(r-1) \times \mathrm{U}(1) \to \mathrm{U}(r)$ be the standard (block diagonal) embedding and let $\tilde{\mathrm{U}}(1)$ (resp. $\tilde{\mathrm{U}}(r-1)$) be the inverse image of $\mathrm{U}(1)$ (resp. $\mathrm{U}(r-1)$) in $K_r'$. Recall that for $k \in \tilde{\mathrm{U}}(r-1)$ we have
	%\begin{equation}
	%\omega(k)\varphi(v_1,\ldots,v_{r-1}) = (\det k)^{m/2} \varphi(v_1,\ldots,v_{r-1})
	%\end{equation}
	%and for $k \in \tilde{\mathrm{U}}(1)$ we have
	%\begin{equation}
	%\omega(k)\nu(v_r) = (\det k)^{m/2-2}\nu(v_r).
	%\end{equation}
	%Note also that $\nu(z)$ has Howe degree $2r-2$; the lemma now follows from \cite[Prop. 4.2.1]{HarrisKudlaGSp2} or by a direct computation.
	%
	%{\bf TO DO: explain in more detail + case 2.}
	Assume first that $r=1$. By \eqref{eq:nu0_and_Siegel_gaussian}, $\nu(z)_{[0]}$ is the Siegel gaussian, which has weight
	\begin{equation} \label{eq:siegel_gaussian_weight_1}
	\lambda_0 = \tfrac{p-q}{2}=\tfrac{m}{2}-2
	\end{equation}
	in case 1 and
	\begin{equation} \label{eq:siegel_gaussian_weight_2}
	\lambda_0 = (\tfrac{p-q+k(\chi)}{2},\tfrac{q-p+k(\chi)}{2})=(\tfrac{m+k(\chi)}{2}-1,\tfrac{-m+k(\chi)}{2}+1)
	\end{equation}
	in case 2. Now assume that $r>1$. By \Cref{prop:phi_basic_properties}.(c), we have
	\begin{equation}
	\nu_r({\bf v})_{[2r-2]} = \nu(v_r)_{[0]} \cdot \varphi(v_1,\ldots,v_{r-1})_{[2r-2]}
	\end{equation} 
	and hence by \eqref{eq:varphi_KM_weight_1}, \eqref{eq:varphi_KM_weight_2}, \eqref{eq:siegel_gaussian_weight_1} and \eqref{eq:siegel_gaussian_weight_2}, the form $\nu_r({\bf v})_{[2r-2]}$ has weight $\lambda_0$. To show that $\nu_r({\bf v})_{[2r-2]}$ is a highest weight vector, one needs to check that $\omega(\alpha)\nu_r({\bf v})_{[2r-2]}=0$ for every compact positive root $\alpha \in \Delta_c^+$ (see \eqref{eq:roots_of_g} and \eqref{eq:roots_of_g_2}). This can be done by a direct computation using \eqref{eq:Weil_rep_formulas} and the explicit formulas \eqref{eq:nu0_and_Siegel_gaussian} and \eqref{eq:varphi_explicit_formula} for $\nu_{[0]}$ and $\varphi$, or alternatively as follows. Evaluating at $z_0$, \eqref{eq:nu0_and_Siegel_gaussian} and \eqref{eq:varphi_explicit_formula} show that
	\begin{equation}
	\nu_r(\mathbf{v},z_0) \in (S(V^r) \otimes \wedge \mathfrak{p}^*)^K;
	\end{equation}
	here $S(V^r) \subset \mathcal{S}(V^r)$ is the subspace spanned by functions of the form $e^{-\pi Q_{z_0}(\mathbf{v},\mathbf{v})}p(\mathbf{v})$, where $p(\mathbf{v})$ is a polynomial on $V^r$. For $\nu_r(\mathbf{v},z_0)$, the degree of these polynomials (called the Howe degree) is $2r-2$. Now it follows from the formulas in \cite[\textsection III.6]{KashiwaraVergne} (see also \cite[Prop. 4.2.1]{HarrisKudlaGSp2} in case 1)
that the only $K_r'$-representation containing the weight $\lambda_0$ realized in $S(V^r)$ in Howe degree $2r-2$ is $\pi_{\lambda_0}$, and the statement follows.
	%and consider case 1 first. Let $\mathrm{U}(r-1) \times \mathrm{U}(1) \to \mathrm{U}(r)$ be the standard (block diagonal) embedding and let $\tilde{\mathrm{U}}(1)$ (resp. $\tilde{\mathrm{U}}(r-1)$) be the inverse image of $\mathrm{U}(1)$ (resp. $\mathrm{U}(r-1)$) in $K_r'$. Recall that for $k \in \tilde{\mathrm{U}}(r-1)$ we have
	%\begin{equation}
	%\omega(k)\varphi(v_1,\ldots,v_{r-1}) = (\det k)^{m/2} \varphi(v_1,\ldots,v_{r-1});
	%\end{equation}
	%together with \eqref{eq:siegel_gaussian_weight_1}, this shows that $\nu_r$ has weight $\lambda_0$.
	%Note also that $\nu(z)$ has Howe degree $2r-2$; the lemma now follows from \cite[Prop. 4.2.1]{HarrisKudlaGSp2} or by a direct computation.
\end{proof}

\subsection{Green forms} \label{subsection:Green forms}

\subsubsection{} \label{subsection:green_form_on_D} In this section we will construct certain currents on $\mathbb{D}$ depending on a parameter $\bfv=(v_1,\ldots,v_r) \in V^r$ and having singularities at $\mathbb{D}_{\bfv}$. We begin with the following special case.

\begin{definition} \label{def:regular tuple}
A tuple $\bfv = (v_1,\ldots,v_r) \in V^r$ is non-degenerate if $\{v_1,\ldots,v_r\}$ is linearly independent. We say that $\bfv$ is regular if $\mathbb{D}_\bfv$ is either empty or of codimension $r$ in $\mathbb{D}$. Note that the latter occurs if and only if $\bfv$ is non-degenerate and $v_1,\ldots,v_r$ span a positive definite subspace of $V$.
\end{definition}

If $\bfv$ is regular, then $s_\bfv$ is a regular section of $(\vbE^r)^\vee$ in the sense of \Cref{subsubsection:Koszul_complexes}, and $\mathbb{D}_\bfv$ is smooth. Therefore, setting
\begin{equation} \label{eq:def_Green form g(v)}
\archGreen(\bfv) := \int_{1}^{\infty} \nuo( \sqrt t \bfv)_{[2r \cdot \mathrm{rk}(\vbE) -2]} \frac{dt}{t} \in A^{r\cdot \mathrm{rk}(\vbE)-1,r\cdot \mathrm{rk}(\vbE)-1}(\mathbb{D}-\mathbb{D}_\bfv),
\end{equation}
Proposition \ref{prop:Green_current_general} shows that $\archGreen(\bfv)=\archGreen(s_{\bfv})$ is smooth on $\mathbb{D}-\mathbb{D}_{\bfv}$, locally integrable on $\mathbb{D}$ and, as a current, satisfies
\begin{equation} \label{eq:Green_form_gv}
\mathrm{dd}^c \archGreen({\bf v})+\delta_{\mathbb{D}_{\bf v}} = \phio({\bf v})_{[2r \, \mathrm{rk}(\vbE)]},
\end{equation}
so that $\archGreen({\bf v})$ is a Green form for $\mathbb{D}_{\bf v}$.

%\begin{definition} \label{definition:Green form g(v)}
%For non-degenerate $\bfv$, define
%\begin{equation*}
%\archGreen(\bfv) := \int_{1}^{\infty} \nuo( \sqrt t \bfv)_{[2r \cdot \mathrm{rk}(\vbE) -2]} \frac{dt}{t} \in A^{r\cdot \mathrm{rk}(\vbE)-1,r\cdot \mathrm{rk}(\vbE)-1}(\mathbb{D}-\mathbb{D}_\bfv).
%\end{equation*}
%\end{definition}
%
%Note that when $\bfv$ is non-degenerate the section $s_\bfv$ is regular, and so Proposition \ref{prop:Green_current_general} shows that $\archGreen(\bfv)=\archGreen(s_{\bfv})$ is smooth on $\mathbb{D}-\mathbb{D}_{\bfv}$, locally integrable on $\mathbb{D}$ and, as a current, satisfies
%\begin{equation} \label{eq:Green_form_gv}
%\mathrm{dd}^c \mathfrak{g}({\bf v})+\delta_{\mathbb{D}_{\bf v}} = \phio({\bf v})_{[2r \, \mathrm{rk}(\vbE)]},
%\end{equation}
%so that $\mathfrak{g}({\bf v})$ is a Green form for $\mathbb{D}_{\bf v}$.

\begin{example}\label{example:g(v)_rkE_1}  Let $v \neq 0 \in V$ and assume that $\mathrm{rk}(\vbE) = 1$. Then, using  \eqref{eq:computation_nu0_degree_0_case_1}, we compute
	\begin{equation}
	\begin{split}
	\archGreen(v)  =  \int_1^{+\infty} \nuo(t^{1/2}v)_{[0]}  \frac{dt}{t}  =  \int_1^{+\infty} e^{-2\pi t h(s_v)} \frac{dt}{t} = -\mathrm{Ei}(-2\pi h(s_v)),
	\end{split}
	\end{equation} 
	where $\mathrm{Ei}(-z) = -\smallint_1^\infty e^{-zt}\tfrac{dt}{t}$ denotes the exponential integral. Thus $\archGreen(v) $ coincides with the Green form defined in \cite[(11.24)]{KudlaCD}. \hfill $\diamond$
\end{example}

\subsubsection{} \label{subsection:regularized_green_current_on_D} If $\bfv$ is no longer assumed regular, then the integral in \eqref{eq:def_Green form g(v)} is often no longer convergent. We shall overcome this deficiency by regularization:
%We now explain how to regularize this integral to define a current $\archGreen(\bfv)$ for arbitrary $\bfv \in V^r$. 
for any $\bfv \in V^r$ and $\rho \in \mathbb{C}$ with $\mathrm{Re}(\rho)>0$, define
\begin{equation} \label{eq:def_g_reg_local}
\archGreen(\bfv;\rho) = \int_1^{+\infty} \nuo(\sqrt{t} \bfv)_{[2r \cdot \mathrm{rk}(\vbE)-2]} \, \frac{dt}{t^{\rho+1}}.
\end{equation}
As $\nuo(t \bfv)$ is  bounded as $t \to \infty$, locally uniformly on $\bbD$, this integral defines a smooth form on $\bbD$. We will show that it admits a meromorphic continuation (as a current) to a neighbourhood of $\rho = 0$, beginning first with the case $r=1$.
%\todo{I moved around some of the discussion here for improved clarity. LG: Yes, this is much better!}
%We will show that the integral converges and determine the basic properties of $\archGreen(\bfv;\rho)$ in the next proposition, whose proof uses the following lemma.

\begin{lemma} \label{lemma:GreenReg r=1 case}
	For $v \in V$ with $v \neq 0$ and a complex parameter $\rho$, consider the integral on $\bbD - \bbD_{v}$:
	\[
	\archGreen(v; \rho) \ := \ \int_1^{\infty} \nuo(\sqrt t v)_{[2 \rank(\vbE)-2]} \, \frac{dt}{t^{1+\rho}}.
	\]
	For $\mathrm{Re}(\rho) > -1/2$,  this integral converges to a locally integrable form on $\bbD$, and the convergence is locally uniform on $\bbD$ and in $\rho$. 

	\begin{proof}
%		If $\mathrm{Re}(\rho) > 0$, the convergence of the integral to a locally integrable form follows immediately from the fact that $\nuo(\sqrt{t} v)$ is bounded as $t \to \infty$ (locally on $\mathbb{D})$. To extend to $\mathrm{Re}(\rho)>-1/2$, 
		
		Recalling that $\nuo(v) = \str(Ne^{\nabla_{v}^2})$, it follows by taking $s=1$ in \eqref{eq:expansion_e_nabla2} that we may write
		\begin{equation}
		\nuo(\sqrt t v)_{[2 \rk(\vbE) - 2]} \ = \ \sum_{k=0}^{\rank(\vbE) - 1} t^k e^{-2 \pi \, t \,  h(s_v)} \, \eta_k(v)
		\end{equation}
		for some  differential forms $\eta_k(v)$ that are smooth on $\bbD$.
		 For convenience, set  $x =  2 \pi h(s_v)$; then
		 \begin{align}
			\archGreen(v;\rho) = \sum_{k=0}^{\rank( \vbE) - 1} \int_1^{\infty} t^{k-\rho} e^{-tx} \frac{dt}{t} \cdot \eta_k(v)
			= \sum_{k=0}^{\rank(\vbE) -1}  x^{\rho - k}  \left( \int_{x}^{\infty} t^{k-\rho}  e^{-t} \frac{dt}{t}  \right) \eta_k(v).
		\end{align}
%		 \begin{align}
%			\archGreen(v;\rho) \ = \ \sum_{k=0}^{\rank( \vbE) - 1} \int_1^{\infty} t^{k-\rho} e^{-tx} \frac{dt}{t} \cdot \eta_k(v) \
%			= \ \sum_{k=0}^{\rank(\vbE) -1} \, x^{\rho - k}  \cdot \left( \int_{x}^{\infty} t^{k-\rho}  e^{-t} \frac{dt}{t}  \right) \cdot \eta_k(v).
%		\end{align}
		 The forms $\eta_k(v)$ are locally bounded since they are smooth, and	a straightforward computation in local coordinates, as in the proof of \Cref{prop:Green_current_general}, implies that $x^{\rho - k }$ is locally integrable for $\mathrm{Re}(\rho) > - 1/2$ and $k \leq \rk(\vbE) - 1$. As for the integrals, write
		 \begin{equation}
		 \begin{split}
			 \int_x^{\infty} t^{k-\rho }e^{-t} \frac{dt}{t} &=  \ \int_x^1 t^{k-\rho} \frac{dt}{t}  - \int_x^1 t^{k - \rho } \left( 1 - e^{-t} \right)\frac{dt}{t} + \int_1^{\infty} t^{k - \rho} e^{-t} \frac{dt}{t}  \\
			 &= \ \frac{1 - x^{k-\rho}}{k-\rho} - \int_x^1 t^{k - \rho } \left( 1 - e^{-t} \right)\frac{dt}{t} + \int_1^{\infty} t^{k - \rho} e^{-t} \frac{dt}{t}.
		\end{split}
		 \end{equation}
		 	The latter two integrals are absolutely bounded uniformly in $x$ and in $\rho$ for $  -\tfrac{1}{2}<\mathrm{Re}(\rho) < \tfrac{1}{2}$, say, while the first term is evidently holomorphic on this region.
% The point is that the numerator $1-x^{k-\rho}$ is a holomorphic function of $k-\rho$ vanishing  at $k=\rho$, and so the fraction is really holomorphic.
In particular,  multiplying each of these terms by $x^{\rho - k}$ yields locally integrable functions for $-\tfrac{1}{2}<\mathrm{Re}(\rho)<\tfrac{1}{2}$; since the more direct estimate covers the case $\mathrm{Re}(\rho)>0$, this proves the lemma.
\end{proof}
\end{lemma}

\begin{proposition} \label{prop:xiRegContinuation} Let $\bfv = (v_1,\ldots,v_r) \in V^r$ and let $r' = \dim \langle v_1,\ldots,v_r \rangle$.
	\begin{enumerate}[(i)]
		\item The integral \eqref{eq:def_g_reg_local} converges to a smooth form on $\bbD$ if $\mathrm{Re}(\rho)>0$.
		\item Let $k \in \mathrm{O}(r)$ (case 1) or $k \in \mathrm{U}(r)$ (case 2). Then
		\begin{equation*}
		\archGreen(\bfv \cdot k;\rho) = \archGreen(\bfv;\rho).
		\end{equation*}
		\item As a current, $\archGreen (\bfv; \rho)$ extends meromorphically\footnote{More precisely, we are asserting that for any compactly supported form $\eta$, the expression $\int_{\bbD}\archGreen (\bfv; \rho) \wedge \eta $ admits a meromorphic extension as a function of $\rho$, and is continuous in $\eta$ in the sense of distributions.} to the right half plane $\mathrm{Re}(\rho)>-\frac12$.
		\item If $\bfv$ is regular, then the current $\archGreen(\bfv;\rho)$ is regular at $\rho=0$ and 
		\begin{equation*}
		\archGreen(\bfv; 0) = \archGreen(\bfv). 
		\end{equation*}
		Similarly, for any $\bfv$, the identity $  \archGreen(\bfv; 0) =  \smallint_1^{+\infty} \nuo(t^{1/2}(\bfv))\tfrac{dt}{t}$ holds on $\bbD - \bbD_{\bfv}$. 
%		\item Let 
%		\begin{equation*}
%		I(\bfv) = \{i \in \{1,\ldots,r\} | v_i \in \langle v_1,\ldots,\widehat{v_i},\ldots,v_r \rangle \}.
%		\end{equation*}
%		 The constant term $\mathrm{CT}_{\rho=0} \ \archGreen(\bfv;\rho)$ is given by
%		 \begin{equation*}
%		 \mathrm{CT}_{\rho=0} \ \archGreen(\bfv;\rho) = \int_1^\infty \left(\nu^0(t^{1/2}\bfv)_{[2r \cdot \mathrm{rk}(\vbE)-2]} - |I(\bfv)| \delta_{\mathbb{D}_\bfv} \wedge c_{\mathrm{rk}(\vbE)-1}(\vbE^\vee,\nabla)^*  \wedge \Omega_{\vbE^\vee}^{r-r'-1} \right)\frac{dt}{t}
%		 \end{equation*}
\item The constant term of $\archGreen(\bfv;\rho)$at $\rho=0$ is given by
		 \begin{equation*}
		 \mathrm{CT}_{\rho=0} \ \archGreen(\bfv;\rho) = \int_1^\infty \left(\nuo(t^{1/2}\bfv)_{[2r \cdot \mathrm{rk}(\vbE)-2]} - (r-r') \delta_{\mathbb{D}_\bfv} \wedge c_{\mathrm{rk}(\vbE)-1}(\vbE^\vee,\nabla)^*  \wedge \Omega_{\vbE^\vee}^{r-r'-1} \right)\frac{dt}{t}
		 \end{equation*}
		\item The constant term $\mathop{\mathrm{CT}}\limits_{\rho=0} \archGreen(\bfv;\rho)$ satisfies the  equation
		\begin{equation*}
		\ddc \mathop{\mathrm{CT}}\limits_{\rho=0} \archGreen(\bfv;\rho) +  \delta_{\mathbb{D}_{\bfv}} \wedge \Omega_{\vbE^{\vee}}^{r-r'} = \phio(\bfv )_{[2r \cdot \rank(\vbE)]}
		\end{equation*}
		of currents on $\bbD$, where $\Omega_{\vbE^{\vee}} =  c^{\mathrm{top}}( \vbE^{\vee},\nabla)^*$.
% \comm{SS: I'm not sure, why is it $- \Omega$ and not $\Omega$ in this formula? From the proof, the term should be $\Omega = \phio(0)$, which is $ c_1(\overline \vbE^{\vee})^* = \frac{1}{- 2 \pi i} c_1(\overline \vbE^{\vee})$, no? }
	\end{enumerate}
\end{proposition}	
\begin{proof}
Part $(i)$ follows immediately from the expression \eqref{eq:expansion_e_nabla2}, which shows that (locally on $\mathbb{D}$) $\nuo(\sqrt{t} \bfv)$ and its partial derivatives stay bounded as $t \to +\infty$. Part $(ii)$ follows from \Cref{prop:nu_properties}.$(f)$.

To show $(iii)$, let $k$ be as in $(ii)$ such that $\bfv \cdot k=({\bf 0}_{r-r'},\bfv')$ with $\bfv'$ non-degenerate. For convenience, set $q'=\mathrm{rk}(\vbE)$. By Propositions \ref{prop:phi_basic_properties}.(c) and \ref{prop:nu_properties}.(c),(f) we can write
\begin{equation}\label{eq:proof_regularization_on_D_two_sums_1}
\begin{split}
\nuo(\bfv)_{[2r q'-2]} &= \nuo(\bfv \cdot k)_{[2r q'-2]} \\
&= \nuo(\bfv')_{[2r'q'-2]} \wedge \phio({\bf 0}_{r-r'})_{[2(r-r')q']} + \nuo({\bf 0}_{r-r'})_{[2(r-r') q'-2]} \wedge \phio(\bfv')_{[2r'q']}.
\end{split}
\end{equation}
The same propositions also show that
\begin{equation}
\begin{split}
\phio({\bf 0}_{r-r'})_{[2(r-r')q']} &= \Omega_{\vbE^\vee}^{r-r'} \\
\nuo({\bf 0}_{r-r'})_{[2(r-r') q'-2]} &= (r-r')c_{\mathrm{rk}(\vbE)-1}(\vbE^\vee,\nabla)^* \wedge \Omega_{\vbE^\vee}^{r-r'-1}
\end{split}
\end{equation}
and hence
\begin{equation} \label{eq:proof_regularization_on_D_two_sums}
\begin{split}
\nuo(\bfv)_{[2rq'-2]} &= \nuo(\bfv')_{[2r'q'-2]} \wedge \Omega_{\vbE^\vee}^{r-r'} \\
& \quad + \phio(\bfv')_{[2r'q']} \wedge (r-r')c_{\mathrm{rk}(\vbE)-1}(\vbE^\vee,\nabla)^* \wedge \Omega_{\vbE^\vee}^{r-r'-1}.
\end{split}
\end{equation}
Consider the contribution of each term in the last expression to $\archGreen(\bfv;\rho)$. Writing $\bfv_i'=(v_1',\ldots,\widehat{v_i'},\ldots,v'_{r'})$, the first term contributes
\begin{equation} \label{eq:proof_regularization_on_D_eq_1}
\sum_{1 \leq i \leq r'} \int_1^{+\infty} \nuo(\sqrt{t}v_i')_{[2q'-2]} \wedge \phio(\sqrt{t}\bfv_i')_{[2(r'-1)q']} \frac{dt}{t^{\rho+1}} \wedge \Omega_{\vbE^\vee}^{r-r'}.
\end{equation}
Since $\phio(\sqrt{t}\bfv_i')$ stays bounded as $t \to +\infty$, \Cref{lemma:GreenReg r=1 case} shows that \eqref{eq:proof_regularization_on_D_eq_1} converges to a locally integrable form on $\mathbb{D}$ for $\mathrm{Re}(\rho)>-1/2$. 

The contribution of the second term is
\begin{equation} \label{eq:proof_regularization_on_D_second_contribution}
\int_1^{+\infty} \phio(\sqrt{t}\bfv')_{[2r'q']} \frac{dt}{t^{\rho+1}} \wedge (r-r')c_{\mathrm{rk}(\vbE)-1}(\vbE^\vee,\nabla)^* \wedge \Omega_{\vbE^\vee}^{r-r'-1},
\end{equation}
hence it suffices to prove meromorphic continuation of the integral in this expression. We rewrite this integral as
%\begin{equation}
%\begin{split}
%\int_1^{+\infty} & \phio(\sqrt{t}\bfv')_{[2r'q']} \frac{dt}{t^{\rho+1}} \\
%&=\int_1^{+\infty} (\phio(\sqrt{t}\bfv')_{[2r'q']}- \delta_{\mathbb{D}_\bfv}) \frac{dt}{t^{\rho+1}} + \frac{1}{\rho} \delta_{\mathbb{D}_\bfv};
%\end{split}
%\end{equation}
\begin{equation} \label{eq:proof_regularization_on_D_eq_2}
\int_1^{+\infty}  \phio(\sqrt{t}\bfv')_{[2r'q']} \frac{dt}{t^{\rho+1}} 
 =\int_1^{+\infty} (\phio(\sqrt{t}\bfv')_{[2r'q']}- \delta_{\mathbb{D}_\bfv}) \frac{dt}{t^{\rho+1}} + \frac{1}{\rho} \delta_{\mathbb{D}_\bfv};
\end{equation}
then Bismut's asymptotic estimate \cite[(3.9)]{BismutInv90} %maybe \cite[(3.7)]{BismutInv90} is more relevant?
	implies that
	\begin{equation} \label{eqn:BismutPhiEstimate}
	\phio(\sqrt{t} \bfv')_{[2r'q']} - \delta_{\mathbb{D}_{\bfv}} = \mathrm{O}(t^{-\frac12})
	\end{equation}
	as currents on $\mathbb{D}$, and hence the right hand side gives the desired meromorphic continuation to $\mathrm{Re}(\rho)>-1/2$, proving $(iii)$.
	
	When $\bfv$ is regular, or for general $\bfv$ upon restriction to $\bbD - \bbD_{\bfv}$, the proof of $(iii)$ shows that the integral defining $\archGreen(\bfv;\rho)$ converges when $\mathrm{Re}(\rho)>-1/2$; thus we can set $\rho=0$ and obtain $(iv)$.		

To prove $(v)$, we proceed as in $(iii)$ and analyze the contribution to $\mathrm{CT}_{\rho=0} \ \archGreen(\bfv;\rho)$ of each summand in the right hand side of \eqref{eq:proof_regularization_on_D_two_sums}. Observe that the constant term at $\rho = 0$ of \eqref{eq:proof_regularization_on_D_eq_1} is simply
\begin{equation} \label{eq:proof_regularization_on_D_eq_1_const_term}
\sum_{1 \leq i \leq r'} \int_1^{+\infty} \nuo(\sqrt{t}v_i')_{[2q'-2]} \wedge \phio(\sqrt{t}\bfv_i')_{[2(r'-1)q']} \frac{dt}{t} \wedge \Omega_{\vbE^\vee}^{r-r'},
\end{equation}
whereas the constant term of \eqref{eq:proof_regularization_on_D_eq_2} is
\begin{equation}
\int_1^{+\infty} (\phio(\sqrt{t}\bfv')_{[2r'q']}- \delta_{\mathbb{D}_\bfv}) \frac{dt}{t}.
\end{equation}
Substituting in \eqref{eq:proof_regularization_on_D_second_contribution} and adding these two contributions gives $(v)$.

Finally, note that (by $(ii)$ and \Cref{prop:phi_basic_properties}.(f)) all terms in $(vi)$ are invariant under replacing $\bfv$ with $\bfv \cdot k$ for any matrix $k \in \mathrm{O}(r)$ (case 1) or $k \in \mathrm{U}(r)$ (case 2). Thus we can assume that $\bfv = ({\bf 0}_{r-r'},\bfv')$, where $\bfv' \in V^{r'}$ is non-degenerate. Then $\nuo(\bfv)_{[2rq'-2]}$ is given by \eqref{eq:proof_regularization_on_D_two_sums}; since $\phio(\bfv')$ is closed, we conclude that
\begin{equation}
\begin{split}
\ddc \nuo(\sqrt{t}\bfv)_{[2rq'-2]} &= \ddc \nuo(\sqrt{t}\bfv')_{[2r'q'-2]} \wedge \Omega_{\vbE^\vee}^{r-r'} \\
&= -t \frac{d}{dt} \phio(\sqrt{t}\bfv')_{[2r'q']} \wedge \Omega_{\vbE^\vee}^{r-r'}.
\end{split}
\end{equation}
Using $(v)$ gives
\begin{equation}
\begin{split}
\ddc \mathrm{CT}_{\rho=0} \ \archGreen(\bfv;\rho) &= \int_1^\infty \ddc\left(\nuo(t^{1/2}\bfv)_{[2rq'-2]}  \right. \\
& \left. \hspace{6em} - (r-r') \delta_{\mathbb{D}_\bfv} \wedge c_{\mathrm{rk}(\vbE)-1}(\vbE^\vee,\nabla)^*  \wedge \Omega_{\vbE^\vee}^{r-r'-1} \right)\frac{dt}{t} \\
& = \int_1^\infty \ddc \nuo(t^{1/2}\bfv)_{[2rq'-2]}  \frac{dt}{t} \\
&= \int_1^\infty - t \frac{d}{dt} \phio(t^{1/2}\bfv')_{[2r'q']}  \frac{dt}{t} \wedge \Omega_{\vbE^\vee}^{r-r'} \\
&= \phio(\bfv)_{[2rq']} - \lim_{t \to \infty} \phio(t^{1/2}\bfv')_{[2r'q']} \wedge \Omega_{\vbE^\vee}^{r-r'} \\
&= \phio(\bfv)_{[2rq']} - \delta_{\mathbb{D}_\bfv} \wedge \Omega_{\vbE^\vee}^{r-r'},
\end{split}
\end{equation}
where the last equality follows from \eqref{eqn:BismutPhiEstimate}, proving $(vi)$.
\end{proof}

\begin{definition}
Let $\bfv=(v_1,\ldots,v_r) \in V^r$.  Define
\begin{equation*}
\archGreen(\bfv) = \mathrm{CT}_{\rho = 0} \ \archGreen(\bfv;\rho) \in D^*(\mathbb{D}).
\end{equation*}
\end{definition}

\begin{example}  \label{xi0WithZeroes}
Suppose that $\bfv \in V^r$ is degenerate and choose $k$ in $O(r)$ or $U(r)$
%be as in \Cref{prop:xiRegContinuation}.(ii) 
such that $\bfv \cdot k = ({\bf 0}_{r-r'},\bfv')$, where $\bfv' \in V^{r'}$ is non-degenerate. Then the proof of \Cref{prop:xiRegContinuation}.(iii) shows that
\begin{equation}
\archGreen(\bfv) = \archGreen(\bfv') \wedge \Omega_{\vbE^\vee}^{r-r'} + \mu(\bfv'),
\end{equation}
where %\todo{This should have $r-r'$ in front. LG: You are right.}
\begin{equation}
\mu(\bfv') =  (r - r') \int_1^{+\infty} \left( \phio(\sqrt{t}\bfv')_{[2r'q']} - \delta_{\mathbb{D}_\bfv} \right) \frac{dt}{t} \cdot c_{\mathrm{rk}(\vbE)-1}(\vbE^\vee,\nabla)^* \wedge \Omega_{\vbE^\vee}^{r-r'-1}.
\end{equation}
\end{example}

\subsection{Star products} \label{subsection:star_products} 

Let $k,l$ be positive integers with $(k+l)\mathrm{rk}(\vbE) \leq \dim(\mathbb{D})+1$. %Given a regular $k$-tuple ${\bf v}=(v_1,\ldots,v_k) \in V^k$, we have defined a Green form $\archGreen(\bfv)$ for the cycle $\mathbb{D}_{\bf v}$ (see \eqref{eq:def_Green form g(v)}). \todo{Is this statement necessary?. LG: You are right, it's not.}
Fix regular tuples ${\bf v}'=(v'_1,\ldots,v'_k)$ and ${\bf v}''=(v''_1,\ldots,v''_l) \in V^l$ such that the tuple 
\begin{equation}
{\bf v}= (v_1,\ldots,v_{k+l}):= (v'_1,\ldots,v'_k,v''_1,\ldots,v''_l) \in V^{k+l}
\end{equation}
is also regular.  
This implies that $\bbD_{\bfv}= \bbD_{\bfv'} \cap \bbD_{\bfv''}$, if non-empty, is a proper intersection, see \Cref{def:regular tuple}.  Define the star product of the Green forms $\archGreen(\bfv')$ and $\archGreen(\bfv'')$, which for the regular case are given by \eqref{eq:def_Green form g(v)}, by
%We define
\begin{equation}
\archGreen({\bf v}') * \archGreen({\bf v}'') = \archGreen({\bf v}') \wedge \delta_{\mathbb{D}_{{\bf v}''}} + \phio({\bf v}')_{[2\mathrm{rk}(\vbE)k]} \wedge \archGreen({\bf v}'') \in D^{2(k+l)\mathrm{rk}(\vbE)-2}(\mathbb{D}).
\end{equation}
Our next goal is to compare the currents $\archGreen(\bfv)$ and $\archGreen(\bfv') * \archGreen(\bfv'')$.  For $t_1,t_2 \in \mathbb{R}_{>0}$, define
	\begin{equation}
	\begin{split}
	\alpha(t_1,t_2,{\bf v}',{\bf v}'')  \ &= \  \frac{i}{2\pi} \   \nuo(t_1^{1/2}{\bf v}')_{[2k \cdot \rank(\vbE)-2]}    \ \wedge  \  \overline{\partial}(\nuo(t_2^{1/2}{\bf v}'')_{[2l \cdot \rank(\vbE)-2]})  \ \wedge \  \frac{dt_1 \, dt_2}{t_1 \, t_2}, \\
	\beta(t_1,t_2,{\bf v}',{\bf v}'') \  &= \  \frac{i}{2\pi} \ \partial (\nuo(t_1^{1/2}{\bf v}')_{[2k\cdot \rank(\vbE)-2]})  \ \wedge \ \nuo(t_2^{1/2}{\bf v}'')_{[2l\cdot \rank(\vbE)-2]}  \ \wedge  \ \frac{dt_1  \, dt_2}{t_1\, t_2}.
	\end{split}
	\end{equation}
	We set
	\begin{equation} \label{eq:def_currents_alpha_beta}
	\begin{split}
	\alpha({\bf v}',{\bf v}'')= \int\limits_{ 1 \leq t_1 \leq t_2 \leq +\infty}\alpha(t_1,t_2,{\bf v}',{\bf v}''), \\
	\beta({\bf v}',{\bf v}'')= \int\limits_{1 \leq t_1 \leq t_2 \leq +\infty}\beta(t_1,t_2,{\bf v}',{\bf v}'').
	\end{split}
	\end{equation}

The estimate \eqref{eq:algebra_norm_estimate} shows that the integrals converge to smooth forms on $\mathbb{D}-(\mathbb{D}_{{\bf v}'} \cup \mathbb{D}_{{\bf v}''})$. These forms are locally integrable on $\mathbb{D}$ by a straightforward adaptation of the proof of \Cref{prop:Green_current_general} 
%(note that $\mathbb{D}_{{\bf v}'}$ and $\mathbb{D}_{{\bf v}''}$ intersect transversely since ${\bf v}$ is regular and use the uniform bound in \eqref{eq:int_nuo_uniformly_bounded}) 
and so they define currents $[\alpha({\bf v}',{\bf v}'')]$ and $[\beta({\bf v}',{\bf v}'')]$ on $\mathbb{D}$.

\begin{theorem} \label{theorem:star_products_arch}
The currents $[\alpha({\bf v}',{\bf v}'')]$ and $[\beta({\bf v}',{\bf v}'')]$ satisfy $g_*\alpha[({\bf v}',{\bf v}'')]=[\alpha(g{\bf v}',g{\bf v}'')]$ and $g_*[\beta({\bf v}',{\bf v}'')]=[\beta(g{\bf v}',g{\bf v}'')]$ for $g \in G$ and
\begin{equation*}
	\archGreen({\bf v}') * \archGreen({\bf v}'') - \archGreen({\bf v}) = \partial [\alpha({\bf v}',{\bf v}'')]+\overline{\partial}[\beta({\bf v}',{\bf v}'')].
\end{equation*}
\end{theorem}
\begin{proof} 
%	We abbreviate $q'=\mathrm{rk}(\vbE)$
Set $q'=\mathrm{rk}(\vbE)$. %\todo{$r$ is a poor choice for $\rk \vbE$ I think,in light of previous section. LG: Agreed.} 
Denote by $t_1,t_2$ the coordinates on $\mathbb{R}^{2}_{>0}$ and consider the form $\tildenuo({\bf v}) \in A^*(\mathbb{D} \times \mathbb{R}_{>0}^{2})$ defined by
		\begin{align}
	\tildenuo({\bf v}) \ = \  \nuo(\sqrt{t_1} & \bfv ')_{[2 k q' - 2]}  \, \wedge \, \phio(\sqrt{t_2} \bfv'')_{[2lq']} \, \frac{dt_1}{t_1}   \notag \\
	\ &+ \ \phio(\sqrt{t_1} \bfv')_{[2kq']}  \, \wedge \, \nuo(\sqrt{t_2} \bfv'')_{[2lq'-2]} \, \frac{dt_2 }{t_2}.
	\end{align}
%	\begin{equation}
%	\tildenuo({\bf v})=\sum_{1 \leq i \leq r} \nuo(t_i^{1/2}v_i)_{[2q'-2]}\wedge \phio(t_1^{1/2}v_1,\ldots,\widehat{t_i^{1/2}v_i},\ldots,t_{r}^{1/2}v_{r})_{[2(r-1)q']}\wedge \frac{dt_i}{t_i}.
%	\end{equation}
	%Here $A^*(\mathbb{D} \times \mathbb{R}_{>0}^r)^{\mathrm{rd}}$ is the subcomplex of differential forms that are rapidly decreasing along the fibers of $\pi$. 
	For a piecewise smooth path $\gamma:I \to \mathbb{R}_{>0}^{2}$ ($I \subset \mathbb{R}$ a closed interval) and $\alpha \in A^*(\mathbb{D}\times \mathbb{R}_{>0}^{2})$, let 
	\begin{equation}
	\int_\gamma \alpha \in A^{*-1}(\mathbb{D})
	\end{equation}
	be the form obtained by integrating $(\mathrm{id} \times \gamma)^*\alpha$ along the fibers of the projection of $\mathbb{D} \times I \to \mathbb{D}$. Fix a real number $M>1$ and consider the paths
	\begin{equation}
		\gamma_{d,M} = (t,t),  \ \ \gamma'_M = (1, t), \ \text{ and } \gamma''_M=(t, M)
	\end{equation}
%	\begin{equation}
%	\begin{split}
%	\gamma_{d,M}(t)&=(t,\ldots,t), \\
%	\gamma_{k,M}'(t)&=(\underbrace{1,\ldots,1}_k,t,\ldots,t), \quad \gamma_{k,M}''(t)=(\underbrace{t,\ldots,t}_k,M,\ldots,M), \\
%%	\gamma_{k,M} &= \gamma_{k,M}' \sqcup \gamma_{k,M}'',
%	\end{split}
%	\end{equation}
%	\todo{dispense with $\gamma_{k,M}$... LG: Great!}
	with $t \in [1,M]$. By \Cref{prop:nu_properties}.(a), we have
	\begin{equation}
\lim_{M \to \infty}	\int_{\gamma_{d,M}} \tildenuo({\bf v}) \ = \  \lim_{M \to \infty} \int_1^M \nuo(t^{1/2}{\bf v})_{[2(k+l)q'-2]} \frac{dt}{t} \ = \ \archGreen(\bfv).
	\end{equation}
%	this approaches $\archGreen({\bf v})$ as $M \to \infty$.
% The reason that this approaches $\archGreen(\bfv)$ is implicit in the proof of \Cref{prop:Green_current_general}. Namely, it's the first equality in the proof of part (3), which uses the uniform bound \eqref{eq:int_nuo_uniformly_bounded}. Should I explain all this?
	Next, note that 
	\begin{equation}
	\begin{split}
	(\mathrm{id} \times \gamma_{M}')^* \tildenuo({\bf v}) &= \nuo(t^{1/2} {\bf v}'')_{[2lq'-2]} \wedge \phio({\bf v}')_{[2kq']} \wedge \frac{dt}{t}, \\
	(\mathrm{id} \times \gamma_{M}'')^* \tildenuo({\bf v})&=  \nuo(t^{1/2} {\bf v}')_{[2kq'-2]} \wedge \phio(M^{1/2}{\bf v}'')_{[2lq']} \wedge \frac{dt}{t}.
	\end{split}
	\end{equation} 

	By \cite[Thm. 3.2]{BismutInv90}, as $M \to \infty$ we have $\phio(M^{1/2}{\bf v}'')_{[2lq']} \rightarrow \delta_{\mathbb{D}_{{\bf v}''}}$
	as currents, and hence
	\begin{equation}
	\lim_{M \to \infty} \int\limits_{\gamma'_{M} + \gamma_M''} \tildenuo({\bf v}) = \archGreen({\bf v}') * \archGreen({\bf v}'').
	\end{equation}
	 Let 
	\begin{equation}
	\Delta_{M} \ =  \ \{  (t_1, t_2) \ | \ 1 \leq t_1 \leq t_2 \leq M\}  \ \subset  \  \mathbb{R}_{>0}^{2},
	\end{equation}
	oriented so that $\partial \Delta_{M}=\gamma_{d,M}-\gamma'_{M} - \gamma''_{M}$.
	Let $d=d_1+d_2$ be the differential on $\mathbb{D} \times \mathbb{R}_{>0}^{2}$, where $d_1=\partial+\overline{\partial}$ is the differential on $\mathbb{D}$ and $d_2$ is the differential on $\mathbb{R}_{>0}^{2}$. Then we have
	\begin{equation} \label{eq:identity_proof_star_product}
	\int_{\gamma_{d,M}} \tildenuo({\bf v})-\int_{\gamma'_{M}+\gamma''_M} \tildenuo({\bf v}) \ = \ \int_{\partial \Delta_M } \tildenuo(\bfv) \ = \ \int_{\Delta_{M}} d_2  \, \tildenuo({\bf v}).
	\end{equation}

		Applying \Cref{prop:nu_properties}.$(b)$ we obtain
		\begin{equation}
		\begin{aligned}
 	d_2 \, \tildenuo({\bf v})  &= \Big( t_1 \frac{d}{dt_1}     \phio(t_1^{1/2}{\bf v}')_{[2kq']} \wedge \nuo(t_2^{1/2}{\bf v}'')_{[2lq'-2]}  \\
 	& \qquad \qquad   - \nuo(t_1^{1/2}{\bf v}')_{[2kq'-2]}\wedge t_2\frac{d}{dt_2}\phio(t_2^{1/2}{\bf v}'')_{[2lq']} \Big) \wedge \frac{dt_1  dt_2}{t_1t_2} \\
		&= (-2\pi i)^{-1} \left( -(\partial \overline{\partial} \nuo(t_1^{1/2} {\bf v}'))_{[2kq']} \wedge \nuo(t_2^{1/2}{\bf v}'')_{[2lq'-2]} \right. \\ 
		& \hspace{8em} \left. + \nuo(t_1^{1/2}{\bf v}')_{[2kq'-2]}\wedge (\partial \overline{\partial} \nuo(t_2^{1/2}{\bf v}''))_{[2lq']} \right) \wedge \frac{dt_1 dt_2}{t_1t_2} \\
		&= \partial \alpha(t_1,t_2,{\bf v}',{\bf v}'') + \overline{\partial} \beta(t_1,t_2,{\bf v}',{\bf v}'').
	\end{aligned}
		\end{equation}

The statement follows by taking the limit as $M \to +\infty$ in \eqref{eq:identity_proof_star_product}, and the equivariance property under $g \in G$ follows from \Cref{prop:nu_properties}.(d).
\end{proof}

As a corollary, we obtain the following invariance property of star products. 

\begin{corollary} \label{corollary:rotation_invariance}
	Let $k \in \mathrm{O}(k+l)$ (case 1) or $k \in \mathrm{U}(k+l)$ (case 2) and suppose that ${\bf v}=({\bf v}',{\bf v}'') \in V^{k+l}$  is non-degenerate. Let ${\bf v}'_k$, ${\bf v}''_k$ be defined by ${\bf v} \cdot k = ({\bf v}'_k,{\bf v}''_k)$ and set
	\begin{equation*}
	[\alpha(k;{\bf v}',{\bf v}'')] =  [\alpha({\bf v}',{\bf v}'')]-[\alpha({\bf v}'_k,{\bf v}''_k)], \quad [\beta(k;{\bf v}',{\bf v}'')]=[\beta({\bf v}',{\bf v}'')]-[\beta({\bf v}'_k,{\bf v}''_k)],
	\end{equation*}
	with $\alpha$ and $\beta$ as in \eqref{eq:def_currents_alpha_beta}. Then
	\begin{equation*}
	\archGreen({\bf v}') * \archGreen({\bf v}'') - \archGreen({\bf v}'_k) * \archGreen({\bf v}''_k) = \partial [\alpha(k;{\bf v}',{\bf v}'')] + \overline{\partial}[\beta(k;{\bf v}',{\bf v}'')] \in D^{2(k+l)\mathrm{rk}(\vbE)-2}(\mathbb{D}).
	\end{equation*}
\end{corollary}
\begin{proof}
	This is a consequence of the theorem and the invariance property $\archGreen(\bfv \cdot k) = \archGreen(\bfv)$, which follows from \Cref{prop:nu_properties}.(f). 
\end{proof}

When $G=\mathrm{SO}(1,2)^0$, this invariance property is one of the main results in \cite{KudlaCD}, where it is shown to hold by a long explicit computation (see also \cite{LiuYifengArithThetaI} for a similar proof in arbitrary dimension in case 2 when $q=1$ and $k+l=p+1$). Our corollary, and its global counterpart  (\Cref{corollary:archimedean_height_pairing_derivative_at_0}), generalizes these results  to arbitrary hermitian symmetric spaces of orthogonal or unitary type and gives a conceptual proof.

\section{Archimedean heights and derivatives of Whittaker functionals} \label{section:archimedean_arith_SW}

Let $r \leq p+1$ and $\bfv=(v_1,\ldots,v_r) \in V^r$ be non-degenerate (recall that by this we mean that $v_1,\ldots,v_r$ are linearly independent) and denote by $\mathrm{Stab}_G \langle v_1,\ldots,v_r \rangle$ the pointwise stabilizer of $\langle v_1,\ldots,v_r \rangle$ in $G$. Let $\archGreen(\bfv)$ be the form given in \Cref{subsection:Green forms}. Assuming that $\mathrm{rk}(\vbE)=1$, we will compute the integral
\begin{equation} \label{eq:integral_xi_0}
\int_{\Gamma_{\bf v} \backslash \mathbb{D}^+} \archGreen({\bf v}) \wedge \Omega^{p-r+1},
\end{equation}
where $\Gamma_{\bfv}$ is a discrete subgroup of finite covolume in $\mathrm{Stab}_G \langle v_1,\ldots,v_r \rangle$, under some additional conditions ensuring that the integral converges. The result is stated in \Cref{thm:height_pairing_and_derivative} and relates this integral to the derivative of a Whittaker functional defined on a degenerate principal series representation of $G_r'$. Despite the length of this section, its proof is conceptually simple and follows easily from \Cref{lemma:weight_lambda_0}, which determines the weight of $\nu(\bfv)_{[2r-2]}$, together with results in \cite{KudlaRallisDegenerate, Lee} concerning reducibility of these representations and multiplicity one for their $K$-types (we review these results in  \Cref{subsection:degenerate principal series of Mp2r,subsection:degenerate principal series of Urr}) and estimates of Shimura \cite{ShimuraConfluent} for Whittaker functionals that we review in \Cref{subsection:Whittaker_functionals}.

\subsection{Degenerate principal series of $\mathrm{Mp}_{2r}(\mathbb{R})$} \label{subsection:degenerate principal series of Mp2r} In this section we fix $r \geq 1$ and let $G'=\mathrm{Mp}_{2r}(\mathbb{R})$. We abbreviate $N=N_r$, $M=M_r$ and $K'=K'_r$ (see \ref{subsubsection:symplectic_group_definitions}).  

\subsubsection{} We denote by $\mathfrak{g}'$ the complexified Lie algebra of $G'$. We have the Harish-Chandra decomposition
\begin{equation}
\mathfrak{g}' = \mathfrak{p}_+ \oplus \mathfrak{p}_- \oplus \mathfrak{k}',
\end{equation}
with
\begin{equation}
\begin{split}
\mathfrak{k}' &= \left\{ \left. \begin{pmatrix} X_1 & X_2 \\ -X_2 & X_1 \end{pmatrix} \right| X_1^t=-X_1, X_2^t = X_2 \right\}, \\
\mathfrak{p_+} &= \left\{ \left. p_+(X)=\frac{1}{2}\begin{pmatrix} X & iX \\ iX & -X \end{pmatrix} \right| X^t=X \right\}, \quad \mathfrak{p}_-=\overline{\mathfrak{p}_+}.
\end{split}
\end{equation}
Note that $\mathfrak{k}'=\mathrm{Lie}(K')_\mathbb{C}$, where $K'$ is the maximal compact subgroup of $G'$ in \Cref{subsubsection:symplectic_group_definitions}.
For $x=(x_1,\ldots,x_r) \in \mathbb{C}^r$, let $d(x)$ denote the diagonal matrix $\mathrm{diag}(x_1,\dots,x_r)$, write
\begin{equation}
h(x) = \begin{pmatrix} & -i\cdot d(x) \\ i\cdot d(x) & \end{pmatrix}
\end{equation}
and define $e_j(h(x)) = x_j$. Then $\mathfrak{h}' = \{h(x)|x \in \mathbb{C}^r\}$ is a Cartan subalgebra of $\mathfrak{k}'$, and we choose the set of positive roots $\Delta^+=\Delta^+_c \sqcup \Delta^+_{nc}$ given by
\begin{equation} \label{eq:roots_of_g}
\begin{split}
\Delta^+_c &= \{e_i-e_j | 1\leq i < j \leq r\}, \\
\Delta^+_{nc} &= \{e_i+e_j | 1 \leq i \leq j \leq r\}, 
\end{split}
\end{equation}
where $\Delta^+_c$ and $\Delta^+_{nc}$ denote the compact and non-compact roots respectively.

% Siegel parabolic and degenerate principal series
\subsubsection{} The group $P=MN$ is a maximal parabolic subgroup of $G'$, the inverse image under the covering map of the standard Siegel parabolic of $\mathrm{Sp}_{2r}(\mathbb{R})$. The group $M$ has a character of order four given by
\begin{equation}
\chi(m(a),\epsilon) = \epsilon \cdot \begin{cases} i, & \text{ if } \det a  <0, \\ 1, & \text{ if } \det a  >0. \end{cases}
\end{equation}
For $\alpha \in \mathbb{Z}/4\mathbb{Z}$ and $s \in \mathbb{C}$, consider the character 
\begin{equation}
\chi^\alpha |\cdot|^s: M \to \mathbb{R}, \quad (m(a),\epsilon) \mapsto \chi(m(a),\epsilon)^\alpha |\det a|^s
\end{equation}
and the smooth induced representation
\begin{equation}
I^\alpha(s) = \mathrm{Ind}_{P}^{G'} \chi^\alpha |\cdot|^s
\end{equation}
with its $\mathcal{C}^\infty$ topology, where the induction is normalized so that $I^\alpha(s)$ is unitary when $\mathrm{Re}(s)=0$. In concrete terms, $I^\alpha(s)$ consists of smooth functions $\Phi\colon G' \to \mathbb{C}$ satisfying 
%\todo{tweaked formatting. LG: OK.}
\begin{equation} \label{eq:principal_series_case_1}
\Phi\left((m(a),\epsilon)\, \underline{n}(b) \, g'\right) = \chi(m(a),\epsilon)^\alpha \, |\det a|^{s+\rho_r}  \, \Phi(g'), \quad \rho_r:=\tfrac{r+1}{2},
\end{equation}
with the action of $G'$ defined by $r(g')\Phi(x)=\Phi(xg')$. Note that, by the Cartan decomposition $G'=PK'$, any such function is determined by its restriction to $K'$; in particular, given $\Phi(s_0) \in I^\alpha(s_0)$, there is a unique family $(\Phi(s) \in I^\alpha(s))_{s \in \mathbb{C}}$ such that $\Phi(s)|_{K'}=\Phi(s_0)|_{K'}$ for all $s$. Such a family is called a standard section of $I^\alpha(s)$.

% Highest weight vectors in the induced representation
\subsubsection{} We denote by $\chi^\alpha$ the character of $K'$ whose differential restricted to $\mathfrak{h}'$ has weight $\tfrac{\alpha}{2}(1,\ldots,1)$. 
%\todo{rephrased for flow. LG: I prefer this too.}
The $K'$-types appearing in $I^{\alpha}(s)$ were determined by Kudla and Rallis \cite{KudlaRallisDegenerate} to be precisely those irreducible representations $\pi_{\lambda}$ of $K'$ with highest weight $\lambda=(l_1,\ldots,l_r)$ (here $l_1 \geq \cdots \geq l_r$) such that $\pi_\lambda \otimes (\chi^{\alpha})^{-1}$ descends to an irreducible representation of $\mathrm{U}(r)$ and satisfies
\begin{equation} \label{eq:restriction_lambda_K_type}
l_i \in \frac{\alpha}{2}+2\mathbb{Z}, \quad 1 \leq i \leq r.
\end{equation} 

%We will use some of the results about the $K'$-types of $I^
%\alpha(s)$ established by Kudla and Rallis. Namely, in \cite{KudlaRallisDegenerate} they show that the $K'$-types that appear are the irreducible representations $\pi_{\lambda}$ of $K'$ with highest weight $\lambda=(l_1,\ldots,l_r)$ (here $l_1 \geq \cdots \geq l_r$) such that $\pi_\lambda \otimes (\chi^{\alpha})^{-1}$ descends to an irreducible representation of $\mathrm{U}(r)$ and satisfies
%\begin{equation} \label{eq:restriction_lambda_K_type}
%l_i \in \frac{\alpha}{2}+2\mathbb{Z}, \quad 1 \leq i \leq r.
%\end{equation}
Moreover, these $K'$-types appear with multiplicity one in $I^\alpha(s)$ (op.\ cit., p.31). If $\Phi^\lambda(\cdot,s) \in I^\alpha(s)$ is a non-zero highest weight vector of weight $\lambda$, then $\Phi^\lambda(e,s) \neq 0$ (op.\ cit., Prop. 1.1), hence from now on we normalize all such highest weight vectors so that $\Phi^\lambda(e,s)=1$. Note that the restriction of $\Phi^\lambda$ to $K'$ is independent of $s$. For scalar weights $\lambda=l(1,\ldots,1)$ with $l \in \tfrac{\alpha}{2}+2\mathbb{Z}$, we have
\begin{equation} \label{eq:def_scalar_weight_vector_case_1}
\Phi^{l}(k',s):= \Phi^{l(1,\ldots,1)}(k',s) = (\det k')^l, \quad k' \in K'.
\end{equation}

% Lowering operators in $I^\alpha(s)$
\subsubsection{} \label{subsection:lowering_ops_case_1} Suppose that $X \in \mathfrak{g}'$ is a highest weight vector for $K'$ of weight $\lambda$. Then $X\Phi^l(s)$ is a highest weight vector of weight $\lambda+l$ and hence multiplicity one of $K'$-types implies that
\begin{equation}
X\Phi^l(s) = c(X,l,s) \Phi^{\lambda+l}(s)
\end{equation}
for some constant $c(X,l,s) \in \mathbb{C}$. We will need to determine this constant explicitly for certain choices of $X$; let %\todo{removed ``namely"}
\begin{equation}
e_i = ( 0, \ldots, 0, \underset{i-\mathrm{th}}{1}, 0, \ldots, 0 )
\end{equation}
and define $X_i^-=p_-(d(e_i))$. When $i=r$, the vector $X_r^-$ is a highest weight vector for $K'$ of weight $-2e_r$. 
%This is just because for any $\alpha \in \Delta_c^+$, the sum $\alpha-2e_r$ does not belong to $\Delta \cup 0$, as one sees immediately from \eqref{eq:roots_of_g}. 
%Let $\iota_i:\mathrm{Mp}_2(\mathbb{R}) \to G$ be the embedding defined by
%\begin{equation}
%\left( \begin{pmatrix} a & b \\ c & d \end{pmatrix}, \epsilon \right) \mapsto \left( \begin{pmatrix} 1_{i-1} & & & 0_{i-1} & &  \\ & a & & & b & \\ & & 1_{r-i} & & & 0_{r-i} \\ 1_{i-1} & & & 0_{i-1} & &  \\ & c & & & b & \\ & & 1_{r-i} & & & 0_{r-i} \end{pmatrix},\epsilon \right).
%\end{equation}
Let $\iota_r \colon \mathrm{Mp}_2(\mathbb{R}) \to G'$ be the embedding defined by
\begin{equation} \label{eq:embedding_iota_r}
\left( \begin{pmatrix} a & b \\ c & d \end{pmatrix}, \,  \epsilon \right) \mapsto \left( \left( \begin{smallmatrix} 1_{r-1} & & 0_{r-1} & \\ & a & & b \\ 0_{r-1} & & 1_{r-1} & \\ & c & & d \end{smallmatrix} \right), \, \epsilon \right).
\end{equation}
Then, for any $\Phi \in I^\alpha(s)$, the value of $X_r^-\Phi(e,s)$ only depends on the pullback function $\iota_r^*\Phi(s):\mathrm{Mp}_2(\mathbb{R}) \to \mathbb{C}$. Note %\todo{``Recall" doesn't seem to fit here. LG: Agreed} 
that every element $g'$ of $\mathrm{Mp}_2(\mathbb{R})$ can be written uniquely in the form
\begin{equation}
g' \ = \  \left(n(x)m(y^{1/2}),\, 1 \right) \tilde{k}_{\theta},
\end{equation}
where $x \in \mathbb{R}$, $y \in \mathbb{R}_{>0}$, $\theta \in \mathbb{R}/4\pi\mathbb{Z}$ and we define
%\todo{tweaked some formatting here. LG: OK.}
\begin{equation}
\tilde{k}_\theta = \begin{cases} (k_\theta,1), & \text{ if } -\pi < \theta \leq \pi, \\ (k_\theta,-1), & \text{ if } \pi < \theta \leq 3\pi, \end{cases} \qquad \text{where } k_\theta =  \begin{pmatrix} \cos \theta & \sin \theta \\ -\sin \theta & \cos \theta \end{pmatrix}.
\end{equation}
We think of $x$, $y$ and $\theta$ as coordinates on $\mathrm{Mp}_2(\mathbb{R})$. In terms of these coordinates, we have
%\begin{equation} \label{eq:formula_lowering_explicit}
%X_r^-\Phi(e) = \left(-2iy \frac{d}{d\overline{\tau}}+\frac{i}{2}\frac{\partial}{\partial \theta}\right)\iota_r^*\Phi(e),
%\end{equation}
%\begin{equation} \label{eq:formula_lowering_explicit}
%X_r^-\Phi \  = \  \left(-i \frac{\partial}{\partial x} \, + \frac{\partial}{\partial y}+\frac{i}{2}\frac{\partial}{\partial \theta}\right)\iota_r^*\Phi,
%\end{equation}
\begin{equation} \label{eq:formula_lowering_explicit}
X_r^-\Phi \  = \  \left(-2i y \frac{d}{d\overline{\tau}}+\frac{i}{2}\frac{d}{d \theta}\right)\iota_r^*\Phi,
\end{equation}
for any $\Phi \in I^\alpha(s)$, where $\tfrac{d}{d\overline{\tau}} = \tfrac{1}{2} \left( \tfrac{d}{dx}+i\tfrac{d}{dy} \right)$. 
%, where $\tfrac{d}{d\overline{\tau}}=\tfrac{1}{2} \left( \tfrac{\partial}{\partial x}+i\tfrac{\partial}{\partial y} \right)$. 
Taking $\Phi=\Phi^l$, the following lemma is a straightforward computation using the explicit expression
\begin{equation}
\iota_r^*\Phi^l(x,y,\theta,s)=y^{\tfrac{1}{2}\left(s+\tfrac{r+1}{2}\right)}e^{i\theta l}.
\end{equation}
\begin{lemma} \label{lemma:lowering_explicit}
Let $l \in \tfrac{\alpha}{2} + 2\mathbb{Z}$. Then
\begin{equation*}
X_r^-\Phi^{l(1,\ldots,1)}(s)=\tfrac{1}{2}\left(s+\tfrac{r+1}{2}-l\right) \Phi^{l(1,\ldots,1)-2e_r}(s).  \qed
\end{equation*}
% \qed \todo{LG: Do you really want the qed sign here?}
\end{lemma}
In other words, this shows that $c(X_r^-,l,s) = \tfrac{1}{2}\left(s+\tfrac{r+1}{2}-l\right)$.
\subsubsection{} \label{subsubsection:rep_R_r_case_1} We now return to the quadratic space $V$ over $\mathbb{R}$ of signature $(p,2)$. Define
%Let $\alpha(V)=p+2$ and define
\begin{equation} \label{eq:def_induced_rep_and_s0_case_1}
I_r(V,s) = I^{ \, \dim V}(s) \qquad  \text{ and }  \qquad 
s_0 = \frac{p-r+1}{2};
\end{equation}
we assume that $r \leq p+1$, so that
\begin{equation}
s_0 \geq 0.
\end{equation} 

Consider the Weil representation $\omega=\omega_\psi$ of $G' \times \mathrm{O}(V)$ on $\mathcal{S}(V^r)$, as in \Cref{section: Arch Weil representation}, and let $R_r(V)$ be its maximal quotient on which  $(\mathfrak{so}(V),\mathrm{O}(V^+) \times \mathrm{O}(V^-))$ acts trivially. The map
\begin{equation}
\lambda \colon \mathcal S(V^r ) \ \to \ I_r(V,s_0), \qquad \lambda(\varphi)(g') :=   \left(\omega(g')\varphi \right)(0)
\end{equation}
is $G'$-intertwining and factors through $R_r(V)$.
%Note that for any $\phi \in \mathcal{S}(V^r)$, the function
%\begin{equation}
%\Phi(g') = \omega(g')\phi(0), \quad g' \in G',
%\end{equation} 
%belongs to $I_r(V,s_0)$; the resulting map $\phi \mapsto \Phi : \mathcal{S}(V^r) \to I_r(V,s_0)$ is $G'$-intertwining and factors through $R_r(V)$. 
By \cite{KudlaRallisDegenerate}, it defines an embedding
\begin{equation}
R_r(V) \hookrightarrow I_r(V,s_0).
\end{equation}
%Since $I^\alpha(0)$ is unitary, this realizes $R_r(V)$ as a direct summand of $I^\alpha(0)$.

\subsection{Degenerate principal series of $\mathrm{U}(r,r)$} \label{subsection:degenerate principal series of Urr} In this section we fix $r \geq 1$ and let $G'=\mathrm{U}(r,r)$. We abbreviate $N=N_r$, $M=M_r$ and $K'=K'_r$ (see \Cref{subsubsection:symplectic_group_definitions}). Our setup follows \cite{Ichino2007}.

\subsubsection{} \label{subsection:Lie_algebra_u_rr} We denote by $\mathfrak{g}'$ the complexified Lie algebra of $G'$ and let $\mathfrak{g}'_{ss} = \{X \in \mathfrak{g}'|\mathrm{tr}X=0\}$. Let %\todo{I think neither $r$ nor $\phi$ are good choices for notation here; tried rephrasing to avoid this. LG: Yes I agree.}
\begin{equation}
u = \frac{1}{\sqrt{2}} \begin{pmatrix} 1_r & 1_r \\ i1_r & -i1_r \end{pmatrix} \in \mathrm{U}(2r),
\end{equation}
and note that $u^{-1} G' u$ is the isometry group of the Hermitian form determined by $\left(\begin{smallmatrix} 1_r & 0 \\ 0 & -1_r \end{smallmatrix}\right)$.
%and define $\phi(g)=rgr^{-1}$ for $g \in G'$; note that $\phi^{-1}(G')$ is the isometry group of the Hermitian form determined by $\left(\begin{smallmatrix} 1_r & 0 \\ 0 & -1_r \end{smallmatrix}\right)$.
In addition, we have
\begin{equation}
K'_r = \left\{  [k_1, k_2] =  u   \begin{pmatrix} k_1 & 0 \\ 0 & k_2 \end{pmatrix}  u^{-1} \ |  \ k_1,k_2 \in \mathrm{U}(r) \right\},
\end{equation}
see \eqref{eq:isom_Kr_case2}, and the Harish-Chandra decomposition
\begin{equation}
\mathfrak{g}' = \mathfrak{p}_+ \oplus \mathfrak{p}_- \oplus \mathfrak{k}',
\end{equation}
with $\mathfrak{g}'=\mathrm{Mat}_{2r}(\mathbb{C})$ and 
%\todo{$M_{2r}$ looks a bit too much like the Levi factor $M$. LG : Agreed.}
\begin{equation}
\begin{split}
\mathfrak{k}' &= \mathrm{Lie}(K'_r)_\mathbb{C} =  \left\{ \left. u   \begin{pmatrix} X_1 & 0 \\ 0 & X_2 \end{pmatrix} u^{-1} \right| X_1, X_2 \in \mathrm{Mat}_r(\mathbb{C}) \right\}, \\
\mathfrak{p_+} &= \left\{ \left. p_+(X):=u \begin{pmatrix} 0 & X \\ 0 & 0 \end{pmatrix} u^{-1}\right| X \in \mathrm{Mat}_r(\mathbb{C}) \right\}, \\ 
\mathfrak{p}_-&=  \left\{ \left. p_-(X):= u \begin{pmatrix} 0 & 0 \\ X & 0 \end{pmatrix} u^{-1} \right| X \in \mathrm{Mat}_r(\mathbb{C}) \right\}.
\end{split}
\end{equation}
Let $\mathfrak{k}'_{ss} = \mathfrak{k'} \cap \mathfrak{g}'_{ss}$ and 
\begin{equation}
\mathfrak{h} = \left\{ u \cdot d(x) \cdot u^{-1} \ | \  x=(x_1,\ldots,x_{2r}) \in \mathbb{C}^{2r}, \ x_1+\ldots+x_{2r} = 0 \right\}.
\end{equation}
Then $\mathfrak{h}$ is a Cartan subalgebra of $\mathfrak{k}'_{ss}$. For $1 \leq i \leq {2r}$, the assignment $u \cdot d(x) \cdot u^{-1} \mapsto x_i$ defines a functional $e_i \colon \mathfrak{h} \to \mathbb{C}$. We write $\Delta$ for the set of roots of $(\mathfrak{g}_{ss},\mathfrak{h})$ and fix the set of positive roots $\Delta^+=\Delta_c^+ \sqcup \Delta_{nc}^+$ given by
\begin{equation} \label{eq:roots_of_g_2}
\begin{split}
\Delta_{c}^+ &= \{e_i-e_j | 1 \leq i < j \leq r\} \cup \{ -e_i+e_j | r<i<j \leq 2r\},\\
\Delta_{nc}^+ &= \{e_i-e_j|1 \leq i \leq r< j \leq 2r\} ,
\end{split}
\end{equation}
where $\Delta_{c}^+$ and $\Delta_{nc}^+$ denote the compact and non-compact roots respectively.

% Siegel parabolic and degenerate principal series representation
\subsubsection{} The group $P=MN$ is the Siegel parabolic of $G'$. For a character $\chi$ of $\mathbb{C}^\times$ and $s \in \mathbb{C}$, define a character $\chi |\cdot|_\mathbb{C}^s:P \to \mathbb{C}^\times$ by
\begin{equation}
\chi |\cdot|_\mathbb{C}^s(m(a)n(b))=\chi(\det a) |\det a|_\mathbb{C}^s
\end{equation}
and let 
\begin{equation}
I(\chi,s)=\mathrm{Ind}_P^{G'}(\chi |\cdot|_\mathbb{C}^s)
\end{equation}
be the degenerate principal series representation of $G'$. Thus $I(\chi,s)$ is the space of smooth functions $\Phi:G' \to \mathbb{C}$ satisfying
\begin{equation} \label{eq:principal_series_case_2}
\Phi(m(a)n(b)g') = \chi(\det a) |\det a|_\mathbb{C}^{s+\rho_r} \Phi(g'), \quad \rho_r := \tfrac{r}{2},
\end{equation}
and the action of $G'$ is via right translation: $r(g')\Phi(x)=\Phi(xg')$. Here the induction is normalized so that the $I(\chi,s)$ is unitary when $\chi$ is a unitary character and $s=0$. Note that, by the Cartan decomposition $G'=PK'$, any such function is determined by its restriction to $K'$; in particular, given $\Phi(s_0) \in I(\chi,s_0)$, there is a unique family $(\Phi(s) \in I(\chi,s))_{s \in \mathbb{C}}$ such that $\Phi(s)|_{K'}=\Phi(s_0)|_{K'}$ for all $s$. Such a family is called a standard section of $I(\chi,s)$.

% Highest weight vectors in the degenerate principal series
\subsubsection{} 
%We will need 
Some useful facts on $K'$-types of $I(\chi,s)$ were proved by Lee \cite{Lee}, namely that
%Namely, 
$I(\chi,s)$ is multiplicity free as a representation of $K'$, and if $\Phi^{(\lambda_1,\lambda_2)}(\cdot,s) \in I(\chi,s)$ is a highest weight vector of weight $(\lambda_1,\lambda_2)$, then $\Phi^{(\lambda_1,\lambda_2)}(e,s) \neq 0$. Hence from now on we normalize all such highest weight vectors so that $\Phi^{(\lambda_1,\lambda_2)}(e,s)=1$. Note that the restriction of $\Phi^{(\lambda_1,\lambda_2)}$ to $K'$ is independent of $s$, and for scalar weights
\begin{equation} \label{eq:def_scalar_weight_case_2}
l=(l_1(1,\ldots,1),l_2(1,\ldots,1)), \quad l_1,l_2 \in \mathbb{Z},
\end{equation}
we have
\begin{equation} \label{eq:def_scalar_weight_vector_case_2}
\Phi^{(l_1,l_2)}([k_1,k_2],s):= \Phi^{(l_1(1,\ldots,1),l_2(1,\ldots,1))}([k_1,k_2],s) = (\det k_1)^{l_1} (\det k_2)^{l_2}, \quad k_1,k_2 \in \mathrm{U}(r).
\end{equation}

\subsubsection{} %{\bf TO DO: Result on lowering operators, i.e. the analog of 3.1.4.} DONE!
Let $X_r^- \in \mathfrak{p}_-$ be as in \Cref{subsection:lowering_ops_case_1}; that is, $X_r^-=p_-(d(e_r))$ with $e_r=(0,\ldots,0,1) \in \mathbb{C}^r$. Then $X_r^-$ is a highest weight vector for $K'$ of weight $(-e_r,e_r)$
% To see this it is easiest to look directly at the action of $\mathfrak{k}'$ on $\mathfrak{p}_-$. We can work without applying $\phi$ and then it's clear.
and so
\begin{equation} \label{eq:def_c_s_case_2}
X_r^-\Phi^l(s) \  = \  c(X_r^-,l,s) \, \Phi^{l+(-e_r,e_r)}(s)
\end{equation}
for some constant $c(X_r^-,l,s) \in \mathbb{C}$. To compute this constant, note that $\mathrm{U}(r,r)\cap \mathrm{GL}_{2r}(\mathbb{R}) = \mathrm{Sp}_{2r}(\mathbb{R})$ and $X_r^- \in \mathfrak{sp}_{2r,\mathbb{C}}$. Moreover, if $\Phi \in I(\chi,s)$, then the restriction $\Phi|_{\mathrm{Sp}_{2r}(\mathbb{R})}$ belongs to $I^\alpha(s')$, where $s'=2s+(r-1)/2$ and $\alpha=0$ if $\chi|_{\mathbb{R}^\times}$ is trivial and $\alpha=2$ otherwise (this follows directly from a comparison of \eqref{eq:principal_series_case_1} and \eqref{eq:principal_series_case_2}). If $l=(l_1(1,\ldots,1),l_2(1,\ldots,1))$ is a scalar weight, then
\begin{equation}
\Phi^l|_{\mathrm{Sp}_{2r}(\mathbb{R})}(s) = \Phi^{(l_1-l_2)\cdot (1,\ldots,1)}(s') \in I^{\alpha}(s') 
\end{equation}
and so \Cref{lemma:lowering_explicit} and the above remarks show that the constant in \eqref{eq:def_c_s_case_2} is given by
\begin{equation} \label{eq:formula_c_s_case_2}
c(X_r^-,l,s) = s+\rho_r-\frac{l_1-l_2}{2}.
\end{equation}

\subsubsection{} \label{subsubsection:rep_R_r_case_2} We now return to the $m$-dimensional hermitian space $V$ over $\mathbb{C}$ of signature $(p,q)$ with $pq \neq 0$. From now on we fix a character $\chi=\chi_V$ of $\mathbb{C}^\times$ such that $\chi|_{\mathbb{R}^\times} = \mathrm{sgn}(\cdot)^m$ and define
\begin{equation} \label{eq:def_induced_rep_and_s0_case_2}
\begin{split}
I_r(V,s) &= I(\chi,s), \\
s_0 &= \frac{m-r}{2}.
\end{split}
\end{equation}
We assume that $r \leq p+1$, so that
\begin{equation}
s_0 \geq \frac{q-1}{2} \geq 0,
\end{equation} 
with equality only when $q=1$ and $r=p+1$.

Consider the Weil representation $\omega=\omega_{\psi,\chi}$ of $G' \times \mathrm{U}(V)$ on $\mathcal{S}(V^r)$,  as in \Cref{section: Arch Weil representation}, and let $R_r(V)$ be its maximal quotient on which $(\mathfrak{u}(V),K)$ acts trivially. The map
	\begin{equation}
	\lambda \colon \mathcal S(V^r ) \ \to \ I_r(V,s_0), \qquad \lambda(\varphi)(g') :=   \left(\omega(g')\varphi \right)(0)
	\end{equation}
	is $G'$-intertwining and factors through $R_r(V)$; by \cite[Thm. 4.1]{LeeZhu1}, it defines an embedding
\begin{equation}
R_r(V) \hookrightarrow I_r(V,s_0).
\end{equation}
%
%Note that for any $\phi \in \mathcal{S}(V^r)$, the function
%\begin{equation}
%\Phi(g') = \omega(g')\phi(0), \quad g' \in G',
%\end{equation} 
%belongs to $I_r(V,s_0)$; the resulting map $\varphi \mapsto \Phi : \mathcal{S}(V^r) \to I_r(V,s_0)$ is $G'$-intertwining and factors through $R_r(V)$. 
%By \cite[Thm. 4.1]{LeeZhu1}, it defines an embedding
%\begin{equation}
%R_r(V) \hookrightarrow I_r(V,s_0).
%\end{equation}

\subsection{Whittaker functionals} \label{subsection:Whittaker_functionals}

% Definitions
\subsubsection{} \label{subsubsection:archimedean_Whittaker_functionals} Let 
\begin{equation}
T \in  \begin{cases} \mathrm{Sym}_r(\mathbb{R}), & \text{ case 1} \\ \mathrm{Her}_r, & \text{ case 2 } \end{cases} 
\end{equation}
be a non-singular matrix and let $\psi_T \colon N_r \to \mathbb{C}^\times$ denote the character defined by %\todo{tweaked formatting. LG: OK}
\begin{equation}
\psi_T(\underline{n}(b))=\psi(\mathrm{tr}(Tb)) = e^{2 \pi i\, \tr(Tb)}.
\end{equation}
 Let $I_r(V,s)$ and $s_0$ be as in \eqref{eq:def_induced_rep_and_s0_case_1} (resp. \eqref{eq:def_induced_rep_and_s0_case_2}) in case 1 (resp. in case 2). A continuous functional $l  \colon I_r(V,s_0) \to \mathbb{C}$ is a called a Whittaker functional if it satisfies 
\begin{equation}
l(r(n)\Phi)=\psi_T(n) \cdot l(\Phi) \quad \text{ for all } n \in N_r, \ \Phi \in I_r(V,s_0).
\end{equation}
Such a functional can be constructed as follows. Let $dn$ be the Haar measure on $N_r$ that is self-dual with respect to the pairing $(\ul{n}(b), \ul{n}(b')) \mapsto \psi(\mathrm{tr}(bb'))$. Embed $\Phi \in I_r(V,s_0)$ in a (unique) standard section $(\Phi(s))_{s \in \mathbb{C}}$ and define
\begin{equation} \label{eq:W_T_local_field}
W_T(\Phi,s)=\int_{N_r} \Phi(\underline{w}_r^{-1}n,s) \, \psi_T(-n)  \, dn
\end{equation}
The integral converges for $\mathrm{Re}(s) \gg 0$ and admits holomorphic continuation to all $s$ \cite{Wallach06}; its value at $s_0$ defines a Whittaker functional $W_T(s_0)$. It is shown in op. cit. that $W_T(s_0)$ spans the space of Whittaker functionals. 
For $g' \in G'_r$ and $\Phi \in I_r(V,s_0)$, define
\begin{equation} \label{eq:def_W_T_coefficient}
W_T(g',\Phi,s) = \int_{N_r} \Phi(\underline{w}_r^{-1}n  g',s) \, \psi_T(-n) \, dn.
\end{equation}
and
\begin{equation} \label{eq:def_W_T_derivative}
W_T'(g',\Phi,s_0) = \left. \frac{d}{ds}W_T(g',\Phi,s)\right|_{s=s_0}.
\end{equation}

% Growth estimates
\subsubsection{}
In this section, we apply results of \cite{ShimuraConfluent} to extract some necessary asymptotic estimates for Whittaker functionals. Assume throughout this section that $T$ is non-degenerate.  As in \Cref{subsubsection:hermitian_domain_definitions}, let
\begin{equation}
	\bbK = \begin{cases} \bbR, & \text{ orthogonal case} \\ \bbC, & \text{ unitary case,} \end{cases} \qquad \text{and} \qquad \iota := [ \bbK : \bbR].
\end{equation}

It will be useful for us to work in symmetric space coordinates, as follows. Let $\bbH_r$ denote the Siegel (resp.\ Hermitian) upper half space of genus $r$, so that in the orthogonal case,
\begin{equation}
\bbH_r \ =\  \left\{ \tau = x + i y \in \Sym_r(\bbC) \ | \ y > 0  \right\}  
\end{equation}
and in the unitary case
\begin{equation}
 \bbH_r \ = \ \left\{  \tau  \in \Mat_{r}(\bbC) \ | \ \frac{1}{2i} (\tau -  \transpose{\overline{\tau}}) > 0  \right\} ;
\end{equation}
in the latter case, write $\tau = x + iy$ with $y = \frac1{2i} (\tau - \transpose{\overline \tau}) \in \Herm_r(\bbC)_{>0}$ and $x = \tau - iy$.

For a point $\tau  = x + iy \in \bbH_r$, fix a matrix $\alpha \in \GL_r(\bbK)$ with $\det \alpha \in \mathbb{R}_{>0}$ and such that $y = \alpha \cdot \transpose{\overline \alpha}$ and let 
\begin{equation} \label{eq:def_g_z'}
g_\tau' := \underline{n}(x) \underline{m}(\alpha) \in G_r'.
\end{equation}
Let $\Phi^{l}(s)$ be the normalized highest weight vector of $I_r(V,s)$ of scalar weight %\todo{tweaked formatting. LG: OK}
\begin{equation} \label{eq:def_scalar_weight_both_cases}
l := \begin{dcases} \frac{m}{2}, & \text{orthogonal case} \\ \left( \frac{m+k(\chi)}{2}, \frac{-m+k(\chi)}{2} \right), & \text{unitary case} \end{dcases}
\end{equation}
(see \eqref{eq:def_scalar_weight_vector_case_1} and \eqref{eq:def_scalar_weight_vector_case_2}) and define
\begin{equation} \label{eqn:ArchWhittakerClassicalNormalization}
	\archW_{T}(y, s) \  := \  (\det y)^{-\iota m/4}  \, W_T(g'_{iy} , \Phi^l,s) \, e^{ 2 \pi \tr(T y)};
\end{equation}
note that this is independent of the choice of $\alpha$ in \eqref{eq:def_g_z'}.

\begin{proposition} \label{prop:NonDegenWhittakerEstimates}
	Suppose $T$ is positive definite of rank $r$. %\todo{LG: remove numbers from eqs?}
		\begin{enumerate}[(i)] 
			\item For any integer $k \geq 0$ and any fixed  $s \in \bbC$, 
				\begin{equation*} 
						\lim_{\lambda \to \infty}  \frac{\partial^k}{\partial s^k} \archW_T(\lambda y, s)    <  \infty. 
				\end{equation*} 
				\item There exists a constant $C>0$, depending on $s, y, T$ and $k$, such that
\begin{equation*}
			 \frac{\partial}{\partial \lambda}  \left[	\frac{\partial^k}{\partial s^k} \archW_T(\lambda y, s)  \right] =  O(\lambda^{-1-C})
\end{equation*}
				 as $\lambda \to \infty$.

%  we have the more precise values at $s_0(n)$:
\suspend{enumerate}
 Let  $\kappa = 1 + \frac{\iota}2 (r -1)$ and 
 \[  s_0(r) = \begin{cases} \frac{m- (r+1)}2, & \text{orthogonal case} \\ \frac{m-r}2, & \text{unitary case.} \end{cases} \]  
 Then we have the following more precise results for the value and derivative at $s_0(r)$:
\resume{enumerate}[{[(i)]}]
	 	\item   
			\begin{equation*} %\label{eqn:PosDefWhittakerValue}
				\archW_T(y,s_0(r)) \ = \ \frac{  (- 2\pi i )^{\iota rm/2}  }{ 2^{r (\kappa-1)/2} \,   \Gamma_r(\iota m/2)}  (\det T)^{\iota  s_0(r)} 
			\end{equation*}
		
		\item  There is an asymptotic formula
			\begin{align*} % \notag
					\archW'_T(\lambda   y, s_0(r)) =  \left(\frac{\iota}{2} \right)   \cdot  \frac{  (- 2\pi i )^{\iota rm/2} \, (\det T)^{\iota \, s_0(r)}  }{ 2^{r (\kappa-1)/2}  \,  \Gamma_r(\iota m/2)}&    \left[ \log \det \pi T   - \frac{ \Gamma_r'(\iota m/2)}{\Gamma_r(\iota m/2)} \right]  \\
					& \ \ \ +   O(\lambda^{-1})  
%				   	\label{eqn:PosDefWhittakerDeriv}
			\end{align*}
		as $\lambda \to \infty$, where the implied constant depends on $y$ and $T$. 

\end{enumerate}

\end{proposition}
%Note that when $k=0$, these results follow immediately from \cite[(4.13.K)]{Shimura}.
%\comm{unify notation with unitry case....}

\begin{proof} 
Write $T = (T^{\frac12}) \cdot \transpose{ (\overline{T^{\frac12}})}$, for some matrix $T^{\frac12} \in \GL_r(\bbK)$. If we define
\begin{equation}
g \ := \ 4 \pi  \ \transpose{ \overline{(T^{\frac12})}} \cdot y \cdot (T^{\frac12}),
\end{equation}
then applying \cite[(1.29), (3.3), (3.6)]{ShimuraConfluent} gives
%\begin{align} \label{eqn:WTPosDefShimura}
%	\archW_{T} (y, s) = \frac{  (- 2\pi i )^{\iota rm/2} \, \pi^{r \beta} }{ 2^{r(\kappa -1 )/2} \, \Gamma_r(\beta + \iota m/2)} \, (\det T)^{\beta +\iota m/2 - \kappa} \omega(g; \beta + \iota m/2, \beta)
%\end{align}
\begin{align} \label{eqn:WTPosDefShimura}
	\archW_{T} (y, s)  \ = \ \frac{  (- 2\pi i )^{\iota rm/2} \, \pi^{r \beta} }{ 2^{r (\kappa-1)/2}  \, \Gamma_r(\beta + \iota m/2)} \, (\det T)^{\beta +\iota m/2 - \kappa} \, \omega(g; \beta + \iota m/2, \beta)
\end{align}
where 
\begin{itemize}
	\item $\beta = \frac{\iota}{2} ( s- s_0(r))$,
	\item $\Gamma_r(\beta) = \pi^{\iota r(r-1)/4}  \prod_{k=0}^{r-1} \Gamma(\beta - \iota k/2)$, and
	\item  $\omega(g; \alpha, \beta)$ is an entire function in $(\alpha, \beta) \in \bbC^2$, initially defined by the formula %
	%\todo{ replaced $N_{\bbR}^+$ with $N^+$. LG: OK}
	\begin{align}
	\omega(g; \alpha, \beta)  = \Gamma_r(\beta)^{-1} \,  \det(g)^{\beta}   \, \int\limits_{N^+}e^{- \tr (g x)} \det(x + 1)^{\alpha - \kappa} \det(x)^{\beta- \kappa} dx 
	\label{eqn:ShimuraOmegaConvergent}
	\end{align}
	on the region $Re(\beta) > \kappa - 1$, where the integral is absolutely convergent; here we abuse notation and write
	\begin{equation}
		 	N^+  = \begin{cases} \Sym_{r}(\bbR)_{>0}, & \text{orthogonal case} \\ \Herm_r(\bbC)_{>0}, & \text{unitary case}. \end{cases}
	\end{equation}
	and take the measure $dx$ to be the standard Euclidean measure, following \cite[Section 1]{ShimuraConfluent}. The analytic continuation of $\omega(g;\alpha, \beta)$ to $(\alpha, \beta) \in \bbC^2$ is proven in \cite[Theorem 3.1]{ShimuraConfluent}. 
\end{itemize}

Replacing $y$ by $\lambda y$ corresponds to replacing $g$ by $\lambda g$; thus, in light of \eqref{eqn:WTPosDefShimura}, in order to prove parts $(i)$ and $(ii)$ of the proposition, it suffices to prove the corresponding estimates for $\omega(g;\alpha,\beta)$. 
More precisely, we shall show that for  fixed $g$, integers $k,k' \geq 0$, and $(\alpha, \beta) \in \bbC^2$,
\begin{equation} \label{eqn:OmegaEstimate1}
		\lim_{\lambda \to \infty} \left[ \frac{\partial^{k+ k'}}{ \partial \alpha^{k} \, \partial \beta^{k'} } \omega(\lambda g; \alpha, \beta) \right] < \infty  
\end{equation}
and
\begin{equation} \label{eqn:OmegaEstimate2}
\frac{\partial}{\partial \lambda} \left[ \frac{\partial^{k+ k'}}{ \partial \alpha^{k} \, \partial \beta^{k'} } \omega(\lambda g; \alpha, \beta) \right] = O(\lambda^{-2});
\end{equation}
the implied constants  may depend on $k, k', \alpha, \beta$ and $g$.

To prove these estimates, we would like to use \eqref{eqn:ShimuraOmegaConvergent}, but our choice of parameters $(\alpha, \beta)$ may place us outside the range of absolute convergence of the integral; to circumvent this, consider the differential operator 
%\todo{added def of $\Delta$. LG: Yes this is better. I changed the reference to Shimura's paper to include both cases.}
\begin{equation}
\Delta = \begin{cases} \det \left(\tfrac12 (1 + \delta_{ij}) \tfrac{\partial}{\partial g_{ij}} \right) , & \text{orthogonal case} \\ \det( \tfrac{\partial}{\partial g_{ij}}), & \text{unitary case}\end{cases}
\end{equation}
as in \cite[(3.10.I-II)]{ShimuraConfluent}, where $g_{ij}$ are the coordinates for the entries of $g \in N^+$; it satisfies the identity 
\begin{equation}
\Delta \, e^{- \mathrm{tr}{(g)}} \ = \ (-1)^r \, e^{- \mathrm{tr}{(g)}}.
\end{equation}

Then (3.12) and (3.7) of \cite{ShimuraConfluent} together imply that for $N \in \bbZ_{>0}$, 
\begin{align}
\notag		\omega(g; \alpha, \beta) \ =& \ (-1)^{rN} e^{ \tr (g)} \det(g)^\beta \cdot  \Delta^{N}  \left( e^{- \tr (g)} \det(g)^{- \beta } \omega(z; \alpha-N, \beta) \right) \\
=& (-1)^{rN} e^{ \tr (g)} \det(g)^\beta \cdot    \Delta^{N}  \left[ e^{- \tr (g)} \det(g)^{- \beta } \omega \left(g;  \kappa- \beta, \kappa - \alpha + N \right) \right].
\label{eqn:ShimuraOmegaRecurrence}
\end{align}

Fixing $N> \mathrm{Re}(\alpha) -1$, the final incarnation of $\omega(...)$ in \eqref{eqn:ShimuraOmegaRecurrence} is now in the range of convergence of \eqref{eqn:ShimuraOmegaConvergent}; for convenience, rewrite \eqref{eqn:ShimuraOmegaConvergent} for these parameters as
\begin{equation}
\widetilde \omega( g ; \alpha, \beta) \ := \ \omega \left(g;  \kappa- \beta, \kappa - \alpha + N \right).
\end{equation}
By passing partial derivatives in the coordinates of $g$ under the integral \eqref{eqn:ShimuraOmegaConvergent} defining $\widetilde \omega(g; \alpha, \beta)$, it follows that $\omega(g;\alpha, \beta)$ can be written as a finite sum  of terms of the form 
\begin{equation}
		f(\alpha, \beta) \, (\det g)^{\kappa - \alpha + N - m} \cdot F_1(g) \cdot \int\limits_{N^+} e^{-\tr(gx)} \, F_2(x) \, \det(x+1)^{-\beta} \det(x)^{-\alpha + N} dx
\end{equation}
where $f (\alpha, \beta)$ is a holomorphic function independent of $g$,  $m \geq 0$ is an integer, and $F_1(g)$ and $F_2(x)$ are homogeneous polynomials with $ \deg F_1 < m r$; here $F_1$ and $F_2$ arise as products of iterated partial derivatives, of $\det(g)$ and $e^{-\tr(gx)}$ respectively, with respect to the entries of $g$. % \todo{this is a bit more precise. LG: OK}

For a parameter $\lambda>0$ and fixed $g$, replacing $g$ by $\lambda g$ and applying a change of variables in the previous display implies that $\omega( \lambda g; \alpha, \beta)$ can be written as a sum of terms of the form
\begin{equation}
  \lambda^{-M}	f(\alpha, \beta) \, (\det g)^{\kappa - \alpha + N - m} \cdot F_1(g) \cdot \int\limits_{N^+} e^{-\tr(gx)} \, F_2(x) \, \det( \lambda^{-1}x+1)^{-\beta} \det(x)^{-\alpha + N} dx
\end{equation}
where $M = mr - \deg F_1 + \deg F_2 \geq 0$.  It follows (again, for fixed $g$, etc.) that $\frac{\partial^{k+k'}}{\partial \alpha^{k} \partial \beta^{k'}} \omega(\lambda g; \alpha, \beta)$ can be written as a finite sum of terms of the form
\begin{align} \label{eqn:omegaDerivsSumOfTerms}
	\lambda^{-M}  & f( \alpha, \beta) \,  ( \log \det g )^A (\det g)^{\kappa - \alpha + N - m}  \,  F_1(g)   \notag \\
	& \times \int\limits_{N^+} e^{-\tr(gx)} \, F_2(x) \,  \log(\det(\lambda^{-1} x + 1))^B \, \det( \lambda^{-1}x+1)^{-\beta} \log(\det x)^C \det(x)^{-\alpha + N} dx
\end{align}
where $f(\alpha, \beta)$ is independent of $\lambda$, and $A,B,C$ are integers. 

 For $\lambda >1$, we have the (rather crude) estimates
\begin{equation}
\begin{split}
	 |\log(\det(\lambda^{-1} x + 1))^B \, \det( \lambda^{-1}x+1)^{-\beta}|  &\leq |\det( \lambda^{-1}x+1)^{-\beta + B }| \\
	  &\leq \begin{cases}  \det( x+1)^{-Re(\beta) + B } , &  \text{if }  B \geq  - Re(\beta) \\ 1, & \text{otherwise.}	 \end{cases}
	 \end{split}
\end{equation}
Hence, by  dominated convergence, we may pass the limit $\lambda \to \infty$ inside the integral in \eqref{eqn:omegaDerivsSumOfTerms}, and, by comparison with \eqref{eqn:ShimuraOmegaConvergent}, conclude that the limit of this integral exists  as $\lambda \to \infty$. Since  $\frac{\partial^{k+k'}}{\partial \alpha^{k} \partial \beta^{k'}} \omega(\lambda g; \alpha, \beta)$ is a sum of  terms as in \eqref{eqn:omegaDerivsSumOfTerms}, this shows the desired existence of the limit   \eqref{eqn:OmegaEstimate1}. To prove the second estimate \eqref{eqn:OmegaEstimate2}, differentiate \eqref{eqn:omegaDerivsSumOfTerms} with respect to $\lambda$; dominated convergence again allows us to pass the derivative under the integral sign, and the estimate in the same way.  These two estimates in turn imply statements $(i)$ and $(ii)$ of the proposition.

We now turn to the more precise versions at $s = s_0(r)$. To prove $(iii)$, recall that $\beta= 0$ when $s=s_0(r)$; since $\omega(g;\iota m/2,0) \equiv 1$ by \cite[(3.15)]{ShimuraConfluent}, evaluating \eqref{eqn:WTPosDefShimura} at $s=s_0(r)$ yields the desired formula. 

Finally to prove $(iv)$, take the logarithmic derivative with respect to $s$ in \eqref{eqn:WTPosDefShimura} to obtain
\begin{align}
&  \frac{ 2^{r (\kappa-1)/2}  \, \Gamma_r(\iota m/2)} {  (- 2\pi i )^{\iota rm/2}  (\det T)^{\iota  s_0(r)}  } \cdot \archW'_T(\lambda  y, s_0(r))  \\
& \ \ \ \ \ = \  \frac{\archW'_T( \lambda  y,s_0(r))}{\archW_T(\lambda  y,s_0(r))}  = \frac{\iota}{2} \left(  \log \det \pi T  - \frac{ \Gamma_r'(\iota m/2)}{\Gamma_r(\iota m/2)} +  \left. \frac{\partial}{\partial \beta} \omega( \lambda \cdot g; \beta+\iota m/2, \beta) \right|_{\beta = 0}\right). \notag
\end{align}
It will therefore suffice to show that $\frac{\partial}{\partial \beta} \omega( \lambda \cdot g; \beta+\iota m/2, \beta)|_{\beta = 0} = O(\lambda^{-1})$. 

To this end, for sufficiently large integer $N$, we have
\begin{equation}
\omega(g; \beta+\iota m/2, \beta) \ = \ (-1)^{rN} e^{ \tr (g)} \det(g)^\beta \cdot    \Delta^{N}  \left[ e^{- \tr (g)} \det(g)^{- \beta } \widetilde \omega \left(g; \beta + \iota m/2, \beta \right) \right]
\end{equation}
with
\begin{equation}
\widetilde \omega( g ; \beta+\iota m/2, \beta) \ := \ \omega \left(g;  \kappa- \beta, \kappa - \beta - \iota m/2 + N \right).
\end{equation}

Consider expanding $\frac{\partial}{\partial \beta} \omega(\lambda g, \beta+\iota m/2, \beta)$ as a sum of terms as in \eqref{eqn:omegaDerivsSumOfTerms}, and note that any terms with either $F_1$ or $F_2$ non-constant are $O(\lambda^{-1})$. As these terms arise from the partial derivatives of $\det g$ or $\widetilde \omega(g, \beta)$ respectively, it follows that
\begin{equation}
\frac{\partial}{\partial \beta} \omega(\lambda g; \beta+\iota m/2, \beta) = \frac{\partial}{\partial \beta} \widetilde	\omega(\lambda g; \beta+\iota m/2,\beta) + O(\lambda^{-1}).
\end{equation}
We are reduced to proving that $ \frac{\partial}{\partial \beta} \widetilde \omega(\lambda g; \beta+\iota m/2, \beta)|_{\beta = 0}$ is itself $O(\lambda^{-1})$; taking $N$ large enough to ensure the convergence of \eqref{eqn:ShimuraOmegaConvergent}, and applying a change of variables in the integral, gives  
\begin{equation} \label{eqn:OmegaTildeLambda}
\widetilde \omega(\lambda \, g; \beta+\iota m/2,\beta) =   \frac{ (\det g)^{\kappa - \beta - \iota m/2+ N}}{ \Gamma_r(\kappa - \beta - \iota m/2+ N)} \int\limits_{N^+}e^{- \tr (g x)} \det(\lambda^{-1}x + 1)^{- \beta } \det(x)^{-\beta- \iota m/2 + N} dx .
\end{equation}
Now  substitute the Taylor expansion
\begin{equation}
\det(\lambda^{-1} x + 1)^{- \beta}   \ = \  1  \ - \ \log(\det( \lambda^{-1}x + 1)) \, \beta \ + \ \dots
\end{equation}
at $\beta = 0$ into the previous expression,  and note  that the integral in \eqref{eqn:OmegaTildeLambda} is uniformly convergent for $\beta$ in a neighbourhood of $\beta = 0$. Moreover, there is a useful integral representation \cite[(1.16)]{ShimuraConfluent}:
\begin{equation}  
\int\limits_{N^+}e^{- \tr (g x)}  \,  \det(x)^{- \beta - \iota m/2 + N } \, dx \ = \  \Gamma_r(\kappa - \beta - \iota m/2 + N) \, \det(g)^{-\kappa + \beta + \iota m/2 - N}
\end{equation}
Combining these observations, the Taylor expansion of $\widetilde \omega(\lambda g; \beta + \iota m/2, \beta)$ around $\beta = 0$ is of the form
\begin{align}
\widetilde \omega(\lambda g; \beta + \iota m/2, \beta) =& 1  -  \left[   \frac{ (\det g)^{\kappa - \iota m/2 + N}}{ \Gamma_r(\kappa  - \iota m/2 + N)} \int\limits_{N^+}e^{- \tr (g x)} \log (\det(\lambda^{-1}x + 1))   \det(x)^{- \iota m/2 + N} dx \right] \beta \notag \\ 
&  \qquad \qquad \qquad  + \text{ higher order terms in } \beta.
\end{align}
For fixed $g$, the coefficient  of $\beta$   is $O(\lambda^{-1})$ as $\lambda \to \infty$, as can be easily deduced from the estimate
\begin{equation}
\log ( \det (\lambda^{-1} x + 1)) < \lambda^{-1} \tr(x) \ \qquad \ \text{ for } x \in N^+.
\end{equation}
This proves $  \frac{\partial}{\partial \beta} \widetilde \omega(\lambda g; \beta+\iota m/2, \beta)|_{\beta = 0} = O(\lambda^{-1})$ as required.
\end{proof}

%
%\comm{LG:Should we keep this remark?}
%\begin{remark}
% A straightforward generalization of the   proof of  $(iv)$ yields the estimate (for fixed $g$ and $k \geq 1$)
%	\[
%	\left.	\frac{\partial^k}{\partial \beta^k} \omega(\lambda g; \beta + \iota m/2, \beta) \right|_{\beta = 0} \ = \  O(\lambda^{-1}), \qquad  \text{ as } \lambda \to \infty.
%	\]
%	This implies the asymptotic  for the special value of the $k$'th derivative of the Whittaker function:
%	\[
%	\archW^{(k)}_T(\lambda y  , s_0(r))  \ = \ \left(  \frac{\iota}{2}\right)^{k} \   (- 2\pi i )^{r\iota m/2} \cdot 
%	\left.	\left[ \frac{\partial^k}{\partial \beta^k}\frac{   \pi^{r \beta}  (\det T)^{\beta +\iota m/2 - \kappa} }{  \Gamma_r(\beta + \iota m/2)}  \right] \right|_{\beta = 0} \ + \ O(\lambda^{-1}) 
%	\]
%\hfill $\diamond$
%\end{remark}

\begin{proposition} \label{prop:NonDegArchWhittakerNonPos}
	Suppose $T$ is non-degenerate of signature $(p,q)$ with $q>0$. 
	\begin{enumerate}[(i)]	
	\item Let $s_0 (r) $ be as in \Cref{prop:NonDegenWhittakerEstimates}. Then $\archW_T(y,s_0(r)) = 0$ for all $y$.
	\item For any fixed $s_0 \in \bbC$, $k \in \mathbb N$ and $y >0$, there are positive constants $C$ and $C'$ such that
	\begin{equation*} \label{eqn:nonpositive Arch Whittaker estimate}
	\left[	\frac{\partial^k}{\partial s^k} \archW_T(\lambda y,s) \right]_{s = s_0} = O(e^{-C\lambda}) 
	\end{equation*}
	and
	\begin{equation*}
	\frac{\partial}{\partial \lambda} \left( \left[	\frac{\partial^k}{\partial s^k} \archW_T(\lambda y,s) \right]_{s = s_0}   \right) = O(e^{-C' \lambda});	
	\end{equation*} 
	here $C, C'$ and the implied constants also depend on $T, s_0, k$ and $y$.
	\end{enumerate}

\begin{proof}
		Following the notation of \cite{ShimuraConfluent}, choose a symmetric positive matrix $y^{\frac12}$ such that $y = (y^{\frac12})^2$, consider the collection
		\begin{equation}
		(\mu_1,\dots \mu_r ) 
		\end{equation}
		of eigenvalues (repeated with multiplicity) of the matrix $y^{\frac12} T y^{\frac12}$, and define
		\begin{equation}
		\delta_+(y, T) \ := \ \prod_{ \substack{\mu_i > 0 }} \mu_i , \qquad \text{ and } \qquad \delta_-(y,T)  \ := \   \prod_{ \substack{\mu_i < 0 }} |\mu_i|. 
		\end{equation}
		Then, by \cite[(4.34K)]{ShimuraConfluent} there is a function $\omega(y, T; \alpha,\beta)$ that is holomorphic in $(\alpha,\beta) \in \bbC^2$ such that 
		%\todo{tweaked formatting. LG: OK}
		\begin{align} \label{eqn:WTNonPosOmega}
		\archW_T(y,s)  \ &= \ C_{p,q}(\beta) \, \left( \frac{ \det(y)^{- \beta - \iota m/2 + \kappa}}{\Gamma_{p}(\beta + \iota m/2) \cdot  \Gamma_{q}(\beta)}  \right) \delta_{+}(y,T)^{- \kappa + \beta + \iota m/2 + \iota q/4} \delta_-(y,T)^{- \kappa + \beta   - \iota p/4 } \notag   \\
	    & \qquad {}	\times \omega(2 \pi y, T; \beta + \iota m/2, \beta) \cdot  e^{2 \pi \, \tr(yT)} 
		\end{align}
		where $C_{p,q}(\beta)$ is an entire, nowhere vanishing, function depending only on $p$ and $q$; here $\beta = \frac{\iota}{2}(s - s_0(r))$. Note that $\Gamma_q(\beta)^{-1}$ vanishes at $\beta = 0$, while the remaining terms are holomorphic; this proves $(i)$. 
		
		Moreover, for fixed $y$ and $\lambda \in \bbR_{>0}$, we may write 
	\begin{equation}
		\archW_T(\lambda y,s) =  f(y, \beta) \, \frac{  \lambda^{-\iota mq/2}}{\Gamma_p(\beta + \iota m/2) \Gamma_q(\beta)} \, \omega(2 \pi \lambda y, T; \beta + \iota m/2, \beta) \, e^{2 \pi \lambda \, \tr(yT)}
	\end{equation}
		for some function $f(y, \beta)$ that is entire and nowhere-vanishing in $\beta$. 
		
		By \cite[Theorem 4.2]{ShimuraConfluent}, for any compact subset $U$ of $\bbC$, there are positive constants $A$ and $B$ (depending only on $T$ and $U$) such that the estimate 
		\begin{equation}
		|	\omega(2 \pi y, T; \beta + \iota m/2, \beta) \, e^{ 2 \pi \, \tr(yT)} | \ < \ A e^{ - (\tau_{-}(y, T))   } \left(1 +\mu( y, T)^{-B} \right)
		\end{equation}
		holds for all $y \in $ and $\beta \in U$; here $\tau_-( y,T) = \sum_{\mu_i < 0} |\mu_i|$ and $\mu( y, T) = \min(|\mu_i|)$. Since $T$ is not positive definite, at least one eigenvalue $\mu_i$ is negative, so $\tau_-(y, T) > 0$. Replacing $y$ by $\lambda y$ and noting that 
		\begin{equation} \tau_{-}(\lambda y, T) \ = \ \lambda  \, \tau_{-}(y,T) \qquad \text{ and } \qquad  \mu(\lambda y, T) \ = \  \lambda \, \mu(y,T)
		\end{equation}
		it follows easily that for fixed $y$, there is a constant $C$ such that
		\begin{equation}
		\archW_T(\lambda y,s)   \ = \ O(e^{-C \lambda})
		\end{equation}
		uniformly for $s$ in some neighbourhood of $s= s_0$, say. This proves \eqref{eqn:nonpositive Arch Whittaker estimate} for the case $k=0$. The estimates for $k \geq 1$ follow immediately from Cauchy's integral formula
		\begin{equation}
		\left[\frac{\partial^k}{\partial s^k}  	\archW_T(\lambda y,s)  \right]_{s = s_0} \ = \ \frac{k!}{2 \pi i} \ \oint  \frac{ \archW_T(\lambda y,s)}{(s -s_0)^{k+1}} \, ds.
		\end{equation}
		Finally, we turn to the derivative with respect to $\lambda$. The results of \cite[Section 5]{ShimuraConfluent}, see especially Lemma 5.7, imply that
		for fixed $y$, the function
		\begin{equation}
		F(\lambda, s) \ := \ \archW_T(\lambda y,s)
		\end{equation}
		is entire in $s$ and extends to a complex function in $\lambda$ that is holomorphic on $Re(\lambda)>0$ and satisfies, in the same manner as before, the asymptotic $\frac{\partial^k}{\partial s^k}   F(\lambda,s)|_{s=s_0} = O(e^{-C'\, Re(\lambda)})$ for some constant $C'$.
		The proposition follows from another application of Cauchy's integral formula: for a point $\lambda_0 >1$,
		\begin{equation}
		\frac{\partial}{\partial \lambda} \left( \left[	\frac{\partial^k}{\partial s^k} \archW_T(\lambda y, s) \right]_{s = s_0}   \right)_{\lambda = \lambda_0} = \frac{1}{2 \pi i} \oint \frac{1}{(\lambda- \lambda_0)^2} \frac{\partial^k}{\partial s^k} F(\lambda,s)|_{s=s_0} \, d \lambda \ =  O(e^{-C' \lambda_0}),
		\end{equation}
		where the integral is taken around a circle of radius one (say) around $\lambda_0$.
\end{proof}

\end{proposition}

\subsection{The integral of $\archGreen$} For the rest of \Cref{section:archimedean_arith_SW} we assume that $\mathrm{rk}(\vbE)=1$. Recall that we had constructed a Schwartz form
	\begin{equation}
		\nu(\bfv) \ =\ \sum_{ i=1}^r \, \nu_i(\bfv)
	\end{equation}
for $\bfv \in V^r$, as in  \Cref{prop:nu_properties}; note that for the permutation matrices $\epsilon_i$ ($1 \leq i \leq r$) as in \eqref{eq:def_epsilon_i}, we have $\nu_i({\bf v}) = \omega(\underline{m}(\epsilon_i))\nu_r({\bf v})$.

%\todo{some recollections are in order here; it's been 20 pages since we last saw $\nu(\bfv)$... . LG: I fully agree!}

\subsubsection{} \label{subsection:lemma_lowering_W_T_derivative} 
%Let $r \geq 1$ and $\epsilon_i$ ($1 \leq i \leq r$) be \change{the permutation matrices} as in \eqref{eq:def_epsilon_i}, so that $\nu_i({\bf v}) = \omega(\underline{m}(\epsilon_i))\nu_r({\bf v})$. 
Let $z_0 \in \bbD$ be the base point fixed in \Cref{subsubsection:hermitian_domain_definitions}, and consider the Schwartz functions $\tilde{\nu}_i, \tilde{\nu} \in \mathcal{S}(V^r)$ defined by
\begin{equation} \label{eq:def_tilde_nu_i}
\nu_i({\bf v},z_0) \wedge \Omega^{p-r+1}(z_0) = \tilde{\nu}_i({\bf v})  \, \Omega^p(z_0), \qquad \qquad  1 \leq i \leq r
\end{equation}
and
\begin{equation}
  \tilde{\nu} := \tilde{\nu}_1+\ldots,+\tilde{\nu}_r;
\end{equation}
here $\Omega = \Omega_{\vbE} = \frac{i}{2 \pi} c_1(\vbE, \nabla)$. 

The next lemma computes the images of $\tilde{\nu}_1,\ldots, \tilde{\nu}_r$ under the map $\lambda\colon \mathcal{S}(V^r) \to I_r(V,s_0)$ described in \Cref{subsubsection:rep_R_r_case_1,subsubsection:rep_R_r_case_2}. 

\begin{lemma} \label{lemma:nu_and_Phi}
Let $\lambda_0$ be the weight of $K'$ given by \eqref{eq:def_lambda_0} and  $\Phi^{\lambda_0}(s)$ be the unique vector in $I_r(V,s)$ of weight $\lambda_0$ with $\Phi^{\lambda_0}(e,s)=1$. For $1 \leq i \leq r$, let $\tilde{\Phi}_i(s) = r(\underline{m}(\epsilon_i))\Phi^{\lambda_0}(s)$. Then
\begin{equation*}
	\lambda \left( \tilde{\nu}_i \right)(g') \  = \ \omega(g')( \tilde \nu_i)(0) = \ (-1)^{r-1} \tilde{\Phi}_i(g', s_0), \quad g' \in G'_r.
%	\omega(g')\tilde{\nu}_i(0) = (-1)^{r-1}\tilde{\Phi}_i(g',s_0), \quad g' \in G'_r.
	\end{equation*}
%\begin{equation*}
%\omega(g')\tilde{\nu}_i(0) = (-1)^{r-1}\tilde{\Phi}_i(g',s_0), \quad g' \in G'_r.
%\end{equation*}
\end{lemma}
\begin{proof}
Since the map $\lambda \colon \mathcal{S}(V^r) \to I_r(V,s_0)$ 
%in \ref{subsubsection:rep_R_r_case_1} and \ref{subsubsection:rep_R_r_case_2} 
is $G'_r$-intertwining, it suffices to consider the case $i=r$. By \Cref{prop:nu_properties} we have $\nu_r(0)_{[2r-2]}=(-\Omega)^{r-1}$, and hence
\begin{equation} \label{eq:nu_and_Phi}
\omega(g')\tilde{\nu}_r(0) = (-1)^{r-1}\Phi^{\lambda_0}(g',s_0), \quad g' \in G'_r,
\end{equation}
since both sides define highest weight vectors of weight $\lambda_0$ in $I_r(V,s_0)$ (see \Cref{lemma:weight_lambda_0}) that moreover  agree for $g'=1$, and the $K'$-types in this representation appear with multiplicity one.
\end{proof}

Thus, setting
\begin{equation} \label{eq:def_tilde_Phi}
\tilde{\Phi}(s)=\sum_{1 \leq i \leq r} r(\underline{m}(\epsilon_i)) \Phi^{\lambda_0}(s) \in I_r(V,s),
\end{equation}
we have
\begin{equation}\label{eq:relation_nu_tilde{Phi}}
\omega(g')\tilde{\nu}(0)=(-1)^{r-1}\tilde{\Phi}(g',s_0).
\end{equation}
%\begin{equation}\label{eq:relation_nu_tilde{Phi}}
%\omega(g')\nu(0) = (-1)^{r-1}\tilde{\Phi}(g',s_0) \Omega^{r-1}.
%\end{equation}

The following lemma relating the Whittaker functional $W_T(\cdot,s_0)$ evaluated at $\Phi^{\lambda_0}$ with the derivative $W_T'(\Phi^{l},s_0)$ is the main ingredient in the proof of \Cref{thm:height_pairing_and_derivative}. It will be convenient to work in classical coordinates: for any $\Phi \in I_r(V,s)$ and  $y \in \Sym_r(\bbR)_{>0}$ (resp.\ $y \in \Herm_r(\bbC)_{>0}$) in the orthogonal (resp.\ unitary) case, let
	\begin{equation}
		\archW_T(y, \Phi, s) \ := \ (\det y)^{-\iota m/4} \, W_T(g'_{iy},\tilde{\Phi},s) \, e^{2 \pi  \tr(Ty)}
	\end{equation}
	where $g'_{iy} = \ul m(\alpha)$ for any matrix $\alpha \in \GL_r(\bbK)$ with $\det \alpha \in \bbR_{>0}$ and $y = \alpha \cdot \transpose{\overline \alpha}$.
 % LG: Yes, this discussion is much better than mine.
%
%For $\tau=x+iy \in \mathbb{H}_r$, define
%\begin{equation}
%\archW_T(y,\tilde{\Phi},s) = (\det y)^{-\iota m/4} W_T(g'_{\tau},\tilde{\Phi},s) e^{-2\pi i \tr(T\tau)}.
%\end{equation}

\begin{lemma} \label{lemma:lowering_W_T_derivative}
 Let $\archW_T(y,s) = \archW_T(y, \Phi^l,s)$ be given by \eqref{eqn:ArchWhittakerClassicalNormalization} and write $\archW_T'(y,s) = \tfrac{d}{ds}\archW_T(y,s)$. For any $t \in \mathbb{R}_{>0}$ we have
\begin{equation*}
\archW_T(ty,\tilde{\Phi},s_0) = \frac{2}{\iota} \cdot t \frac{d}{dt} \archW_T'(ty,s_0).
\end{equation*}
%\begin{equation*}
%e^{2\pi \mathrm{tr}(d({\bf t})T)}W_T(\underline{m}(d({\bf t}^{1/2})),\tilde{\Phi},s_0)y^{-\frac{\iota rm}{4}}\frac{dy}{y} =\frac{2}{\iota}\frac{d}{dy}f_T'(y,s_0) dy.
%\end{equation*}
\end{lemma}
\begin{proof}
%Since $\archW_T$ is real analytic \cite[Thm. 5.9]{ShimuraConfluent}, it suffices to show that
%\begin{equation} \label{eq:proof_derivative_Whittaker_1}
%\archW_T(ty,\tilde{\Phi},s)  =\frac{2}{\iota} \cdot t \frac{d}{dt}\left(\frac{\archW_T(ty,s)-\archW_T(ty,s_0)}{s-s_0} \right)
%\end{equation}
%For $s \neq s_0$. 
We begin with the orthogonal case.
% If $k \in \mathrm{O}(r)$ (orthogonal case) or $k \in \mathrm{U}(r)$ (unitary case), a change of variables in \eqref{eq:def_W_T_coefficient} shows that
 If $k \in \mathrm{O}(r)$, a change of variables in \eqref{eq:def_W_T_coefficient} shows that
\begin{equation}
W_T(\underline{m}(k)g,\Phi,s) = W_{{^t}\overline{k}T k}(g,\Phi,s), \quad \Phi \in I_r(V,s).
\end{equation}
Hence it suffices to consider the case where $y$ is diagonal, i.e.\, we may assume  \begin{equation}
y=d(y_1,\ldots,y_r).
\end{equation}

 The statement is now a direct computation using \eqref{eq:formula_lowering_explicit}, as follows. For $1 \leq j \leq r$, let $\tilde{\Phi}_j(s)=r(\underline{m}(\epsilon_j))\Phi^{\lambda_0}(s)$. By \Cref{lemma:lowering_explicit},
% or \eqref{eq:formula_c_s_case_2},} 
we can write 
\begin{equation}
\tilde{\Phi}_j(s) = 2 \, (s-s_0)^{-1}  \, \mathrm{Ad}(\underline{m}(\epsilon_j))X_r^- \Phi^{l}(s).
%\tilde{\Phi}_j(s) = \frac{2}{\iota}(s-s_0)^{-1} \mathrm{Ad}(\underline{m}(\epsilon_j))X_r^- \Phi^{l}(s).
\end{equation}
Let $F(g)=W_T(g,\Phi^{l},s)$ and $F_j(g)=W_T(g,\tilde{\Phi}_j,s)$ and set $y^{1/2}:=d(y_1^{1/2},\ldots,y_r^{1/2})$. Using coordinates $x_j$, $y_j$ and $\theta_j$ for the embedding 
\begin{equation}
\iota_j := \underline{m}(\epsilon_j)\iota_r \underline{m}(\epsilon_j)^{-1} \colon \Mp_2(\bbR) \to G'_r = \Mp_{2r}(\bbR)
\end{equation} 
as in \eqref{eq:embedding_iota_r}-\eqref{eq:formula_lowering_explicit}, we compute
\begin{equation} \label{eq:main_lemma_W_T_derivative_eq1}
\begin{split}
 2^{-1}(s-s_0)F_j(\underline{m}(y^{1/2})) &= \mathrm{Ad}(\underline{m}(\epsilon_j)) X_r^-F(\underline{m}({y}^{1/2})) \\
%&= (-2iy_j\frac{d}{d\overline{\tau}_j}+\frac{i}{2}\frac{d}{d\theta_j})F(\underline{m}({\bf y}^{1/2})) \\
&= \left(-iy_j\frac{d}{dx_j}+y_j \frac{d}{dy_j}+\frac{i}{2}\frac{d}{d\theta_j}\right) F(\underline{m}({y}^{1/2})).
%&=(-iy_j\frac{d}{dx_j}+y_j \frac{d}{dy_j}-\frac{\iota m}{4})F(\underline{m}({\bf y}^{1/2})) \\
%%&=\left(\left(\sum_{1 \leq j \leq r} -iy_j\frac{d}{dx_j}\right)+y \frac{d}{dy}-\frac{\iota rm}{4}\right)F(\underline{m}(y^{1/2} \cdot 1_r)) \\
%&= (2\pi y_j T_{jj}+y_j \frac{d}{dy_j}-\frac{\iota m}{4})F(\underline{m}({\bf y}^{1/2}))
\end{split}
\end{equation}
Note that for $\ul n(x) \in N_r$ and $k' \in K'_r$ we have
\begin{equation}
W_{T} \left( \ul n(x) \ul m(y^{\frac12 }) k', \Phi^l,s \right)  \ = \ e^{2 \pi i \, \tr \left(T x\right)} \, (\det k')^l \, W_T \left(\ul m(y^{\frac12}), \Phi^l,s \right)
\end{equation}
and hence, in coordinates $(x_j,y_j, \theta_j)$ as above, we find
\begin{equation}
 \left(-iy_j\frac{d}{dx_j}+y_j \frac{d}{dy_j}+\frac{i}{2}\frac{d}{d\theta_j}\right) F (\underline{m}({y}^{1/2})) \ = \ \left( 2 \pi y_j T_{jj}  +  y_j \frac{d}{d y_j}  - \frac{m}4 \right) F (\underline{m}({y}^{1/2})).
\end{equation}
Substituting this expression in \eqref{eq:main_lemma_W_T_derivative_eq1} we conclude that
\begin{align}
 2^{-1} \, (s-s_0) \, \archW_T(y,\tilde{\Phi}_j,s)  &= \prod_{i=1}^r y_i^{- m/4} \cdot \left[(2\pi y_j T_{jj}+y_j \frac{d}{dy_j}-\frac{ m}{4})F(\underline{m}({\bf y}^{1/2})) \right]\cdot e^{2\pi  \tr(T y))} \notag  \\
 &=  y_j \frac{d}{dy_j} \archW_T(y,s) .
\end{align}
Adding these equations for $j=1, \dots, r$, the lemma follows in the orthogonal case by comparing the Taylor expansions around $s=s_0$. The unitary case follows from analogous considerations, using \eqref{eq:formula_c_s_case_2} instead of \Cref{lemma:lowering_explicit}.
%
%\begin{equation}
%\begin{split}
%&\iota 2^{-1} (s-s_0)\archW_T(d({\bf y}),\tilde{\Phi}_j,s) \\
%&=\iota 2^{-1} (s-s_0)\prod_{j=1}^r y_j^{-\iota m/4} F_j(\underline{m}({\bf y}^{1/2})) e^{2\pi  \tr(Td({\bf y}))}\\
%&= \prod_{j=1}^r y_j^{-\iota m/4} \cdot (2\pi y_j T_{jj}+y_j \frac{d}{dy_j}-\frac{\iota m}{4})F(\underline{m}({\bf y}^{1/2})) \cdot e^{2\pi  \tr(Td({\bf y}))}  \\
%&= y_j \frac{d}{dy_j} (\prod_{j=1}^r y_j^{-\iota m/4}F(\underline{m}({\bf y}^{1/2}))e^{2\pi  \tr(Td({\bf y}))}) \\
%&= y_j \frac{d}{dy_j} \archW_T(d({\bf y}),s) = y_j \frac{d}{dy_j} \left( \archW_T(d({\bf y}),s) - \archW_T(d({\bf y}),s_0)  \right),
%\end{split}
%\end{equation}
%where the last equality holds since $\archW_T(y,s_0)$ is independent of $y$ 
%\todo{Why is this observation necessary? Can we not conclude by comparing Taylor expansions?} 
%by  \eqref{eqn:PosDefWhittakerValue} and Proposition \ref{prop:NonDegArchWhittakerNonPos}.(i). The statement follows by adding these equations for $1 \leq j \leq r$. 
\end{proof}

\subsubsection{} 

%For $T \in X_r$ with $\det T \neq 0$,
Supppose $\det T \neq 0$ and let
\begin{equation}
\Omega_T(V) = \{(v_1,\ldots,v_r) \in V^r \ |  \ (Q(v_i,v_j))_{i,j}=2T\}.
\end{equation}
Thus $\Omega_T(V) \neq \emptyset$ if and only if $(V,Q)$ represents $T$, and in this case $\mathrm{U}(V) = \mathrm{Aut}(V,Q)$ acts transitively on $\Omega_T(V)$; assuming this, let $d\mu({\bf v})$ be an $\mathrm{U}(V)$-invariant measure on $\Omega_T(V)$ and consider the functional $\mathcal{S}(V^r) \to \mathbb{C}$ defined by %\todo{tweaked formatting. LG: OK}
\begin{equation} 
\phi  \ \mapsto \  \int_{\Omega_T(V)} \phi({\bf v}) \, d\mu({\bf v}).
\end{equation}
This functional is obviously $\mathrm{U}(V)$-invariant and non-zero, and hence defines a Whittaker functional on $R_r(V)$. We denote by $d\mu(\bf v)^{\mathrm{SW}}$ the unique $\mathrm{U}(V)$-invariant measure on $\Omega_T(V)$ such that, for any $\phi \in \mathcal{S}(V^r)$, 
\begin{equation} \label{eq:def_SW_measure}
\int_{\Omega_T(V)} \phi({\bf v}) \, d\mu({\bf v})^{\mathrm{SW}} = \gamma_{V^r}^{-1} \cdot W_T(e,\Phi,s_0), \quad \text{ where } \Phi(g):=\omega(g)\phi(0).
\end{equation}

We denote by $dg^{(2)}$ the invariant measure on $\mathrm{U}(V)$ defined as $dg^{(2)}=dp \,  dk$, where $dk$ is the unique Haar measure on $\mathrm{O}(V^+) \times \mathrm{O}(V^-)$ (resp. $\mathrm{U}(V^+) \times \mathrm{U}(V^-)$) 
with total volume one in case 1 (resp. case 2), and $dp$ is the left Haar measure on $P$ induced by the invariant volume form $\Omega^p$ on $\mathbb{D}^+ \cong G/K$.

\begin{lemma} \label{lemma:local_SW}
Let $\lambda_0$ be given by \eqref{eq:def_lambda_0} and let ${\bf v}=(v_1,\ldots,v_r) \in \Omega_T(V)$, where $\det T \neq 0$. Let $G_{\bf v} \subset \mathrm{U}(V)$ be the pointwise stabilizer of $\langle v_1,\ldots,v_r \rangle$ and $\Gamma_{\bf v} \subset G_{\mathbf{v}}^0$ be a torsion free subgroup of finite covolume. Then, for any $g' \in G_r'$, we have
\begin{equation*}
\int_{\Gamma_{\bf v} \backslash \mathbb{D}^+} \omega(g')\nu({\bf v}) \wedge \Omega^{p-r+1} = -\frac{\iota}{2}C_{T,\Gamma_{\bf v}} W_T(g',\tilde{\Phi},s_0),
\end{equation*}
where $\tilde{\Phi}$ is as in \eqref{eq:def_tilde_Phi} and $C_{T,\Gamma_{\bf v}}$ is the non-zero constant given by
\begin{equation*} %\label{eq:C_T_as_measure_quotient}
C_{T,\Gamma_{\bf v}} \ =  \ \frac{2}{\iota} \ (-1)^r \  \mathrm{Vol}(\Gamma_{\bf v} \backslash G_{\bf v}, dg_{\bf v}) \  \frac{dg^{(2)}/dg_{\bf v}}{d\mu({\bf v})^{\mathrm{SW}}}  \ \gamma_{V^r}^{-1}.
\end{equation*}
Here $dg_{\bf v}$ is an arbitrary Haar measure on $G_{\bf v}$ and $dg^{(2)}/dg_{\bf v}$ denotes the quotient measure on $\Omega_{T}(V) \cong \mathrm{U}(V)/G_{\bf v}$ induced by $dg^{(2)}$ and $dg_{\bf v}$.
\end{lemma}
\begin{proof}
The hypothesis that $\det T \neq 0$ 
 implies that $G_{\bf v}^0$ is reductive, hence the invariant measure $dg^{(2)}/dg_{\bf v}$ exists. By \eqref{eq:def_tilde_nu_i} and \Cref{prop:nu_properties}.(d), we have $\tilde{\nu}(gv)=\tilde{\nu}(v)$ for any 
 $g \in \mathrm{U}(V)$ stabilizing $V^-$. 
We compute %\todo{tweaked formatting. LG: OK}
\begin{equation}
\begin{split}
\int_{\Gamma_{\bf v} \backslash \mathbb{D}^+} \omega(g')\nu({\bf v}) \wedge \Omega^{p-r+1} &= \int_{\Gamma_{\bf v} \backslash \mathrm{U}(V)} \omega(g')\tilde{\nu}(g^{-1}{\bf v}) \, dg^{(2)} \\
&= \mathrm{Vol}(\Gamma_{\bf v} \backslash G_{\bf v}, dg_{\bf v}) \int\limits_{G_{\bf v} \backslash \mathrm{U}(V)} \omega(g')\tilde{\nu}(g^{-1}{\bf v}) \,  \frac{dg^{(2)}}{dg_{\bf v}} \\
&= \mathrm{Vol}(\Gamma_{\bf v} \backslash G_{\bf v}, dg_{\bf v}) \, \frac{dg^{(2)}/dg_{\bf v}}{d\mu({\bf v})^{\mathrm{SW}}}  \int_{\Omega_T(V)} \omega(g')\tilde{\nu}({\bf v}) \, d\mu({\bf v})^{\mathrm{SW}} \\
&= (-1)^{r-1} \ \mathrm{Vol}(\Gamma_{\bf v} \backslash G_{\bf v}, dg_{\bf v}) \  \frac{dg^{(2)}/dg_{\bf v}}{d\mu({\bf v})^{\mathrm{SW}}} \  \gamma_{V^r}^{-1}  \  W_T(g',\tilde{\Phi},s_0),%\\
\end{split}
\end{equation}
where the last equality follows from \eqref{eq:def_SW_measure} and \eqref{eq:relation_nu_tilde{Phi}}. 
\end{proof}

\subsubsection{} We will now apply the results obtained in this section to compute the integral \eqref{eq:integral_xi_0} under our standing assumption that $\mathrm{rk}(\vbE)=1$. Rather than aiming for the most general result, we restrict to cocompact arithmetic subgroups $\Gamma$ and integral vectors $\bfv$ (see below); this suffices for the application to compact Shimura varieties in \Cref{section:arith_Siegel_Weil}.

Fix a lattice $L \subset V$ and a torsion free cocompact arithmetic subgroup $\Gamma \subset G$ stabilizing $L$, and let $\Gamma_{\bfv}=\mathrm{Stab}_\Gamma \langle v_1,\ldots,v_r \rangle$ and $X_\Gamma=\Gamma \backslash \mathbb{D}^+$.

\begin{lemma}\label{lemma:convergence_arith_quotient_1}
Assume that $\bfv \in L^r$ is non-degenerate and let $Z(\bfv)_\Gamma$ be the image of the natural map $\Gamma_{\bfv} \backslash \mathbb{D}_{\bfv}^+ \to X_\Gamma$. 
\begin{enumerate}[(a)]
\item The sum
\begin{equation*}
\archGreen(\bfv)_\Gamma = \sum_{\gamma \in \Gamma_{\bfv} \backslash \Gamma} \archGreen(\gamma^{-1}\bfv)
\end{equation*}
converges absolutely to a smooth form on $X_\Gamma-Z(\bfv)_\Gamma$ that is locally integrable on $X_\Gamma$. 
\item As currents on $X_\Gamma$ we have
\begin{equation*}
\sum_{\gamma \in \Gamma_{\bfv} \backslash \Gamma} \int_{1}^M \nuo(t^{1/2}\gamma^{-1}\bfv)_{[2r-2]} \frac{dt}{t} \xrightarrow{M \to +\infty} \archGreen(\bfv)_\Gamma.
\end{equation*}
\end{enumerate}
\begin{proof}
Let us first show that $\archGreen(\bfv)_\Gamma$ converges as claimed. Let $z \in \mathbb{D}^+$ and pick a relatively compact neighborhood $U$ of $z$. Recall that there is a positive definite form $Q_z$ defined by \eqref{eq:def_Siegel_majorant} that varies continuously with $z$. Write $\Gamma\bfv = S_1 \sqcup S_2$ (disjoint union), with
\begin{equation}
S_1 = \{ \bfv' \in \Gamma\bfv|\min_{z \in U}h_z(s_{\bfv'})<1 \};
\end{equation}
then $S_1$ is finite since $\Gamma\bfv \subset L^r$. We can split the sum defining $\archGreen(\bfv)_\Gamma$ accordingly; the sum over $S_1$ converges to a locally integrable form on $\mathbb{D}^+$ that is smooth outside $\cup_{\bfv' \in S_1} \mathbb{D}_{\bfv'}^+$ by \Cref{prop:Green_current_general}. The sum over $S_2$ converges to a smooth form on $U$ by the estimate \eqref{eq:algebra_norm_estimate} and the standard argument of convergence of theta series. By \Cref{prop:nu_properties}.(d), the current defined by $\archGreen(\bfv)_\Gamma$ is invariant under $\Gamma$, and this shows $(a)$. For part (b), note that for each $\bfv' \in S_1$ we have $\smallint_1^M \nuo(t^{1/2}\bfv')_{[2r-2]}\tfrac{dt}{t} \to \archGreen(\bfv')$ as $M\to +\infty$ (as currents on $\mathbb{D}^+$) by dominated convergence, as remarked in the proof of \Cref{prop:Green_current_general}. For the sum over $S_2$ we apply the bound \eqref{eq:bound_Chern_2}, and (b) follows.
\end{proof}
\end{lemma}

\begin{lemma}\label{lemma:exchange_integrals}
Assume that $\bfv \in L^r$ is non-degenerate. Then the integral \eqref{eq:integral_xi_0} converges and
\begin{equation*}
\int_{\Gamma_{\bf v} \backslash \mathbb{D}^+} \archGreen({\bf v}) \wedge \Omega^{p-r+1} = \lim_{M \to +\infty} \int_1^M \int_{\Gamma_{\bf v} \backslash \mathbb{D}^+} \nuo(t^{1/2}{\bf v}) \wedge \Omega^{p-r+1} \frac{dt}{t}.
\end{equation*}
\end{lemma}
\begin{proof}
By \Cref{lemma:convergence_arith_quotient_1}.(a), the integral \eqref{eq:integral_xi_0} equals $\smallint_{\Gamma \backslash \mathbb{D}^+} \archGreen(\bfv)_\Gamma \wedge \Omega^{p-r+1}$ and so it converges since $X_\Gamma$ is compact. Applying \Cref{lemma:convergence_arith_quotient_1}.(b) and unfolding the sum and the integral proves the claim. 
\end{proof}

\begin{theorem} \label{thm:height_pairing_and_derivative}
Suppose $\det T \neq 0$ and 
%Let $T \in X_r$ with non-zero determinant and 
${\bf v} \in \Omega_T(V) \cap L^r$. Let $C_{T,\Gamma_{\bf v}}$ be as in \Cref{lemma:local_SW} and $l$ be the weight in \eqref{eq:def_scalar_weight_both_cases}. Recall that we assume that $\mathrm{rk}(\vbE)=1$.
\begin{enumerate}[(a)]
\item If $T$ is not positive definite (so that $\mathbb{D}_{\mathbf{v}} =\emptyset$), then %\todo{tweaked formatting. LG: OK}
\begin{equation*}
e^{-2\pi \tr(T)} \int_{\Gamma_{\bf v} \backslash \mathbb{D}^+} \archGreen({\bf v}) \wedge \Omega^{p-r+1} = C_{T,\Gamma_{\bf v}} \, W_T'(e,\Phi^{l},s_0).
\end{equation*}
\item If $T$ is positive definite (so that $\mathbb{D}_{\mathbf{v}} \neq \emptyset$), then
\begin{equation*}
\begin{split}
e^{-2\pi \tr(T)} \int_{\Gamma_{\bf v} \backslash \mathbb{D}^+} \archGreen({\bf v}) \wedge \Omega^{p-r+1} &= C_{T,\Gamma_{\bf v}} \, W_T'(e,\Phi^{l},s_0)  \\ 
& \quad - C_{T,\Gamma_{\bf v}}\,  W_T(e,\Phi^l,s_0) \, \frac{\iota}{2} \, \left(\log \det (\pi T) - \frac{\Gamma_r'(\iota m/2)}{\Gamma_r(\iota m/2)} \right).
\end{split}
\end{equation*}
%\begin{equation*}
%\begin{split}
%\int_{\Gamma_{\bf v} \backslash \mathbb{D}^+} \xi({\bf v}) \wedge \Omega^{p-r+1} = & C_{T,\Gamma_{\bf v}}  \left( W_T'(e,\Phi^{l},s_0) \right. \\
%& \ - \left. \frac{(-2\pi i)^{rm/2}}{2^{r(r-1)/4}\Gamma_r\left( \frac{m}{2} \right)}\det(T)^{s_0}e^{-2\pi \mathrm{tr}(T)} \lim_{y \to \infty} \frac{W_T'(\underline{m}(y^{1/2} \cdot 1_r),\Phi^{l},s_0)}{W_T(\underline{m}(y^{1/2} \cdot 1_r),\Phi^{l},s_0)} \right).
%\end{split}
%\end{equation*}
\end{enumerate}
%{\bf TO DO: Should we remove the torsion free hypothesis by adding an explanation as in KRY's paper on Faltings heights, equation (4.17)? Also, explain that some special cases have been proved before: Steve's and Yifeng's and KRY's results. Are there any others? Also, is this result new for $r=1$? For $r=p=1$ it's in KRY's paper, but I'm not sure for other $p$.}
\end{theorem}
\begin{proof} The integral converges by \Cref{lemma:exchange_integrals}. We compute
\begin{equation}
\begin{split}
\int_{\Gamma_{\bf v} \backslash \mathbb{D}^+} \archGreen({\bf v}) \wedge \Omega^{p-r+1} &= \int_1^{\infty} \left( \int_{\Gamma_{\bf v} \backslash \mathbb{D}^+} \nuo(t^{1/2}{\bf v}) \wedge \Omega^{p-r+1} \right) \frac{dt}{t} \\
&= \int_1^{\infty} e^{2\pi \mathrm{tr}(tT)} \left(\int_{\Gamma_{\bf v} \backslash \mathbb{D}^+} \nu(t^{1/2}{\bf v}) \wedge \Omega^{p-r+1} \right) \frac{dt}{t} \\
&=\int_1^{\infty} e^{2\pi \mathrm{tr}(tT)} \left(\int_{\Gamma_{\bf v} \backslash \mathbb{D}^+} \omega(\underline{m}(t^{1/2}\cdot1_{r}))\nu({\bf v}) \wedge \Omega^{p-r+1} \right) t^{-\tfrac{\iota rm}{4}} \, \frac{dt}{t} \\
&=-\frac{\iota}{2}C_{T,\Gamma_{\bf v}} \int_1^{\infty} e^{2\pi \mathrm{tr}(tT)} \, W_T(\underline{m}(t^{1/2} \cdot 1_r),\tilde{\Phi},s_0) \, t^{-\tfrac{\iota rm}{4}} \, \frac{dt}{t},
\end{split}
\end{equation}
where we have used \Cref{lemma:exchange_integrals} and \Cref{lemma:local_SW} for the first and last equality respectively. The last integrand equals $\archW_T(t \cdot 1_r,\tilde{\Phi},s_0)$, and so applying \Cref{lemma:lowering_W_T_derivative} we conclude that
\begin{equation}
\int_{\Gamma_{\bf v} \backslash \mathbb{D}^+} \archGreen({\bf v}) \wedge \Omega^{p-r+1} = C_{T,\Gamma_\mathbf{v}} (\archW_T'(e,s_0)- \lim_{t \to \infty} \archW_T'(t \cdot 1_r,s_0)).
\end{equation}
If $T$ is not positive definite, then the limit in the above expression vanishes by part (ii) of \Cref{prop:NonDegArchWhittakerNonPos}; this proves $(a)$. If $T$ is positive definite, then $(b)$ follows from \Cref{prop:NonDegenWhittakerEstimates}.(iii)-(iv).
%Now assume that $T$ is positive definite. 
\end{proof}

\section{Green forms for special cycles on Shimura varieties} \label{section:Green_currents_global}

We now shift focus from the Hermitian symmetric domain $\mathbb{D}$ to its quotients $\Gamma \backslash \mathbb{D}$ by arithmetic subgroups, and apply the results of the previous sections to construct Green forms for the special cycles $Z(T,\varphi_f)$ on orthogonal and unitary Shimura varieties introduced by Kudla in \cite{KudlaOrthogonal}. 

%
%We review the definition of $\mathrm{GSpin}$ Shimura varieties and their special cycles in Sections \ref{subsection:orthogonal_Shim_vars} and \ref{sec:SpecialCyclesOrthogonal} and we construct a natural Green current $\lie g(T,\bfy,\varphi_f)$ for $Z(T,\varphi_f)$ in \Cref{subsection:Green forms for special cycles_orthogonal case}. \Cref{subsection:unitary_Shim_vars} proves analogous results in the unitary case. In \Cref{subsection:star_products_X_K} we prove our main result concerning star products of  Green forms $\lie g(T,\bfy,\varphi_f)$ for non-degenerate $T$.
%\todo{There is no information in this paragraph that is not contained in the intro or the section titles; is this really necessary? I'd rather put the emphasis on outline in the introduction. LG: I find it helpful to have a short outline of the subsections here, but you have a point, let's discuss it.}

%Throughout, we fix a totally real number field $F$ with real embeddings $\sigma_1, \dots, \sigma_d$. 
\subsection{Orthogonal Shimura varieties} \label{subsection:orthogonal_Shim_vars}
%\subsubsection{} 
Let us briefly recall the definition and basic properties of orthogonal Shimura varieties attached to quadratic spaces over $F$; see \cite{KudlaOrthogonal} for more detail.

\subsubsection{}  Let $F$ be a totally real field with real embeddings $\sigma_1, \dots, \sigma_d$. 

Let $p \geq 1$ and $\bbV$ be a quadratic space over $F$ of dimension $p+2$, with corresponding bilinear form $Q(\cdot,\cdot)$. Assume that 
\begin{equation} \label{eqn:VSigCondition}
\mathrm{signature}(\bbV_{\sigma_i}) \ = \ \begin{cases} (p,2) & \text{ if } i = 1 \\ (p+2, 0 ) & \text{ if } i \neq 1 ,\end{cases}
\end{equation}
where we abbreviate $\bbV_{\sigma_i} := \bbV \otimes_{F, \sigma_i}  \bbR$ for the real quadratic spaces at each place. 

Attached to $\bbV_{\sigma_1}$ is the symmetric space $\mathbb D(\bbV) = \mathbb{D}(\bbV_{\sigma_1})$ defined in \eqref{eq:isomorphism_modelsofD_orthogonal}; recall that it is defined to be
\begin{equation}
\bbD(\bbV)  = \bbD(\bbV_{\sigma_1}) \ = \ \{ [v] \in \bbP(\bbV_{\sigma_1}(\bbC))  | Q(v, v) = 0 , \, Q(v, \overline{v}) \rangle < 0  \}.
\end{equation}
Here $Q(\cdot,\cdot)$ is the $\bbC-$bilinear extension of the bilinear form on $\bbV_{\sigma_1}$. 
%	In particular, $\bbD(\bbV)$ is an open subset of the set of isotropic lines in $\bbV_{\sigma_1} \otimes \bbC$, and thereby acquires the structure of a complex manifold of complex dimension $p$.
%	Let
%	\begin{equation}
%	\bfH := \Res_{F/\bbQ} \, \GSpin(\bbV)
%	\end{equation}
%	viewed as an algebraic group over $\bbQ$.
Let
\begin{equation}
\bfH_{\bbV} := \Res_{F/\bbQ} \, \GSpin(\bbV).
\end{equation}
Then 
\begin{equation}
\bfH_{\bbV}(\bbR) \simeq \GSpin(\bbV_{\sigma_1}) \times \cdots \times \GSpin(\bbV_{\sigma_d})
\end{equation}
acts transitively on $\bbD(\bbV)$ via the first factor.%, and the stabilizer $K_{\infty}$ in $\bfH(\bbR)$ of any fixed point on $\bbD(\bbV)$ is a maximal compact subgroup of $\bfH(\bbR)$. THIS IS NOT CORRECT.

\begin{definition}
	For a compact open subgroup $K \subset \bfH(\bbA_f)$, consider the \emph{Shimura variety}
	\begin{equation*} 
%	 \label{eqn:ComplexOrthogonalShimuraVarietyDefn}
	 X_{\bbV, K} := \bfH_{\bbV}(\bbQ) \backslash \bbD(\bbV) \times \bfH_{\bbV}(\bbA_f) / K
	\end{equation*}
	If $K$ is neat, then $X_{\bbV, K}$ is a complex quasi-projective algebraic variety. If $\bbV$ is moreover anisotropic, then $X_{\bbV, K}$ is projective.\footnote{Our assumptions on the signature of $\bbV$ imply that this is always the case when $F \neq \bbQ$.}

\end{definition}
%From now on, we shall always tacitly assume $K$ is neat. %{\color{red} Should we say neat instead of sufficiently small? It's precise.}

The space $X_{\bbV, K}$ may be written in a perhaps more familiar fashion: fix a connected component $\bbD^+ \subset \bbD(\bbV)$
and let $\bfH_{\bbV}(\bbR)^+$ denote its stabilizer in $\bfH_{\bbV}(\bbR)$. Setting $\bfH_{\bbV}(\bbQ)^+ = \bfH_{\bbV}(\bbR)^+ \cap \bfH_{\bbV}(\bbQ)$, 
there exist finitely many elements $h_1, \dots h_t \in \bfH_{\bbV}(\bbA_f)$ such that 
\begin{equation}
\bfH_{\bbV}(\bbA_f) = \coprod_j \bfH_{\bbV}(\bbQ)^+ h_j K;
\end{equation}
then we may write 
\begin{equation} \label{eqn:ShVarConnComponents}
X_{\bbV,K} \simeq \coprod_{j} \Gamma_j \big\backslash \bbD^+ = : \coprod_j \ X_{\bbV,j}
\end{equation}
as a disjoint union of quotients  of $\bbD^+$ by discrete subgroups 
\begin{equation} \label{eq:Gamma_j_definition}
\Gamma_j =  \bfH_\bbV(\bbQ)^+ \cap \left( h_j K h_j^{-1} \right).
\end{equation}

In general, we regard the quotients $X_{\bbV,K}$ and $\Gamma_j \backslash \mathbb{D}^+$ as orbifolds. In particular, for a $\Gamma_j$-invariant differential form $\eta$ of top degree on $\mathbb{D}^+$, we define %\todo{tweaked formatting. LG: OK}
\begin{equation}
\int\limits_{[\Gamma_j \backslash \mathbb{D}^+]} \eta = [\Gamma_j:\Gamma']^{-1} \int\limits_{\Gamma' \backslash \mathbb{D}^+} \eta,
\end{equation}
where $\Gamma' \subset \Gamma_j$ is any neat subgroup of finite index, and set $\smallint_{[X_K]} = \sum_j \int_{[\Gamma_j \backslash \mathbb{D}^+]}$.

\subsubsection{}  The theory of canonical models of Shimura varieties (see \cite{Shih}) implies the existence of a quasi-projective model $\mathcal X_K$ over $\mathrm{Spec} (F)$, which is projective when $\bbV$ is anisotropic, such that 
\[ \calX_K \otimes_{F, \sigma_1} \bbC \simeq X_{\bbV, K}. \] 
From the point of view of arithmetic intersection theory, it will be important to work with all the complex fibres of $\mathcal X_K$ simultaneously; the remaining fibres have the following concrete description. 

%Let $\sigma_1, \dots ,\sigma_d $ denote the set of real embeddings of $F$. 
For each $k = 2, \dots, d$, let $\bbV[k]$ denote a quadratic space over $F$ such that 
	\begin{enumerate}[(i)]
		\item $\bbV[k]_{\sigma_k} \simeq \bbV_{\sigma_1}$, i.e.\ the signature of $\bbV[k]_{\sigma_k}$ is $(p,2)$;
		\item  $\bbV[k]_{\sigma_1}\simeq \bbV_{\sigma_k}$; 
		\item   and $(\bbV[k])_w \simeq \bbV_w$ at all other places. 
	\end{enumerate} 
The space $\bbV[k]$ is unique up to isometry, and we have $\bbV[1] \simeq \bbV$.

 Fix, once and for all, identifications 
	 \begin{equation} \label{eqn:VkFiniteAdeleIdentification}
		 \bbV[k] \otimes_F \bbA_f \simeq  \bbV \otimes_F \bbA_f 
	 \end{equation}
	  inducing identifications \[ \bfH_{\bbV[k]}(\bbA_f) \ \simeq \ \bfH_{\bbV}(\bbA_f) \] 
 for all $k$, and so in particular we may view $K \subset \bfH_{\bbV[k]}(\bbA_f)$. Then, setting $\bbD(\bbV[k]) = \bbD(\bbV[k]_{\sigma_k}) \simeq \bbD(\bbV)$, the theory of conjugation of Shimura varieties (see \cite{MilneCanonical, MilneSuh}, as well as \cite[Section 7]{BruinierYangCMTotReal} for our particular situation) gives identifications
 \begin{equation}
	 \calX_K \times_{F, \sigma_k} \bbC \ \simeq \ X_{\bbV[k], K} \ = \ \bfH_{\bbV[k]}(\bbQ) \Big\backslash \bbD(\bbV[k]_{\sigma_k}) \times  \bfH_{\bbV[k]}(\bbA_f) \Big/ K.
 \end{equation}
In particular, viewing $\calX_K$ as a scheme over $\bbQ$ via the map $\mathrm{Spec}(F) \to \mathrm{Spec}(\bbQ)$, we have
\begin{equation}
	\calX_K(\bbC) \ = \ \coprod_{k=1}^d \, X_{\bbV[k],K}.
\end{equation}

\subsection{Special cycles} \label{sec:SpecialCyclesOrthogonal}
Recall that in \Cref{subsubsection:taut_bundle_definitions} we have defined the tautological bundle $\vbE$ over $\bbD(\bbV[k])$ and a global section $s_{\bfv}$ of $(\vbE^r)^\vee$ for any $\bfv \in (\bbV[k]_{\sigma_k})^r$, whose zero locus $Z(s_{\bfv})$ we denote by $\mathbb{D}_{\bfv}$. Given a rational vector $\bfv=(v_1,\ldots,v_r) \in \bbV[k]^r$, we set
\begin{equation}
\mathbb{D}_{\bfv} = Z(s_{\sigma_k(\bfv)})
\end{equation}
and $\mathbb{D}_{\bfv}^+=\mathbb{D}_{\bfv} \cap \mathbb{D}(\bbV[k])^+$.

Let $\bfH_{\bfv}(\mathbb{Q})$ be the pointwise stabilizer of $\mathrm{span}_F \left\{  v_1,\ldots,v_r \right\} $ in $\bfH_{\bbV[k]}(\mathbb{Q})$. Given a component $X_j = \Gamma_j \backslash \bbD^+ \subset X_{\bbV[k],K}$ associated to $h_j$ as in \eqref{eqn:ShVarConnComponents}, 
let $\Gamma_j(\bfv) = \Gamma_j \cap H_{\bfv}(\mathbb{Q})$; then the natural map
\begin{equation}
\Gamma_j(\bfv) \backslash  \bbD_{\bfv}^+ \ \to \ \Gamma_j \backslash \bbD^+  = X_j
\end{equation}
defines a (complex algebraic) cycle on $X_{j}$ that we denote by $c(\bfv,X_j)$. 
In addition, recall that for a matrix $T \in \Sym_r(F)$, we had defined
\begin{equation}
\Omega_T(\bbV[k]) \  = \  \{\bfv \in \bbV[k]^r  \, |  \ T(\bfv) =  T \};
\end{equation}
where $T(\bfv) = (\frac12 Q(v_i,v_j))_{i,j} \in \Sym_r(F)$ is the moment matrix of $\bfv$. 

\begin{definition}[\cite{KudlaOrthogonal}]  \label{def:GlobalSpecialCycle}
Let $\varphi_f \in \mathcal{S}(\bbV(\bbA_f)^r)^{K}$ be  a $K$-invariant Schwartz function  which, for each $k = 1, \dots, d$, may be viewed as a Schwartz function on $\bbV[k](\bbA_f)^r$ via \eqref{eqn:VkFiniteAdeleIdentification}.  For $T \in \Sym_r(F)$, define the weighted special cycle $Z(T, \varphi_f,K)$ on $\calX_K(\bbC) = \coprod X_{\bbV[k],K}$ by
	\begin{equation*}
	Z(T, \varphi_f,K) \ = \ \sum_{k=1}^d \  \sum_{X_j \subset X_{\bbV[k],K}} \ \sum_{\substack{ \bfv \in \Omega_T(\bbV[k]) \\ \mathrm{mod } \ \Gamma_j }} \varphi_f( h_j^{-1} \bfv) \, 	c(\bfv,X_j).
	\end{equation*} 
	This is a complex algebraic cycle on $\calX_K(\bbC)$ that is in fact defined over $F$. 
\end{definition}

It follows from the discussion after \eqref{eq:def_D_bfv} that if $Z(T, \varphi_f, K)$ is non-empty, then $T$ is totally positive semi-definite and the codimension of $Z(T, \varphi_f, K)$ is equal to the rank of $T$. 

Note that this definition is independent of all choices. Moreover, if $K' \subset K$ is an open subgroup of finite index and $\pi \colon \calX_{K'} \to \calX_K$ is the natural covering map, then $\pi^*Z(T,\varphi_f,K)=Z(T,\varphi_f,K')$.  See \cite[\textsection 5]{KudlaOrthogonal} for a proof of this and further properties of these cycles.
{\ \\}

To lighten notation, we will once and for all fix a compact open subgroup $K \subset \bfH_{\bbV}(\bbA_f)$, and write, for example, $Z(T, \varphi_f) = Z(T,\varphi_f, K)$, $X_{\bbV[k]} = X_{\bbV[k],K}$ and $\calX = \calX_K$, etc.

\subsection{Green forms for special cycles}  \label{subsection:Green forms for special cycles_orthogonal case}

\subsubsection{} For the moment, fix a real embedding $\sigma_k \colon F \to \bbR$ and a component 
\[X_j = \Gamma_j \backslash \bbD^+ \subset X_{\bbV[k]} \simeq \calX_{\sigma_k}(\bbC) \]
with $\bbD^+ = \bbD(\bbV[k]_{\sigma_k})^+$; here $\Gamma_j$ is attached to $h_j \in \bfH_\bbV(\mathbb{A}_f)$ as in \eqref{eq:Gamma_j_definition}.

Let $T \in \Sym_r(F)$ with $\det T \neq 0$. Any collection of vectors $\bfv =(v_1, \dots, v_r)\in \bbV[k]^r$ with $T(\bfv) = T$ is necessarily linearly independent. %;  thus either $\bbD_{\bfv}$ is empty or has codimension $r$ in $\mathbb{D}$ with the latter case occurring if and only if $T$ is totally positive definite. 
For such $\bfv$, we defined in \Cref{subsection:green_form_on_D} 
%	we have defined 
a form  satisfying the  equation 
\begin{equation} \label{eqn:nondegGreenEquationD}
\ddc \archGreen(\sigma_k(\bfv)) + \delta_{\mathbb{D}_\bfv} \ = \phio(\sigma_k(\bfv))_{[2r]}
\end{equation}
of currents on $\bbD^+$.

We introduce Green forms that depend on an auxiliary parameter $\bfy \in \Sym_r(F \otimes_{\bbQ}\bbR)_{\gg 0}$
%\begin{equation} \bfy = (y_{\sigma_1}, \dots, y_{\sigma_d}) \in \Sym_n(F_{\bbR})_{>0}, \end{equation} 
as follows: fix some element $\bm \alpha \in \GL_r(F_{\bbR})$ with totally positive determinant such that $\bfy  = \bm \alpha \cdot  \transpose{\bm \alpha}$
and, for a Schwartz function $\varphi_f \in S(\bbV(\bbA_f)^r)^{K}$, define a form $\lie g_j(T, \bfy, \varphi_f)_{\sigma_k}$ on $\bbD^+  $ by setting
\begin{equation} \label{eq:def_global_Green_form_connected_cycle}
\lie g_j(T, \bfy, \varphi_f)_{\sigma_k}  \ := \ \sum_{\bfv \in \Omega_T(\bbV[k])} \, \varphi_f( h_j^{-1} \bfv ) \ \archGreen(\sigma_k(\bfv) \cdot \sigma_k(\bm \alpha));
\end{equation}
here we view  $\sigma_k(\bm{\alpha}) \in \GL_r(\bbR)$ via the $\bbR$-linear map $\sigma_{k} \colon F\otimes_{\bbQ} \bbR \to \bbR$ induced by $\sigma_k$. 
%\comm{LG: is it really necessary to say this? What is the possible confusion?} 
The convergence of this sum to a locally integrable form on $\mathbb{D}^+$ follows from \Cref{lemma:convergence_arith_quotient_1}.(a) and the fact that the number of orbits of $\Gamma_j$ on $\mathrm{Supp}(\varphi_f) \cap \Omega_T(\bbV[k])$ is finite.

%\comm{normalize by $|\det \bfy|^*$  here?}
%The convergence of this sum on the complement of the pullback of $Z(T,\varphi_f)$ to $\bbD^+$ follows from \Cref{lemma:convergence_arith_quotient_1}.$(a)$ and the fact that the number of orbits of $\Gamma_j$ on $\mathrm{Supp}(\varphi_f) \cap \Omega_T(\bbV)$ is finite. 
%% \comm{discuss convergence later...}
Note that the definition is independent of the choice of $\bm \alpha$, by \Cref{prop:nu_properties}.$(f)$. Moreover, the form $\lie g_j(T, \bfy, \varphi_f)_{\sigma_k}$ is invariant under the action of $\Gamma_j$ by \Cref{prop:nu_properties}.$(d)$, and so 
it descends to a form on the connected component $X_j$ that, abusing notation, we also denote by $\lie g_j(T, y, \varphi_f)_{\sigma_k}$.

Finally, let
\begin{equation}
	\lie g(T, \bfy, \varphi_f)_{\sigma_k}
\end{equation}
denote the form on $X_{\bbV[k]}\simeq \calX_{\sigma_k}(\bbC)$ whose restriction to $X_j$ is $\lie g_j(T, \bfy , \varphi_f)_{\sigma_k}$. Essentially by construction, it is a Green form for the cycle $Z(T, \varphi_f)_{\sigma_k}$; more precisely, let $\omega(T, \bfy, \varphi_f)_{\sigma_k}$ be the differential form on $X_{\bbV[k]}$ whose restriction to the component $X_j$ is
\begin{equation}
\omega(T, \bfy, \varphi_f)_{\sigma_k} \big|_{X_j} \ = \ \sum_{ \bfv \in \Omega_T(\bbV[k])}  \varphi_f( h_j^{-1} \bfv ) \ \phio(\sigma_k(\bfv) \cdot \sigma_k(\bm{\alpha}))_{[2r]} ,
\end{equation}
%Let $\mathcal{H}_r$ be the symmetric space of $G_r'$ (see \eqref{eq:def_Gr'_archimedean}); for $\tau=x+iy \in \mathcal{H}_r$, we set 
%\begin{equation} \label{eq:g_tau'_def}
%g'_\tau=\underline{n}(x)\underline{m}(y^{1/2}) \in P_r'.
%\end{equation}
Then the form $\omega(T, \bfy, \varphi_f)_{\sigma_k}$ is the $T$'th coefficient of the theta function
%\todo{tweaked formatting. LG: OK}
\begin{equation} \label{eq:def_Theta_KM}
\Theta_{\KM}(\bftau; \varphi_f)_{\sigma_k} \ := \ \sum_{T \in \Sym_r(F)} \omega(T, \bfy, \varphi_f)_{\sigma_k} \, q^T, 
%\\\text{where }\quad \bftau \in (\bbH_r)^d, \quad \text{ and }  q^T := \prod_{i=1}^d e^{2\pi i \mathrm{tr}(\tau_i \sigma_i(T))}.
\end{equation} 
where $\bftau \in (\bbH_r)^d$ and $ q^T := \prod_{i=1}^d e^{2\pi i \mathrm{tr}(\tau_i \sigma_i(T))}$. 
%Here $\bbH_r = \{\bftau = \bfx + i \bfy \in \Sym_{r}(F \otimes_{\bbQ} \bbC) \ | \ \bfy \gg 0  \}$ is the Siegel upper half-space of genus $r$.
This theta function was considered (in much greater generality) by Kudla and Millson \cite{KudlaMillson3}.

%, who showed that $\omega(T, \bfy, \varphi_f)_{\sigma_k}$ is a Poincar\'e dual form to the special cycle $Z(T, \varphi_f)_{\sigma_k}$; hence $\omega(T, \bfy, \varphi_f)_{\sigma_k} - \delta_{Z(T, \varphi_f)_{\sigma_k}}$ is exact.

Applying the identity \eqref{eqn:nondegGreenEquationD} of currents on $\bbD$, summing over $\bfx$ with $T(\bfx) = T$ and descending to the Shimura variety $X_{\bbV[k]}$ yields the equation 
\begin{equation} \label{eq:global_current_equation}
\ddc \lie g(T, \bfy, \varphi_f)_{\sigma_k} + \delta_{Z(T, \varphi_f)_{\sigma_k}}  =  \omega(T, \bfy, \varphi_f)_{\sigma_k}
\end{equation}
of currents on $X_{\bbV[k]} \simeq \calX_{\sigma_k}(\bbC)$. The collection
\begin{equation} \label{def:global_g_T_Y_phi_case_1_nondeg_T}
		\{ \lie g(T,\bfy,\varphi_f)_{\sigma_k} \ | \ k = 1, \dots, d  \}
\end{equation}
defines a Green form $\lie g(T,\bfy, \varphi_f)$ for the cycle $Z(T,\varphi_f)$ on $\calX$, in the sense of \cite[Chap. II]{SouleBook}. 

\subsubsection{} We will next construct a current $\lie g(T, \bfy; \varphi_f) = \left\{ \lie g(T, \bfy, \varphi_f)_{\sigma_k} \right\}$ for an arbitrary matrix $T \in \Sym_r(F)$ and $\bfy \in \Sym_{r}(F_{\bbR})_{\gg0}$.
For the moment, choose an embedding $\sigma_k$ and a component $X_j = \Gamma_j \backslash \bbD \subset \calX_{\sigma_k}(\bbC)$.  Recall that for any $\bfv = (v_1, \dots, v_r) \in \bbV[k]^r$, we had defined the current
\begin{equation}
		\archGreen (\sigma_k(\bfv);\rho) \ := \ \int_1^{\infty} \nuo(\sqrt{t} \bfv)_{[2r \cdot \mathrm{rk}(\vbE)-2]} \, \frac{dt}{t^{\rho+1}}
\end{equation}
on $\bbD$, see \eqref{eq:def_g_reg_local}.
%Given a collection of vectors $\bfv = (v_1,\ldots,v_r) \in \bbV[k]^r$ with $T(\bfv)=T$, in \ref{subsection:regularized_green_current_on_D} we have defined a current $\lie g(\sigma_k(\bfv);\rho)$ on $\mathbb{D}(\bbV[k])$; 
here $\rho$ is a complex parameter.

Let $\bfy \in \mathrm{Sym}_r(F_\mathbb{R})_{\gg 0}$ and let $\bfalpha \in \mathrm{GL}_r(F_\mathbb{R})_{\gg 0}$ such that $\bfy = \bfalpha \cdot {^t}\bfalpha$. Given a Schwartz function $\varphi_f \in S(\bbV(\bbA_f)^r)^{K }$, consider the sum 
\begin{equation} \label{eqn:DegCurrentDef}
\lie g^{\mathtt{o}}_j(T, \bfy, \varphi_f;  \, \rho )_{\sigma_k}  := \sum_{\Omega_T(\bbV[k])}\varphi_f( h_j^{-1} \bfv) \ \archGreen \left( \sigma_k(\bfv) \cdot \sigma_k({\bfalpha});\rho \right),
\end{equation}
viewed as a current on $\bbD^+ = \bbD(\bbV[k]_{\sigma_k})^+$. Note that the right hand side is independent of the choice of $\bfalpha$ by \Cref{prop:xiRegContinuation}.(ii).

\begin{proposition} \label{prop:DegCurrentDef}
The sum \eqref{eqn:DegCurrentDef} converges for $\mathrm{Re}(\rho) \gg 0$ to a $\Gamma_j$-invariant  current on $\bbD^+$ that has a meromorphic continuation to $\mathrm{Re}(\rho)> -1/2$. In particular, the constant term in the Laurent expansion
\begin{equation*}
		\lie g^{\mathtt{o}}_j(T, \bfy, \varphi_f)_{\sigma_k}  :=  \mathrm{CT}_{\rho=0}   \   \lie g_{j}^{\mathtt{o}}(T,\bfy, \varphi_f;\, \rho)_{\sigma_k}
\end{equation*}
		descends to a current on $X_j = \Gamma_j \backslash \bbD^+$.
	\begin{proof}
		For convenience,  we take $\sigma_k=\sigma_1$ and suppress this index  from the notation, writing $y = y_1$ etc. The proof for the other embeddings $\sigma_k$ is identical.
		 
%	 We proceed as in the proof of \Cref{lemma:convergence_arith_quotient_1}. 
		 Let $\eta \in A_c^*(\mathbb{D}^+)$ with compact support. Choose an open set $U \subset \bbD^+$ with compact closure such that $\mathrm{supp}(\eta)\subset U$.
		 Let
		 \begin{equation}
			 S_1 \ := \ \{\bfv \in \bbV^r \ | \ T(\bfv) = T, \ \varphi_f(h_j^{-1} \bfv) \neq 0 , \text{ and } \bbD_{\bfv} \cap U \neq \emptyset \},
		 \end{equation}
		 and
		 \begin{equation}
		 	S_2 \ := \ \{\bfv \in \bbV^r \ | \ T(\bfv) = T, \ \varphi_f(h_j^{-1} \bfv) \neq 0 , \text{ and } \bbD_{\bfv} \cap U = \emptyset \},
		 \end{equation}
		 so that $S_1 \sqcup S_2$ indexes the non-zero terms appearing on the right hand side of \eqref{eqn:DegCurrentDef}. Note also that $S_1$ is finite, 
while there exists a bound $C>0$ such that
		 \begin{equation}
			 \min_{z \in \overline U} \sum_{i=1}^r h_{z} (\sigma(v_i), \sigma(v_i)) > C
		 \end{equation}
		 for all $\bfv = (v_1, \dots, v_r) \in S_2$. 		 

By \Cref{prop:xiRegContinuation}, the finite sum 
		\begin{equation}
		\sum_{\bfx \in S_1} \varphi_f(h_j^{-1}\bfv) \int_\mathbb{D} \archGreen \left(\sigma(\bfv) \cdot \sigma
		(\bfalpha); \rho\right) \wedge  \eta
		\end{equation} 
		converges for $\rho$ large, and has meromorphic continuation to $\mathrm{Re}(\rho)>-1/2$. For the same sum where $\bfv$ now runs over $S_2$, the exponential decay estimate \eqref{eq:algebra_norm_estimate}  and standard arguments for the convergence of theta series  imply that the sum
\begin{equation}
\begin{split}
			\sum_{\bfv \in S_2} \varphi_f(h_j^{-1} \bfv) \ & \int_{\bbD^+} \archGreen \left(\sigma(\bfv) \cdot \sigma(\bfalpha) ; \rho\right)\wedge \eta  \\
			& = \sum_{\bfv \in S_2} \varphi_f(h_j^{-1} \bfv) \int_1^{\infty} \int_{\bbD^+} \nuo \left( \sqrt{t}\sigma(\bfv) \cdot \sigma(\bfalpha)\right) \wedge \eta  \, \frac{dt}{t^{\rho+1}}
\end{split}
\end{equation}
		converges absolutely and locally uniformly in $\rho$, and hence defines an entire function of $\rho$.

\end{proof}
\end{proposition}
%%
%%\comm{fix this...}
%%Using the collection of embeddings $\sigma_1, \dots, \sigma_d$, we may identify $\bfy \in \Sym_{r}(F_{\bbR})$ with a tuple $\bfy = (y_1, \dots, y_d)$ with $y_k \in \Sym_{r}(\bbR)_{>0}$. For each $k$,  fix a matrix $\theta_k \in \SL_r(\bbR)$ of the form
%%\[
%%\theta_k = \begin{pmatrix} A & * \\ & B \end{pmatrix}, \qquad \text{ with } A \in O_{n-t}(\bbR) , \ B \in O_t(\bbR)
%%\]
%%such that
%%\[
%%\theta_k \cdot \sigma_k( \gamma )\cdot y_k \cdot \sigma_k( \transpose{ \gamma} )\cdot \transpose{\theta_k} \ = \ \begin{pmatrix} a_1 & & \\ & \diagdown & \\ & & a_n \end{pmatrix} \ =: \ \bfa_k
%%\]
%is diagonal.

We next show that by patching together the $\lie g_j^{\mathtt{o}}(T,\bfy,\varphi_f)_{\sigma_k}$ on $X_{\bbV[k]}$, we obtain a current satisfying an analogue of Green's equation \eqref{eq:global_current_equation}.
\begin{proposition} \label{prop:GlobalGreenOrthogonalDegenerateCase}
	
	Let $\lie g^{\mathtt{o}}(T, \bfy, \varphi_f)$ denote the  current on $\calX(\bbC)= \coprod_{k} \calX_{\sigma_k}(\bbC)$ whose restriction to a connected component $X_j\subset \calX_{\sigma_k}(\bbC)$ is  $\lie g_j^{\mathtt{o}}(T, \bfy, \varphi_f)_{\sigma_k}$. 
	Then there is an identity of currents
	\begin{equation} \label{eqn:general global current Green}
	\ddc  \lie g^{\mathtt{o}}(T,\bfy,\varphi_f) + \delta_{Z(T, \varphi_f)} \wedge \Omega_{\vbE^{\vee}}^{r - \rank(T)} =  \omega(T, \bfy, \varphi_f)
	\end{equation}
	where $\Omega_{\vbE^{\vee}} = c_1(\vbE^\vee,\nabla)^*$ and $\omega(T,\bfy,\varphi_f)$ is the differential form on $\mathcal{X}(\mathbb{C})$ whose restriction to $\mathcal{X}_{\sigma_k}(\mathbb{C})$ is $\omega(T,\bfy,\varphi_f)_{\sigma_k}$.
	
	\begin{proof}
		%	 \comm{todo: adjust proof taking into account $\alpha_0$ in def of $\lie g_j$} 
		It suffices to prove the given identity, for each component $X_j$, at the level of $\Gamma_j$-invariant currents on $\bbD$. The estimates in the proof of \Cref{prop:DegCurrentDef} allow us to write
		\begin{align}
		\ddc  \lie g_j^{\mathtt{o}}(T, \bfy, \varphi_f)_{\sigma_k} \Big|_{X_j} \ =& \ \ddc \left[ \mathrm{CT}_{\rho=0} \ \sum_{\bfv} \varphi_f(h_j^{-1} \bfv) \, \archGreen \left( \sigma_k(\bfv) \cdot \sigma_k(\bfalpha);   \rho \right) \right] \notag \\ 
		=& \  \sum_{\bfv} \ \varphi_f(h_j^{-1} \bfv) \ \ddc \, \mathrm{CT}_{\rho=0} \, \archGreen \left( \sigma_k(\bfv) \cdot \sigma_k(\bfalpha);   \rho \right). \label{eqn:RegGreenSumonX}
		\end{align}
		The proposition follows immediately from \Cref{prop:xiRegContinuation}.(vi).

	\end{proof}
\end{proposition}

Finally, in order to obtain agreement with the derivatives of Eisenstein series in our main theorem, we introduce a modified version of $\lie g^\mathtt{o}(T,\bfy,\varphi_f)$. We write $\det' A$ for the product of non-zero eigenvalues of a square matrix $A$, with the convention  $\det'(\textbf 0) = 1$.
\begin{definition} \label{def:global_g_T_Y_phi_case_1}
%Let $\bfy = (y_v)_{v|\infty}$ and, for an archimedean place $v$, let
Let $\bfy = (y_v)_{v|\infty} \in \mathrm{Sym}_r(F_\mathbb{R})_{\gg 0}$ and $T \in \mathrm{Sym}_r(F)$, and define a current 
	$\lie g(T, \bfy, \varphi_f)  \in   D^{*}(\calX(\bbC))$
%	\begin{equation}
%		\lie g(T, \bfy, \varphi_f)  \in   D^{\, 2r \cdot \rank(\vbE)-2}(\calX(\bbC))
%	\end{equation}
	 as follows: 
 if $T$ is not totally positive semidefinite, set
\begin{equation*}
\lie g(T,\bfy,\varphi_f) := \lie g^{\mathtt o}(T,\bfy,\varphi_f)
\end{equation*} %\todo{LG: remove number from eq?}
and if $T$ is  totally positive semidefinite, set %\todo{I think the sign should be $-$ here, to obtain agreement with the previous definition (see the example below). LG: I think you are right.}
\begin{align*}
\lie g(T,\bfy,\varphi_f)& := \lie g^\mathtt{o}(T,\bfy,\varphi_f)  \\
 & \hspace{4em} {}  -  \sum_{v|\infty} \log \left( \frac{\det' \sigma_v(T) \cdot \det y_v}{\det'  \left(\sigma_v(T)y_v\right)}  \right)   \delta_{Z(T,\varphi_f)_{\sigma_v}} \wedge \Omega_{\calE^{\vee}}^{r-\mathrm{rk}(T)-1}  \notag
\end{align*} %\todo{LG: remove number from eq?}
where $\Omega_{\vbE^{\vee}} = c_{1}(\vbE^{\vee}, \nabla)^* = \frac{i}{2 \pi} c_1(\vbE^{\vee}, \nabla)$. 
Note that the additional term is closed, and vanishes if $T$ is non-degenerate.
\end{definition}

Note also that when $\det T \neq 0$,  the current defined by the Green form \eqref{def:global_g_T_Y_phi_case_1_nondeg_T} agrees with the one in \Cref{def:global_g_T_Y_phi_case_1} by \Cref{prop:xiRegContinuation}.$(iv)$.

\subsection{Unitary Shimura varieties} \label{subsection:unitary_Shim_vars}
The results in the previous section carry over, essentially verbatim, to the unitary case. To describe the setup, suppose that $E$ is a CM extension of the totally real field $F$ with $[F:\bbQ] = d$, and that $\bbV$ is a Hermitian space over $E$, with Hermitian form $Q(\cdot , \cdot)$. 

%Fx a CM type $\Phi$. Each real embedding $\sigma_i \colon F \to \bbR$ admits a unique extension to a complex embedding of $E$ belonging to $\Phi$; we \change{still} denote this embedding  by $\sigma_i \colon E \to \bbC$. 
%\todo{I found the previous notation a little too abusive, how is this? LG: Sounds good.}
Fix a CM type $\Phi = \{ \sigma_1, \dots,\sigma_d \} \subset \mathrm{Hom}(E, \bbC)$. For each $i$, 
the space $\bbV_{\sigma_i} = \bbV \otimes_{\sigma_i, E} \bbC$ is a complex Hermitian space; we assume 
\begin{equation}
\mathrm{signature}  \ \bbV_{\sigma_i} \ = \ \begin{cases} (p,q), & \text{ if } i = 1 \\ (p+q,0) , & \text{ if } i = 2, \dots d. \end{cases}
\end{equation}
for some integers $p,q > 0$. 

Let 
\begin{equation*}
  \bbD(\bbV_{\sigma_1}) \ := \ \left\{ z \subset \bbV_{\sigma_1} \text{ negative-definite subspace}, \ \dim_{\bbC} z = q   \right\},
\end{equation*}
(see \eqref{eq:def_unitary_D}) be the symmetric space attached to the real points of the unitary group
\begin{equation}
\bfH_{\bbV} \ := \ \mathrm{Res}_{F/\bbQ} \ \mathrm{U}(\bbV).
\end{equation}
Just as in \Cref{subsection:orthogonal_Shim_vars}, a fixed compact open subgroup $K \subset \bfH_{\bbV}(\bbA_f)$ determines a complex Shimura variety
\begin{equation} \label{eqn:ComplexUnitaryShimuraVarietyDefn}
X_{\bbV} = X_{\bbV, K} := \bfH_{\bbV}(\bbQ) \backslash \bbD({\bbV}_{\sigma_1}) \times \bfH_{\bbV}(\bbA_f) / K
 \end{equation}
which is quasi-projective, and projective when $\bbV$ is anisotropic; choosing representatives  $h_1, \dots, h_t $ for the double coset space $ \bfH_{\bbV}(\bbQ) \backslash \bfH_{\bbV}(\bbA_f) / K$ gives a decomposition
\begin{equation}
X_{\bbV} \  \simeq \ \coprod_j \Gamma_j \backslash \bbD \ =:  \coprod_j X_j , \qquad \text{where } \Gamma_j := \bfH_{\bbV}(\bbQ) \cap \left( h_j K h_j^{-1} \right). 
\end{equation}
Let $\calX$ denote the canonical model over $E$, so that  $\calX_{\sigma_1}(\bbC) \simeq X_{\bbV}$.
%By the theory of canonical models, there is a model\footnote{In general, the reflex field is either $E$ or $F$; in the latter case, we work with the base change of the canonical model to $E$. \comm{LG: is this footnote necessary?}
%} $\calX$ over $E$ such that $\calX_{\sigma_1}(\bbC) \simeq X_{\bbV}$. 
For the other complex embeddings, the story is similar to \Cref{subsection:orthogonal_Shim_vars}. For each $k = 1, \dots, d$, let $\bbV[k]$ denote the (unique up to isometry) $E$-Hermitian space such that 
\begin{itemize}
	\item $\bbV[k]_{\sigma_k} \simeq \bbV_{\sigma_1}$ ;
	\item $\bbV[k]_{\sigma_j} $ is positive definite, for $j \neq k$; and
	\item $\bbV[k]_v \simeq \bbV_v$ at all finite places $v$.
\end{itemize} 
Identifying $\bbV[k] \otimes_{\bbQ} \bbA_f \simeq \bbV \otimes_{\bbQ}\bbA_f$, and in particular, viewing $K$ as a subgroup of $\bfH_{\bbV[k]}(\bbA_f)$, we may define the complex Shimura variety $X_{\bbV[k] } = X_{\bbV[k], K}$ in the same way as $X_{\bbV,K}$. 

Now suppose $\rho  \in \mathrm{Hom}(E, \bbC )$  and let $\sigma_k$ be the element of the CM type  such that $\rho|_F = \sigma_k|_F$. Then  there is an identification
\begin{equation}
	\calX_{\rho }(\bbC) \simeq   X_{\bbV[k]};
%	\bfH_{\bbV[k]}(\bbQ) \big\backslash \bbD(\bbV[k]_{\rho}) \times \bfH_{\bbV[k]}(\bbA_f)  \big/ K.
\end{equation}
this follows from the general considerations of \cite{MilneSuh} and \cite[Section II.4]{MilneCanonical}, or \cite[Section 3A]{LiuYifengArithThetaI} for the case at hand.

The special cycles are defined just as in \Cref{sec:SpecialCyclesOrthogonal}: recall that a tuple $\bfv = (v_1, \dots, v_r) \in (\bbV[k])^r$ determines a  section of $(\vbE^r)^\vee$, where $\vbE$ is the rank $q$ tautological bundle on $\bbD(\bbV[k])$. Its vanishing locus $\bbD_{\bfv} \subset \bbD(\bbV[k])$ determines a cycle $c(\bfv, X_j)$ on each component $X_j = \Gamma_j \backslash \bbD(\bbV[k])$, which is either empty or of codimension $r' q$, where $r' = \dim \mathrm{span} \{ v_1, \dots, v_r\}$.  Given a $K$-invariant Schwartz function $\varphi_f \in \mathcal{S}( \bbV(\bbA_f)^r)^{K}$ and a matrix $T \in \Herm_r(E)$,  define the (complex) special cycle
\begin{equation}
Z(T, \varphi_f) \ = \ \sum_{\substack{  X_j \subset X_{\bbV[k]}}  \subset \calX(\bbC)  } \, \sum_{\substack{ \Omega_T(\bbV[k]) \\ \text{mod } \Gamma_j }} \varphi_f( h_j^{-1} \bfv) \, 	c(\bfv,X_j)
\end{equation}
exactly as in \Cref{def:GlobalSpecialCycle}, with the sum taken over all connected components of $\calX(\bbC)$; as before, these are the complex points of a rational cycle.

Given $\bfy \in \Herm_r(E\otimes_{\bbQ}\bbR)_{\gg0} $, use the CM type $\Phi$ to identify
\begin{equation}
\bfy = (y_{1}, \dots, y_{d}) \in (\Herm_r(\bbC)_{>0})^d \simeq \Herm_r(E\otimes_{\bbQ}\bbR)_{\gg0} .
\end{equation}
For  $\rho \colon E \hookrightarrow \bbC$, let $\lie g(T, \bfy, \varphi_f)_{\rho}$ denote the current on $\calX_{\rho}(\bbC) \simeq X_{\bbV[k]}$ whose restriction to a component $X_j \subset X_{\bbV[k]}$ is given by
%(this isomorphism uses the CM type $\Phi$)
%Given 
%\begin{equation}
%\bfy = (y_{1}, \dots, y_{d}) \in (\Herm_r(\bbC)_{>0})^d \simeq \Herm_r(E\otimes_{\bbQ}\bbR)_{\gg0} 
%\end{equation}
%(this isomorphism uses the CM type $\Phi$) and $\rho \colon E \hookrightarrow \bbC$, let $\lie g(T, \bfy, \varphi_f)_{\rho}$ denote the current on $\calX_{\rho}(\bbC) \simeq X_{\bbV[k]}$ whose restriction to a component $X_j \subset X_{\bbV[k]}$ is given by
\begin{equation}
 	\lie g^\mathtt{o}(T, \bfy, \varphi_f)_{\rho} \big|_{X_j}  = \mathrm{CT}_{s=0} \sum_{\substack{\bfv \in \Omega_T(\bbV[k])} }\varphi_f( h_j^{-1} \bfv) \  \archGreen \left( \rho (\bfv) \cdot \sigma_k(\bfalpha); s \right) .
\end{equation}
Here $\bfalpha \in \mathrm{GL}_r(E_\mathbb{R})$ is any matrix with totally positive determinant such that $\bfy=\bfalpha \cdot {^t}\overline{\bfalpha}$.
%and $\sigma_k \colon E \to \bbC$ is the distinguished extension of $\sigma_k \colon F \to \bbR$ contained in the fixed CM type. 
The independence of the choice of $\bfalpha$ follows again from \Cref{prop:xiRegContinuation}.(ii).

The  analogue of \Cref{prop:GlobalGreenOrthogonalDegenerateCase}, which can be proved with straightforward modifications to the arguments in the previous section, is:
\begin{proposition} \label{prop:GlobalGreenUnitary} The following equation of currents on $\calX_{\rho}(\bbC)$ holds:
	\[
	\ddc \lie g^\mathtt{o}(T, \bfy, \varphi_f)_{\rho}  \ + \ \delta_{Z(T, \varphi_f)_{\rho}} \wedge \Omega^{r - \rank T}_{\vbE^{\vee}}  \ = \ \omega(T, \bfy, \varphi_f)_{\rho}.
	\]
	Here
	$\Omega_{\vbE^{\vee}}  =  c^{\mathrm{top}}(\vbE^\vee,\nabla)^*$
	is the top Chern-Weil form of the Hermitian bundle $\vbE^{\vee}$. 
\end{proposition}

As in the orthogonal case, we introduce a modified version of $\lie g^\mathtt{o}(T,\bfy,\varphi_f)$ by adding a closed current. 
\begin{definition} \label{def:global_g_T_Y_phi_case_2}
%Let $\bfy = (y_v)_{v|\infty}$ and, for an archimedean place $v$, let
Let $\bfy = (y_v)_{v|\infty} \in \mathrm{Her}_r(E_\mathbb{R})_{\gg 0}$ and $T \in \mathrm{Her}_r(E)$. We set
\begin{equation*}
\lie g(T,\bfy,\varphi_f) := \lie g^\mathtt{o}(T,\bfy,\varphi_f)
\end{equation*} %\todo{LG: remove number from eq?}
if $T$ is not totally positive semidefinite and %\todo{should be - here for second term? LG: I think so too.}
\begin{align*}
\lie g(T,\bfy,\varphi_f) &:= \lie g^\mathtt{o}(T,\bfy,\varphi_f) \\
&\qquad -  \sum_{v|\infty} \log \left( \frac{\det' \sigma_v(T) \cdot \det y_v}{\det' \sigma_v(T)y_v}\right) \delta_{Z(T,\varphi_f)_{\sigma_v}} \wedge c_{\mathrm{rk}(\vbE)-1}(\vbE^{\vee},\nabla)^* \wedge \Omega_{\calE^{\vee}}^{r-\mathrm{rk}(T)-1} \notag
\end{align*} %\todo{LG: remove number from eq?}
if $T$ is totally positive semidefinite; when $T$ is non-degenerate, the additional term vanishes. 
\end{definition}

	\begin{example} 
		\label{ex:Green form for block diag} 
		
		We treat the orthogonal and unitary cases simultaneously here.
		Assume that $\bbV$ is anisotropic and that 
		\begin{equation}
		T = \begin{pmatrix} 0 & \\ & S \end{pmatrix}
		\end{equation} with $S $ non-degenerate, where $S \in \Sym_t(F)$ (resp.\ $\Herm_t(E)$) in the orthogonal (resp.\ unitary) cases, and $t = \rank T$.  Then any $\bfx$ with $T(\bfx) = T$ is of the form $\bfx = ( 0, \dots, 0, \bfx')$ with $T(\bfx') = S$. 
		
		Suppose  that $\varphi_f  = \varphi_f' \otimes \varphi_f''$ with $\varphi_f' \in S(\bbV(\bbA_f)^{r-t})$ and $\varphi_f '' \in S(\bbV(\bbA_f)^t)$,  so that $Z(T,\varphi_f) = \varphi_f'(0) \cdot Z(S, \varphi_f'')$. Suppose furthermore that  $\bfy$ is of the form
		\begin{equation}
		\bfy = \begin{pmatrix} \bfy' & \\ & \bfy'' \end{pmatrix}, 
		%	\qquad \text{where } \bfy'' \in \Sym_{r-r'}(F_{\bbR})_{\gg0} \text{ and }  \bfy' \in \Sym_{r'}(F_{\bbR})_{\gg0}.
		\end{equation}
		where $\bfy'$ and $\bfy''$ are totally positive definite of rank $r-t$ and $t$, respectively. 
		
		It follows from \Cref{xi0WithZeroes} that, after descending to the Shimura variety $\calX_{\sigma_k}(\bbC)$, we have the equation of currents
		\begin{equation}
		\begin{split}
		\lie g(T, & \bfy, \varphi_f)_{\sigma_k} \\ 
		&=  \ \varphi_f'(0) \cdot  \Big[  \lie g(S, \bfy'',  \varphi_f'')_{\sigma_k} \wedge \Omega_{\vbE^{\vee}}^{r-t} - \log (\det y' ) \, \delta_{Z(S, \varphi_f'')_{\sigma_k}}  \wedge c_{\mathrm{rk}(\vbE)-1}(\vbE^{\vee},\nabla)^*  \wedge \Omega_{\vbE^{\vee}}^{r-t-1}  \\
		&  \hspace{6em}   + (r - t)\int_1^{\infty} \left[ \omega(S, u \bfy'', \varphi_f'')_{\sigma_k.} - \delta_{Z(S, \varphi_f'')_{\sigma_k}} \right] \frac{du}{u}  \wedge c_{\mathrm{rk}(\vbE)-1}(\vbE^{\vee},\nabla)^*  \wedge \Omega_{\vbE^{\vee}}^{r-t-1} \ \Big].
		\end{split}
		\end{equation}
		Since $\omega(S, u\bfy'', \varphi_f')$ and $\delta_{Z(\tau, \varphi_f'')}$ are cohomologous, the term appearing on the second line above is exact, i.e.\
		\begin{align}
		\lie g(T, \bfy, \varphi_f)_{\sigma_k}  \equiv   \varphi_f'''(0) \cdot &  \left[ \lie g(S, \bfy'', \varphi_f'')_{\sigma_k} \wedge \Omega_{\vbE^{\vee}}^{r-t} - \log (\det y'_k ) \, \delta_{Z(S, \varphi_f'')_{\sigma_k}} \wedge c_{\mathrm{rk}(\vbE)-1}(\vbE^{\vee},\nabla)^*    \wedge \Omega_{\vbE^{\vee}}^{r-t-1} \right] \notag
		\\ 
		& \in D^{r-1,r-1}(\calX_{\sigma_k}(\bbC)) \big/ \mathrm{im} \partial + \mathrm{im} \overline{\partial}. 	\label{eqn:global deg current mod exact} 
		\end{align}
		This expression generalizes a similar term appearing indirectly in the work of Kudla-Rapoport-Yang  \cite[p.\ 178]{KudlaRapoportYang}, which dealt with the case of a Shimura curve over $\bbQ$ and  $r=2, r'=1$. 
		
		Finally, we note that for $T = \mathbf 0_r$ and $\bbV$ anisotropic, a computation along the same lines gives $\lie g^{\mathtt o}(\textbf0_r, \bfy, \varphi_f) = 0$ and the pleasant expression %\todo{tweaked formatting LG: OK}
		\begin{equation} \label{eqn:Green Zero example}
		\lie g( \mathbf 0_r, \bfy, \varphi_f)_{\sigma_k} \ =\ -\log (\det y_k  ) \cdot  \varphi_f(0) \cdot  c_{\mathrm{rk}(\vbE)-1}(\vbE^\vee,\nabla)^*  \wedge  \Omega_{\vbE^{\vee}}^{r-1}.
		\end{equation}
	\end{example}
	
	\begin{remark} \label{rmk:Green invariance} We continue to assume $\bbV$ is anisotropic, and note two useful invariance properties for $\lie g(T, \bfy, \varphi_f)$. Set $\bfk=F$ in the orthogonal case  and $\bfk=E$ in the unitary one.
		\begin{enumerate}
			\item Suppose $T = \left( \sm{ 0 & \\ & S} \right)$ for a non-degenerate matrix $S \in \Sym_t(F)$ (resp. $S \in \mathrm{Her}_t(E)$), and $\bm \theta \in \SL_r(\bfk_{\bbR})$ is of the form
			\begin{equation}
			\bm \theta = \begin{pmatrix} 1_{r-t} & * \\ & 1_t \end{pmatrix}.
			\end{equation}
			Then 
			\begin{equation}
			\lie g \left(T, \, \bm\theta \, \bfy \,  \transpose{\overline{\bm\theta}}, \varphi_f\right) \ = \ \lie g(T, \bfy, \varphi_f) .
			\end{equation}
			\item Suppose $\gamma \in \SL_r(\bfk)$, and let $T[\gamma] := \transpose{\overline{\gamma}}^{-1} \, T \, \gamma^{-1}$. Assume that $\gamma$ is chosen so that $\det' T[\gamma] = \det'T$. Then
			\begin{equation}
			\lie g \left(T[\gamma], \, \gamma \, \bfy \, \transpose{\overline{\gamma}} , \varphi'_f \right) \ = \ \lie g \left(T, \bfy, \varphi_f \right) 
			\end{equation}
			where $\varphi'_f(\bfx) = \varphi_f(\bfx \cdot \gamma) $.
		\end{enumerate}
		Given $T$, one can always find an element $\gamma \in \SL_r(\bfk)$ as above such that $T[\gamma] = \left( \sm{0 & \\ & S} \right)$ for some non-degenerate matrix $S$. Similarly, we may choose $\bm \theta$ as above, such that $\bm \theta \bfy \transpose{\overline{\bm \theta}} = \left( \sm{\bfy' & \\ & \bfy''} \right)$ with $\bfy''$ of the same rank as $S$;  thus we may always place ourselves in the setting of \Cref{ex:Green form for block diag}.
	\end{remark}

\subsection{Star products on $X_K$} \label{subsection:star_products_X_K} 

In this section, we continue to treat both the orthogonal and unitary cases.
Let $\lie g(T_1, \bfy_1, \varphi_1)$ and $\lie g(T_2, \bfy_2, \varphi_2)$ denote two currents attached to special cycles $Z(T_1, \varphi_1)$ and $Z(T_2,\varphi_2)$. Assume that $T_1$ and $T_2$ are non-degenerate and that $Z(T_1, \varphi_1)$ and $Z(T_2, \varphi_2)$ intersect properly, and consider the star product %\todo{tweaked formatting. LG: OK}
	\begin{align}
\lie g(T_1,  \bfy_1,\varphi_{1})  \, * \,  & \lie g(T_2,\bfy_2,\varphi_{2}) \\  
& := \  \lie g(T_1,\bfy_1,\varphi_{1})  \wedge  \delta_{Z(T_2,\varphi_{2})} \  + \  \lie g(T_2,\bfy_2,\varphi_{2}) \wedge \omega(T_1,\bfy_1,\varphi_1)  \notag
\end{align}
in $D^*(X_K)$.

\begin{theorem} \label{theorem:star_product_formula_global}
Let $\varphi = \varphi_1 \otimes \varphi_2$. With assumptions as above,
	\begin{equation*}
	\lie g(T_1,\bfy_1,\varphi_{1}) * \lie g(T_2,\bfy_2,\varphi_{2}) \equiv \sum_{\substack{  T= \left(\begin{smallmatrix} T_1 & * \\ * & T_2 \end{smallmatrix}\right)}} \lie g \left(T,\left(\begin{smallmatrix} \bfy_1 &  \\ & \bfy_2 \end{smallmatrix}\right),\varphi  \right)
	\end{equation*}
in $\tilde{D}^*(X_K):= D^*(X_K)/(\mathrm{im}  \partial + \mathrm{im}\overline{\partial})$.
\end{theorem}
\begin{proof} 
%Let $r_i=\mathrm{rk} \ T_i$. 
Fix an embedding $\sigma_k$, matrices $\alpha_i \in \GL_{r_i}(\bbK)$ such that  $\sigma_k(y_{i})= \alpha_i \cdot \transpose{\overline{\alpha_i}}$ for $i=1,2$, and a component $X_j = \Gamma_j \backslash \bbD^+ \subset \calX_{\sigma_k}(\bbC)$; working with $\Gamma_j$-invariant currents on $\bbD^+$,  the proof of \Cref{prop:DegCurrentDef}(i) implies that 
	\begin{equation}
	\lie g(T_1,\bfy_1,\varphi_{1}) * \lie g(T_2,\bfy_2,\varphi_{2})\big|_{X_j} \ = \ \sum_{\substack{\bfx_1 \in \Omega_{T_1}(\bbV[k]) \\ \bfx_2 \in \Omega_{T_2}(\bbV[k]) }} \varphi_1(h_j^{-1}\bfx_1) \,  \varphi_2(h_j^{-1}\bfx_2) \   \archGreen( \bfv_1) *  \archGreen( \bfv_2)
	\end{equation}
	where we write $\bfv_i := \sigma_k(\bfx_i) \cdot \alpha_i$; note that the proper intersection assumption implies that $\bfv = (\bfv_1,\bfv_2)$ is regular for any non-zero term above, i.e. $\bbD_{\bfv}$ is either empty or of codimension $(r_1+r_2)\mathrm{rk}(\vbE)$. By \Cref{theorem:star_products_arch}, the previous line becomes 
	\begin{equation}
	\begin{split}
		\sum_{\substack{\bfx_1 \in \Omega_{T_1}(\bbV[k]) \\\bfx_2 \in \Omega_{T_2}(\bbV[k]) }} & \varphi_1(h_j^{-1}\bfx_1) \varphi_2( h_j^{-1}\bfx_2) \left\{ \archGreen(\bfv)   - \partial \alpha(\bfv_1, \bfv_2) - \overline \partial \beta(\bfv_1, \bfv_2) 		\right\}	 \\
		=& \sum_{T = \left( \sm{ T_1 &*\\ *& T_2 } \right) } \sum_{\bfx \in \Omega_T(\bbV[k])} \varphi(h_j^{-1}\bfx) \archGreen(\bfv)  -   	\sum_{\substack{\bfx_1 \in \Omega_{T_1}(\bbV[k]) \\\bfx_2 \in \Omega_{T_2}(\bbV[k]) }}  \varphi_1(h_j^{-1}\bfx_1) \varphi_2( h_j^{-1}\bfx_2) \left\{ \partial \alpha(\bfv_1, \bfv_2) + \overline{\partial} \beta(\bfv_1, \bfv_2) \right\} \\
		=&  \sum_{T = \left( \sm{ T_1 &*\\ *& T_2 } \right) } \lie g \left(T,\left(\begin{smallmatrix} \bfy_1 &  \\ & \bfy_2 \end{smallmatrix}\right),\varphi  \right)  -   	\sum_{\substack{\bfx_1 \in \Omega_{T_1}(\bbV[k]) \\\bfx_2 \in \Omega_{T_2}(\bbV[k]) }}  \varphi_1(h_j^{-1}\bfx_1) \varphi_2( h_j^{-1}\bfx_2) \left\{ \partial \alpha(\bfv_1, \bfv_2) + \overline{\partial} \beta(\bfv_1, \bfv_2) \right\} 
	\end{split}	
	\end{equation}
	where $\bfx =(\bfx_1, \bfx_2)$ and $\alpha(\bfv_1, \bfv_2)$ and $\beta(\bfv_1, \bfv_2)$ are as in \Cref{theorem:star_products_arch}. Again, an argument along the lines of  \Cref{prop:DegCurrentDef}(i) shows that the sum
\begin{equation}
		\sum_{\substack{\bfx_1 \in \Omega_{T_1}(\bbV[k]) \\\bfx_2 \in \Omega_{T_2}(\bbV[k]) }}  \varphi_1(h_j^{-1}\bfx_1) \varphi_2( h_j^{-1}\bfx_2)  \alpha(\bfv_1, \bfv_2),
\end{equation}
and its analogue with $\alpha$ replaced by $\beta$,	converge to currents on $\bbD^+$, that are moreover $\Gamma_j$-invariant by \Cref{theorem:star_products_arch}. The theorem follows upon descending to $X_j$. 
	\end{proof}

\section{Local archimedean heights and derivatives of Siegel Eisenstein series} \label{section:arith_Siegel_Weil}

Here we prove \Cref{thm:GlobalGreenIntegral}, our main global result relating archimedean local heights and derivatives of Siegel Eisenstein series. We review the definition of these Eisenstein series and the Siegel-Weil formula in \Cref{subsection:Siegel Eisenstein series and the Siegel-Weil formula}. For the proof we also need to explicitly determine the asymptotics of the Fourier coefficients $\Eis'_T(\lambda \bfy,\Phi_f,s_0)$ as $\lambda \to \infty$; we do this in \Cref{subsection:Fourier coefficients of scalar weight Eisenstein series} and give the proof of \Cref{thm:GlobalGreenIntegral} in \Cref{subsection:Archimedean height pairings}. 

In \Cref{sec:Arithmetic Height conjecture} we explain how, using our results,  Kudla's conjectural arithmetic Siegel-Weil fomula can be rephrased in terms of Faltings heights of special cycles.

Fix a totally real number field $F$ of degree $d$, and a CM extension $E$, and set
%	 and a CM type $\Phi$ of $E$; 
%we denote by $\eta \colon F^\times \backslash \mathbb{A}_F^\times \to \{\pm 1\}$ the quadratic character corresponding to $E$. We set  
\begin{equation}
{\bf k} = \begin{cases} F, & \text{orthogonal case,} \\ E, & \text{unitary case}  \end{cases}
\qquad \text{ and }   \qquad
\mathbb{K} =  \begin{cases}\mathbb{R}, & \text{orthogonal case} \\ \mathbb{C}, & \text{unitary case.}  \end{cases} 
\end{equation}
Let $\sigma_1,\ldots,\sigma_d$ the archimedean places of $F$ in the orthogonal case, or the elements of a fixed CM type of $E$ in the unitary case.
% we identify each $\sigma_i$ with the complex embedding of $E$ belonging to $\Phi$. 

 We fix an $m$-dimensional Hermitian ${\bf k}$-vector space $(\bbV,Q)$ such that $\bbV_{\sigma_i}:= \bbV \otimes_{{\bf k},\sigma_i} \mathbb{K}$ is positive definite when $i>1$ and
\begin{equation}
\mathrm{sig } \ \bbV_{\sigma_1} =  \begin{cases} (p,2), & \text{orthogonal case} \\ (p,1), & \text{unitary case} \end{cases}
\end{equation}
with $p \geq 1$. \emph{From now on we assume that $\bbV$ is anisotropic.} 

Finally, let  $\eta \colon F^\times \backslash \mathbb{A}_F^\times \to \{\pm 1\}$ the quadratic character corresponding to $E$ and fix a unitary character $\chi \colon E^\times \backslash \mathbb{A}_E^\times \to \mathbb{C}^\times$ such that $\chi|_{\mathbb{A}_F^\times} = \eta^m$.

\subsection{Siegel Eisenstein series and the Siegel-Weil formula} \label{subsection:Siegel Eisenstein series and the Siegel-Weil formula}

\subsubsection{} \label{subsubsection:Siegel_Eisenstein_definition_and_Siegel_Weil} 
Let $\mathrm{Mp}_{2r}(\mathbb{A}_F)$ be the metaplectic double cover of $\mathrm{Sp}_{2r}(\mathbb{A}_F)$ and for a positive integer $r$, set
\begin{equation}
{\bf G}'_r(\mathbb{A}) =  \begin{cases} \mathrm{Mp}_{2r}(\mathbb{A}_F), & \text{ case 1,} \\ \mathrm{U}(r,r)(\mathbb{A}_F), & \text{ case 2,} \end{cases}
\end{equation}
Denote by $P_r(\mathbb{A})$ the standard Siegel parabolic of ${\bf G}'_r(\mathbb{A})$; then $P_r(\mathbb{A})=M_r(\mathbb{A}) \ltimes N_r(\mathbb{A})$, where
\begin{equation}
\begin{split}
M_r(\mathbb{A}) &= \{(m(a),\epsilon) | a \in \mathrm{GL}_r(\mathbb{A}_F), \ \epsilon=\pm 1\}, \\
N_r(\mathbb{A}) &= \left\{(n(b),1) | b \in \mathrm{Sym}_r(\mathbb{A}_F)\right\} 
\end{split}
\end{equation}
in case 1 and
\begin{equation}
M_r(\mathbb{A}) = \{m(a) | a \in \mathrm{GL}_r(\mathbb{A}_E)\}, \quad N_r(\mathbb{A}) = \left\{ n(b) | b \in \mathrm{Her}_r(\mathbb{A}_E) \right\}
\end{equation}
%\begin{equation}
%\begin{split}
%M_r(\mathbb{A}) &= \{m(a) | a \in \mathrm{GL}_r(\mathbb{A}_E)\}, \\
%N_r(\mathbb{A}) &= \left\{ n(b) | b \in \mathrm{Her}_r(\mathbb{A}_E) \right\}
%\end{split}
%\end{equation}
in case 2. We also write
\begin{equation}
\begin{split}
\underline{m}(a) &= \begin{cases} (m(a),1), & \text{ for } a \in \mathrm{GL}_r(\mathbb{A}_F) \text{ in case 1,} \\ m(a), & \text{ for } a \in \mathrm{GL}_r(\mathbb{A}_E) \text{ in case 2,} \end{cases} \\
\underline{n}(b) &= \begin{cases}(n(b),1), & \text{ for } b \in \mathrm{Sym}_r(\mathbb{A}_F) \text{ in case 1,} \\ n(b), & \text{ for } b \in \Herm(\mathbb{A}_E)  \text{ in case 2,} \end{cases} \\
\underline{w}_r &=   \begin{cases} (w_r,1), & \text{ case 1,} \\ w_r, & \text{ case 2.} \end{cases} 
\end{split}
\end{equation}
where $w_r = ( \sm{& 1_r \\ -1_r & })$. 
The multiplication in $M_r(\mathbb{A})$ in case 1 is defined by
\begin{equation}
(m(a_1),\epsilon_1) \cdot (m(a_2),\epsilon_2) =(m(a_1a_2),\epsilon_1\epsilon_2 (\det a_1,\det a_2)_\mathbb{A}),
\end{equation}
where $(\cdot,\cdot)_\mathbb{A}$ denotes the Hilbert symbol of $F$. 

Define a character $\chi_{\bbV}$ of $M_r(\mathbb{A})$ as follows: in case 1, set
\begin{equation} \label{eq:def_chi_bbV_orthogonal}
\chi_{\bbV} \left(m(a), \epsilon \right) =\left(\det a,(-1)^{m(m-1)/2} \det V \right)_{\bbA} \cdot 
\begin{cases} 
\epsilon \cdot  \gamma_{\bbA}(\det a, \psi)^{-1}, & \text{ if } m \text{ is odd,} \\
1, & \text{ if } m \text{ is even,}
\end{cases}
\end{equation}
where $\gamma_{\mathbb{A}}$ denotes the Weil index, see \cite{KudlaSplitting}, and in case 2, set  
\begin{equation} \chi_{\bbV}(m(a)) = \chi(\det a). \end{equation}
 We may extend $\chi_{\bbV}$ to a character of $P_r(\mathbb{A})$ by declaring it trivial on $N_r(\mathbb{A})$, and define 
\begin{equation}
I_r(\bbV,s) = \mathrm{Ind}_{P_r(\mathbb{A})}^{{\bf G}'_r(\mathbb{A})}(\chi_{\bbV} |\cdot|^s)
\end{equation}
(smooth induction), where the induction is normalized so that $s=0$ belongs to the unitary axis. Concretely, elements of $I_r(\bbV,s)$ are smooth functions $\Phi(\cdot, s) \colon \bfG'_r(\bbA) \to \bbC$ satisfying 
\begin{equation}
\Phi((m(a),\epsilon)(n(b),1)g',s) =  |\det a|_{\bbA_{\bfk}}^{s + \rho} \cdot \chi_{\bbV}(m(a),\epsilon) \cdot \Phi(g',s), \quad \rho=\frac{r+1}{2}
\end{equation}
in case 1 and
\begin{equation}
\Phi(m(a)n(b) g', s)  = |\det a|_{\bbA_{\bfk}}^{s + \rho} \cdot \chi_{\bbV}(m(a)) \cdot \Phi(g',s), \quad \rho=\frac{r}{2}
\end{equation}
in case 2.
%for all $m = m(a) \in \sfM_r(\bbA)$ and $n \in \sfN_r(\bbA)$; here $\rho = \frac{r+1}{2}$ in the orthogonal case, and $\rho = \frac{r}2$ in the unitary case.

We say that a section $\Phi(s) \in I_r(\bbV,s)$ is standard if its restriction to the standard maximal compact $K'_{r,\mathbb{A}}$ of $G_r'(\mathbb{A})$ is $K'_{r,\mathbb{A}}$-finite and independent of $s$.

Let
\begin{equation}
\bfG'_r(F) = \begin{cases} \mathrm{Sp}_{2r}(F), & \text{orthogonal case,} \\ \mathrm{U}(r,r)(F), & \text{unitary case;} \end{cases}
\end{equation}
then there is an embedding $\bfG'_r(F) \to \bfG'_r(\mathbb{A})$; given by simply by the diagonal embedding in case 2, and the canonical splitting of the metaplectic cover $\mathrm{Mp}_{2r}(\mathbb{A}_F) \to \mathrm{Sp}_{2r}(\mathbb{A}_F)$ over $\mathrm{Sp}_{2r}(F)$ in case 1. 
%See my note Metaplectics v2 for an explanation of the canonical splitting.
In the sequel we will tacitly identify $\bfG'_r(F)$ with its image under this embedding.

Given a standard section $\Phi(s) \in I_r(\bbV,s)$ and $g' \in G_r'(\mathbb{A})$, the Siegel Eisenstein series
\begin{equation}
E(g',\Phi,s) = \sum_{\gamma \in P_r(F) \backslash {\bf G}'_r(F)} \Phi(\gamma g',s)
\end{equation}
converges for $\mathrm{Re}(s) \gg 0$ and admits meromorphic continuation to $s \in \mathbb{C}$.
%; the function $E(g',\Phi,s)$ is the Siegel Eisenstein series.  
It admits a Fourier expansion 
\begin{equation}
E(g',\Phi,s) = \sum_{T} E_T(g',\Phi,s).
\end{equation}
where $T$ ranges over $\Sym_r(F)$ in case 1 (resp.\ $\Herm_r(E)$ in case 2),
%Here 
%\begin{equation}
%X_r({\bf k})=\left\{ \begin{array}{rr} \mathrm{Sym}_r(F), & \text{orthogonal case,} \\ \mathrm{Her}_r(E), & \text{unitary case,} \end{array} \right.
%\end{equation}
and 
\begin{equation} \label{eq:def_E_T}
E_T(g',\Phi,s) = \int\limits_{N_r(F) \backslash N_r(\mathbb{A})} E(\underline{n}(b)g',\Phi,s) \,  \psi(-\mathrm{tr}(Tb)) \, d\underline{n}(b),
\end{equation}
where $d\underline{n}(b)$ denotes the Haar measure on $N_r(\mathbb{A}_F)$ that is self-dual with respect to the pairing $(b,b') \mapsto \psi(\mathrm{tr}(bb'))$.

\subsubsection{}  \label{subsubsection:Siegel Weil redux}
Let $\omega=\omega_{\psi,\chi}$ be the Weil representation of $\bfG_r'(\mathbb{A}_F) \times \mathrm{U}(\bbV(\mathbb{A}))$ on $\mathcal{S}(\bbV(\mathbb{A}_F)^r)$. For $\phi \in \mathcal{S}(\bbV(\mathbb{A}_F)^r)$, $g' \in \bfG_r'(\mathbb{A}_F)$ and $h \in \mathrm{U}(\bbV(\mathbb{A}))$, define the theta series
\begin{equation}\label{eq:def_adelic_theta_series}
\Theta(g',h;\phi) = \sum_{\bfv \in \bbV(\bfk)^r} \omega(g',h) \, \phi(\bfv).
\end{equation}
The Siegel-Weil formula relates the integral of this function over $\mathrm{U}(\bbV)(F) \backslash \mathrm{U}(\bbV)(\mathbb{A})$  to the value of an Eisenstein series. Rather than discussing the formula in full generality, it will be convenient to recast the theta integral in the context of the Shimura varieties discussed above. For $k=1,\dots,d$, we have the ``nearby" spaces $\bbV[k]$, obtained by switching invariants at $\sigma_1$ and $\sigma_k$. It follows immediately from definitions that $I_{r}(\bbV[k], s) = I_r(\bbV, s)$; in the sequel, we will implicitly identify these spaces without further mention. 

Fix such a $k$ and a compact open subgroup $K \subset \bfH_{\bbV[k]}(\bbA_f)$, and let 
 \begin{equation}
 X_{\bbV[k]} \  \ = \ 	X_{\bbV[k], K}  \ =   \  \bfH_{\bbV[k]}(\bbQ) \big\backslash \bbD(\bbV[k]) \times \bfH_{\bbV[k]}(\bbA_f) \big/ K.
 \end{equation}
 Since $\bfH_{\bbV[k]}(\bbR)$ acts transitively on $\bbD(\bbV[k])$,  we may identify 
 \begin{equation}
 X_{\bbV[k]} \simeq  \bfH_{\bbV[k]}(\bbQ) \big\backslash \bfH_{\bbV[k]}(\bbA_{\bbQ})\big/ K_{\infty} K
 \end{equation}
 where $K_{\infty} \subset \bfH_{\bbV[k]}(\bbR)$ is the stablilizer of a fixed point $z_0 \in \bbD(\bbV[k])$. Thus, if $\phi $ is $K_{\infty} K$-invariant, then the theta function $\Theta(g', h; \phi)$ descends to a well defined function $\Theta(g',h;\phi)$ on $\bfG'_r(\bbA) \times X_{\bbV[k]}$. 

For any $\phi \in \mathcal{S}(\bbV(\mathbb{A}_F)^r)$, the function $\Phi(g')=\omega(g')\phi(0)$ belongs to $I_r(\bbV,s_0)$, where
\begin{equation}
s_0=s_0(r) = \begin{cases} (m-r-1)/2, & \text{orthogonal case,} \\ (m-r)/2, & \text{unitary case;} \end{cases}
\end{equation}
%These maps define a $\bfG_r'(\mathbb{A})$-intertwining map  $\lambda:\mathcal{S}(\bbV(\mathbb{A}_F)^r) \to I_r(\bbV,s_0)$. 
this construction defines a $\bfG_r'(\mathbb{A})$-intertwining map  that we denote $\lambda \colon \mathcal{S}(\bbV(\mathbb{A}_F)^r) \to I_r(\bbV,s_0)$. 

\begin{theorem}[Siegel-Weil formula] \label{thm:Siegel Weil redux} Suppose $\bbV$ is anisotropic  and let $\phi  \in \mathcal S(\bbV(\bbA)^r)^{K_{\infty}K}$.  Denote by $\Phi \in I_r(\bbV,s)$  the unique standard section such that $\Phi(\cdot, s_0(r)) = \lambda(\phi)$. Let $\Omega$ be a positive $\bfG(\mathbb{R})$-invariant differential form on $\mathbb{D}(\bbV)$ of top degree. Then $E(g', \Phi, s)$ is regular (in the variable $s$) at $s=s_0(r)$, and
	\[
		 \frac{\kappa_0}{\mathrm{vol}(X_{\bbV}, \Omega)}	\int_{[X_{\bbV}]}  \Theta(g',\cdot; \phi) \, \Omega    =   E(g', \Phi, s_0(r)),
	\]
	where
		\begin{equation} \label{eq:Siegel_Weil_kappa_constant}
			\kappa_0 =  \begin{cases} 1, & \text{ if } s_0(r) >0, \\ 2, & \text{ if } s_0(r)=0. \end{cases}
			\end{equation}
\begin{proof}
	Recall the usual formulation of the Siegel-Weil formula: set $\bfH^1_{\bbV} := \mathrm{Res}_{F/\bbQ} \, \mathrm{O}(\bbV)$ in the orthogonal case and $\bfH^1_{\bbV} = \bfH_{\bbV} =  \mathrm{Res}_{F/\bbQ}  \, \mathrm{U}(\bbV)$ in the unitary case. The Siegel-Weil formula reads
	\begin{equation} \label{eq:Siegel_Weil}
					E(g', \Phi, s_0(r)) = \kappa_0 \int\limits_{\bfH^1_{\bbV}(\bbQ) \backslash \bfH^1_{\bbV}(\mathbb{A})} \Theta(g',h;\phi) \,  dh,
	\end{equation}
	where the Haar measure $dh$ is normalized so that $\bfH^1_{\bbV}(\bbQ) \backslash \bfH^1_{\bbV}(\mathbb{A})$ has volume one. This is proved in \cite{Ichino2004, Ichino2007} in the unitary case, and in \cite{KudlaRallis1} and \cite{SweetThesis} for the metaplectic cases with $m$ even and odd, respectively; a convenient reference treating all cases simultaneously is \cite{GanQiuTakeda}.
	
We now claim that
	\begin{equation} \label{eqn:geometricThetaIntegral}
			\int_{X_{\bbV}} \Theta(g',h;\phi)  \, \Omega = C \cdot E(g', \Phi, s_0(r))
	\end{equation}
	for some constant $C$ independent of $\phi$. To see this, note that
	\begin{equation}
	\begin{split}
			\int_{X_{\bbV}} \Theta(g',\cdot;\phi) \, \Omega &= 
			\int\limits_{\bfH_{\bbV}(\bbQ) \backslash \bfH_{\bbV}(\bbA) / K_{\infty}K } \Theta(g',h;\phi) \, d'h  \\
			 &=  \mathrm{vol}(K_{\infty} K) \int\limits_{\bfH_{\bbV}(\bbQ) \backslash \bfH_{\bbV}(\bbA)} \Theta(g',h; \phi) \, d'h
	\end{split}
	\end{equation}
	for some Haar measure $d'h$. By \eqref{eq:Siegel_Weil}, this establishes the claim in the unitary case. The orthogonal case follows from the fact that the action of $\bfH_\bbV(\mathbb{A}) = \mathrm{GSpin}(\bbV(\mathbb{A}))$ on $\mathcal{S}(\bbV(\mathbb{A})^r)$ factors through its quotient $\mathrm{SO}(\bbV(\mathbb{A}))$, together with \cite[Thm. 4.1.(ii)]{KudlaIntegrals}, which shows that,
%	(up to multiplying by a non-zero constant), 
up to multiplying by a non-zero constant, the integral over $\mathrm{O}(\bbV(\mathbb{Q})) \backslash \mathrm{O}(\bbV(\mathbb{A}))$ in \eqref{eq:Siegel_Weil} can be replaced by integration over $\mathrm{SO}(\bbV(\mathbb{Q})) \backslash \mathrm{SO}(\bbV(\mathbb{A}))$.
	
	To evaluate the constant $C$, compare the constant terms in the Fourier expansion on both sides of \eqref{eqn:geometricThetaIntegral}; the left hand side is $\mathrm{vol}(X_{\bbV[k]}, \Omega) \cdot \omega(g')\phi(0)$, while, again using \eqref{eq:Siegel_Weil}, the right hand side is $C \cdot \kappa_0 \cdot \omega(g') \phi(0)$.
\end{proof}
\end{theorem}

\subsection{Fourier coefficients of scalar weight Eisenstein series} \label{subsection:Fourier coefficients of scalar weight Eisenstein series}

In this section we study the asymptotic behaviour of the Fourier coefficients $E_T(g',\Phi,s)$ as $g'$ goes to infinity, under certain hypotheses on $\Phi$. More precisely,let $K_r'$ be the standard maximal compact subgroup of $G'_r = \bfG_r'(F_v)$ (where $v$ is archimedean) described in \Cref{subsubsection:symplectic_group_definitions}, and let  
\begin{equation}
l = \begin{dcases} \frac{m}{2}, & \text{orthogonal case} \\ \left( \frac{m+k(\chi)}{2}, \frac{-m+k(\chi)}{2} \right), & \text{unitary case} \end{dcases}
\end{equation}
%$l$ be the weight of $K_r'$ in  
as in
\eqref{eq:def_scalar_weight_both_cases}. Assume that $\Phi=\Phi_f \otimes \Phi_\infty$, with
\begin{equation} \label{eq:Siegel_section_conditions}
\begin{split}
\Phi_\infty&:=\Phi^l \otimes \cdots \otimes \Phi^l \in I_r(\bbV(\mathbb{R}),s), \\
\Phi_f & \ = \lambda(\varphi_f) \text{ for some Schwartz form } \varphi_f \in \mathcal{S}(\bbV(\mathbb{A}_f)^r).
\end{split}
\end{equation}
Here $\lambda$ is as in \Cref{subsubsection:Siegel_Eisenstein_definition_and_Siegel_Weil}. 

 \begin{lemma} \label{lemma:Eis is hol}
With $\Phi$ as above, the Eisenstein series $E(g',\Phi, s)$ is regular at $s=s_0$.
 	\begin{proof}
 		When $s_0 = 0$, this follows from \cite[Thm. 1.1]{KudlaRallisSiegel1} and \cite{Tan}. Suppose $s_0 > 0$ so that $r \leq p$: let $z_0 \in \bbD$ denote the fixed based point  as in \Cref{subsubsection:hermitian_domain_definitions}, and consider the Schwartz form $\tilde{\varphi} \in \mathcal{S}(\bbV_{\sigma_1}^r)$ defined by
 		\begin{equation} \label{eqn:def of phi tilde}
 		\varphi(\bfv,z_0) \wedge \Omega^{p-r}(z_0) = \tilde{\varphi}(\bfv)  \, \Omega^{p}(z_0), \quad \bfv \in V^r.
 		\end{equation}
 		where $\Omega  =c_1(\calE, \nabla)^*$.
 		As remarked in \Cref{subsubsection:varphi_and_nu_under_Weil_rep}, it has weight $l$ under $K_r'$. Since $\tilde{\varphi}({\bf 0})=(-1)^r$ , an argument as in the proof of \Cref{lemma:nu_and_Phi} shows that $\lambda(\tilde{\varphi})=(-1)^r \Phi^l(s_0)$. Thus the global element $\Phi = \otimes_v \Phi^l \otimes \Phi_f$ is in the image of $\lambda$, and the lemma follows from \Cref{thm:Siegel Weil redux}.
 	\end{proof}
 \end{lemma}

For $\bftau=(x_j+iy_j)_{1 \leq j \leq d} \in \mathbb{H}_r^d$, let $g'_{\bftau}=(g'_{\tau_j}) \in \bfG'_r(\mathbb{R})$ with $g'_{\tau_j}=\underline{n}(x_j)\underline{m}(\alpha_j)$ as in \eqref{eq:def_g_z'} and define
\begin{equation} \label{eq:def_Ecal_Eis_series}
\Eis(\bftau,\Phi_f,s) := (\det y_1\cdots y_d)^{-\iota m/4} \, E(g'_{\bftau},\Phi_f \otimes \Phi_\infty^l,s).
\end{equation}
Then $\Eis(\bftau,\Phi_f,s)$ has a Fourier expansion
\begin{equation} \label{eqn:classical Eis series}
\begin{split}
\Eis(\bftau,\Phi_f,s) &= \sum_{T } \Eis_T(\bftau,\Phi_f,s)  \\
& = \sum_{T } C_T(\bfy,\Phi_f,s) \, q^T, \qquad q^T := \prod_{v|\infty} e^{2\pi i \mathrm{tr}(\tau_v \sigma_v(T))},
\end{split}
\end{equation}
where $T $ runs over $\Sym_r(F)$  (resp.\ $\Herm_r(E)$), and
% $\Eis_T(\bftau,\Phi_f,s):=(\det y_1 \cdots y_d)^{-\iota m/4} E_T(g'_{\bftau},\Phi_f,s)$ and
\begin{equation} \label{eq:def_Eis_coeff_C_T}
C_T(\bfy,\Phi_f,s) := (\det y_1 \cdots y_d)^{-\iota m/4} \, E_T(g'_{\bftau},\Phi_f,s) \, q^{-T}.
\end{equation}
%{\bf TO DO: add explanation about $\lambda$ and remark that the Eisenstein series is holomorphic at $s_0$ +  defintion of $C_T$ (used in 5.5.2).}
In the rest of this section we will determine the asymptotic behaviour of $\tfrac{d}{ds}C_T(\lambda \bfy,\Phi_f,s_0)$ as $\lambda \to +\infty$.

\subsubsection{} Consider first the case $\det T \neq 0$. Assuming that $\Phi=\prod_v \Phi_v$, we have 
\begin{equation} \label{eqn:Eis coeff product}
E_T(g',\Phi,s)\ = \ \prod_v W_{T,v}(g_v',\Phi_v,s),
\end{equation}
 where
\begin{equation}
W_{T,v}(g_v',\Phi_v,s) = \int_{N_r(F_v)} \Phi_v(\underline{w}_r^{-1}\underline{n}(b)g',s) \, \psi(-\mathrm{tr}(T\underline{n}(b))) \, d\underline{n}(b), \quad \mathrm{Re}(s) \gg 0.
\end{equation}
Since $T$ is nonsingular, each $W_{T,v}(g_v',\Phi_v,s)$ admits analytic continuation to $s \in \mathbb{C}$ and we can write
\begin{equation} \label{eq:E_T'_non_deg_sum}
\left.\frac{d}{ds}E_T(g',\Phi,s)\right|_{s=s_0} = \sum_{v} E_T'(g',\Phi,s_0)_v,
\end{equation}
where the sum runs over places of $F$ and we set
\begin{equation} \label{eq:def_E'_T_v}
E'_T(g',\Phi,s_0)_v = \prod_{w \neq v}  W_{T,w}(g_w',\Phi_w,s_0) \cdot \left. \frac{d}{ds} W_{T,v}(g_v',\Phi_v,s)\right|_{s=s_0}.
\end{equation}
%Thus the asymptotic behaviour of $\Eis_T(\lambda\bfy,\Phi_f,s_0)$ and $\tfrac{d}{ds}\Eis_T(\lambda\bfy,\Phi_f,s_0)$ is given by Propositions \ref{prop:NonDegenWhittakerEstimates} and \ref{prop:NonDegArchWhittakerNonPos}.

\begin{lemma}  \label{lem:nonDegKappaBeta}
	Suppose $\Phi$ satisfies \eqref{eq:Siegel_section_conditions} and $\det T \neq 0$. Then
	
	\begin{align*}
		 \beta(T, \Phi_f)  &:=    C_{T}( \bfy, \Phi_f,s_0(r))   \\ \intertext{and}
	 \kappa(T, \Phi_f) &:=  \lim_{\lambda \to \infty}  \frac{d}{d s }C_{T}(\lambda \bfy, \Phi_f,s)|_{s = s_0(r)} %\\
	\end{align*}
	are independent of $\bfy$.  Let $\iota=1$ in the orthogonal case and $\iota=2$ in the unitary one and set $\kappa=1+\tfrac{\iota}{2}(r-1)$.  Then explicit values for these quantities %limits 
	are as follows:
	\begin{enumerate}[(i)]
		\item If $T$ is not  totally positive definite, then $\beta(T,\Phi_f)  = \kappa(T, \Phi_f) = 0$.
		\item Suppose $T$ is totally positive definite and set $W_{T,f}(e, \Phi_f,s) = \prod_{v < \infty} W_{T,v}(e, \Phi_v,s)$ and
		
		\[
					c(T) :=   \left( \frac{  (- 2\pi i )^{rl}  }{2^{r (\kappa-1)/2} \, \Gamma_r( l)} \right)^d  \cdot N_{F/\bbQ}(\det T)^{\iota \, s_0(r)}.
		\] 
		Then
	\[
		\beta(T, \Phi_f) \ = \ c(T) \cdot W_{T, f}(e, \Phi_f,s_0(r))
	\]
	and
	\begin{align*}
		 \kappa(T, \Phi_f)   =  \left( \frac{\iota d}{2} \left( r \log \pi - \frac{\Gamma_r'(l)}{\Gamma_r(l)} \right) + \frac{\iota}{2} \log N_{F/\mathbb{Q}} \det T \right)  & \beta(T,\Phi_f)   \\
		 &  + c(T)  \, W_{T,f}'(e,\Phi_f,s_0(r)).
	\end{align*}
	Note that if $\beta(T, \Phi_f) \neq 0$, then the above may be rewritten more suggestively as
	\begin{equation*}
	\begin{split}
		 \kappa(T, \Phi_f)  &=  \left[  \frac{ \iota d}{2} \left( r \log \pi - \frac{\Gamma_r'(l)}{\Gamma_r(l)} \right) + \frac{\iota}{2} \log N_{F/\mathbb{Q}} \det T \right. \\
		& \hspace{12em}\left.  + \, \frac{W'_{T,f}(e, \Phi_f,s_0(r))}{W_{T,f}(e,\Phi_f,s_0(r))} \right] \beta(T,\Phi_f).
    \end{split}
	\end{equation*}
	\item In both cases, $ \frac{\partial}{\partial \lambda} \frac{\partial}{\partial s }C_{T}(\lambda \bfy,\Phi_f,s)|_{s = s_0(r)}  =  O(\lambda^{-1-C})$ as $\lambda \to \infty$, for some $C>0$. 
		\end{enumerate} 
\begin{proof}
%	As in \Cref{rmk:BjTopAndBottom}(ii), and taking our normalizations into account, we have
 Writing \eqref{eqn:Eis coeff product} in classical coordinates, we have
	\begin{equation}
		C_T(\lambda \bfy,\Phi_f,s)  = \left( \prod_{v | \infty}   \archW_{\sigma_v(T)}(\lambda y_v, s)  \right) \cdot W_{T,f}(e,\Phi_f,s)
	\end{equation}
where $\archW_{\sigma_v(T)}(\lambda y_v, s)$ is the normalized archimedean Whittaker functional as in \eqref{eqn:ArchWhittakerClassicalNormalization};  all claims in the lemma follow easily from Propositions  \ref{prop:NonDegenWhittakerEstimates} and \ref{prop:NonDegArchWhittakerNonPos}.
\end{proof}
\end{lemma}

\subsubsection{} We now consider the asymptotic behaviour of $\tfrac{d}{ds}C_T(\lambda\bfy,\Phi_f,s_0)$ as $\lambda \to +\infty$ for a general matrix $T$. Our approach, which follows that of \cite{KudlaRallis1}, involves relating these coefficients to certain Eisenstein series on $\GL_r$.

Fix a matrix $T$ and let $t=\mathrm{rk}(T) \leq r$.
A change of variables in \eqref{eq:def_E_T} shows that, for $\gamma \in \mathrm{SL}_r(\bfk)$ and $g \in \bfG'_r(\mathbb{A})$ we have
\begin{equation} \label{eq:equivariance_E_T_gamma}
E_{T[\gamma]}(g,\Phi,s) = E_{T}(\underline{m}(\gamma^{-1})g,\Phi,s), \quad T[\gamma] := {^t}\overline{\gamma}^{-1}T\gamma^{-1}.
\end{equation}
Hence, by choosing an appropriate $\gamma$,
%, when studying $E_T(\underline{m}(\lambda 
%\cdot \mathrm{Id})g,\Phi,s)$ as $\lambda \to +\infty$
 it suffices to consider the case when $T$ is of the form 
\begin{equation} \label{eq:T_block_diagonal_form}
T=\begin{pmatrix} 0_{r-t} &  \\  & \lbT \end{pmatrix},
\end{equation}
where $\lbT$ is non-degenerate of rank $t$  
% $0 \leq t \leq r$ and $\lbT \in X_t(\bfk)$ satisfies 
when $t \neq 0$. We assume $T$ is of this form until \Cref{prop:EisGrowthEstimate}. 

For integers $1 \leq k \leq k'$ and any ring $R$, consider the embedding $\Mat_{2k}(R) \to  \Mat_{2k'}(R) $ given by
\begin{equation}
\begin{pmatrix} A& B \\ C & D\end{pmatrix} \mapsto \left( \begin{array}{cc|cc} \mathbf{1}_{k'-k}  & & \mathbf 0_{k'-k} &   \\ & A & & B \\ \hline \mathbf 0_{k'-k} & & \mathbf 1_{k'-k} &   \\ & C & & D\end{array} \right), \qquad A, B ,C, D \in \Mat_k(R)
\end{equation}
It is easily checked that in the orthogonal case, this map induces an embedding
\begin{equation} \label{eqn:etaMetaplecticEmbeddingDefn}
	\eta_k^{k'} \colon \Mp_{2k}(\bbA )\to \Mp_{2k'}(\bbA)
\end{equation}
with $\eta_k^{k'}([1, \epsilon]) = [1, \epsilon]$, and in the unitary case an embedding $\eta_k^{k'} \colon \mathrm{U}(k,k)(\bbA) \to \mathrm{U}(k',k')(\bbA)$; the same is true over $F_v$ for any place $v$.  % Moreover, this map is compatible with the formation of the subgroups $\sfM_k, \sfN_k$, etc. 

For integers $0 \leq j\leq r$,  define a parabolic subgroup of $\GL_r$ by
\begin{equation}
	\mathcal P_{r,j} \ = \ \left\{  \begin{pmatrix} * & * \\ \mathbf 0_{j, r-j} & * \end{pmatrix}  \right\} \cap \GL_r.
\end{equation}

\begin{lemma} \label{lem:GlobalEisFCDecomp} 
	% Let $t\geq 0$ and $T=\left(\sm{0_{r-t} & \\ & \lbT} \right) \in \sfX_r(\bfk)$, where $\lbT \in \sfX_t(\bfk)$ is non-degenerate.
Suppose $T$ is of the form \eqref{eq:T_block_diagonal_form} and $t = \rank(T)$. For a standard section $\Phi(s) \in I_r(\bbV,s)$, there is a decomposition
	\[
			E_T(g, \Phi,s) = \sum_{j = t}^r  \sum_{ \sfa \in \mathcal P_{r-t,j-t}(\bfk) \backslash \GL_{r-t}(\bfk)}   B^{j}_{T} \left(  \ul m    (\sm{ \sfa & \\ & \mathbf 1_{t} }) g , \Phi,s \right)
	\]
	where 
	\begin{equation} \label{eqn: Bj Global Def}
	B^{j}_{T} (g, \Phi,s) := \int_{\ul n(b) \in \sfN_j(\bbA)}  \Phi \left(  \eta_j^r( \ul w_j^{-1} \ul n(b))  \cdot g, s\right) \psi_{T} \left( \sm{\mathbf 0_{r-j}  & \\ & - b} \right)  d \ul n(b).
	\end{equation}
	and 
$ w_j = (  \sm{&  \mathbf 1_j \\  -\mathbf 1_j &})$.  
%\todo{added underlines; changed $w$ to $w^{-1}$ but I think the formula is true either way (see Section 1.6 of Sweet's thesis) -- I'm feeling a little lazy, if it looks correct this way I'm fine with it}

	\begin{proof}
		This follows from the standard unfolding argument for the Fourier coefficients of Eisenstein series and the Bruhat decomposition. See e.g.\ \cite[Lemma 2.4]{KudlaRallis1} for the symplectic case; the proofs for the cases required here are identical.
	\end{proof}
\end{lemma}

If the section $\Phi  = \otimes_v \Phi_v$ is factorizable, then there is a product expansion
\begin{equation}
	B^{j}_{T} (g, \Phi,s) = \prod_{v} B^{j}_{T, v} (g_v, \Phi_v,s) 
\end{equation}
taken over the places of $F$, where
\begin{equation}
	B^{j}_{T, v} (g_v, \Phi_v,s)   := \int_{\sfN_j(F_v)}  \Phi_v \left(  \eta_j^r( \ul w_j ^{-1} \ul n(b)) g_v, s\right) \,  \psi_{T, v} \left( \sm{\mathbf 0_{r-j}  & \\ & - b} \right) \, d \ul n(b).
\end{equation}
Note that if $j = r$, then these factors are the usual local Whittaker functions 
	\begin{equation} \label{eqn:LocalWhittakerDef}
		W_{T,v}(g_v, \Phi_v,s) = \int_{\sfN_r(F_v)} \Phi_v( \ul w_r^{-1} \ul n(b) g, s) \psi_{T,v}(-b)  \, d \ul n(b).
	\end{equation}
We shall relate these, as well as the remaining terms, to Whittaker functions of lower rank; we require a bit more notation.

	For positive integers $k \leq k'$, pullback by $\eta_k^{k'}$ induces a map
	\begin{equation} \label{eqn:etaStarDef}
	(\eta^{k'}_k)^* \colon I_{k',v}(\bbV, s) \to I_{k,v}\left(\bbV, s+ \frac{k'-k}{2}  \right)
	\end{equation}
	preserving holomorphic standard sections. Given a place $v$ of ${\bf k}$, define an operator $U_{k,k',v}(s)$ as follows: for $\Phi \in I_{k'}(\bbV, s)$ and $ g \in \bfG'_{k}(F_v)$, and supposing $\mathrm{Re}(s)$ is sufficiently large, let
	\begin{equation} \label{eqn:UIntertwiningDef}
		\left(U_{k,k',v}(s) \Phi \right) ( g ) = \int\limits_{ \substack{b_1 \in \sfN_{k'-k}(\bfk_v)   \\ b_2 \in \Mat_{k'-k,k}(\bfk_v)} }   \Phi  \left( \ul w^{-1}_{k'}   \ul n \left( \sm{b_1 & b_2 \\ b_2^* & \mathbf 0_k }\right)  \eta_k^{k'}( \ul w_{k}^{-1} \cdot g) ,  s \right) \, db_1 \, db_2,
	\end{equation}
generalizing the construction in \cite[Section 7]{KudlaRallis1}.
%	\comm{[check for $Mp$]}
	Also, let $\calM_{k}(s)=\otimes_v \calM_{k,v}(s)$ denote the standard intertwining operator, where
	\begin{equation}
	\calM_{k,v}(s) \colon I_{k,v}(\bbV, s) \to I_{k,v}(\bbV,-s)
	\end{equation}
	is defined for $\mathrm{Re}(s)$ sufficiently large by the integral
	\begin{equation} \label{eqn:localIntOp}
	\calM_{k,v}(s)  \Phi (g) = \int\limits_{ \sfN_k(F_v)}  \Phi \left(  \ul w_k^{-1} \,  \ul n(b) \, g, s \right) \,   d \ul n(b), \qquad \Phi \in I_{k,v}(\bbV,s).
	\end{equation}
	Both $\calM_{k,v}(s)$ and $\calM_{k}(s)$ admit meromorphic continuation to $\bbC$.

\begin{lemma} \label{lem:BjGlobalWhittakerReln} Let $v$ be a place of $F$ and suppose $T = (\sm{ 0& \\ & \lbT } )$ as in \eqref{eq:T_block_diagonal_form}.
	\begin{enumerate}[(i)]
		\item If $\Phi \in I_{k',v}(\bbV,s)$, then $U_{k, k',v}(s)\Phi \in I_{k, v}(\bbV, s-\frac{k' - k}2).$ 	
		
		\item
			Fix an integer $j$ with $t \leq j \leq r$. If $t \neq 0$,  define
				\[
				U_v(s) :=  \left[ U_{t, j,v}\left(s + \frac{r-j}2 \right) \circ (\eta_j^r)^*  \right]   
				\]
		which, by $(i)$, defines a $\bfG_t'(F_v)$-intertwining map
		\[
			 U_v(s)	\colon I_{r,v}(\bbV, s) \to I_{t,v}(\bbV, \sigma ) , \qquad \text{ where }  \sigma = s + \frac{r+t}{2} - j.
		\]
		Then for any $\Phi_{r,v} \in I_{r,v}(\bbV,s)$ and $g \in \bfG'_r(F_v)$,
		\[
					B_{T, v}^{j}(g,  \Phi_{r,v},s)   = W_{\lbT, v}\left(e,  U_v(s) \left( r(g) \Phi_{r,v} \right),\sigma  \right).
		\]
		\item If $T \neq 0$, then
			\[
					\calM_{t,v}(\sigma) \circ U_v(s)  \ = \ (\eta_t^j)^* \circ \calM_{j,v}\left(s + \frac{r-j}2 \right) \circ  (\eta_j^r)^* .
			\]
		\item	If $T = 0$, then
		\[
			B_{ \mathbf 0,v}^{j}(g, s, \Phi_{r,v}) \ =\ \left[ \calM_{j,v} \left(s + \frac{r-j}2 \right) \circ (\eta_j^r)^*  \right] \left( r(g)\Phi_{r,v} \right)(e).
		\]
		
	\end{enumerate}
	\begin{proof} Consider the orthogonal case first. Suppose $b_1 \in \Sym_{k'-k}(F_v)$ and $b_2 \in \Mat_{k'-k, k}(F_v)$. 
		For $\ul n(\beta) \in N_k(F_v)$, a direct computation yields
		\begin{equation}
			\ul w_{k'}^{-1}  \,  \ul n \left( \sm{b_1 & b_2 \\ \transpose{ b_2}  &  0  }\right) \, \eta_k^{k'} \left( \ul w_{k}^{-1} \cdot  \ul n(\beta) \right) \ = \ m' \cdot \ul n \left( \sm{ 0 &  \\  & \beta }\right) \cdot \ul w^{-1}_{k'}  \cdot \ul n \left( \sm{b_1 + b_2 \beta \transpose{b_2} & b_2 \\ \transpose{b_2} &  0 }\right) \cdot \eta_k^{k'}( \ul w^{-1}_{k} )
		\end{equation}
		where $m' =  \ul m( \sm{\mathbf 1 & - b_2 \beta \\ & \mathbf 1}) \in M_{k'}(F_v)$.
		Similarly, if $\ul m(\alpha) \in M_k(F_v)$, then 
		\begin{equation}
			\ul w_{k'}^{-1}  \,  \ul n \left( \sm{b_1 & b_2 \\ \transpose{b_2} & 0 }\right) \, \eta_k^{k'} \left( \ul w_{k}^{-1} \cdot \ul m(\alpha) \right) \ = \ \eta_k^{k'}( \ul m(\alpha)) \cdot \ul w^{-1}_{k'} \cdot  \ul n \left( \sm{b_1 & b_2 \alpha \\ \transpose{ \alpha} \transpose{b_2} &0 }\right) \cdot \eta_k^{k'}(\ul w^{-1}_{k} ).
		\end{equation}
		Using these relations and applying an appropriate change of variables in \eqref{eqn:UIntertwiningDef} implies that for any $\Phi \in I_{k', v}(\bbV,s)$, we have the transformation formula
		\begin{equation}
		 [	U_{k,k',v}(s) \Phi] \left( \ul m(\alpha) \ul n(\beta) g \right) \ = \ \chi_{\bbV}( \ul m(\alpha)) \cdot |\alpha|^{s -\frac{k'}2 + k} \cdot  [	U_{k,k',v}(s) \Phi] \left(   g \right);
		\end{equation}
	thus $U_{k,k',v}(s) \Phi \in I_{k, v}(\bbV, s - \frac{k'-k}2)$, proving $(i)$.
	
	To prove $(ii)$, let $\Phi = \Phi_{r,v} \in I_{r, v}(\bbV,s)$ and set
	\begin{equation} \Phi'(s') := (\eta_{j}^r)^* \left( r(g) \Phi(s) \right) \ \in \ I_{j,v}(\bbV, s + (r-j)/2) \end{equation}
	with 
	$s' = s + (j-r)/2 $. Then, identifying $N_k(F_v) \simeq \Sym_k(F_v)$, we may write
	\begin{equation}
	\begin{split}
		W_{\lbT,v}&(e,  U(s)\Phi, \sigma)  \\ 
		&=   \int_{\beta \in \Sym_{t}(F_v)}  \left[ U(s) \Phi \right] \left( \ul w^{-1}_t \cdot \ul n(\beta) \right)  \  \psi_{\lbT}(-\beta) \, d \beta \\
		&=   \int\limits_{\beta_t \in \Sym_{t}(F_v)}  \int\limits_{ \substack{b_1 \in \Sym_{j-t}(F_v)   \\ b_2 \in \Mat_{j-t,t}(F_v)} } \Phi'\left( \ul w^{-1}_j \, \ul n \left( \sm{b_1 & b_2 \\ \transpose b_2 & 0 }\right) \, \eta_t^{j}\left(  \ul w_{t}^{-1} \cdot  \ul w_{t}^{-1} \cdot \ul n(\beta)) \right) ,s' \right)  \, db_1  \,  db_2 \,   \psi_{\lbT}(-\beta) \,  d\beta \\
		&=  \int_{\beta_t \in \Sym_{t}(F_v)}  \int_{ \substack{b_1 \in \Sym_{j-t}(F_v)   \\ b_2 \in \Mat_{j-t,t}(F_v)} } \Phi'\left( \ul w^{-1}_j \, \ul n \left( \sm{b_1 & b_2 \\ b_2^* & \beta }\right), \ s' \right) \ db_1 \  db_2 \ \psi_{\lbT}(-\beta) \ d\beta \\
		&= \int_{b \in \Sym_j(F_v) } \Phi'( \ul w^{-1}_j \cdot \ul n(b), \ s') \cdot \psi_{\lbT}(-\beta) \ db \\ 
		&=  \int_{b \in \Sym_j(F_v) } \Phi \left(\eta_{j}^r \left( \ul w^{-1}_j \cdot \ul n(b) \right) \cdot g, \ s \right)\psi_{T} \left( \sm{\mathbf 0   & \\ & - b} \right)  \, db\
		= \ B^{j}_{T,v}(g, \Phi,s).
	\end{split}
	\end{equation}
 The proofs of $(iii)$ and $(iv)$ are similar, as are the statements in the unitary case.
	\end{proof}
\end{lemma}

\begin{remark} \label{rmk:BjTopAndBottom}
\begin{enumerate}[(i)]
\item	Since $\mathcal{M}_{t,v}(-\sigma) \circ \mathcal{M}_{t,v}(\sigma)$ is given by a meromorphic function in $\sigma$, part (iii) shows that $U_v(s)$ admits meromorphic continuation to $s \in \mathbb{C}$. Setting $U(s) = \otimes_v U_v(s)$ and using the meromorphic continuation of the global intertwining operator $\mathcal{M}_t(\sigma)$ we also conclude that the global operator $U(s)$ admits meromorphic continuation.
\item If $T \neq 0$ and $j= t$  or $j=r$, then  $\calP_{r-t, j-t} = \GL_{r-t}$ and so the sum over $\sfa$ for these terms in \Cref{lem:GlobalEisFCDecomp} is trivial. Noting that $U(s) =  (\eta_t^r)^*  $ when $j = t $, and $U(s) =  U_{t,r}(s)$ when $j = r$, 
	%	 On the other hand, in these cases the operator $U(s)$ is given by
	%	\[
	%		U(s) \ = \ \begin{cases} (\eta_t^n)^*  & \text{ if } j = t \\ U_{t,n}(s) & \text{ if } j = n,\end{cases}
	%	\] 
	%	cf. \ \Cref{lem:BjGlobalWhittakerReln}. Thus, 
	the corresponding summands in \Cref{lem:GlobalEisFCDecomp} are
	\[
	B^{ \, t}_{T}( g,\Phi,s)  =    W_{\lbT} \left( e, \left[ (\eta_t^r)^*   \circ r(  g) \right]\Phi ,\ s + \frac{r-t}{2}\right) 
	\]
	and
	\[
		B^{ \, r}_{T}(  g,  \Phi, s ) =   W_{\lbT} \left( e,  \left[U_{t,r}(s)  \circ r(g) \right]\Phi, s - \frac{r-t}{2} \right) .
	\]
	In particular, when $T$ is non-degenerate we recover the expression $E_T(g,\Phi,s) = \prod_v W_{T,v}(g_v,\Phi_v,s)$.
\end{enumerate}
\end{remark}

%By the above remark, for $\Phi_r = \otimes_v \Phi_{r,v} \in I_r(\bbV,s)$ and $g \in \bfG_r'(\mathbb{A})$ we have 
In particular, the lemma implies the global identity 
\begin{equation} \label{eq:B_T^j_relate_to_W_tau_U(s)}
	B_{T}^{j}(g,  \Phi_{r},s)   = W_{\lbT}\left(e,  U(s) \left( r(g) \Phi_{r} \right),\sigma  \right).
\end{equation}
%We will also use the following invariance property: given $x \in \mathrm{Mat}_{r-t,t}(\mathbb{A}_{\bfk})$, let  \todo{bit less clutterd this way}
%\begin{equation}
%\theta= \left(\sm{ 1_{r-t} & x \\ 0 & 1_t } \right) \in \mathrm{GL}_r(\mathbb{A}_{\bfk}).
%\end{equation}
%Then, for any $j$ with $t \leq j \leq r$, we have 
%\begin{equation} \label{eq:invariance_U(s)_theta} 
%U(s)(r(\underline{m}(\theta(x))\Phi) = U(s) \Phi;
%\end{equation}
%the proof follows from a simple change of variables in the definition of $U(s)$.
%\todo[inline]{can you explain this, I don't see why this is true; in any case the only thing we need is the corresponding transformation rule for $B_T^j$, which can be seen most easily from the definition (not entirely trivially I might add)). }

We will also need the following invariance property: suppose $x  \in \Mat_{r-t,t}(\bbA_k)$ and let
	\begin{equation}
	\theta= \left(\sm{ 1_{r-t} & x \\ 0 & 1_t } \right) \in \mathrm{GL}_r(\mathbb{A}_{\bfk}).
	\end{equation}
Then a direct computation using \eqref{eqn: Bj Global Def} yields the transformation formula
\begin{equation} \label{eq:invariance_B_theta} 
B_T^j \left( \ul m(\theta) g , \Phi, s \right) = B_T^j \left( g,\Phi, s \right).
\end{equation}

	\subsubsection{} Our next step is to relate the individual terms in \Cref{lem:GlobalEisFCDecomp} to Eisenstein series on $\GL_{r-t}$, generalizing the discussion in \cite{KudlaRallis1}. 
%	We begin with a little notation. 
	Consider first the orthogonal case and let $\GL'_{r-t}(\bbA)$ denote the metaplectic double cover of $\GL_{r-t}(\bbA_F)$: as a set, $ \GL'_{r-t}( \bbA) \ = \ \GL_{r-t}(\bbA) \times \{ \pm 1 \}$, with multiplication
	\begin{equation}
			(\sfa_1, \epsilon_1) \cdot ( \sfa_2, \epsilon_2) \ =\ \left( \sfa_1 \sfa_2, \, (\det \sfa_1, \det \sfa_2)_{\bbA} \, \epsilon_1 \, \epsilon_2 \right)
	\end{equation}
It follows from the formulas in \cite[\S 5]{Rao} that there is an embedding
\begin{equation}
 \iota \colon	\GL'_{r-t} (\bbA) \hookrightarrow \sfM_{r}( \bbA),  \qquad  \text{given by}  \qquad (a, \epsilon)  \mapsto \left( m \left( \sm{ a & \\ & \mathbf 1_t } \right), \,  \epsilon  \right)
\end{equation}

Abusing notation, let $\chi_{\bbV} \colon \GL'_{r-t}( \bbA) \to \bbC$ denote the character
$\chi_{\bbV}(\sfa) = \chi_{\bbV}(\iota(\sfa))$, where the latter $\chi_{\bbV}$ is defined in \eqref{eq:def_chi_bbV_orthogonal}. 

For the moment, fix an integer $j$ with $t \leq j \leq r$; to lighten notation,  we write
\begin{equation} \label{eqn:parabolic GLr}
\calP = \calP_{r-t, j-t} \ =\  \left\{ \begin{pmatrix} \sfp_1 & * \\ \mathbf 0_{j-t, r-j} &  \sfp_2  \end{pmatrix} \ \Big| \  \, \sfp_1 \in \GL_{r-j}, \, \sfp_2 \in \GL_{j-t} \right\} \ \subset \ \GL_{r-t}.
\end{equation}
	Let $\calP'_{\bbA}$ denote the  inverse image of $\calP(\bbA)$ with respect to the projection $\GL'_{r-t}(\bbA) \to \GL_{r-t}(\bbA)$. Consider the (smooth normalized) induced representation 
\begin{equation}
\widetilde I^{ \, j}_{\calP}(\bbV, s ) \ := \ \mathrm{Ind}_{\calP'_{\bbA}}^{\GL'_{r-t}(\bbA)} \left( \chi \,  | \det(\sfp_1)|^{s + \frac{r+1}2 - \frac{j-t}2}_{\bbA} \,    | \det(\sfp_2)|^{-s+\frac{j+1}2}_{\bbA} \right);
\end{equation}
concretely, $\widetilde I^{ \, j}_{\calP,v}(\bbV, s )$  consists of smooth functions $\Psi(\cdot, s) \colon \GL'_{r-t}( \bbA) \to \bbC$ such that 
\begin{equation} \label{eqn:MetaplecticGLnEisCentralChar}
\Psi( \sfp \sfa, s) \ = \ \chi_{\bbV}(\sfp)  \,  | \det(\sfp_1)|^{s + \frac{r+1}2  }_{\bbA} \,    | \det(\sfp_2)|^{-s- \frac{r-1}2 +j}_{\bbA} \, \Psi(\sfa)
\end{equation}
for all $ \sfa   \in \GL'_{r-t}(\bbA)$ and $\sfp = [ ( \sm{ \sfp_1 & * \\ \mathbf 0  &  \sfp_2} ), \epsilon] \in \mathcal P'_\mathbb{A}$. 

In the unitary case, we may be more direct: let $\widetilde I^j_{\calP}(\bbV,s)$ denote the space of smooth functions $\Psi(\cdot, s) \colon \GL_{r-t}(\bbA_E)\to \bbC$ such that 
\begin{equation}  \label{eqn:UnitaryGLnEisCentralChar}
\Psi( \sfp \sfa, s) \ = \ \chi_{\bbV}(\sfp)  \,  | \det(\sfp_1)|^{s + \frac{r}2  }_{\bbA_E} \,    | \det(\sfp_2)|^{-s- \frac{r}2 +j}_{\bbA_E} \, \Psi(\sfa), \qquad \text{ for all } \sfp = \left( \sm{\sfp_1 & * \\ & \sfp_2} \right)\in \calP(\bbA_E).
\end{equation}
We also write $\iota \colon \GL_{r-t}(\bbA_E) \to \bfG'_r(\bbA)$ for the embedding $ a \mapsto m \left( \sm{a & \\ & \mathbf 1_t } \right)$.

Finally, for a finite place $v$ of $F$, let 
\begin{equation} \widetilde \bfK_v = \GL_{r-t}(\calO_{\bfk, v}) \subset \GL_{r-t}(\bfk_v); \end{equation} 
in the orthogonal case we identify $\widetilde \bfK_v$ as a subset of $\GL'_{r-t}(F_v)$ via the map $ k \mapsto [k,1]$. For a real place $v$, set $\widetilde \bfK_v = \mathrm{O}(r)$ or $\widetilde \bfK_v  = \mathrm{U}(r)$, and finally define $\widetilde \bfK = \prod \widetilde \bfK_v$.

\subsubsection{}
To formulate the connection between $B_T^{j}(g,\Phi,s)$ and  Eisenstein series on $\mathrm{GL}_{r-t}$, assume for the moment that $t= \rank(T) > 0$, and fix an integer $j$ with $t \leq j\leq r$.  Let $\Phi_r = \otimes_v \Phi_{r,v} \in I_r(\bbV,s)$ be a standard section such that $\Phi_{r,v} = \Phi_r^l$ at each archimedean place $v$.

Suppose that $g \in \bfG'_r(\bbA)$ is of the form
\begin{equation}
g=  \iota(g' ) \, \eta_t^r(g''),
\end{equation}
where $g'' \in \bfG_t'(\bbA)$ and $g' \in \mathrm{GL}_{r-t}'(\bbA)$  or $g' \in \mathrm{GL}_{r-t}(\bbA_E)$ in the orthogonal or unitary cases, respectively. As a function of $g''$, the expression
\begin{equation} \label{eq:proof_asymptotics_eq_1}
U(s) \left( r(g) \Phi_r \right) (e) \ = \ U(s) \left( r ( \iota(g')) \Phi_r \right) ( g'') 
\end{equation}
defines an element of $I_t(\bbV,\sigma)$ by \Cref{lem:BjGlobalWhittakerReln}, and a change of variables in the definition shows that it defines an element of $\widetilde I^{ \, j}_{\calP}(\bbV, s )$ as a function of $g'$.  In particular, the function
\begin{equation}
B_{T}^j \left( g, \, \Phi, \, s\right)  = B_{T}^j \left( \iota(g')  , \, r(g'') \Phi_v, \, s \right)
\end{equation}
is in $\widetilde I_{\calP}(\bbV,s)$ when viewed as a function of $g' $ with $g''$ fixed.

It follows from multiplicity one for $K'_t$-types in $I_t(\bbV_v,\sigma)$ for archimedean $v$ that, as a function of $g_\infty''$, the expression \eqref{eq:proof_asymptotics_eq_1} is proportional to $\prod_{v|\infty} \Phi_t^l(g_v'',\sigma)$; evaluating at $g_\infty''=e$ to determine the constant of proportionality, we find that
\begin{equation} \label{eq:U(s)_Phi_r}
U(s)\left( r(g)\Phi_r\right) (e) = \widetilde{\Phi}(g',g_f'',s) \prod_{v|\infty} \Phi_t^l(g_v'',\sigma),
\end{equation}
where $\widetilde{\Phi}(g',g_f'',s) = U(s)\left( r( \iota(g' ) \, \eta_t^r(g_f''))\Phi_r\right) (e)$. Note that $\tilde{\Phi}(g',g_f'',s)$ is meromorphic in $s$ and, as a function of $g'$, is $\widetilde \bfK_f$-finite and $\widetilde \bfK_\infty$-invariant. Moreover
$\tilde{\Phi}(\cdot,g_f'',s) \in \widetilde I^{ \, j}_{\calP}(\bbV, s )$
for fixed $g''_f$ and so we can write
\begin{equation} \label{eq:phi_GL_r-t_standard_linear_combination}
\tilde{\Phi}(g',g_f'',s) = \alpha_1(s) \, \tilde{\Phi}_1(g',s) \Phi''_1(g_f'',\sigma) + \cdots  + \alpha_N(s)  \, \tilde{\Phi}_N(g',s)\Phi''_N(g_f'',\sigma),
\end{equation}
where $\tilde{\Phi}_i(\cdot,s)$ are $\widetilde \bfK_\infty$-invariant standard sections of  $\widetilde I^{ \, j}_{\calP}(\bbV, s )$, $\Phi''_i(\cdot,\sigma)=\otimes_{v \nmid \infty} \Phi''_{i,v}(\cdot,\sigma)$ are standard sections of $I_t(\bbV(\mathbb{A}_f),\sigma)$, and the coefficients $\alpha_i(s)$ meromorphic in $s$.

Substituting \eqref{eq:U(s)_Phi_r} and \eqref{eq:phi_GL_r-t_standard_linear_combination} in \eqref{eq:B_T^j_relate_to_W_tau_U(s)} we conclude that, for $g$ as above with $g_f''=1$,
\begin{equation} \label{eq:GlobalBjEis}
\sum_{ \sfa \in \mathcal P(\bfk) \backslash \GL_{r-t}(\bfk)}   B^{j}_{T} \left(  m    (\sm{ \sfa & \\ & \mathbf 1_{t} }) g , \Phi,s \right) =  \, \left(\sum_{1 \leq i \leq N} \gamma_i(s) \, \mathcal{G}^j(g',\tilde{\Phi}_i,s) \right) \,  \prod_{v | \infty} W_{\lbT,v}(g_v'',\Phi_t^l,\sigma),
\end{equation}
where
\begin{equation}
\mathcal{G}^j(g',\tilde{\Phi}_i,s) := \sum_{ \sfa \in \mathcal P(\bfk) \backslash \GL_{r-t}(\bfk)} \tilde{\Phi}_i(\sfa g',s)
\end{equation}
is an Eisenstein series on $\mathrm{GL}_{r-t}'(\mathbb{A})$ (case 1) or $\mathrm{GL}_{r-t}(\mathbb{A}_E)$ (case 2), $\mathcal P =\mathcal P_{r-t,j-t}$ is as in \eqref{eqn:parabolic GLr}, 
and 
\begin{equation}
\gamma_i(s) \ := \alpha_i(s) \prod_{v \nmid \infty} \, W_{S, v}\left( e, \Phi''_{i,v},\sigma \right), \qquad 1 \leq i \leq N,
\end{equation}
is meromorphic in $s$.

We turn now to the case $T= 0$. If $1 \leq j \leq r-1$, then the same argument as above shows that, for $g' \in \mathrm{GL}_{r}'(\mathbb{A})$ (orthogonal case) or $g' \in \mathrm{GL}_{r}(\mathbb{A}_E)$ (unitary case), we can write
\begin{equation} \label{eq:GlobalBjEisT0}
\sum_{ \sfa \in \mathcal P_{r,j}(\bfk) \backslash \GL_{r}(\bfk)}   B^{j}_{0} \left(    \ul m    (\sfa) \iota( g') , \Phi,s \right) = \sum_{1 \leq i \leq N} \gamma_i(s) \, \mathcal{G}^j( g',\tilde{\Phi}_i,s) 
\end{equation}
for some standard $\tilde{\bfK}_\infty$-invariant sections $\tilde{\Phi}_i(s) \in \tilde{I}_{\mathcal{P}}^j(\bbV,s)$ and meromorphic functions $\gamma_i(s)$. In addition, \Cref{lem:BjGlobalWhittakerReln} shows that
\begin{equation} \label{eq:GlobalBjEisT0_bis}
\begin{split}
B_0^0(\iota(g'),\Phi_r,s) &= \Phi_r(\iota(g'),s), \\
B_0^r(\iota(g'),\Phi_r,s) &= \mathcal{M}(s)\Phi(\iota(g')) = W_0(\iota(g'),\Phi_r,s).
\end{split}
\end{equation}

\subsubsection{}
We can now generalize \Cref{lem:nonDegKappaBeta} to arbitrary matrices $T$; we consider the case $T \neq 0$ first.

\begin{definition} \label{def:KappaTilde}
Assume that $\Phi_f \in I_r(\bbV(\mathbb{A}_f),s)$ satisfies \eqref{eq:Siegel_section_conditions} and $T = (\sm{0 & \\ & S}) \neq 0$ with $S$ non-degenerate, cf.\ \eqref{eq:T_block_diagonal_form}. Let $\Phi_f', \Phi_f'' \in I_t(\bbV(\mathbb{A}_f), \sigma)$ be given by
	\[
			\Phi'_f(\sigma)  =  (\eta_t^r)^* \Phi_f(s) \quad \text{and} \quad 	\Phi''_f(\sigma)   =    (\eta_t^r)^*(\calM_r (-s) \Phi_f(-s)),
	\]
	where $(\eta_t^r)^*$ is defined as in \eqref{eqn:etaStarDef}, $\calM_r(s)$ is the standard intertwining operator  \eqref{eqn:localIntOp}, and $\sigma = s + \frac{r-t}2$.
	
	Define constants $ \beta(T, \Phi_f)$ and $ \kappa(T, \Phi_f)$ as follows:
	\begin{enumerate}[(i)]
		\item If $s_0(r) > 0$ set $ \beta(T, \Phi_f) = \beta(\lbT, \Phi_f')$ and $ \kappa(T, \Phi_f) = \kappa(\lbT, \Phi_f')$, where these  quantities are defined for the non-degenerate matrix $S$ as in \Cref{lem:nonDegKappaBeta}.
		\item Suppose $s_0(r) = 0$, so that $r=m-1$ (resp.\ $r = m$)  in the orthogonal (resp.\ unitary) case, and let 
		\begin{equation} \label{eqn:archIntertwiningCentralPtFactor}
					 \ttd(s) \ = \  2^{- \frac{r}2(\frac{\iota m}{2}-1) - \iota s} (2\pi i)^{\iota m r /2} \frac{\Gamma_r(\iota s)}{\Gamma_r \left( \frac{\iota }2(s+m)  \right) \Gamma_r\left( \frac{\iota s}2 \right) }.
		\end{equation} 
		Note that $\ttd(s)$ is holomorphic and non-vanishing at $s=0$. Define
			\[
			\beta(T,\Phi_f)  \  = \ 2 \beta(S,\Phi_f')
			\]
		and
			\[
			\kappa (T, \Phi_f) \ :=\ \kappa(\lbT, \Phi_f') \ - \  \left[ d \cdot  \frac{\ttd'(0)}{\ttd(0)}\beta(S, \Phi_f') + \ttd(0)^d \kappa(\lbT, \Phi_f'') \right]. 
			\]
				\hfill $\diamond$
	\end{enumerate}
\end{definition}	

 Let us now define constants $\kappa(T,\Phi_f)$ and $\beta(T,\Phi_f)$ for an arbitrary (i.e. not necessarily block diagonal) matrix $T $ of rank $t>0$. 
 
Let $\gamma \in \mathrm{SL}_r(\bfk)$ such that
\begin{equation} \label{eqn:choice of gamma}
T[\gamma] := {^t}\overline{\gamma}^{-1}T \gamma^{-1} = \begin{pmatrix} 0_{r-t} & \\ & \lbT \end{pmatrix},
\end{equation}
where $\lbT $ is non-degenerate and
\begin{equation} \label{eqn:gamma normalization}
\det \lbT =  {\det}'  \, T := \text{ product of non-zero eigenvalues of } T;
\end{equation} 
such $\gamma$ always exists for fixed $T$. Define
\begin{equation} \label{eq:def_beta_kappa_arbitrary_T}
\begin{split}
\beta(T,\Phi_f) &:=\beta(T[\gamma],\underline{m}(\gamma)\Phi_f)  \\
\kappa(T,\Phi_f) &:=  \kappa(T[\gamma],\underline{m}(\gamma)\Phi_f).
\end{split}
\end{equation}
That $\beta(T,\Phi_f)$ and $\kappa(T,\Phi_f)$ are independent of $\gamma$ (and hence are well-defined) follows from a direct computation, or alternatively from the invariance property 
\begin{equation} \label{eq:invariance_C_T}
C_T(\bfy,\Phi_f,s) = C_{T[\gamma]}(\bfy[{^t}\overline{\gamma}^{-1}],\underline{m}(\gamma)\Phi_f,s),
\end{equation}
(cf.\ \eqref{eq:equivariance_E_T_gamma}) and the next proposition, which shows that the constants $\beta(T,\Phi_f)$ and $\kappa(T,\Phi_f)$ determine the asymptotic behaviour of the derivative $\tfrac{d}{d s} C_T(\bfy,\Phi_f,s_0)$.

\begin{proposition} \label{prop:EisGrowthEstimate} 
Assume that $\Phi_f$ satisfies \eqref{eq:Siegel_section_conditions} and let $\bfy=(y_v)_{v|\infty} \in \mathrm{Sym}_r(F_\mathbb{R})_{\gg 0}$ (resp. $\mathrm{Her}(E_\mathbb{R})_{\gg 0}$). Write $\det' A$ for the product of non-zero eigenvalues of a square matrix $A$, and let 
	\[
	F(\bfy) := \left. \frac{d}{d s} C_T(\bfy,\Phi_f,s) \right|_{s = s_0(r)} -  \kappa(T, \Phi_f)  -   \frac{\iota}2 \beta(T, \Phi_f) \ \sum_{v|\infty} \log \left(\frac{\det' \sigma_v(T) \cdot \det y_v}{\det' \left( \sigma_v(T)y_v\right)} \right).
	\] 
	Then for every fixed $\bfy$, we have 
	\begin{equation} \label{eqn:CondRadialDecay}
		\lim_{\lambda \to \infty} F(\lambda \bfy) = 0 \qquad \text{ and } \qquad \frac{\partial}{\partial \lambda} F(\lambda \bfy) = O(\lambda^{-1-C})
	\end{equation} 
   as $\lambda \to \infty$, for some $C>0$ .
	
\begin{proof} 
	We begin with a few preliminary reductions. Choose $\gamma \in \SL_r(\bfk)$ satisfying \cref{eqn:gamma normalization,eqn:choice of gamma}, and note that the quantities $\frac{\det' \sigma_v(T) \cdot \det y_v}{\det'( \sigma_v(T)y_v)}$ are unchanged upon simultaneously replacing $T$ with $T[\gamma]$ and $\bfy $ with $\gamma \, \bfy \, \transpose{\overline \gamma}$. Combining this observation with  \eqref{eq:invariance_C_T},  it suffices to prove the proposition for $T$ of the form $T = (\sm{0 & \\ & S})$ with $S$ non-degenerate of rank $t = \rank (T)$.

	In this case, \Cref{lem:GlobalEisFCDecomp} gives a decomposition $C_T(\bfy, \Phi_f,s)  = \sum_{j =t}^r C^j_T(\bfy,\Phi_f,s)$,
	where
	\begin{equation} \label{eqn:CjT main theorem setup}
	C^j_T(\bfy, \Phi_f,s)  =  \prod_{v| \infty}\det(y_v)^{-\iota m/4} \sum_{ \sfa \in \mathcal P_{r-t,j-t} \backslash \GL_{r-t}(\bfk)}   B^{\, j}_T \left( \ul m(\sm{ \sfa & \\ & \mathbf 1_{t} })g'_{\bfy} ,\Phi,s \right) \, q^{-T}.
	\end{equation}
	
		Next,  we may fix an element $\bm \theta  = \left( \sm{1_{r-t} & * \\ & 1_t } \right) \in \SL_{r}(\bfk_{\bbR}) \simeq \SL_r(\bbK)^d$ such that
\begin{equation}
\bm \theta \cdot \bfy \cdot \transpose{\overline {\bm \theta}} \ =\  \begin{pmatrix} \bfy' & \\ & \bfy'' \end{pmatrix}	
\end{equation}
is block diagonal, where $\bfy'$ and $\bfy''$ are totally positive definite of rank $r-t$ and $t$, respectively. Let $\widetilde{\bm \theta} = (\bm \theta, e, \dots) \in \SL_r(\bbA)$; then, using the invariance property \eqref{eq:invariance_B_theta}  and the relation $g'_{ \bm \theta \bfy \transpose{\overline{\bm \theta}}} = \ul m(\widetilde{\bm \theta}) g'_{\bfy}$, we have
\begin{equation}
B^{\, j}_T \left(  \ul m (\sm{ \sfa & \\ & \mathbf 1_{t} }) g'_{\bm \theta \bfy \transpose{\overline{\bm \theta}}} ,\Phi,s \right)  =
%=  B^{\, j}_T \left(  \ul m (\widetilde{\bm \theta'})m    (\sm{ \sfa & \\ & \mathbf 1_{t} })g'_{\bfy} ,\Phi,s \right) = 
B^{\, j}_T \left(  \ul m    (\sm{ \sfa & \\ & \mathbf 1_{t} })g'_{\bfy} ,\Phi,s \right).
\end{equation}
%where $\widetilde{\bm \theta '} = \left( \sm{ \sfa & \\ & 1 } \right) \widetilde{\bm \theta} \cdot \left( \sm{ \sfa^{-1} & \\ & 1 } \right)  =\left( \sm{ 1& * \\ & 1 } \right). $
Moreover, for $T = ( \sm{0 & \\ & S })$, the quantities $\det' (\sigma_v(T) y_v)$ are unchanged upon replacing $\bfy $ by $\bm \theta \bfy  \transpose{\overline{\bm \theta}}$, as are the determinants $\det y_v$.

Thus, we may reduce to the case where $T= ( \sm{0 & \\ & S })$ and $\bfy =  ( \sm{\bfy' & \\ & \bfy''})$ are both block diagonal, so that
\begin{equation}
\frac{\det' \sigma_v(T) \cdot \det y_v}{\det' \sigma_v(T)y_{v}}=\frac{\det \sigma_v(\lbT) \cdot \det y_v}{\det \sigma_v(S) \det y''_{v}} = \det y'_v.
\end{equation}
Our approach will be to show that the value at $s_0(r)$ of the derivative of the expression \eqref{eqn:CjT main theorem setup}, after possibly subtracting off a constant and a multiple of  $\sum \log(\det y'_v)$, satisfies the condition \eqref{eqn:CondRadialDecay}.

%
%By \eqref{eq:invariance_C_T}, we may assume that $T=\left(\begin{smallmatrix} 0_{r-t} & \\ & \lbT \end{smallmatrix}\right)$, with $S $ non-degenerate.  As in  \Cref{lem:GlobalEisFCDecomp}, there is a decomposition $C_T(\bfy, \Phi_f,s)  = \sum_{j =t}^r C^j_T(\bfy,\Phi_f,s)$,
%where
%\begin{equation}
%C^j_T(\bfy, \Phi_f,s)  =  \prod_{v| \infty}\det(y_v)^{-\iota m/4} \sum_{ \sfa \in \mathcal P_{n-t,j-t} \backslash \GL_{r-t}(\bfk)}   B^{\, j}_T \left(m(\sm{ \sfa & \\ & \mathbf 1_{t} })g'_{\bfy} ,\Phi,s \right) \, q^{-T}.
%\end{equation}
%Given $\bfx \in \Mat_{r-t,t}(\mathbb{R})^d$, note that
%\begin{equation}
%B^{\, j}_T \left( m (\sm{ \sfa & \\ & \mathbf 1_{t} })\underline{m}(\sm{ 1_{r-t} & \bfx \\ & \mathbf 1_{t} })g'_{\bfy} ,\Phi,s \right) = B^{\, j}_T \left(  m    (\sm{ \sfa & \\ & \mathbf 1_{t} })g'_{\bfy} ,\Phi,s \right)
%\end{equation}
%by \Cref{lem:BjGlobalWhittakerReln}.(ii) and the invariance property \eqref{lem:BjGlobalWhittakerReln}. Choosing $\bfx$ appropriately we assume that $\bfy$ is block diagonal, of the form $\bfy = \left( \sm{ \bfY' & 0 \\ 0 & \bfy'_t} \right)$, so that 
%\begin{equation}
%\frac{\det' \sigma_v(T) \cdot \det y_v}{\det' \sigma_v(T)y_{v}}=\frac{\det \sigma_v(\lbT) \cdot \det y_v}{\det \sigma_v(S) \det y'_{t,v}} = \det Y'_v.
%\end{equation}
%Our approach will be to show that the value at $s_0(r)$ of the derivative of each of these functions, after possibly subtracting off a constant and a multiple of  $\sum \log(\det Y'_v)$, satisfies the condition \eqref{eqn:CondRadialDecay}.
Fix an element  $\bm\mu   \in \prod_{v | \infty} \GL_{r-t}(\bfk_v)^d$ of totally positive determinant with $\bfy' =\bm \mu \cdot \transpose{\overline{\bm \mu}}$, and let $\tilde{ 	\bm \mu} =  \left((\bm \mu, \Id, \dots),1\right)  \in  \GL'_{r-t, \bbA_F}$ in the orthogonal case, and $\tilde{ 	\bm \mu} =( \bm \mu, \Id, \dots) \in \GL_{r-t}(\bbA_E)$ in the unitary case. Then 
	\begin{equation}
	 g'_{\bfy} \ = \ \iota\left(\tilde{\bm \mu} \right) \cdot \eta_t^r \left( g'_{\bfy''} \right).
	\end{equation}
	
Taking into account our normalizations, cf. \eqref{eqn:ArchWhittakerClassicalNormalization}, we may use \eqref{eq:GlobalBjEis} to write
\begin{equation}
		C_T^j(\bfy, \Phi_f,s) =  \left( \sum \gamma_i(s) \, \calG^j(\tilde{\bm \mu} , s, \tilde{\Phi}_i)  \right)  \prod_{v|\infty} \det(y'_v)^{-\iota m/4} \archW_{\sigma_v( \lbT)}(y''_v, \sigma) 
\end{equation}
	for some meromorphic functions $\gamma_i(s)$ and standard sections $\tilde{\Phi}_i \in \widetilde I_{\calP}^j(\bbV, s)$, where $\sigma = \sigma(s) = s + \frac{r+t}{2} - j$, 
%	\begin{itemize} 
%			\item $\sigma = \sigma(s) = s + \frac{r+t}{2} - j$; and
%%		\item $c(T, \sigma) = \prod_{v_i \neq v_1} W_{\sigma_i(\lbT)}(\Id, \sigma) \,  e^{ 2 \pi \, \tr{( \sigma_i(\lbT))}}$ is a holomorphic function independent of $y$; 
%%	 fix $\mu \in \GL_{n-t}(\bbR)$ such that $Y' = \mu \transpose \mu$, and let
%	\item $\tilde{ 	\bm \mu} =  \left((\bm \mu, \Id, \dots),1\right)  \in  \GL'_{r-t, \bbA_F}$ in the orthogonal case, and $\tilde{ 	\bm \mu} =( \bm \mu, \Id, \dots) \in \GL_{r-t}(\bbA_E)$ in the unitary case, for some fixed $\bm\mu   \in \prod_{v | \infty} \GL_{r-t}(\bfk_v)^d$ of totally positive determinant with $\bfy' =\bm \mu \cdot \transpose{\overline{\bm \mu}}$.
%
%	\end{itemize}

	Now for fixed $\bfy$ and a parameter $\lambda > 0$, replacing $\bfy$ by $\lambda \bfy$ corresponds to replacing $\bfy''$ by $\lambda \, \bfy''$, and replacing $\bm \mu$ by $\sqrt{\lambda} \, \bm \mu$. 
		Using the transformation formulas \eqref{eqn:MetaplecticGLnEisCentralChar} or \eqref{eqn:UnitaryGLnEisCentralChar} to determine the central character of the Eisenstein series $\calG^j$ in the orthogonal or unitary case, respectively, a short computation yields 
	\begin{equation}  \label{eqn:GLnCentralCharacter}
		\calG^j( \sqrt{\lambda} \tilde{\bm \mu} , \tilde{\Phi},s) \ = \ \lambda^{\frac{\iota d}2(r+t - 2j)(s - s_0(r)) + \frac{\iota d(r-t)m}{4} \ - \  \iota d (j-t) s_0(j)}  \, \calG^{j}( \tilde{\bm \mu}, \tilde{\Phi},s) 
	\end{equation}
	for any section $\tilde{\Phi}$. Since each $\calG^{j}(\tilde{\bm \mu}, \tilde{\Phi}_i,s)$ is itself meromorphic, it follows that for fixed $\bfy$, we have
\begin{equation} \label{eq:calEis_Tj_lambda_multiples}
		C_T^j(\lambda \bfy,\Phi_f,s) = f(\bfy, s) \, \lambda^{\frac{\iota d}2(r+t - 2j) (s - s_0(r)) - \iota  d(j-t)s_0(j)} \prod_{v|\infty}  \archW_{\sigma_v( \lbT)}(\lambda y''_{v}, \sigma) 
\end{equation}
	for some meromorphic function $f(\bfy,s)$ independent of $\lambda$. A priori, we do not know whether $f(\bfy,s)$ has a pole at $s= s_0(r)$, though one could imagine that an analysis along the lines of  \cite{KudlaRallis1, SweetThesis} could be used to determine its order. In any case, let
	\begin{equation} \mathcal A_1^j( \bfy, \Phi_f)  :=  \mathcal A_1( C_T^j( \bfy,\Phi_f,s)  )  \end{equation}
	denote the coefficient of $(s-s_0(r))$  in the Laurent expansion at $s=s_0(r)$, so that
	\begin{equation}
 		\left.	  \frac{\partial}{\partial s}	C_T (\lambda \bfy,\Phi_f,s) \right|_{s= s_0(r)} \ =\  \sum_{j=t}^r \mathcal A_1^{j}(\lambda \bfy, \Phi_f).
	\end{equation}
	Then $ \mathcal A_1^j(\lambda \bfy, \Phi_f)$ can be written as a sum of terms of the form
	\begin{equation} \label{eqn:DerivFCEisJthTerm}
			a(\bfy) \, (\log\lambda)^k \, \lambda^{- \iota d (j-t)s_0(j)} \,  \left. \frac{\partial^{k'}}{\partial s^{k'}}  \left(	  \prod_{v|\infty}   \archW_{\sigma_v( \lbT)}\left(\lambda y''_{v}, \sigma(s) \right)  \right) \right|_{s= s_0(r )}
	\end{equation}
	for some Laurent coefficients $a(\bfy)$ of $f(\bfy,s)$ at $s = s_0(r)$, and integers $k, k' \geq 0$. 
	
	First, consider the case where $\lbT$ is not totally positive definite. If $v$ is a place such that $\sigma_v(\lbT)$ is not positive, then, by \Cref{prop:NonDegArchWhittakerNonPos}, the derivatives of $ \archW_{\sigma_v(\lbT)}(\lambda y''_{v}, \sigma) $ in $s$ and $\lambda$ are all of exponential decay as $\lambda \to \infty$. It follows that  $ \mathcal A_1^j(\lambda \bfy , \Phi_f)$ satisfies the properties \eqref{eqn:CondRadialDecay} for each $t \leq j \leq r$. On the other hand, for non-positive $\lbT$ the constants $\kappa(\lbT,\Phi_f)$ and $\beta(\lbT, \Phi_f)$ are both zero by definition; this proves the proposition in this case.

Next, suppose $\lbT$ is totally positive definite. Assume further that $j>t$ and $s_0(j) \neq 0$, which implies that $-\iota d(j-t) s_0(j) \leq - 1/2$; this is the exponent of $\lambda$ in \eqref{eqn:DerivFCEisJthTerm}. For fixed $\bfy$,  \Cref{prop:NonDegenWhittakerEstimates}(i) implies that  each $\archW_{\sigma_v(\lbT)}(\lambda y''_{v}, \sigma) $, along with  its derivatives in $s$, are bounded as $\lambda \to \infty$. In light of \eqref{eqn:DerivFCEisJthTerm}, this implies
	\begin{equation}
		\lim_{\lambda \to \infty} \mathcal A^j_1(\lambda \bfy, \Phi_f) \ = \ 0
	\end{equation}
	for such $j$. 
	Similarly, differentiating \eqref{eqn:DerivFCEisJthTerm} with respect to $\lambda$ and applying \Cref{prop:NonDegenWhittakerEstimates}(ii) yields the estimate
	\begin{equation}
		\frac{\partial}{\partial \lambda}\mathcal A^j_1(\lambda \bfy, \Phi_f) \ =  \ O(\lambda^{-1-C})
	\end{equation}
	for some $C>0$. In other words, for $j> t$ with $s_0(j) \neq 0$, the term $\mathcal A^j_1( \bfy, \Phi_f) $ satisfies the condition \eqref{eqn:CondRadialDecay}.
	
Now suppose that $\lbT$ is totally positive definite and $j = t$; the corresponding term can be written more concretely, using \Cref{rmk:BjTopAndBottom}, as
\begin{equation}
C_T^{\, t}(\bfy,\Phi_f,s)  = \prod_{v| \infty}  \det(y_v)^{-\iota m/4} \, W_{\lbT} \left(e, \left[(\eta^r_t)^* \circ r( g'_{\bfy}) \right] \Phi,  \sigma\right)  \, q^{-T} 
\end{equation}
where $\sigma = s + \frac{r-t}{2}$ and $(\eta_t^r)^* \colon I_r(\bbV, s) \to I_t(\bbV,\sigma)$ is the map defined in \eqref{eqn:etaStarDef}. For an archimedean place $v$, note that multiplicty one for $K'_{t,v}$-types immediately implies the relation $(\eta_t^r)^* \Phi_{r}^l = \Phi_t^l$. More generally, for block diagonal $\bfy$ as above and any $h_v \in \bfG_{t,v}'$ we have
%\begin{align} \label{eq:eta_t_r_computation_1}
%(\eta_t^r)^*(r(g'_{y_v})\Phi_{r,v}^l)(h_v,\sigma) = (\det Y_v')^{\frac{\iota}{2}(s+\rho_r)} \Phi_{t,v}^l(g'_{y'_{t,v}},s) 
%\end{align}
\begin{equation} \label{eq:eta_t_r_computation_1}
\begin{aligned}
(\eta_t^r)^*(r(g'_{y_v})\Phi_{r}^l)(h_v,\sigma) &= \Phi_{r}^l \left(  \eta_t^r(h_v) \, \iota(\tilde{\mu}_v) \, \eta_t^r (g'_{y''_v}), \, s \right) \\
&= \Phi_{r}^l \left( \iota(\tilde{\mu}_v) \,  \eta_t^r (h_v \, g'_{y''_v}), \, s \right)  \\
&= \det(y'_v)^{\frac{\iota}{2}(s+\rho_r)} \, \Phi_{r}^l \left(   \eta_t^r ( h_v \, g'_{y''_v}), \, s \right)  \\
&=  \det(y'_v)^{\frac{\iota}{2}(s+\rho_r)} \, \Phi_t^l\left(h_v \, g'_{y''_v}, \, \sigma\right).
\end{aligned}
\end{equation}
Hence  
\begin{equation}  \label{eqn:CTt}
C_T^{\, t}(\bfy, \Phi_f,s)  = \prod_{v|\infty} \left(\det y_v' \right)^{\frac{\iota}2(s - s_0(r))} C_{\lbT}(\bfy'',  \Phi_f',\sigma),
\end{equation}
  where
\begin{equation}
C_{\lbT}(\bfy'',  \Phi_f',\sigma) :=  W_{\lbT, f}(e, \Phi'_f,\sigma) \prod_{v|\infty} \archW_{\sigma_v(\lbT)}(y''_{v} , \sigma) 
\end{equation}
and $\Phi'_f =(\eta_t^r)^*\Phi_f  \in  I_t(\bbV(\mathbb{A}_f), \sigma)$.
%\begin{align*}
%	C_T^{\, t}(\bfy, \Phi_f,s)   \ &= \  \prod_{v|\infty} \left(\det Y_v' \right)^{\frac{\iota}2(s - s_0(r))}  \, \archW_{\sigma_v(\lbT)}(y'_{t,v} , \sigma) \cdot W_{\lbT, f}(e, \Phi'_f,\sigma) \\
%	&= \ \prod_{v|\infty} \left(\det Y_v' \right)^{\frac{\iota}2(s - s_0(r))} \cdot C_{\lbT}(\bfy'_t,  \Phi_f',\sigma)
%\end{align*}
%where $\Phi'_f =(\eta_t^r)^*\Phi_f  \ \in \ I_t(\bbV(\mathbb{A}_f), \sigma)$.

Now consider the special point $s=s_0(r)$, and note that 
\begin{equation}
\sigma|_{s=s_0(r)} \ = \ s_0(r) + \frac{r-t}{2} \ = \ s_0(t).
\end{equation}
Since $\lbT$ is non-degenerate, the term $ C_{\lbT}(\bfy'',\Phi_f',\sigma)$ is holomorphic at $s=s_0(r)$, and therefore the same is true for $C_T^{\, t}(\bfy, \Phi_f,s)  $.
% Here are some additional explanations concerning this:
% The point here is to realize that: 1) Eisenstein series $E(g,\Phi,s_0)$ are holomorphic if $\Phi$ is a Siegel-Weil section coming from an anisotropic quadratic space; 2) the map \eta_t^r preserves Siegel-Weil sections. Namely, if $\Phi=\lambda(\varphi)$ with $\varphi \in \mathcal{S}(V(\mathbb{A}^r))$, then (\eta_t^r)^*(\Phi)$ is the Siegel-Weil section attached to the restriction of $\varphi$ to $0_{r-t} \times V(\mathbb{A})^t \simeq V(\mathbb{A})^t$. However, this only takes care of the finite places. At infinity, note that s_0(t) >0$ and hence the scalar $K'_t$-type of weight $l$ is also in the image of $\lambda$ because it is the image of \varphi_KM.
 Therefore,
\begin{equation}
	  \left. \frac{\partial}{\partial s} C_T^t( \bfy,\Phi_f,s) \right|_{s = s_0(r)}  =   \frac{\iota}2 \sum_{v | \infty} \log \det (y'_v) \ \beta(\lbT, \Phi_f' ) \ + \   \left. \frac{\partial}{\partial \sigma} C_{\lbT}(\bfy'', \Phi_f',\sigma) \right|_{\sigma = s_0(t)}.
\end{equation}
Applying \Cref{lem:nonDegKappaBeta} to $C_{\lbT}(...)$, it follows immediately that the difference
\begin{equation}
		  \left. \frac{\partial}{\partial s}C_T^t(\bfy,s, \Phi) \right|_{s = s_0(r)}  -  \frac{\iota}2  \beta(\lbT, \Phi_f' ) \sum_{v | \infty} \log \det y_v' - \kappa(\lbT, \Phi'_f) 
\end{equation}
satisfies the conditions in \eqref{eqn:CondRadialDecay}.

Finally, it remains to consider the case $j>t$ and $s_0(j) = 0$. Since we are assuming $s_0(r) \geq 0$, we must have $s_0(j) = s_0(r) = 0$  and hence $j=r$ as well. On the other hand, if $s_0(r) > 0$, then such a term does not arise and we have completed the proof of the proposition at this stage.

 Otherwise, by \Cref{rmk:BjTopAndBottom}, the corresponding term is 
\begin{equation}
C^{r}_T(\bfy, \Phi_f,s)  =  \prod_{v| \infty}  \det(y_v)^{-\iota m/4}  \ W_{\lbT} \left( e, \left[U_{r,t}(s)  \circ r(g'_{\bfy}) \right]\Phi, \sigma \right)  \ q^{-T} 
\end{equation}
with $ \sigma = s - \frac{r-t}{2}$. Since $s_0(r) = 0$, we have $m = r+1  $ in the orthogonal case, and $m= r$ in the unitary case; it follows  in both cases that
\begin{equation}
\sigma|_{s = s_0(r)=0}  \  = \ - \frac{r-t}2 \ = \ - s_0(t).
\end{equation}
To compute this term, note that the functional equation of the genus $t$ Eisenstein series implies
\begin{equation}
		W_{\lbT}(e, \Phi_t,\sigma) \ =\ W_{\lbT}( e, \calM_t(\sigma) \Phi_t,- \sigma)
\end{equation}
for any $\Phi_t \in I_t(\bbV,\sigma)$. To apply this to the present case, recall that 
\begin{equation}
	\calM_{t}(\sigma) \circ U(s)  \ = \ (\eta_t^j)^* \circ \calM_{j}\left(s + \frac{r-j}2 \right) \circ  (\eta_j^r)^* 
\end{equation}
as in \Cref{lem:BjGlobalWhittakerReln}(iii).  Consider  an archimedean place $v | \infty$, and define a meromorphic function $\ttd(s)$ by the relation
\begin{equation}
\calM_r(s) \Phi_{r,v}^l(s)  \ =  \ \ttd(s) \Phi_{r,v}^l(-s).
\end{equation}
To determine the function $\ttd(s)$, we may evaluate both sides at the identity, and apply \cite[(1.31)]{ShimuraConfluent}; a little algebra, using the fact $s_0(r)=0$, shows that $\ttd(s)$ is given by the formula \eqref{eqn:archIntertwiningCentralPtFactor}. 
%\comm{a bit sketchy...}

 On the other hand, by \Cref{lem:BjGlobalWhittakerReln}(iii) with $j=r$, we find
\begin{equation}
\begin{split}
\left[ \calM_{t,v}(\sigma) \circ U_{r,t,v}(s) \right] \left( r(g'_{y_v}  )   \, \Phi_{r,v}^l(s) \right) \ &= \ \left[ (\eta_t^r)^* \circ \calM_{r,v}(s)  \right] \left( r(g'_{y_v}  )   \, \Phi_{r,v}^l(s) \right) \\
&= \ \ttd(s) \cdot  (\eta_t^{r})^* \left[  r(g'_{y_v}  )  \, \Phi_{r,v}^l(-s)  \right]  \\
&= \ \ttd(s)  \left(\det y'_v \right)^{\frac{\iota}2( -s+\rho_r)}  \left( r(g'_{y''_{v}})\Phi_{t,v}^l \right)(-\sigma).
\end{split}
\end{equation}
Thus, applying the functional equation for $W_{\lbT}(...)$, taking into account our normalizations, and noting that $s_0(r)  = 0$, we have
	\begin{align}  \label{eqn:CTr}
		C_T^{r}(\bfy, \Phi_f,s) \ =\ \ttd(s)^d \cdot  \prod_{v|\infty} \det(y'_v)^{-  \frac{ \iota s}2} \cdot C_{\lbT}(\bfy'',\Phi''_f,-\sigma),
	\end{align}
where $\Phi''_f(-\sigma) =(\eta_t^r)^*(\calM_r(s) \Phi_f(s))$. Therefore
	\begin{equation}
	\begin{aligned}
	\left.	\frac{\partial}{\partial s} C_T^{r}(\bfy, \Phi_f,s) \right|_{s = s_0(r) = 0} \ &= \  \ttd(0)^d  \left[ d\,   \frac{ \ttd'(0)}{\ttd(0)}   -\frac{\iota}2 \sum_{v|\infty} \log \det y'_v  \right] \cdot  C_{\lbT}(\bfy'', \Phi''_f,s_0(t)) \\
	&  \hspace{8em} -   \left. \ttd(0)^d  \,  \frac{\partial}{\partial \sigma } C_{\lbT}(\bfy'', \Phi''_f,\sigma) \right|_{\sigma = s_0(t)}.
	\end{aligned}
	\end{equation}	
Setting 
\begin{equation}
	\beta'  =  - \ttd(0)^d \cdot C_{\lbT}(\bfy'',\Phi''_f,s_0(t)) =   - \ttd(0)^d \cdot \beta(\lbT, \Phi_f'')
\end{equation}
and 
\begin{equation}
	\kappa' =    \ttd(0)^{d}  \left[ d \,  \frac{ \ttd'(0)}{\ttd(0)}\beta(\lbT, \Phi_f'') \  - \  \kappa(\lbT, \Phi_f'') \right]
\end{equation}
it follows, in the same manner as the previous case, that the difference

	\begin{equation}
	\left.	\frac{\partial}{\partial s} C_T^{r}(\lambda \bfy, \Phi_f,s) \right|_{s = s_0(r) = 0}  \  - \ \frac{\iota}{2} \cdot \beta' \sum_{v|\infty} \log \det y_v' - \kappa'
	\end{equation}
satisfies the conditions of \eqref{eqn:CondRadialDecay}. 
Combining this fact with the previous computations for the terms with $t \leq j < r$, we conclude that
	\begin{equation}
	\left.	\frac{\partial}{\partial s} C_T(\lambda \bfy, \Phi_f,s) \right|_{s =0}    - \frac{\iota}{2} \cdot \left[\beta(S, \Phi') + \beta'  \right]  \sum_{v|\infty} \log \det y_v' - \left[\kappa(S,\Phi_f') + \kappa'\right]
	\end{equation}
	satisfies \eqref{eqn:CondRadialDecay}.  
	
	It remains to identify the two terms in square brackets in the preceding display as $\beta(T,\Phi_f)$ and $\kappa(T,\Phi_f)$, respectively. First, consider the value of $C_T(\lambda \bfy, \Phi, s)$ at $s=s_0(r) = 0$; by \cref{eqn:CTt,eqn:CTr} and an argument analogous to the one leading to \eqref{eqn:DerivFCEisJthTerm}, we may write
	\begin{align}
	C_T(\lambda \bfy, \Phi_f, 0) \ &= \ \sum_{j=1}^t \ \mathrm{CT}_{s = 0} \, C_T^j(\lambda \bfy, \Phi_f, s) \notag \\
	&= \ \beta(S,\Phi_f') - \beta' + O(\lambda^{-C})
	\end{align}
	for some $C>0$. On the other hand, the Eisenstein series $E(g, \Phi,s)$ is incoherent, in the sense of \cite{KudlaRallisSiegel1}, and hence vanishes identically at $s=0$ (this result is \cite[Theorem 4.10]{KudlaRallisSiegel1} in the orthgonal case with $m$ even; see \cite[Proposition 6.2]{GanQiuTakeda} and the references therein for more details in the remaining cases). 
	In particular,    $C_T(\bfy, \Phi_f,0)= 0$ and $\beta(S,\Phi_f') = \beta'$, and, comparing with \Cref{def:KappaTilde}, we find
	\begin{equation}
	\beta(S,\Phi'_f) + \beta' = 2 \beta(S,\Phi_f') = \beta(T,\Phi_f')
	\end{equation}
	and 
	\begin{equation}
	\kappa(S,\Phi_f') + \kappa' = \kappa(T,\Phi_f)
	\end{equation}
	as required.

\end{proof}
\end{proposition}

It remains to consider the case $T = \bf 0$.
\begin{proposition} \label{prop:EisGrowthEstimateConstTerm}
	Let 
		\[
		\beta(\mathbf 0, \Phi_f) \ := \ \begin{cases} \Phi_f(e), & \text{if } s_0(r)>0 \\ 2 \, \Phi_f(e)   , & \text{if }s_0(r)  = 0 .\end{cases}
		\]
	and 
		\[
		\kappa(\mathbf 0, \Phi_f) \ := \ \begin{cases} 0, & \text{if } s_0(r) > 0 \\ 
		- d \cdot  \frac{\ttd'(0)}{\ttd(0)} \Phi_f(e) \ - \ \ttd(0)^d  \, W'_{0,f}(e,\Phi_f,0), & \text{if } s_0(r) = 0. \end{cases}
		\]
	Then the difference
	\begin{align*}
	F(\bfy) =	\left.	\frac{d}{d s}	C_{\bf0 }( \bfy,\Phi_f,s) \right|_{s=s_0(r)}  \ & - \ \frac{\iota}2 \sum_{v|\infty} \log \det y_v \cdot \beta(\mathbf 0, \Phi_f) -  \kappa(\mathbf 0, \Phi_f)  
	%		& \ \ \  -  \ttd_n(0;l)^d \cdot \left[ d \cdot  \frac{\ttd'(0;l)}{\ttd(0;l)}W_{0,f}(e,0,\Phi_f) - W'_{0,f}(e, 0, \Phi_f) \right].
	\end{align*}
	satisfies the conditions \eqref{eqn:CondRadialDecay}.
	
	%$\lim_{\lambda \to \infty}  F(\lambda \bfy) = 0$ and, for fixed $\bfy$, there exists a constant $C>0$ such that $\frac{\partial}{\partial \lambda} F(\lambda \bfy) = O(\lambda^{-1-C})$. 
	\begin{proof}
		Using \Cref{lem:GlobalEisFCDecomp} and \eqref{eq:GlobalBjEisT0}, we may write $C_{\mathbf 0} (\bfy,\Phi,s) = \sum_{j=0}^r		C_{\mathbf 0}^j (\bfy,\Phi,s)$, where
		%	\[
		%		C_{\mathbf 0} (\bfy,\Phi,s) \ =\ \sum_{j=0}^r		C_{\mathbf 0}^j (\bfy,\Phi,s)
		%	\]
		\begin{equation}
		C^j_{\mathbf 0}(\bfy, \Phi_f,s)  =   \prod_{v| \infty}\det(y_v)^{-\iota m/4} \cdot \sum_{1\leq i\leq N} \gamma_i(s) \calG^j(\bm \alpha,\tilde{\Phi}_i,s),
		\end{equation}
		for some meromorphic functions $\gamma_i(s)$ and standard sections $\tilde{\Phi}_i \in \widetilde I_{\calP}^j(\bbV, s)$.
		%	As before, we may write $C_{\mathbf 0} (\bfy,\Phi,s) \ =\ \sum_{j=0}^r		C_{\mathbf 0}^j (\bfy,\Phi,s)$
		%%	\[
		%%		C_{\mathbf 0} (\bfy,\Phi,s) \ =\ \sum_{j=0}^r		C_{\mathbf 0}^j (\bfy,\Phi,s)
		%%	\]
		%	with
		%	\[
		%		 	C^j_{\mathbf 0}(\bfy,\Phi,s) \ = \  \prod_{v| \infty}\det(y_v)^{-l/2} \sum_{ \sfa \in \mathcal P_{r,j} \backslash \GL_{r}(\bfk)}   B^{\, j}_{\mathbf 0} \left( \underline m (\sfa)     \cdot g'_{\bfy},\Phi,s\right) .
		%	\]
		%	For $1 \leq j\leq r$, \eqref{eq:GlobalBjEisT0} implies that
		%	\[
		%			C^j_{\mathbf 0}(\bfy, s, \Phi) \ = \  \prod_{v| \infty}\det(y_v)^{-l/2} \cdot \sum_{1\leq i\leq N} \gamma_i(s) \calG^j(\bm \alpha,\tilde{\Phi}_i,s)
		%	\]
		%	for some meromorphic functions $\gamma_i(s)$ and standard sections $\tilde{\Phi}_i \in \widetilde I_{\calP}^j(\bbV, s)$; 
		Applying the transformation formula \eqref{eqn:GLnCentralCharacter}, with $t=0$, we have 
		\begin{equation}
		C^j_{\mathbf 0}(\lambda \bfy,\Phi_f,s)  =  \lambda^{\frac{d}2( r - 2j)(s-s_0(r))  - d j s_0(j) } \cdot C^j_{\mathbf 0}(\bfy,\Phi_f,s) .
		\end{equation}
		Thus  $\frac{\partial}{\partial s}	C_{\bf0 }^j( \lambda \bfy, \Phi_f,s)|_{s=s_0(r)}$ is a finite sum of terms of the form
		\begin{equation}
		\log(\lambda)^k \cdot \lambda^{-d \, j \, s_0(j)} \cdot a(\bfy)
		\end{equation}
		for some integer $k \geq 0$ and Laurent coefficient $a(\bfy)$ of $C^j_{\mathbf 0}(\bfy,\Phi_f,s) $.
		
		In particular, if $j\geq 1$ and $s_0(j) \neq 0$, then
		\begin{equation}
		\lim_{\lambda \to \infty}\frac{\partial}{\partial s}		C_{\bf0 }^j( \lambda \bfy, \Phi_f,s)|_{s=s_0(r)} = 0  \qquad \text{ and } \qquad  \frac{\partial}{\partial \lambda} \frac{\partial}{\partial s}	C_{\bf0 }^j( \lambda \bfy,\Phi_f,s)|_{s=s_0(r)} = O(\lambda^{-1-C}).
		\end{equation}
		Next, consider the case $j=0$, so that, by \eqref{eq:GlobalBjEisT0_bis}, we have
		
		%	\comm{\hrulefill}
		
		\begin{align}
		C_{\mathbf 0}^0(\bfy, \Phi_f,s) \ =& \  \prod_{v| \infty}\det(y_v)^{-\iota m/4}   B^{0}_{\mathbf 0} \left(g'_{\bfy},\Phi_f,s \right)  \nonumber \\
		=& \  \prod_{v| \infty}\det(y_v)^{-\iota m/4}  \cdot \Phi(g'_{\bfy}, s)  \\
		=& \  \prod_{v|\infty} \det(y_v)^{\frac{\iota}2(s-s_0(r))} \Phi_f(e). \nonumber
		\end{align}
		Note here that $\Phi_f(e,s) = \Phi_f(e)$ is independent of $s$. Therefore
		\begin{equation}
		\frac{\partial}{\partial s}		C_{\mathbf 0 }^0(  \bfy,\Phi_f,s)|_{s=s_0(r)} \ = \ \frac{\iota}2 \sum_{v|\infty} \log \det y_v \cdot \Phi_f(e).
		\end{equation}

		Finally, consider the case $s_0(j) = 0$, so that $j = r $ with $s_0(r)=0$ and 
			\begin{equation}
			C_{\mathbf 0}^{r}(\bfy,\Phi_f,s) = \prod_{v| \infty}\det(y_v)^{-\iota m/4} \cdot [\calM(s) \Phi(s)](g'_{\bfy}).
			\end{equation}
		We have
		\begin{equation}
		\begin{split}
		[\calM(s) \Phi(s)](g'_{\bfy}) &=  \prod_{v|\infty} [\calM_v(s) \Phi_{r,v}^l(s)](g'_{y_v}) \cdot[ \calM_{f}\Phi_f(s)](e) \\
		&= \ttd(s)^d\prod_{v| \infty}\Phi_{r,v}^l(g'_{y_v}, -s) \cdot W_{0,f}(e, \Phi_f,-s) \\
		&= \ttd(s)^d \prod_{v| \infty} \det(y_v)^{\frac{\iota}{2}( -s + \rho_r)} W_{0,f}(e, \Phi_f,-s)
		\end{split}
		\end{equation}
Moreover, applying \cite[Lemma 6.3]{GanQiuTakeda} to the incoherent section $\Phi$ gives 
			\begin{equation}
			\calM(0) \Phi(0) (g'_{\bfy}) \  = \   -  \prod_{v | \infty}\det (y_v)^{\iota m/4} \cdot \Phi(e) \ = \ - \prod_{v| \infty} \det (y_v)^{\iota m/4} \cdot \Phi_f(e)
			\end{equation}
			and hence, after a little algebra we obtain
			\begin{equation}
			\begin{split}
			\left.		\frac{d }{d s}C_{\mathbf 0}^{r}(\bfy, \Phi_f,s) \right|_{s=s_0(r) = 0} \ = \ \frac{\iota}2 & \sum_{v|\infty} \det \log y_v \, \Phi_f(e) \\ 
			& - d \cdot  \frac{\ttd'(0)}{\ttd(0)} \Phi_f(e)  - \ttd(0)^d W'_{0,f}(e,\Phi_f,0). 
			\end{split}
			\end{equation}
		The proposition follows immediately from these observations.
	\end{proof}
\end{proposition}

\subsection{Archimedean height pairings} \label{subsection:Archimedean height pairings} 
In this section, we prove our main theorem relating the integrals of the Green forms $\lie g(T,\bfy, \varphi)$ constructed in \Cref{section:Green_currents_global} to Eisenstein series. 
	
	We continue to assume $\bbV$ is anisotropic and $\rank(\vbE)= 1$.
Let $\calX $ denote the canonical model of the Shimura variety $X_\bbV = X_{\bbV,K}$, for a fixed open compact subgroup $K \subset \bfH(\bbA_f)$, and recall that there is a decomposition $\calX(\bbC) = \coprod X_{\bbV[k]}$ in terms of the nearby spaces $\bbV[k]$ for $k=1,\dots, d$, see \Cref{subsection:orthogonal_Shim_vars,subsection:unitary_Shim_vars}. For each $k$, let
\begin{equation}
\mathrm{Vol}(X_{\bbV[k]},\Omega) = \int_{[X_{\bbV[k]}]} \Omega^p
\end{equation}
where $\Omega =\Omega_{\calE} = c_1(\calE, \nabla)^*$ is the first Chern form of the positive line bundle $\vbE$ on $X_{\bbV[k]}$.
Then $\mathrm{Vol}(X_{\bbV[k]},\Omega)=\deg_{\vbE}(X_{\bbV[k]})$ is a positive rational number (a positive integer if $K$ is neat). As remarked in \ref{subsubsection:taut_bundle_definitions}, the line bundle $\vbE$
is a rational multiple of the canonical bundle; in particular we have $\mathrm{Vol}(X_{\bbV[k]},\Omega)=\mathrm{Vol}(X_{\bbV},\Omega)$ for every $k$. 

\begin{theorem} \label{thm:GlobalGreenIntegral}
	Assume that $\bbV$ is anisotropic and $\rank(\vbE)=1$. Let $r \geq 1$ and
	\begin{equation*}
	s_0(r) = \begin{cases} (m-r-1)/2, & \text{orthogonal case,} \\ (m-r)/2, & \text{unitary case}, \end{cases}
	\end{equation*}
	and assume that $s_0(r) \geq 0$. Given $\varphi_f \in \mathcal{S}(\bbV(\mathbb{A}_f)^r)^K$, let $\Phi_f(\cdot, s)$ be the unique standard section of $I_r(\bbV(\mathbb{A}_f),s)$ such that $\Phi_f(\cdot, s_0) = \lambda(\varphi_f)$ and set
	\begin{equation} \label{eqn:GlobalMainThmPhiDef}
	\Phi = \otimes_{v|\infty} \Phi_{r}^l \otimes \Phi_f \in I_r(\bbV(\mathbb{A}),s),
	\end{equation}
	where $\Phi^l_r$ is the standard weight $l$ section given by \eqref{eq:def_scalar_weight_both_cases}, \eqref{eq:def_scalar_weight_vector_case_1} and \eqref{eq:def_scalar_weight_vector_case_2}. For $\bftau = \bfx + i \bfy \in \bbH_r^d$, let
	\begin{equation*}
	\Eis(\bftau,\Phi_f,s)=  \left( \det y_1 \cdots y_d \right)^{-\iota m/4} \Eis(g'_{\bftau}, \Phi_f,s) \ = \ \sum_{T } \Eis_T(\bftau,\Phi_f,s)
	\end{equation*}
	be the corresponding Eisenstein series of scalar weight $l$ defined in \eqref{eq:def_Ecal_Eis_series}, and write
	$\Eis_T'(\bftau,\Phi_f,s)=\tfrac{d}{ds}\Eis_T(\bftau,\Phi_f,s)$. Then for any $T \in \mathrm{Sym}_r(F)$ (resp. $T \in \Herm_r(E)$) we have 
	%\begin{equation*}
	%\frac{(-1)^r \iota \kappa_0}{2 \mathrm{Vol}(X_{\bbV},\Omega)} \int_{\mathcal{X}(\mathbb{C})}  \lie g(T,\bfy,\varphi_f) \wedge \Omega^{p+1-r} \,  q^T\ =  \  \Eis_T'( i \bfy,\Phi_f,s_0(r) ) - \kappa(T,\Phi_f) \, q^T,
	%\end{equation*}
	\begin{equation*}
	\frac{(-1)^r  \kappa_0}{2 \mathrm{Vol}(X_{\bbV},\Omega)} \int_{[\mathcal{X}(\mathbb{C})]}  \lie g(T,\bfy,\varphi_f) \wedge \Omega^{p+1-r} \,  q^T\ =  \  \Eis_T'( \bftau,\Phi_f,s_0(r) ) - \kappa(T,\Phi_f) \, q^T,
	\end{equation*}
	%where $\Omega$ denotes the first Chern form of the hermitian line bundle $\mathcal{L}$, the constant 
	where $\kappa(T,\Phi_f)$ is explicit and defined in \Cref{def:KappaTilde}, and $\kappa_0=2$ if $s_0(r)=0$ and $\kappa_0=1$ otherwise.
\end{theorem}
\begin{proof}
	Fix an archimedean embedding $\sigma_k$, and a component $X_{\bbV[k]} \subset \calX(\bbC)$. Recall that the restriction of the Green form $\lie g(T,\bfy,\varphi_f)$  to $X_{\bbV[k]}$ is given by
	\begin{align}
	\label{eqn:green form correction redux}
	\lie g(T,\bfy, \varphi_f)|_{X_{\bbV[k]}} \ &= \ \archGreen(T,\bfy,\varphi_f)_{\sigma_k} -    \log \left( \frac{\det' \sigma_k(T) \cdot \det y_k}{\det'  \left(\sigma_k(T)y_k\right)}  \right)   \delta_{Z(T,\varphi_f)_{\sigma_k}} \wedge \Omega_{\calE^{\vee}}^{r-\mathrm{rk}(T)-1} 
	\end{align}
	when $T$ is positive semi-definite, and $	\lie g(T,\bfy, \varphi_f)|_{X_{\bbV[k]}}  =  \archGreen(T,\bfy,\varphi_f)_{\sigma_k}$ otherwise (see \Cref{def:global_g_T_Y_phi_case_1,def:global_g_T_Y_phi_case_2}); here
	\begin{equation}
	\archGreen(T,\bfy, \varphi_f) \  =  \ \mathop{\mathrm{CT}}_{\rho=0}  \archGreen(T,\bfy, \varphi_f; \rho)
	\end{equation}
	as in \Cref{prop:GlobalGreenOrthogonalDegenerateCase,prop:GlobalGreenUnitary}.
	
	Consider the contribution of $\lie g^\mathtt{o}(T,\bfy,\varphi_f)_{\sigma_k}$ to the integral over $\mathcal{X}(\mathbb{C})$;
	% a fixed archimedean place $\sigma_j$ ($1 \leq j \leq d)$. 
	by definition, it equals $\mathrm{CT}_{\rho=0} \ I(\rho,\sigma_k)$, where
	\begin{equation} \label{eq:proof_main_thm_def_I_rho} 
	I(\rho,\sigma_k) = \frac{(-1)^r  \kappa_0}{2\mathrm{Vol}(X_{\bbV[k]},\Omega)} \int_{[X_{\bbV[k]}]} \archGreen(T,\bfy,\varphi_f;\rho)_{\sigma_k} \wedge \Omega^{p+1-r}.
	\end{equation}
	Let us compute $I(\rho,\sigma_1)$. Define the archimedean Schwartz function  
	\begin{equation}
	\phi=\tilde{\nu} \otimes \varphi_+ \otimes \cdots \otimes \varphi_+  \ \in \  \mathcal{S}(\bbV^r \otimes_{F} \mathbb{R}) = \bigotimes_{ 1 \leq i \leq d} \mathcal{S}(\bbV^r_{\sigma_i}),
	\end{equation}
	with $\tilde{\nu} \in \mathcal{S}(\bbV^r_{\sigma_1})$ as in \eqref{eq:def_tilde_nu_i} and where $\varphi_+$ denotes the Gaussian for the positive definite spaces $\bbV^r_{\sigma_k}$ ($k>1$), given by
	\begin{equation}
	\varphi_+(\bfv) = e^{-\pi \sum_i Q(v_i,v_i)}, 
	\end{equation} 
	for $\bfv = (v_1, \dots, v_r) \in \bbV_{\sigma_k}^r.$
	
	Consider the theta series 
	\begin{equation}
	\Theta(g'_{\bftau},h;\phi \otimes \varphi_f) \ = \ \sum_T \Theta_T(g'_{\bftau},h;\phi \otimes \varphi_f)
	\end{equation}
	defined in \eqref{eq:def_adelic_theta_series}, and write
	\begin{equation}
	C_{\Theta,T} (\bfy,h; \phi \otimes \varphi_f) \ := \ \left( \det y_1 \cdots y_d \right)^{-\iota m/4}  \cdot \Theta_T(g'_{\bftau}, h; \phi \otimes \varphi_f) \ q^{-T}.
	\end{equation}
	For $z=hz_0 \in \mathbb{D}^+$ ($h \in \mathrm{U}(\bbV_{\sigma_1})$) and $\mathrm{Re}(\rho) \gg 0$, we have
	\begin{equation}
	\archGreen(T,\bfy,\varphi_f;\rho)_{\sigma_1}(z) \wedge \Omega^{p+1-r}(z)  \\
	= \int_1^\infty C_{\Theta,T}(t \bfy, h;\phi \otimes \varphi_f)\frac{dt}{t^{\rho+1}} \wedge \Omega^{p}(z),
	\end{equation}
	where the estimates in the proof of \Cref{prop:DegCurrentDef} justify the interchange of sum and integral.
	%where we write $c(\bfy)= (\det y_1\cdots y_d)^{-\iota m/4}$.
	Using the Siegel-Weil formula (\Cref{thm:Siegel Weil redux})  to compute the integral over $X_{\bbV}$, we find 
	\begin{equation}
	I(\rho,\sigma_1) = \frac{(-1)^r  }{2} \int_1^\infty C_{E,T}\left(t\bfy, \lambda (\phi \otimes \varphi_f), s_0(r) \right) \,  \frac{dt}{t^{\rho+1}}; %\\
	\end{equation}
	here the relevant Eisenstein series is
	\begin{equation}
	E(g', \lambda(\phi \otimes \varphi_f),s ) = \sum_{T}E_T(g', \lambda(\phi \otimes \varphi_f), s) 
	\end{equation}
	where $\lambda \colon S(\bbV(\bbA)^r) \to I_r(\bbV,s_0(r))$ is as in \Cref{subsubsection:Siegel Weil redux}, and we have written
	\begin{equation}
	C_{E,T}(\bfy, \lambda(\phi \otimes \varphi_f)),s) \ := \ \left(\det y_1 \cdots y_d \right)^{-\iota m/4}  \cdot E_T(g'_{\bftau}, \lambda(\phi \otimes \varphi_f), s) \ q^{-T}.
	\end{equation}
	
	Our next step is to relate this expression to the coefficient
	\begin{equation}
	C_T(\bfy, \Phi_f, s) \ :=  E_T( \bftau, \Phi_f, s) \ q^{-T}
	\end{equation}
	of the scalar weight $l$ Eisenstein series in the statement. Comparing archimedean components, for $i \geq 2$ we have $\lambda(\phi_{\sigma_i})=\lambda(\varphi_+)=\Phi^l(s_0)$, while $\lambda(\phi_{\sigma_1}) = \lambda(\tilde \nu) =(-1)^{r-1}\tilde{\Phi}(s_0)$ as in  \eqref{eq:relation_nu_tilde{Phi}}. Thus, writing $C'_T(\bfy,\Phi_f,s_0) = \tfrac{d}{ds}C_T(\bfy,\Phi_f,s)|_{s=s_0}$ and $\bfy = (y_1, \dots, y_d)$,
	%with $C_T$ given by \eqref{eq:def_Eis_coeff_C_T}, 
	the argument in the proof of \Cref{lemma:lowering_W_T_derivative} shows that
	\begin{equation}
	C_{E,T}(t\bfy, \lambda(\phi \otimes \varphi_f), s_0 ) \ 
	% = \ (-1)^{r-1}  \Phi_f, s)  C_{E,T}(t\bfy,\tilde \Phi \otimes_{i>1} \Phi^l \otimes \Phi_f, s_0 )
	= \ (-1)^{r-1} \frac{2}{\iota} \cdot \left. t\frac{d}{dt} C_T'((t y_1,t'y_2,\ldots,t'y_r)),\Phi_f,s_0)\right|_{t=t'}.
	\end{equation}

	Adding the contributions from all archimedean places, we conclude that
	\begin{equation} 
	\begin{split}
	\frac{(-1)^r  \kappa_0}{2 \mathrm{Vol}(X_K,\Omega)}&\int_{[\mathcal{X}(\mathbb{C})]} \lie g^\mathtt{o}(T,\bfy,\varphi_f) \wedge \Omega^{p+1-r} \\
	&= - \mathrm{CT}_{\rho = 0} \int_1^\infty \frac{d}{dt} \left(C_T'(t \bfy,\Phi_f,s_0)  \right) \frac{dt}{t^{\rho}}.
	\end{split}
	\end{equation}
	This integral can be evaluated using the following lemma, whose proof is straightforward.
	\begin{lemma}
		Let $f\colon \mathbb{R}_{>0} \to \mathbb{C}$ be a smooth function such that for some constants $a, b \in \mathbb{C}$, the function $F(t)=f(t)-a-b \log t$ satisfies $\lim_{t \to \infty} F(t)=0$ and $F'(t) = O(t^{-1-C})$
		as $t \to \infty$, for some positive constant $C$. Then
			\begin{equation*}
			-\mathrm{CT}_{\rho = 0} \int_1^\infty  f'(t)  \, \frac{dt}{t^\rho} = f(1)-a.
			\end{equation*}
	\end{lemma}
	By \Cref{prop:EisGrowthEstimate}, the function $C_T'(t\bfy,\Phi_f,s_0)$, regarded as a function of $t$, satisfies the hypotheses of the lemma with 
	\begin{equation}
	a = \kappa(T,\Phi_f) + \frac{\iota}{2}\beta(T,\Phi_f) \sum_{v|\infty} \log \left( \frac{\det' \sigma_v(T) \cdot \det y_v}{ \det' \sigma_v(T)y_v }  \right).
	\end{equation}

		To finish the proof, it suffices to show that  the contributions from the second terms in \eqref{eqn:green form correction redux} for the various components $X_{\bbV[k]} \subset \calX(\bbC)$ sum to
		\begin{equation} \label{eq:last_identity_proof_main_theorem} 
		\begin{split}
		\iota \sum_{k=1}^d\log \left( \frac{\det' \sigma_k(T) \cdot \det y_k}{\det'  \left(\sigma_k(T)y_k\right)}  \right)   \frac{(-1)^{\rank(T) } \kappa_0}{ 2 \mathrm{Vol}(X_{\bbV},\Omega)}&\int_{[X_{\bbV[k]}]} \delta_{Z(T,\varphi_f)_{\sigma_k}} \wedge \Omega^{p-\mathrm{rk}(T)} \\ 
		& \stackrel{?}{=}  \frac{\iota}{2} \sum_k \log \left( \frac{\det' \sigma_k(T) \cdot \det y_k}{\det'  \left(\sigma_k(T)y_k\right)}  \right)   \beta(T,\Phi_f),
		\end{split}
		\end{equation} 
		where we used the identity $\Omega_{\calE^{\vee}} =- \Omega_{\calE} = - \Omega$.  Note that on account of the logarithms, both sides vanish if $T$ is non-degenerate. When $T$ is degenerate, the claim is essentially contained in \cite{KudlaMillson3}; we outline the argument in \Cref{lemma:degree of Z is beta} below.
\end{proof}

When the matrix $T$ is non-degenerate, we recover the identities \eqref{eq:intro_thm_non_deg_T_not_pos_def} and \eqref{eq:intro_thm_non_deg_T_pos_def} by using the explicit expression for $\kappa(T,\Phi_f)$ given by \Cref{lem:nonDegKappaBeta}.

Combining Theorems \ref{theorem:star_product_formula_global} and \ref{thm:GlobalGreenIntegral}, we obtain the following corollary generalizing the main result of \cite{KudlaCD} (for $\mathrm{U}(p,1)$ this corollary is due to Liu \cite{LiuYifengArithThetaI}; recently Bruinier and Yang \cite{BruinierYangArithmeticDegrees} have treated the $\mathrm{O}(p,2)$ case).

\begin{corollary}\label{corollary:archimedean_height_pairing_derivative_at_0}
Assume that $T_1$ and $T_2$ are non-degenerate and that $Z(T_1,\varphi_{f,1})$ and $Z(T_2,\varphi_{f,2})$ intersect transversely. Assume also that $r_1+r_2=p+1$. Then
\begin{equation*}
\begin{split}
\frac{(-1)^{p+1}}{\mathrm{Vol}(X_{\bbV},\Omega)} & \int_{[\mathcal{X}(\mathbb{C})]} \lie g(T_1,\bfy_1,\varphi_{f,1}) * \lie g(T_2,\bfy_2,\varphi_{f,2}) \, q^T \\ 
&= \sum_{T= \left(\begin{smallmatrix} T_1 & * \\ * & T_2 \end{smallmatrix}\right)} \Eis'_T\left(\left(\begin{smallmatrix} \bftau_1 &  \\  & \bftau_2 \end{smallmatrix}\right),\lambda(\varphi_{f,1} \otimes \varphi_{f,2}),0\right)_{\infty},
\end{split}
\end{equation*}
\end{corollary}

It remains to establish the following lemma, which is an application of the results of \cite{KudlaMillson3}.
\begin{lemma} \label{lemma:degree of Z is beta} Suppose $T = \left( \sm{0 & \\ & S } \right)$ is degenerate. Then for any $k = 1, \dots, d$, 
	\begin{equation*}
	\frac{(-1)^{\rank(T) } \kappa_0}{  \mathrm{Vol}(X_{\bbV},\Omega)} \int_{[X_{\bbV[k]}]} \delta_{Z(T,\varphi_f)_{\sigma_k}} \wedge \Omega^{p-\mathrm{rk}(T)}  = \beta(T,\Phi_f),
	\end{equation*}
	with  notation as in   \Cref{thm:GlobalGreenIntegral}.
	
	\begin{proof}
		Suppose first that $t = \rank(T) > 0$, so that $S$ is non-degenerate of rank $t$. 
		As both sides of the desired identity are linear in $\varphi_f$, we may  also assume that $\varphi_f$ is of the form
		\begin{equation}
		\varphi_f = \varphi_f' \otimes \varphi_f '' \in S(\bbV(\bbA_f)^{r-t}) \otimes S(\bbV(\bbA_f)^t).
		\end{equation}
		By construction, we have $Z(T,\varphi_f) = \varphi_f'(0) Z(S, \varphi_f'')$. On the other hand,   let $\Phi(\varphi'_f)$ and $\Phi(\varphi''_f)$ denote the standard sections corresponding to $\varphi_f'$ and $\varphi_f''$, respectively. Then \Cref{def:KappaTilde} and a direct computation using explicit formulas for the Weil representation (see e.g.\ \cite[Proposition 2.2.5]{SweetThesis}) give  
		\begin{align}
		\beta(T,\Phi_f)  = \kappa_0 \cdot \beta \left( S, \eta^* \Phi_f \right)  & = \kappa_0 \cdot C_S(e, \eta^* \Phi_f, s_0(t)) ) \notag \\
		&= \kappa_0 \cdot \varphi_f'(0) \cdot C_S\left(e, \Phi(\varphi_f''), s_0(t) \right) \notag\\
		&= \kappa_0 \cdot \varphi_f'(0) \cdot \beta(S, \Phi(\varphi_f'')). \label{eqn:beta of eta}
		\end{align}
		where $\eta^* = (\eta_t^r)^*$ as in \eqref{eqn:etaStarDef}.
		
		To prove the lemma, say for $k=1$, let
		\begin{equation}
		\Theta_{\mathrm{KM}}(\bftau'',\varphi''_f)_{\sigma_1} \ = \ \sum_S \omega(S, \bfy'',  \varphi''_f)_{\sigma_1} \, q^S
		\end{equation} be the Kudla-Millson theta series of genus $t$, as defined in \eqref{eq:def_Theta_KM}; here $\bftau'' = \bfx'' + i \bfy'' \in \bbH_t^d$. It is shown in \cite{KudlaMillson3} that $\omega(S, \bfy'',  \varphi''_f)_{\sigma_1}$ is a closed form on $X_{\bbV}$ whose cohomology class is $[Z(S,\varphi''_f)_{\sigma_1}]$. In particular, we have
		\begin{equation} \label{eqn:degree of Z degenerate}
		\int_{[X_{\bbV}]}  \delta_{Z(T,\varphi_f)_{\sigma_1}} \wedge \Omega^{p-t}  =  \varphi_f'(0)\int_{[X_{\bbV}]}  \delta_{Z(S,\varphi''_f)_{\sigma_1}} \wedge \Omega^{p-t}  =   \varphi_f'(0) \int_{[X_{\bbV}]} \omega(S,\bfy'', \varphi_f)_{\sigma_1} \wedge  \Omega^{p-t};
		\end{equation}
		this also follows from \Cref{prop:GlobalGreenOrthogonalDegenerateCase,prop:GlobalGreenUnitary}.
		To compute the latter integral, define an archimedean Schwartz function 
		\begin{equation}
		\varphi_{\infty} \ := \ \widetilde \varphi \otimes \varphi_+ \otimes \cdots \otimes \varphi_+ \ \in \ \bigotimes S(\bbV_{\sigma_k}^t)
		\end{equation}
		where $\widetilde \varphi$ is the Schwartz function on $\bbV_{\sigma_1}^t$ defined by \eqref{eqn:def of phi tilde}, and $\varphi_+$ is the standard Gaussian. Note that 
		\begin{equation}
		\lambda(\varphi_{\infty}) = (-1)^t \otimes_{v|\infty} \Phi_t^l(s_0),
		\end{equation} 
		as in the proof of \Cref{lemma:Eis is hol}, and  
		\begin{equation}
		\omega(S, \bfy'', \Phi''_f)_{\sigma_1}(z) \wedge \Omega^{p-t}(z) \ = \ C_{\Theta, S}(\bfy'',h, \varphi_{\infty} \otimes \varphi''_f)  \, \Omega^p(z), \quad z=hz_0,
		\end{equation}
		where $C_{\Theta,S}(\bfy'', h ,\varphi_{\infty} \otimes \varphi''_f)$ is the coefficient of $q^S$ of the  theta series attached to $\varphi_{\infty} \otimes \varphi''_f$. Applying the Siegel-Weil formula (\Cref{thm:Siegel Weil redux}) again, noting that $s_0(t) > 0$ here, we conclude 
		\begin{equation}
		\frac{(-1)^t}{\mathrm{Vol}(X_{\bbV}, \Omega)}\int_{[X_{\bbV}]} \omega(S,\bfy'', \varphi_f'')_{\sigma_1} \wedge \Omega^{p-t} \ =\ C_{S}\left(\bfy'', \Phi(\varphi_f), s_0(t) \right) \ = \  \beta(S,\Phi''_f).
		\end{equation}
		Comparing this with \eqref{eqn:beta of eta} and \eqref{eqn:degree of Z degenerate} proves the lemma for $k=1$, and the proof for all other values of $k$ follows in exactly the same way. 
		
		Finally, when $T=0_r$, the left hand side is $\kappa_0 \varphi_f(0) = \kappa_0 \Phi_f(e)$, which is by definition equal to $\beta(0,\Phi_f)$, cf.\ \Cref{prop:EisGrowthEstimateConstTerm}.
	\end{proof}
\end{lemma}

\subsection{Classes in arithmetic Chow groups} \label{sec:arith Chow groups}

In this section, we describe how the currents $\lie g(T, \bfy , \varphi_f)$ arise as the archimedean parts of classes in arithmetic Chow groups lifting the cycles $Z(T,\varphi_f)$. As we will ultimately have little to say about arithmetic aspects of the theory, we shall  gloss over many serious difficulties  regarding suitable integral models, bad reduction, etc. A key point is  a natural geometric context for the analogue \eqref{eqn:general global current Green} of Green's equation in the degenerate case.

We continue to assume that $\bbV$ is anisotropic, but  drop the assumption $\rank(\vbE) = 1$. We also assume the level structure $K \subset \bfH(\bbA_f)$ is neat.

Let $\bfk = F$ (resp.\ $\bfk = E$) in the orthogonal (resp.\ unitary) case, so that $\calX$  is proper over $\mathrm{Spec} \, \bfk $. Suppose $\lie X$ is a regular integral model, proper and flat over $\mathrm{Spec}(\calO_{\bfk})$, with an extension of the tautological bundle that we continue to denote $\vbE$. Finally, for each $T$ and $\varphi_f$, let $\mathcal Z(T, \varphi_f)$ denote a  cycle on $\lie X$ extending $Z(T,\varphi_f)$ on the generic fibre
%\footnote{For concreteness one may take the Zariski closure of $Z(T,\varphi_f)$ in $\lie X$, though for some arithmetic applications and particular choices of level structures, moduli interpretations are available, e.g.\ \cite{KudlaRapoportUnitaryII,KudlaRapoportYang,aghmp-compo}.}  
whose codimension in $\lie X$ is equal to the codimension of $Z(T,\varphi_f)$ in $\mathcal X = \lie X_{\bfk}$.

Let $\ChowHat{ \bullet}_{\bbC}(\lie X) = \bigoplus_r \ChowHat{r}_{\bbC}(\lie X)$ denote the Gillet-Soul\'e arithmetic Chow ring (with $\bbC$ coefficients), as in \cite{SouleBook}. Classes in $\ChowHat{r}_{\bbC}(\lie X)$ are represented by pairs $(Z, g_Z)$, where $Z$ is codimension $r$ cycle on $\mathscr X$ (with $\bbC$-coefficients) and 
%\[ g_Z = \{ g_{Z,\rho} \ | \rho \colon \bfk \to \bbC\} 
%\]
$g_Z$ is an $(r-1,r-1)$ current on $\calX(\bbC) = \coprod_{\sigma\colon \bfk \to \bbC} \calX_{\sigma}(\bbC)$ that is invariant under complex conjugation, and satisfies Green's equation
\[
\ddc g_{Z} \ + \ \delta_{Z(\bbC)} \ = [\omega_Z]
\]
for some smooth differential form $\omega_{Z}$ on $\calX(\bbC)$; we may also view $g_Z = \{ g_{Z, \sigma} \}_{\sigma}$ as a collection consisting of a current $g_{Z,\sigma}$ on $\calX_{\sigma}(\bbC)$ for each complex embedding $\sigma$. 

%If $(Z_1, \lie g_1)$ and $(Z_2, \lie g_2)$ are two cycles such that $Z_1$ and $Z_2$ intersect properly on the generic fibre, then the intersection product is given by the formula \cite[Section III.2]{SouleBook}
%\begin{equation} \label{eqn:Chow Hat Intersection}
%(Z_1, \lie g_1) \, \bigcdot \, (Z_2, \lie g_2) \ = \ (Z_1\cdot Z_2, \, \lie g_1 * \lie g_2).
%\end{equation}

When $T$ is non-degenerate,  \Cref{prop:GlobalGreenOrthogonalDegenerateCase} (and the discussion around \eqref{eq:global_current_equation}) or \Cref{prop:GlobalGreenUnitary}  gives rise to a class
\begin{equation}
\left( \calZ(T, \varphi_f), \, \lie g(T,\bfy, \varphi_f ) \right) \ \in \ \ChowHat{rq'}_{\bbC}(\lie X),
\end{equation}
where $q' = \rank(\vbE)$. 

Now consider a pair $(T,\varphi_f)$ with $T \in \Sym_r(F)$ (resp.\ $T \in \Herm_r(E)$) a degenerate matrix, and set $t= \mathrm{rank}(T)$. Let
\[
\widehat c_{q'}(\vbE^{\vee})  \ \in \ \ChowHat{q'}_{\bbC}(\lie X)
\]
denote the arithmetic Chern class attached to  $\vbE^{\vee}$, as in \cite[Section IV]{SouleBook}. The class of $\widehat c_{q'}(\vbE^{\vee})^{r-t}$ may be represented by a pair $(\calZ_0, g_{0})$ such that the generic fibre of $\calZ_0$ intersects properly with $Z(T,\varphi_f)$, and where the current $g_0$ satisfies the equation
\[
\ddc g_0 \ + \ \delta_{\calZ_0(\bbC)} \ = \ \Omega_{\vbE^{\vee}}^{r-t}
\]
and is of logarithmic type, see \cite[\S II.2]{SouleBook}. On the other hand, consider the set of currents $\lie g$ of degree $(r-1,r-1)$ satisfying the analogue
\begin{equation} \label{eqn:GenGreenCurrentEqn}
\ddc \lie g \ +  \ \delta_{Z(T, \varphi_f)(\bbC)} \wedge \Omega_{\vbE^{\vee}}^{r-t} \ = \ [\omega]
\end{equation}
of Green's equation, for some smooth form $\omega$; a short computation reveals that the map
\[
\lie g \ \mapsto \ \lie g  \, + \,  g_0 \wedge \delta_{Z(T, \varphi_f)(\bbC)} 
\]
defines a bijection between the solutions of \eqref{eqn:GenGreenCurrentEqn} and Green currents for the intersection $\calZ(T, \varphi_f)  \, \bigcdot \, \calZ_0$. Therefore, applying \Cref{prop:GlobalGreenOrthogonalDegenerateCase,prop:GlobalGreenUnitary}, we obtain a class
\[
\widehat{\lie Z}(T, \bfy, \varphi_f) \ := \ \left( \calZ(T,\varphi_f) \cdot \calZ_0 , \  \lie g(T, \bfy, \varphi_f) + g_0 \wedge \delta_{Z(T, \varphi_f)(\bbC)} \right) \ \in \ \ChowHat{rq'}_{\bbC}(\lie  X).
\]
To see that this construction is independent of the choice of $(\calZ_0, g_0)$ representing $\widehat c_{q'}(\vbE^{\vee})^{r-t}$, choose any Green current $\lie g'$ for $\calZ(T,\varphi_f)$, and note that
\begin{align} \notag
\widehat{\lie Z} (T, \bfy, \varphi_f) \  &= \ (\calZ(T,\varphi_f), \lie g')  \,  \bigcdot \,  (\calZ_0, g_0) \ + \ ( 0 , \, \lie g(T, \bfy, \varphi_f) - \lie g' \wedge \Omega_{\vbE^{\vee}}^{r-t}) \\
&= \ (\calZ(T,\varphi_f), \lie g')  \,  \bigcdot \,  \widehat c_{q'}(\vbE^{\vee})^{r-t}  \ + \ ( 0 , \, \lie g(T, \bfy, \varphi_f) - \lie g' \wedge \Omega_{\vbE^{\vee}}^{r-t}) 
\label{eqn:Chow Hat shifted Z}
\end{align}

To preserve uniformity of notation, set $\widehat{\lie Z}(T,\bfy, \varphi_f) = (\calZ(T,\varphi_f), \lie g(T, \bfy, \varphi_f))$ when $T$ is non-degenerate. For any (possibly degenerate) $T$, restricting the cycle $\widehat{\lie Z}(T, \bfy, \varphi_f)$ to the generic fibre (and forgetting the current) yields a cycle that coincides with the construction in \cite{KudlaOrthogonal}.

As an example of our construction, take $T=\mathbf 0_r$; applying the computation \eqref{eqn:Green Zero example} for $\lie g(0, \bfy, \varphi_f)$ gives the concrete expression
\begin{equation} \label{eqn: arith Z(0)  explicit}
\widehat{\lie Z}(\mathbf 0_r, \bfy, \varphi_f) \ =\  \varphi_f(0) \cdot \left( \widehat{c}_{q'}\left(\vbE^{\vee} \right)^{r} \ - \ \left( 0, \left\{   \log \det y_k  \cdot  c_{q'-1}(\vbE^{\vee}, \nabla)^*\wedge\Omega_{\vbE^{\vee}}^{r-1}  \right\}_{k=1,\dots, d} \right) \right).
\end{equation}

\subsection{Kudla's arithmetic height conjecture}  \label{sec:Arithmetic Height conjecture}

We recast our results in the setting of Kudla's conjectures on the arithmetic heights of the cycles $\widehat{\lie Z}(T, \bfy, \varphi_f) $ considered in  the previous section. 

Assume that $\bbV$ is anisotropic and $\rank(\vbE) = 1$. Let
% ; as our aim is to highlight the role of the Green currents in this context,  we continue to gloss over the arithmetic-geometric difficulties in the construction of integral models for these cycles.
% Let
\begin{equation}
\widehat \omega \ := \ \widehat c_1(\vbE^{\vee}) \ \in \ \ChowHat{1}_{\bbC}(\mathscr X)
\end{equation}
denote the arithmetic class attached to $\vbE^{\vee}$, or more precisely, to its integral extension as in the preceding section,
%  the cotautological bundle on the  integral model $\mathscr X$, 
and consider the generating series
\begin{equation}
\phi_{\widehat \omega} (\tau) \ := \ \sum_{T} \   \widehat\deg \left(    \widehat{\lie Z}(T, \bfy, \varphi_f) \, \bigcdot \,  \widehat\omega^{p+1-r} \right) \, q^T
\end{equation}
where $\widehat\deg \colon \ChowHat{p+1}_{\bbC}(\mathscr X) \to \bbC$ is the arithmetic degree map \cite[Section III]{SouleBook}.  A rough form of Kudla's conjectural programme, as outlined in e.g.\  \cite{KudlaMSRI}, suggests that $\phi_{\widehat \omega}(\tau)$ is, up to a normalization, equal to the special derivative of an Eisenstein series.

More precisely, let  $ \Phi_f \in I_r(\bbV(\mathbb{A}_f),s)$ denote the standard section of parallel scalar weight $l$ determined by $\varphi$ as in \eqref{eqn:GlobalMainThmPhiDef}, and consider the parallel weight $l$ Eisenstein series
\begin{equation}
\normEis(\bftau, \Phi_f,s) = A_r(s) \Eis(\bftau, \Phi_f,s ) =: \sum_T \scrC_T(\bfy, \Phi_f,s) q^T,
\end{equation}
for an appropriate normalizing factor $A_r(s)$; then one should have an identity
\begin{equation} \label{eqn:KudlaGlobalConjecture}
\widehat\deg \left(    \widehat{\lie Z}(T, \bfy, \varphi_f) \cdot \widehat\omega^{p+1-r} \right)  \ 
\stackrel{?}{\sim} \  \scrC_T'(\bfy,\Phi_f,s_0(r)) 
\end{equation}
up to correction terms involving rational multiples of $\log p$ with $p$ in a fixed, finite set of primes that might depend on   $\varphi$, the level structure $K$, and the choices of integral models. As these correction terms are expected to arise as contributions from (components of) cycles at primes of bad reduction, it is reasonable to assume that they are independent of the parameter $\bfy$.

Let
\begin{equation}
h(\calZ(T,\varphi_f)) :=  \widehat{\deg} \, \left(\widehat{\omega}^{p+1-\rk(T)}  \, \big| \, \calZ(T,\varphi_f)\right)
\end{equation} 
denote the Bost-Gillet-Soul\'e height of $\calZ(T,\varphi_f)$ along $\widehat{\omega}$, as defined in \cite[Proposition 2.3.1]{BostGilletSoule}. Using \cite[(2.3.3)]{BostGilletSoule} and \eqref{eqn:Chow Hat shifted Z} above, a brief computation gives
\begin{align*}
\widehat\deg \left(    \widehat {\lie Z}(T, \bfy, \varphi_f) \cdot \widehat\omega^{p+1-r} \right) \ =& \ h( \calZ(T,\varphi_f)  )  \ + \  \frac{1}{2} \int_{\calX(\bbC)} \lie g(T,\bfy, \varphi_f)\wedge \Omega_{\vbE^{\vee}}^{p+1-r}  \\
=& \ h( \calZ(T,\varphi_f)  )  \ + \  \frac{(-1)^{p+1-r}}{2} \int_{\calX(\bbC)} \lie g(T,\bfy, \varphi)\wedge \Omega_{\vbE}^{p+1-r}.
\end{align*}
%\comm{More details: choose any Green current $\lie{g}'$ for $\calZ(T, \varphi)$, then
%	\begin{align*}
%	h( \calZ(T,\varphi)) \ &= \ \widehat{\deg} \left[  \widehat{\omega}^{p+1 -\rk T} \, \bigcdot \,   (\calZ(T,\varphi), \lie g') \ -  \ (0, \lie g' \wedge \Omega^{p+1 - \rk T})    \right]\\
%	&= \ \widehat{\deg} \left[  \widehat{\omega}^{p+1 -r}  \, \bigcdot \, \left( \widehat{\lie Z}(T, y, \varphi)  - \left(0, \lie g(T, y, \varphi) - \lie g' \wedge \Omega^{r - \rk(T)}  \right) \right)  - (0, \lie g' \wedge \Omega^{p+1 - \rk T})   \right]
%	\end{align*}
%	which gives the formula above -- should we include this calculation?
%}

By  \Cref{thm:GlobalGreenIntegral}, the conjecture \eqref{eqn:KudlaGlobalConjecture} is equivalent to the statement
\begin{align}   
\notag	h(\calZ(T,\varphi_f)) \stackrel{?}{\sim} & \left( A_r(s_0(r)) - \frac{(-1)^{p+1} \, \mathrm{vol}(X_{\bbV})} { \kappa_0} \right)  C'_T(\bfy, \Phi_f,s_0(r))  \\
& + \ A'_r(s_0(r)) \cdot C_T(\bfy,\Phi_f,s_0(r)) \ +\frac{(-1)^{p+1} \, \mathrm{vol}(X_{\bbV})} { \kappa_0} \cdot \kappa(T, \Phi_f)  \label{eqn:FaltingsHeightConjectureInitial}
\end{align}
for all $T$ and $\varphi_f$, where $\mathrm{vol} (X_{\bbV}) = \mathrm{vol} (X_{\bbV}, \Omega_{\vbE}^p)$.

Note that only the values $A_r(s_0(r))$ and $A'_r(s_0(r))$ appear in this expression. To pin these values down further, suppose for the moment that $s_0(r)>0$, and take  $\bfy = \lambda \cdot \Id$, say; then  \Cref{prop:EisGrowthEstimate} implies that 
\begin{equation}
C'_T( \lambda \cdot \Id,\Phi_f,s_0(r)) \ = \ \kappa(T, \Phi_f) + \frac{\iota d \cdot(r - {\rk T})}{2} \cdot \log \lambda \cdot \beta(T,\Phi_f) + F(\lambda) 
\end{equation}
for some function $F$ satisfying $\lim_{\lambda \to \infty} F(\lambda ) = 0$. Similarly, $C_T(\lambda \cdot Id,\Phi_f, s_0(r))  = \beta(T, \Phi_f) $.

Thus, choosing $T$ and $\varphi_f$ such that $\rank(T) < r$ and $\beta(T, \Phi_f) \neq 0$, and noting that $h( \calZ(T, \varphi_f))$ is evidently independent of $\bfy$,   a necessary condition for  \eqref{eqn:FaltingsHeightConjectureInitial} to hold is
\begin{equation} \label{eqn:A_r value}
A_r(s_0(r)) \ = \ \frac{(-1)^{p+1} \, \mathrm{vol}(X_{\bbV})} { \kappa_0} .
\end{equation}
Now for the derivative $A'_r(s_0(r))$, assume  that  $\varphi(0) = 1$, and consider the matrix $T=0$. Then
\begin{equation}
h(\calZ(0,\varphi_f)) \ \sim \ \widehat{\deg}_{\mathscr X} \, \widehat{\omega}^{p+1}  \ =: \ \widehat{\mathrm{vol}}_{\, \widehat{\omega}}(\mathscr X)
\end{equation}
On the other hand, by \Cref{prop:EisGrowthEstimateConstTerm},
\begin{equation}
\kappa(\mathbf 0, \Phi_f) = 0, 
\end{equation}
and $C_{\mathbf 0}(\bfy, s_0, \Phi) = \varphi(0)=1$. Therefore  \eqref{eqn:FaltingsHeightConjectureInitial} for $T = \mathbf 0$ and $s_0(r)>0$ gives the further necessary condition
\begin{equation}
A'_r(s_0(r)) \sim \widehat{\mathrm{vol}}_{\, \widehat{\omega}} (\mathscr X).
\end{equation}
In other words, taking $A_r(s)$ such that $A_r(s_0(r)) =(-1)^{p+1} \mathrm{vol}(X_{\bbV})  \kappa_0^{-1} $ and $A_r'(s_0(r)) = \widehat{\mathrm{vol}}_{\widehat{\omega}}(\mathscr X)$, the conjecture \eqref{eqn:KudlaGlobalConjecture}, for $s_0(r)>0$, takes the form
\begin{equation} \label{eqn:KudlaHeightConjecture}
h(\calZ(T,\varphi_f)) \ \stackrel{?}{\sim} \ \widehat{\mathrm{vol}}_{\, \widehat{\omega}}(\mathscr X) \cdot \beta(T, \Phi_f) +  \frac{(-1)^{p+1} \, \mathrm{vol}(X_{\bbV})} { \kappa_0} \kappa(T, \Phi_f).
\end{equation}
When $s_0(r) = 0$, note that $C_T(\bfy, \Phi_f, 0) = 0$ and the conjecture \eqref{eqn:FaltingsHeightConjectureInitial} becomes
\begin{equation}
h(\calZ(T,\varphi_f)) \ \stackrel{?}{\sim} \ \frac{(-1)^{p+1} \, \mathrm{vol}(X_{\bbV})} { \kappa_0} \cdot \kappa(T, \Phi_f)  
\end{equation}

The point here is that the right hand sides of these relations involve  explict constants depending only on $T$ and $\varphi_f$. This conjecture has been verified in the case of full level Shimura curves for $r=1$ \cite{KRYFaltingsHeights}, for $r=2$ \cite{KudlaRapoportYang}, and for Hilbert modular surfaces with $r=1$ in \cite{BruinierBurgosKuhn} (these results also include contributions from places of bad reduction). A slightly weaker version of this conjecture (i.e.\ away from an explicit set of primes determined by $T$ and $\varphi$) was proved for general orthogonal Shimura varieties over $\bbQ$ by H\"ormann \cite{HoermannThesis}. Several cases involving cycles of top arithmetic codimension supported at finite primes were also established by Kudla and Rapoport, see e.g.\  \cite{KudlaRapoportUnitaryII}, or the discussion in \cite[\S II]{KudlaMSRI} for more details.

\bibliographystyle{plainurl}
\bibliography{refs2}
%\bibliography{superconn, refs}

\author{\noindent \small \textsc{Department of Mathematics, University of Toronto, 40 St. George Street,
  Toronto, Canada} \\ 
  {\it E-mail address:} \texttt{lgarcia@math.toronto.edu}}
    
\author{\noindent \small \textsc{Department of Mathematics, University of Manitoba, 420 Machray Hall, Winnipeg, Canada} \\
  {\it E-mail address:} \texttt{siddarth.sankaran@umanitoba.ca}}

\end{document}